\documentclass[10pt,twoside]{book}
\usepackage{etex}

\usepackage[paperwidth=170mm, paperheight=240mm, hmarginratio=10:5, textwidth=125mm, top=25mm, bottom=25mm]{geometry}
\usepackage{changepage}

\usepackage[english]{babel}
\usepackage{amsmath, amsthm, amsfonts, mathrsfs, amsfonts}
\usepackage{mathtools}
\usepackage{xcolor}
\usepackage{bm}
\usepackage{amscd}
\usepackage{amssymb}
\usepackage{imakeidx}
\usepackage{enumerate} 
\usepackage{fancyhdr}

\usepackage{array}
\newcolumntype{P}[1]{>{\centering\arraybackslash}p{#1}}
\usepackage{multirow}

\usepackage{slashbox} 

\usepackage{textcomp}

\usepackage{hyperref}
\usepackage{tikz} 
\usepackage{diagrams}
\diagramstyle[labelstyle=\scriptstyle]

\usepackage[T1]{fontenc}
\usepackage{ae,aecompl}
\usepackage{xtab}
\usepackage{tabularx} 
\usepackage{calc} 
\usepackage{graphicx}
\usepackage{lmodern}
\rmfamily

\usepackage{mathtools,booktabs}
\usepackage{amscd}

\usepackage{titlesec}

\usepackage{etoolbox}

\makeatletter
\patchcmd{\ttlh@hang}{\parindent\z@}{\parindent\z@\leavevmode}{}{}
\patchcmd{\ttlh@hang}{\noindent}{}{}{}
\makeatother

\titlespacing*{\chapter}{0pt}{60pt}{60pt}
\titleformat{\chapter}[display]{\scshape\large\bfseries}{\chaptertitlename\ \thechapter}{5pt}{\normalfont\LARGE\bfseries}
[\vspace{1pt}\titlerule]

\renewcommand{\thesection}{\arabic{section}-}

\titleformat{\section}[hang]{\normalfont\Large\bfseries}{\thesection}{8pt}{}

\numberwithin{section}{chapter}
\numberwithin{equation}{chapter}

\newcommand{\F}{\mathbb{F}}

\newcommand{\Z}{\ensuremath{\mathbb{Z}}}\vspace{12pt}

\newcommand{\A}{\ensuremath{\mathbb{A}}}

\newcommand{\cor}[1]{\mathcal{#1}}

\newcommand{\pgp}{Let $p$ be a prime number and let $G$ be a finite $p$-group.}

\DeclareMathOperator{\Aut}{Aut} 
\DeclareMathOperator{\GL}{GL} 
\DeclareMathOperator{\hc}{H} 
\DeclareMathOperator{\ZG}{Z} 
\DeclareMathOperator{\Res}{Res} 
\DeclareMathOperator{\End}{End} 
\DeclareMathOperator{\Hom}{Hom} 
\DeclareMathOperator{\nor}{N} 
\DeclareMathOperator{\id}{id} 
\DeclareMathOperator{\Sym}{Sym} 
\DeclareMathOperator{\Cyc}{C} 
\DeclareMathOperator{\SL}{SL} 
\DeclareMathOperator{\cl}{cl} 
\DeclareMathOperator{\Int}{Int} 
\DeclareMathOperator{\Inn}{Inn} 
\DeclareMathOperator{\inte}{int} 
\DeclareMathOperator{\dpt}{dpt} 
\DeclareMathOperator{\rk}{rk} 
\DeclareMathOperator{\yo}{MC(3)}
\DeclareMathOperator{\wt}{wt}
\DeclareMathOperator{\Br}{Br}
\DeclareMathOperator{\Sn}{S}
\DeclareMathOperator{\cc}{c}

\newcommand{\indexx}[1]{\index{#1}\emph{#1}}
\newcommand{\gen}[1]{\langle{#1}\rangle}

\newcommand{\graffe}[1]{\{{#1}\}}

\newtheorem{definition}{Definition}
\newtheorem{lemma}[definition]{Lemma}
\newtheorem{theorem}[definition]{Theorem}
\newtheorem{proposition}[definition]{Proposition}

\newtheorem{corollary}[definition]{Corollary}

\newtheorem*{lemma*}{Lemma}
\newtheorem*{theorem*}{Theorem}
\newtheorem*{proposition*}{Proposition}
\newtheorem*{remark*}{Remark}
\newtheorem*{corollary*}{Corollary}
\newtheorem*{notation*}{Notation}

\newtheorem{theoremA}{Theorem}

\renewcommand{\chaptermark}[1]{\uppercase{\markboth{#1}{}}}

\makeindex

\begin{document}


\begin{titlepage}

\pagestyle{empty}

~
\vspace{1.2cm}

\begin{center}
{\Large\bf INTENSE AUTOMORPHISMS OF FINITE GROUPS}
\vspace {3cm}

{\large Proefschrift

ter verkrijging van

de graad van Doctor aan de Universiteit Leiden

op gezag van Rector Magnificus prof.\ mr.\ C.J.J.M.\ Stolker,

volgens besluit van het College voor Promoties

te verdedigen op dinsdag 5 september 2017

klokke 10:00 uur
\vspace {2cm}

 door
\vspace{2cm}

 {\bf Mima Stanojkovski}

 geboren te Sarajevo, Joegoslavi\"{e}, in 1989
 }

\end{center}

\newpage

{\Large\noindent Samenstelling van de promotiecommissie:}
\vspace{30pt}
\begin{center}
\begin{tabular}{m{6cm} m{5.5cm}} 

{\textbf{Promotor:}} &   \\
 & \\
Prof.\ dr.\ Hendrik W.\ Lenstra &  Universiteit Leiden \\
 & \\
 & \\
\textbf{Overige leden:} & \\
 & \\
 Dr.\ Jon Gonz\'{a}lez S\'{a}nchez  & Euskal Herriko Unibertsitatea \\
 Dr.\ Ellen Henke  & University of Aberdeen  \\
 Prof.\ dr.\ Andrea Lucchini & Universit\`{a} degli Studi di Padova \\
 Prof.\ dr.\ Bart de Smit & Universiteit Leiden \\
 Prof.\ dr.\ Aad van der Vaart & Universiteit Leiden \\
\end{tabular}
\end{center}

\vfill

\clearpage

\newpage

~

\vspace{150pt}

\hfill Mojim najdra\v{z}ima, mami, tati, i dragom gusaru.





\newpage

~

\vfill

\noindent
The background image of the cover of this thesis has been drawn by the author. The vertices of the illustrated graph represent equivalence classes of subgroups of $\yo$, as defined in Section \ref{section yo}, with respect to the equivalence relation $H\sim K$ if and only if, for each $j\in\graffe{1,2,3,4}$, the $j$-th widths, as defined in Section \ref{section jumps}, of $H$ and $K$ in $\yo$ are the same. An edge is drawn between two vertices if there are representatives $H$ and $K$ of the given vertices such that $H$ is contained in $K$ with index $3$ or vice versa. The color of each vertex is determined by the number of conjugates in $\yo$ of any representative of the equivalence class associated to the vertex.
Subgroups corresponding to white, blue, red, and yellow vertices have respectively $1$, $3$, $9$, and $27$ conjugates in $\yo$.

\end{titlepage}

\newpage
\frontmatter

\pagestyle{fancy}
\fancyhf{}
\fancyhead[LE,RO]{\leftmark}
\fancyfoot[CE,CO]{}
\fancyfoot[LE,RO]{\thepage}

\renewcommand{\contentsname}{Contents}
\tableofcontents



\renewcommand{\chaptermark}[1]{\uppercase{\markboth{#1}{}}}

\chapter*{List of Symbols}\label{notation}
\chaptermark{List of symbols}
\addcontentsline{toc}{chapter}{List of symbols}

\vspace{-20pt}

{\Large\textbf{General}}
\vspace{8pt}

\begin{itemize}
 \item[] $p$, a prime number
 \item[] $\Z$, ring of integers
 \item[] $\Z_{\geq x}$, set of integers that are at least $x$
 \item[] $\Z_{> x}$, set of integers that are larger than $x$
 \item[] $\Z_p$, ring of $p$-adic integers
 \item[] $\F_q$, finite field of $q$ elements
 \item[] $R^*$, group of units of the ring $R$
 \item[] $|X|$, the cardinality of the set $X$
 \item[] $\gen{X}$, the subgroup generated by the set $X$ \\ 
 		 and we write $\gen{a,b,c,\ldots}$ instead of $\gen{\graffe{a,b,c,\ldots}}$
 \item[] $\id_X$, the identity map on the set $X$
 \item[] $\alpha_{|X}$, the map $\alpha$ restricted to the set $X$
 \item[] $\cl(X)$, the closure of the set $X$
 \item[] $\bigsqcup$, disjoint union
 \item[] $\otimes=\otimes_\Z$
 \item[] $\bigwedge=\bigwedge_\Z$
 \item[] $\GL_n(k)$, the general linear group of degree $n$ over $k$
 \item[] $\yo$, definition in Chapter \ref{chapter 3gps}
 \item[] $\SL_2^\triangle(\Z_p)$, definition in Section \ref{section sl2} 
 \item[] $\Sn(\Delta_p)$, definition in Section \ref{section sl1}
\end{itemize}
\vspace{12pt}

\noindent
{\Large\textbf{For any group $G$}}
\vspace{8pt}

\begin{itemize}
 \item[] $[x,y]=xyx^{-1}y^{-1}$, for any $x,y\in G$
 \item[] $\ZG(G)$, the centre of $G$
 \item[] $\Phi(G)$, the Frattini subgroup of $G$, see Section \ref{section p-gps}
 \item[] $(G_i)_{i\geq 1}$, the lower central series of $G$, see Definitions \ref{LCS} and \ref{lcs profinite def}
 \item[] $G^{n}=\gen{x^{n}\ :\ x\in G}$
 \item[] $\mu_n(G)=\gen{x\in G\ :\ x^{n}=1}$
 \item[] $|G:H|$, the index of $H$ in $G$
 \item[] $\Cyc_G(X)=\bigcap_{x\in X}G_x$, where $G_x$ is the stabilizer of $x$
 \item[] $\nor_G(H)$, the normalizer of $H$ in $G$
 \item[] $\End(G)$, the set of endomorphisms of $G$
 \item[] $\Aut(G)$, the automorphism group of $G$
 \item[] $\Inn(G)$, the inner automorphism group of $G$
 \item[] $\Int(G)$, the intense automorphism group of $G$, see Chapter \ref{CH formulation}
 \item[] $\rk(G)$, the rank of $G$, see Sections \ref{rank} and \ref{section background}
\end{itemize}

\vspace{8pt}
\noindent
{\Large\textbf{For a finite $p$-group $G$}}
\vspace{8pt}

\begin{itemize}
 \item[] $\rho$, the map $x\mapsto x^{p}$ 
 \item[] $\dpt_G(x)$, the depth of $x$ in $G$, see Section \ref{section jumps}
 \item[] $\wt_H^G(j)$, the $j$-th width of $H$ in $G$, see Section \ref{section jumps}
 \item[] $\wt_G(j)=\wt_G^G(j)$ 
 \item[] $\chi_G$, the intense character of $G$, see Section \ref{section main}
 \item[] $\inte(G)$, the intensity of $G$, see Section \ref{section main}
\end{itemize}

\vspace{8pt}
\noindent
{\Large\textbf{Exceptions}}
\vspace{8pt}
\begin{itemize}
 \item[] $(F_i)_{i\geq 1}$, the $p$-central series of the free group $F$, in Sections \ref{section construction}, \ref{section strutture free}, and \ref{section unique with automorphism}
\end{itemize}

\renewcommand{\chaptermark}[1]{\uppercase{\markboth{#1}{}}}



\chapter*{Introduction}
\chaptermark{Introduction}
\addcontentsline{toc}{chapter}{Introduction}

Let $G$ be a group and let $\Aut(G)$ denote its group of automorphisms.
An automorphism $\alpha\in\Aut(G)$ is \emph{intense} if it sends each subgroup of $G$ to a conjugate, i.e., for every subgroup $H$ of $G$ there exists $g\in G$ such that $\alpha(H)=gHg^{-1}$. The collection of intense automorphisms is a normal subgroup of $\Aut(G)$, which is denoted by $\Int(G)$. 
\vspace{10pt}\\
\noindent
Such automorphisms come to light in the field of Galois cohomology, as we will see at the end of this introductory section. Additionally, they give rise to a very rich theory. We study the case in which $G$ is a finite $p$-group and show that, if $\Int(G)$ is not itself a $p$-group, then the structure of $G$ is almost completely determined by its ``class''.
\vspace{10pt} \\
\noindent
If $G$ is a finite abelian group, then the inversion map $x\mapsto x^{-1}$ is an intense automorphism of $G$ and therefore, unless the exponent of $G$ divides $2$, the order of $\Int(G)$ is even. 
It follows, for example, that if $G$ is non-trivial abelian of odd order, then $G$ always has a non-trivial intense automorphism of order coprime to its order.  
In Chapter \ref{CH formulation} we prove the following result for groups of prime power order.

\begin{theoremA}\label{intro 1}
Let $p$ be a prime number and let $G$ be a finite $p$-group. Then $\Int(G)$ is isomorphic to a semidirect product $S_G\rtimes C_G$, where $S_G$ is a Sylow 
$p$-subgroup of $\Int(G)$ and $C_G$ is a subgroup of the unit group $\F_p^*$ of the finite field $\F_p$. Moreover, if $G$ is non-trivial abelian, then $C_G=\F_p^*$.
\end{theoremA}

\noindent
Theorem \ref{intro 1} is the same as Theorem \ref{theorem abelian} and is proven in Section \ref{section abelian}. If $p$ is an odd prime number, then Theorem \ref{intro 1}
guarantees the existence of infinitely many $p$-groups, up to isomorphism, whose group of intense automorphisms is not itself a $p$-group. Moreover, it is also clear from Theorem \ref{intro 1} that the order of the intense automorphism group of a $2$-group can never have prime divisors other than $2$. 
We define the \emph{intensity} of a finite $p$-group $G$ to be the order of $C_G$ and we denote it by $\inte(G)$.
The main goal of this thesis is to classify all pairs $(p,G)$ such that $p$ is a prime number and $G$ is a finite $p$-group of intensity greater than $1$. 
Theorem \ref{intro 1} classifies all such pairs $(p,G)$ for which $G$ is abelian\ldots but what happens in general?
\vspace{10pt} \\
\noindent
We proceed by separating into cases based on ``how non-abelian" a group is. 
We define the \emph{lower central series} $(G_i)_{i\geq 1}$ of a group $G$ by
\[ G_1 = G \ \ \ \text{and} \ \ \
G_{i+1}=[G,G_i]=\gen{xyx^{-1}y^{-1}\ :\ x\in G, y\in G_i}
\]
and we define the \emph{(nilpotency) class} of $G$ to be $$\cl(G)=\#\graffe{k\in\Z_{\geq 1} :  G_k\neq 1}.$$ 
In other words, the class of a group $G$ is the number of non-trivial elements of the lower central series.  The only group of class $0$ is the trivial group and the groups of class $1$ are the non-trivial abelian groups. 
It is a classical result that, for any finite $p$-group, the lower central series stabilizes at $\graffe{1}$ and so the class is finite.
\vspace{10pt} \\ 
\noindent
In Chapter \ref{CH class 2} we look at finite $p$-groups of class $2$ -- the first non-abelian case we treat -- and prove the following result. 

\begin{theoremA}\label{intro 3}
Let $p$ be a prime number and let $G$ be a finite $p$-group of class~$2$. Then the following are equivalent.
\begin{itemize}
 \item[$1$.] One has $\inte(G)> 1$.
 \item[$2$.] One has $\inte(G)=p-1$ and $p$ is odd.
 \item[$3$.] The group $G$ is extraspecial of exponent $p$.
\end{itemize}
\end{theoremA}

\noindent
Theorem \ref{intro 3} is the same as Theorem \ref{theorem class2 complete} and is proven in Section \ref{section class 2 extraspecial}.
As we explain in Chapter \ref{CH class 2}, \emph{extraspecial groups} of exponent $p$ are exactly those of the form $(\F_p^{2n+1},\ast)$, where $\ast$ is a twist of the usual $+$ by an inner product on $\F_p^n$. 
Thanks to their pleasant shape, it is not a surprise that they carry intense automorphisms of order coprime to $p$. Moreover, they provide, for each odd prime $p$, an infinite class of examples of $p$-groups of class $2$ and intensity different from $1$. 
\vspace{10pt}\\
\noindent
Passing to class at least $3$, things drastically change: in Chapter \ref{chapter class 3}, we prove the following very restrictive result.

\begin{theoremA}\label{intro 3.5}
Let $p$ be a prime number and let $G$ be a finite $p$-group of class at least $3$. Then the following hold.
\begin{itemize}
 \item[$1$.] One has $\inte(G)\leq 2$. 
 \item[$2$.] If $\inte(G)=2$, then $p$ is odd and $|G:G_2|=p^2$.
\end{itemize}
\end{theoremA}

\noindent
Theorem \ref{intro 3.5} is a reformulation of Theorem \ref{theorem class at least 3}, which is proven in Section \ref{section gps class 3}. Moreover, Theorem \ref{intro 3.5} tells us that, for class greater than $2$, a $p$-group $G$ always has intensity $1$ or $2$; in the latter case, if $p$ is odd, then 
the order of the abelianization of $G$ is ``small''.
\vspace{10pt}\\
\noindent
Starting from class $3$, we want to understand the structure of the groups from Theorem \ref{intro 3.5}($2$). To this end, let $p$ be an odd prime number and let $G$ be a finite $p$-group of class $3$ with $|G:G_2|=p^2$. 
In Section \ref{section gps class 3}, we prove that $G/G_3$ is extraspecial of exponent $p$ and that the order of $G$ is $p^4$ or $p^5$. Moreover, if we write $w_i=\log_p|G_i:G_{i+1}|$, then either $(w_1,w_2,w_3)=(2,1,1)$ or $(w_1,w_2,w_3)=(2,1,2)$. As a consequence, for the given prime number $p$, there are, up to isomorphism, only finitely many possibilities for $G$ (for a sharp bound see for example \cite{p5}) and so, contrarily to what happens for class $1$ and $2$, there are only finitely many isomorphism classes of finite $p$-groups of class $3$ and intensity greater than $1$.
The fortunate outcome of our investigation in class $3$ is the following.

\begin{theoremA}\label{intro 4}
Let $p$ be an odd prime number and let $G$ be a finite $p$-group of class $3$. Then the following are equivalent.
\begin{itemize}
 \item[$1$.] One has $\inte(G)=2$.
 \item[$2$.] One has $|G:G_2|=p^2$.
\end{itemize} 
\end{theoremA}

\noindent
The last theorem is a simplification of Theorem \ref{theorem intensity class 3}, whose proof is given in Section \ref{section construction}.  
Thanks to Theorem \ref{intro 4}, we now know that the only condition, given an odd prime number $p$, for a finite $p$-group of class $3$ to have intensity $2$ is just that of having an abelianization of order $p^2$.  
The most urgent problem at this point is that of constructing examples of $p$-groups of class greater than $3$ and intensity $2$: those will serve as a model for further investigation.
\vspace{10pt} \\
\noindent
\textbf{Example.} Let $p>3$ be a prime number and let $\Z_p$ denote the ring of $p$-adic integers.
Let $t$ be a quadratic non-residue modulo $p$ and denote by $\Delta_p$ the quaternion algebra 
$\Delta_p=\Z_p+\Z_p\mathrm{i}+\Z_p\mathrm{j}+\Z_p\mathrm{ij}$ with defining relations
$\mathrm{i}^2=t$, $\mathrm{j}^2=p$, and $\mathrm{ji}=-\mathrm{ij}$. The algebra $\Delta_p$ is equipped with a \emph{standard involution}, which is given by 
$$x=a+b\mathrm{i}+c\mathrm{j}+d\mathrm{ij} \ \mapsto\ 
\overline{x}=a-b\mathrm{i}-c\mathrm{j}-d\mathrm{ij}$$
and which is an anti-ring-automorphism of $\Delta_p$.
Moreover, $\mathfrak{m}=\Delta_p\mathrm{j}$ is the unique ($2$-sided/left/right) maximal ideal of $\Delta_p$ and the residue field $\Delta_p/\mathfrak{m}$, as well as every quotient $\mathfrak{m}^k/\mathfrak{m}^{k+1}$, has cardinality $p^2$. 
Via the natural isomorphisms of groups $(1+\mathfrak{m}^k)/(1+\mathfrak{m}^{k+1})\rightarrow \mathfrak{m}^k/\mathfrak{m}^{k+1}$, the multiplicative group $1+\mathfrak{m}$ is then seen to be a pro-$p$-subgroup of $\Delta_p^*$.
We define 
$\Sn(\Delta_p)$ to be the subgroup of $1+\mathfrak{m}$ consisting of those elements $x$ satisfying $\overline{x}=x^{-1}$. Being closed in $1+\mathfrak{m}$, the group $\Sn(\Delta_p)$ is itself a pro-$p$-subgroup of $\Delta_p^*$ and, if 
$(\Sn(\Delta_p)_i)_{i\geq 1}$ denotes the lower central series of $\Sn(\Delta_p)$, then 
$$(\log_p|\Sn(\Delta_p)_i:\Sn(\Delta_p)_{i+1}|)_{i\geq 1}=(2,1,2,1,2,1,\ldots).$$
We prove in Section \ref{section sl1} that each non-trivial discrete quotient of $\Sn(\Delta_p)$ has intensity greater than $1$. 
\vspace{10pt} \\
\noindent
Because of the last example, we know that, whenever $p$ is a prime larger than $3$ and $c$ is a positive integer, then there always exists a finite $p$-group of class $c$ and intensity greater than $1$. We cannot however use the same strategy to build examples of high class $3$-groups of intensity $2$. As a matter of fact, even though the group $\Sn(\Delta_p)$ can be defined also for $p=3$, the image of the $3$-torsion of $\Sn(\Delta_3)$ in $\Sn(\Delta_3)/\Sn(\Delta_3)_2$ is non-trivial. The next result, which is obtained by combining Theorem \ref{theorem dimensions} and Lemma \ref{corollary bijection rho}($1$), explains why this is a problem.  

\begin{theoremA}\label{intro 4.5}
Let $p$ be an odd prime number and let $G$ be a finite $p$-group. Let $(G_i)_{i\geq 1}$ denote the lower central series of $G$ and  write $w_i=\log_p|G_i:G_{i+1}|$.
Assume that the class of $G$ is at least $4$ and that $\inte(G)=2$. Then the following conditions are satisfied.
\begin{itemize}
 \item[$1$.] One has $(w_1,w_2,w_3,w_4)=(2,1,2,1)$.
 \item[$2$.] The map $x\mapsto x^p$ induces a bijection 
 			 $\overline{\rho}:G/G_2\rightarrow G_3/G_4$.
\end{itemize}
\end{theoremA}

\noindent
Relying on results coming from Section \ref{section regular}, one can prove that, whenever $p>3$, the map $\overline{\rho}$ from Theorem \ref{intro 4.5} is a group isomorphism, while in the case of $3$-groups it never is: because of this structural difference, we separate the two cases. 
\vspace{10pt} \\
\noindent
We define a \emph{$\kappa$-group} to be a finite $3$-group $G$ such that 
$|G:G_2|=9$ and such that cubing induces a bijection $\kappa:G/G_2\rightarrow G_3/G_4$. In particular, $\kappa$ coincides with $\overline{\rho}$ from Theorem \ref{intro 4.5}($2$). In Chapter \ref{chapter 3gps}, we prove several structural results about $\kappa$-groups: we show, for example, that in class $3$ there is, up to isomorphism, a unique $\kappa$-group and that the minimal extensions of that group to class $4$ (which then have order $729$) have an elementary abelian commutator subgroup. The just-mentioned results are presented in the form of Theorems \ref{theorem unique kappa} and \ref{theorem kappa G2}.
Our investigation of $\kappa$-groups leads to the construction of the following example.
\vspace{10pt}\\
\noindent
\textbf{Example.} Let $R=\F_3[\epsilon]$ be of cardinality $9$, with $\epsilon^2=0$.
Denote by $\Delta$ the quaternion algebra 
$\Delta=R+R\mathrm{i}+R\mathrm{j}+R\mathrm{ij}$ with defining relations
$\mathrm{i}^2=\mathrm{j}^2=\epsilon$ and $\mathrm{ji}=-\mathrm{ij}$. Let moreover the \emph{standard involution} on $\Delta$ be the $R$-linear map that is given by 
$(\overline{1},\overline{\mathrm{i}},\overline{\mathrm{j}},\overline{\mathrm{ij}})=
(1,-\mathrm{i},-\mathrm{j},-\mathrm{ij})$. Then, for each $x,y\in\Delta$, one has $\overline{xy}=\overline{y}\, \overline{x}$.
We write $\mathfrak{m}=\Delta\mathrm{i}+\Delta\mathrm{j}$, which is a nilpotent maximal $2$-sided ideal of $\Delta$ with $\Delta/\mathfrak{m}$ isomorphic to $\F_3$. We define additionally
$\yo$ to be the subgroup of the multiplicative group $1+\mathfrak{m}$ consisting of those elements $x$ satisfying $\overline{x}=x^{-1}$. 
The group $\yo$ has order $729$, class $4$, and it is a $\kappa$-group.
Moreover, $\inte(\yo)=2$.
\vspace{10pt}\\
\noindent
In Chapter \ref{chapter 3gps} we prove the following result, which is a simplified version of Theorem \ref{theorem 3-groups}. 

\begin{theoremA}\label{intro 5}
Let $G$ be a finite $3$-group of class at least $4$. 
Then the following conditions are equivalent.
\begin{itemize}
 \item[$1$.] One has $\inte(G)=2$.
 \item[$2$.] The group $G$ is isomorphic to $\yo$.
\end{itemize}
\end{theoremA}

\noindent
Theorem \ref{intro 5} concludes the classification of finite $3$-groups of intensity greater than $1$. Except for the two infinite families of finite non-trivial abelian $3$-groups and extraspecial $3$-groups of exponent $3$, there are, up to isomorphism, exactly $17$ groups in class $3$ (specifically $4$ of order $81$ and $13$ of order $243$), and $1$, namely $\yo$, in class $4$. In class higher than $4$, there are no $3$-groups of intensity greater than $1$.
\vspace{10pt} \\
\noindent
To continue our investigation, we let $p>3$ be a prime number. In Chapter \ref{chapter obelisks}, we define a \emph{$p$-obelisk} to be a finite non-abelian $p$-group $G$ such that $|G:G_3|=p^3$ and $G^p=G_3$. 
Among other things, we prove that $p$-obelisks of class at least $4$ satisfy both $(1)$ and $(2)$ from Theorem \ref{intro 4.5} and it is in fact true that, for each $p$-obelisk $G$, one has 
\[
(\log_p|G_i:G_{i+1}|)_{i\geq 1}=(2,1,2,1,\ldots,2,1,f,0,0,\ldots) \ \ \text{with} \ \ f\in\graffe{0,1,2},
\]
where the index $i\in\graffe{\cl(G),\cl(G)+1}$ to which $f$ corresponds is odd and larger than $2$. 
We will see in Chapter \ref{chapter bilbao} that, for every prime number $p>3$, each non-abelian quotient of $\Sn(\Delta_p)$ is a special kind of $p$-obelisk that we call ``framed''.  
\vspace{10pt}\\
\noindent
Let $p$ be a prime number and let $G$ be a finite $p$-group. The \emph{Frattini subgroup} of $G$ is $\Phi(G)=[G,G]G^p$; then $G/\Phi(G)$ is the largest possible quotient of $G$ that is vector space over $\F_p$. 
If $p>3$, then a $p$-obelisk $G$ is \emph{framed} if the Frattini subgroup of each maximal subgroup of $G$ coincides with $G_3$, i.e. for each maximal subgroup $M$ of $G$, one has $\Phi(M)=G_3$. Though it might not be evident at first sight, asking for a $p$-obelisk to be framed is equivalent to imposing strong limitations to the interaction of commutator maps and power maps in the group.
\vspace{10pt}\\
\noindent
Using a wide range of techniques, we are able to prove the following characterization for $p$-groups of class at least $4$, which coincides with the combination of Theorems \ref{theorem class 4 iff}, \ref{theorem we need lines}, and \ref{proposition last odd case}. We denote by $\Cyc_G(G_4)$ the centralizer of $G_4$ in the group $G$, i.e. $\Cyc_G(G_4)=\bigcap_{g\in G_4}\graffe{x\in G : [x,g]=1}$. 

\begin{theoremA}\label{intro 6}
Let $p>3$ be a prime number and let $G$ be a finite $p$-group of class at least $4$.
For each $i\in\Z_{\geq 1}$, write $w_i=\log_p|G_i:G_{i+1}|$.
Then $\inte(G)=2$ if and only if there exists $\alpha\in\Aut(G)$ of order $2$ such that $\alpha$ induces the inversion map $x\mapsto x^{-1}$ on $G/G_2$ and one of the following holds.
\begin{itemize}
 \item[$1$.] The group $G$ is a $p$-obelisk of class $4$.
 \item[$2$.] The group $G$ is a $p$-obelisk with $w_5=1$ and $\Phi(\Cyc_G(G_4))=G_3$.
 \item[$3$.] The group $G$ is a framed $p$-obelisk with $w_5=2$.
\end{itemize}  
\end{theoremA}

\noindent
Theorem \ref{intro 6} makes the role of $p$-obelisks in our theory clear and it can be used to prove that any $p$-group of class at least $6$ and intensity greater than $1$ is a framed $p$-obelisk. 
Class $5$ is the highest class in which there still exist $p$-obelisks of intensity greater than $1$ that are not framed \ldots but ``semi-framed''. 
More precisely, if, as in Theorem \ref{intro 6}($2$), the group $G$ is a $p$-obelisk with $w_5=1$, then the class is $5$, the order of $G_5$ is $p$, and $\Cyc_G(G_4)$ is a maximal subgroup; it is the only maximal subgroup whose Frattini subgroup is required to coincide with $G_3$.
\vspace{10pt} \\
\noindent
Theorem \ref{intro 6} completes the classification of prime power order groups of intensity greater than $1$, modulo the existence of some special automorphism. 
Because of their relevance in the theory of intense automorphisms, we give a name to such an automorphism. If $G$ is a group, we call an automorphism $\alpha\in\Aut(G)$  \emph{concrete} if it has order $2$ and the automorphism of $G/G_2$ that is induced by $\alpha$ coincides with the inversion map $x\mapsto x^{-1}$. To the present day, we know very little about concrete automorphisms and how to construct them in general: finding necessary and sufficient conditions for a $p$-obelisks to possess a concrete automorphism is an interesting problem that we have not yet addressed.
\vspace{10pt}\\
\noindent
In the following table, we summarize the results we have formulated so far. We denote by $p$ a prime number and by $G$ a finite $p$-group of class $c$. 
\vspace{10pt}

\begin{table}[h]

\resizebox{\textwidth}{4cm}{

\begin{tabular}{| P{2cm} | P{0.6cm} | P{2cm} | P{10cm} |}\hline
\multicolumn{4}{|c|}{\textbf{Intensity}} \\
\hline
\multicolumn{1}{|c|}{\backslashbox{$c$}{$p$}} & \multicolumn{1}{|c|}{$2$}
 & \multicolumn{1}{|c|}{$3$} & \multicolumn{1}{|c|}{$\geq 5$}  \\  
\hline
\multicolumn{1}{|c|}{$0$} & \multirow{13}{*}{$1$} & \multicolumn{2}{|c|}{$1$} \\ 
\cline{1-1} \cline{3-4} 
\multicolumn{1}{|c|}{$1$} &  & \multicolumn{2}{|c|}{$p-1$} \\ 
\cline{1-1} \cline{3-4} 
\multicolumn{1}{|c|}{$2$} &  & \multicolumn{2}{|c|}{$p-1$ if $G$ extraspecial of exponent $p$; } \\ 
\multicolumn{1}{|c|}{} & \multicolumn{1}{|c|}{} & \multicolumn{2}{|c|}{ $1$ otherwise} \\ 
\cline{1-1} \cline{3-4} 
\multicolumn{1}{|c|}{$3$} & \multicolumn{1}{|c|}{} & \multicolumn{2}{|c|}{$2$ if $|G:G_2|=p^2$; } \\ 
\multicolumn{1}{|c|}{} & \multicolumn{1}{|c|}{} & \multicolumn{2}{|c|}{ $1$ otherwise} \\ 
\cline{1-1} \cline{3-4} 
\multicolumn{1}{|c|}{$4$} & \multicolumn{1}{|c|}{} & \multicolumn{1}{|c|}{$2$ if $G\cong\yo$; } & \multicolumn{1}{|c|}{$2$ if $G$ is a $p$-obelisk with a concrete automorphism;} \\ 
\multicolumn{1}{|c|}{} & \multicolumn{1}{|c|}{} & \multicolumn{1}{|c|}{$1$ otherwise} & \multicolumn{1}{|c|}{$1$ otherwise} \\ 
\cline{1-1} \cline{3-4} 
\multicolumn{1}{|c|}{} & \multicolumn{1}{|c|}{} & \multicolumn{1}{|c|}{} & \multicolumn{1}{|c|}{$2$ if $G$ is a $p$-obelisk with $|G_5|=p$, $\Phi(\Cyc_G(G_4))=G_3$,} \\  
\multicolumn{1}{|c|}{} & \multicolumn{1}{|c|}{} & \multicolumn{1}{|c|}{} & \multicolumn{1}{|c|}{and $G$ has a concrete automorphism;} \\ 
\multicolumn{1}{|c|}{$\geq 5$} & \multicolumn{1}{|c|}{} & \multicolumn{1}{|c|}{$1$} & \multicolumn{1}{|c|}{$2$ if $G$ is framed $p$-obelisk with $|G_5:G_6|=p^2$} \\ 
\multicolumn{1}{|c|}{} & \multicolumn{1}{|c|}{} & \multicolumn{1}{|c|}{} & \multicolumn{1}{|c|}{ and $G$ has a concrete automorphism;} \\ 
\multicolumn{1}{|c|}{} & \multicolumn{1}{|c|}{} & \multicolumn{1}{|c|}{} & \multicolumn{1}{|c|}{$1$ in all other cases} \\ 
\hline
\end{tabular}

}

\end{table}



\noindent
We now have a clear picture of the intensity of groups of prime power order, according to their (finite) class. However, the theory of intense automorphisms can be extended to a larger family of groups with a striking result.
In Chapter \ref{chapter bilbao}, we complete the picture by moving to infinite class and computing the ``intensity'' of infinite pro-$p$-groups. 
\vspace{10pt}\\
\noindent
We call an automorphism $\alpha$ of a profinite group $G$ \emph{topologically intense} if, for each closed subgroup $H$ of $G$, there exists an element $g$ in $G$ such that $\alpha(H)=gHg^{-1}$. 
The group of topologically intense automorphisms of a profinite group $G$ is denoted by $\Int_{\cc}(G)$ and it is itself profinite.
As a consequence, several results concerning intense automorphisms of finite $p$-groups can be generalized to topologically intense automorphisms of pro-$p$-groups. For example, Theorem \ref{inte profinite} asserts that, if $p$ is a prime number and $G$ is a pro-$p$-group, then $\Int_{\cc}(G)$ decomposes as 
$$\Int_{\cc}(G)=S_G\rtimes C_G,$$ where $S_G$ is a Sylow pro-$p$-subgroup of $\Int_{\cc}(G)$ and $C_G$ is isomorphic to a subgroup of $\F_p^*$. Similarly to the finite case, we define the intensity $\inte(G)$ of a pro-$p$-group $G$ to be the order of $C_G$ and we ask which are the infinite pro-$p$-groups of intensity greater than $1$. Surprisingly, this question can be answered much more exhaustively than in the finite case, as follows.

\begin{theoremA}\label{intro 7}
Let $p$ be a prime number and let $G$ be an infinite pro-$p$-group. Then $\inte(G)>1$ if and only if exactly one of the following holds. 
\begin{itemize}
 \item[$1$.] One has $p>2$ and $G$ is abelian.
 \item[$2$.] One has $p>3$ and $G$ is topologically isomorphic to $\Sn(\Delta_p)$.
\end{itemize}
Moreover, one has $\inte(\Sn(\Delta_p))=2$ and, if $G$ is abelian, then $\inte(G)=p-1$.
\end{theoremA}

\noindent
Theorem \ref{intro 7} tells us that, ``in the limit'', for a given prime number $p>3$, there is a unique non-abelian pro-$p$-group, up to isomorphism, of intensity greater than $1$. From the point of view of finite groups, this last statement translates into saying that, if $p>3$ is a prime number, then each finite {$p$-group} $G$ with $\inte(G)> 1$ shares a ``relatively big'' quotient (growing in size with the class of $G$) with the infinite group $\Sn(\Delta_p)$. 
In a more definite way, we present this result in Section \ref{iii}, under the name of Proposition \ref{proposition function}.
\vspace{10pt}\\
\noindent
We conclude our introductory section by giving a ``cohomological context'' to intense automorphism. As we already mentioned at the beginning of this thesis, intense automorphisms arise naturally as solutions to certain problems coming from the field of Galois cohomology and we would like, with these last lines, to make this statement a little less vague. 
We start by looking at some examples.
\vspace{8pt} \\
\noindent
\textbf{Example.} Let $k$ be a field and let $n$ be a positive integer.
Moreover, let $a$ be a non-zero element of $k$. Then the least degree of the irreducible factors of $x^n-a$ divides all other degrees. 
\vspace{8pt} \\
\noindent
\textbf{Example.} Let $k$ be a field and let $\Br(k)$ denote the group of similarity classes of central simple algebras over $k$, endowed with the multiplication $\otimes_k$.
If $[A]\in\Br(k)$, then an extension $\ell/k$ is said to \emph{split} $A$ if $[A\otimes_k\ell]=[\ell]$. In 
\cite[Ch.~$4.5$]{algebras}, it is proven that the minimal degree of finite separable extensions of $k$ that split a given central simple algebra $A$ over $k$ divides all other degrees. 
\vspace{8pt} \\
\noindent
\textbf{Example.} Let $k$ be a field and let $C$ be a smooth projective absolutely irreducible curve of genus $1$ over $k$. As a consequence of the Riemann-Roch theorem, as explained for example in \cite[\S $2$]{langtate}, the least degree of the finite extensions of $k$ for which $C$ has a rational point divides all other degrees. 
\vspace{10pt} \\
\noindent
In a quite simplified manner, the last three examples suggest the following question: \emph{When does it happen that ``a problem'', defined on a base field $k$, is solvable over a field extension $\ell/k$ whose degree divides the degrees of all extensions $m/k$ over which the given problem can be solved?} 
The difficulty of translating this last question into rigorous mathematics is given by the fact that the known examples are quite diverse; however, we can try to unify them from the perspective of Galois cohomology. A first attempt of getting closer to the observed phenomena is Theorem \ref{master thesis}($1$) from \cite{evolving}.

\begin{theoremA}\label{master thesis}
Let $G$ be a finite group. Then the following are equivalent:
\begin{itemize}
 \item[$1$.] For every $G$-module $M$, integer $q$, and $c\in\widehat{\hc}{^{q}}(G,M)$, the minimum of the set
           $\{\,|G:H|\, :\, H\leq G \ \text{with} \ \Res_H^G(c)=0\,\}$ coincides with its greatest common divisor.
 \item[$2$.] There exist nilpotent groups $N$ and $T$ of coprime orders and a homomorphism $\phi:T\rightarrow\Int(N)$ such that 
 $G\cong N\rtimes_\phi T$.   
\end{itemize}
\end{theoremA}

\noindent
A way to interpret ($1$) from Theorem \ref{master thesis} is the following. In some sense, the non-zero elements of a cohomology group are the obstructions to having solutions so, ideally, each subgroup $H$ of $G$ for which $\Res_H^G(c)=0$ corresponds to a field extension ``solving the problem''. 
The merit of Theorem \ref{master thesis} is that of giving a splendid correspondence between a rather technical cohomological condition and a very concrete requirement regarding intense automorphisms. 
More about Theorem \ref{master thesis} and its proof can be found in \cite{evolving}. 
\vspace{10pt}\\
\noindent
Generalizing Theorem \ref{master thesis} to profinite groups
is an intriguing problem that is to the present day still open.


\mainmatter

\renewcommand{\chaptermark}[1]{\uppercase{\markboth{\thechapter.\ #1}{}}}

\chapter{Important tools}

\pagestyle{fancy}
\fancyhf{}
\fancyhead[LE,RO]{}
\fancyhead[LE,RO]{\leftmark}
\fancyfoot[CE,CO]{}
\fancyfoot[LE,RO]{\thepage}

This chapter consists of a miscellaneous collection of definitions and easy facts. Throughout the whole chapter we will fully respect the notation given in the List of Symbols.

\section{Bilinear maps and isotropic spaces}\label{section linear algebra}

Let $K$ be a field and let $V$, $W$, and $Z$ be 
vector spaces over $K$. A map $\phi:V\times W\rightarrow Z$ is said to be $K$-\emph{bilinear} (or simply \emph{bilinear})
\index{bilinear map} 
if, for all $v\in V$, the map ${_{v}}\phi:W\rightarrow Z$, defined by $t\mapsto\phi(v,t)$, is $K$-linear and, for all
$w\in W$, the map $\phi_w:V\rightarrow Z$, defined by 
$u\mapsto\phi(u,w)$, is $K$-linear.

\begin{definition}
Let $V$, $W$, and $Z$ be vector spaces over a field $K$ and let $\phi:V\times W\rightarrow Z$ be a map. Then $\phi$ is 
\emph{non-degenerate} if it is bilinear and \index{non-degenerate map}
both maps $V\rightarrow\Hom(W,Z)$, defined by $v\mapsto{_{v}}\phi$, and 
$W\rightarrow\Hom(V,Z)$, defined by $w\mapsto\phi_w$, are injective.
\end{definition}

\begin{lemma}\label{non-degenerate dim 1}
Let $V$, $W$, $Z$ be finite-dimensional vector spaces over a field $K$ and let $\phi:V\times W\rightarrow Z$ be a non-degenerate map. Assume moreover that $\dim_KZ=1$. 
Then the dimensions of $V$ and $W$ over $K$ are the same.
\end{lemma}

\begin{proof}
The maps $V\rightarrow\Hom(W,Z)$ and 
$W\rightarrow\Hom(V,Z)$ are injective. It follows that 
$\dim V\leq\dim \Hom(W,Z)=\dim W\leq\dim \Hom(V,Z)=\dim V$, and therefore the dimensions of $V$ and $W$ are the same. 
\end{proof}

\begin{definition}
Let $V$ and $Z$ be vector spaces over a field $K$. A map 
$\phi:V\times V\rightarrow Z$ is \index{alternating map}
\emph{alternating} if it is bilinear and, for 
all $v\in V$, one has $\phi(v,v)=0$.
\end{definition}

\noindent
A map $\phi:V\times V\rightarrow Z$ is \emph{antisymmetric} if it is bilinear and, for every \index{antisymmetric map}
$v,w\in V$ one has $\phi(v,w)=-\phi(w,v)$. As a direct consequence of their definitions, every alternating map is antisymmetric.




\begin{lemma}\label{wedge}
Let $K$ be a field and let $U$ be a $K$-vector space of dimension $3$. Then $\wedge:U\times U\rightarrow\bigwedge^2 U$ is surjective.
\end{lemma}

\begin{proof}
Let $(x,y,z)$ be a basis for $U$ over $K$. Then $(x\wedge y, y\wedge z, x\wedge z)$ is a basis for $\bigwedge^2U$. 
Let now $a=\lambda(x\wedge y)+\mu(y\wedge z)+\nu(x\wedge z)$ be an arbitrary element of $\bigwedge^2U$. 
If $a=0$, then clearly $a$ belongs to the image of $\wedge$, so, without loss of generality, we assume that $\lambda\in K\setminus\graffe{0}$. Since $(\lambda x-\mu z)\wedge(y+\lambda^{-1}\nu z)=a$, we are done.
\end{proof}

\begin{definition}
Let $K$ be a field and let $V$ and $Z$ be vector spaces over $K$. 
Let $\phi:V\times V\rightarrow Z$ be an alternating map. A subspace $T$ of $V$ is called \emph{isotropic}
\index{isotropic subspace} 
if $\phi$ restricted to $T\times T$ equals the zero map. A subspace $T$ is said to be
\emph{maximal isotropic} if it is isotropic and it is not properly contained in any other isotropic subspace of $V$.
\end{definition}

\begin{definition}\label{def induced map on quotient isotropic}
Let $K$ be a field and let $V$ and $Z$ be vector spaces over $K$. 
Let $\phi:V\times V\rightarrow Z$ be an alternating map and let $T$ be an isotropic subspace of $V$. 
Then $\phi_T:V/T\rightarrow\Hom(T,Z)$ is defined by $v+T\mapsto (t\mapsto\phi(v,t))$. 
\end{definition}

\noindent 
The map $\phi_T$ is well defined for every isotropic subspace $T$ of $V$, because $\phi(T\times T)=0$, and it is linear.

\begin{lemma}\label{max isotropic}
Let $K$ be a field and let $V$ and $Z$ be vector spaces over $K$. 
Let $\phi:V\times V\rightarrow Z$ be an alternating map and let $T$ be an isotropic subspace of $V$. 
Then $T$ is maximal isotropic if and only if $\phi_T$ is injective.
\end{lemma}

\begin{proof}
The subspace $T$ is not maximal isotropic if and only if there exists an element $v\in V\setminus T$ such that $T\oplus Kv$ is isotropic, which happens if and only if $v+T$ belongs to the kernel of $\phi_T$.
\end{proof}

\begin{definition}\label{definition orthogonal complement}
Let $K$ be a field and let $V$ and $Z$ be vector spaces over $K$. 
Let $\phi:V\times V\rightarrow Z$ be an alternating map. 
Let moreover $W$ be a linear subspace of $V$. The \indexx{orthogonal complement} $W^\perp$ of $W$ with respect to $\phi$ is the kernel of the map $V\rightarrow\Hom(W,Z)$ that is defined by $v\mapsto (w\mapsto\phi(v,w))$.
\end{definition}

\noindent
With the notation of Definition \ref{definition orthogonal complement}, the orthogonal complement of a subspace $W$ of $V$ is itself a subspace of $V$. Moreover, an alternating map being antisymmetric, $W^\perp$ is equal to the collection of all vectors $v\in V$, such that, for all $w\in W$, one has $\phi(w,v)=0$. It follows directly from the definition that $W\subseteq(W^\perp)^\perp$ and that, if $U\subseteq W$, then $U^\perp\supseteq W^\perp$.

\begin{lemma}\label{isotropic in orthogonal}
Let $K$ be a field and let $V$ and $Z$ be vector spaces over $K$. 
Let $\phi:V\times V\rightarrow Z$ be an alternating map. 
Let moreover $W$ be a linear subspace of $V$. 
Then $W$ is isotropic if and only if $W\subseteq W^\perp$.
Moreover, $W$ is maximal isotropic if and only if $W=W^\perp$.
\end{lemma}

\begin{proof}
Easy exercise.
\end{proof}

\begin{lemma}\label{dimension isotropic}
Let $V$ and $Z$ be finite-dimensional vector spaces over a field $K$ and let $\phi:V\times V\rightarrow Z$ be a non-degenerate alternating map. 
Assume that $Z$ has dimension $1$. Let $W$ be a linear subspace of $V$.
Then $\dim W+\dim W^\perp=\dim V$. 
Moreover, if $W$ is maximal isotropic, then $2\dim W=\dim V$. 
\end{lemma}

\begin{proof}
By definition of orthogonal complement, we have $\phi(W^\perp\times W)=\graffe{0}$, and hence the bilinear map 
$V/W^\perp\times W\rightarrow Z$ that is induced from $\phi$ is non-degenerate. It follows from Lemma \ref{non-degenerate dim 1} that $\dim V-\dim W^\perp=\dim W$. If $W$ is maximal isotropic, then Lemma \ref{isotropic in orthogonal} yields $\dim V=2\dim W$.
\end{proof}

\begin{lemma}\label{doppio perpendicolare}
Let $V$ and $Z$ be finite-dimensional vector spaces over a field $K$ and let $\phi:V\times V\rightarrow Z$ be a non-degenerate alternating map. 
Assume that $Z$ has dimension $1$. Let $W$ be a linear subspace of $V$. 
Then $(W^\perp)^\perp=W$.
\end{lemma}

\begin{proof}
The subspace $W$ is always contained in $(W^\perp)^\perp$, and therefore we always have $\dim W\leq \dim (W^\perp)^\perp$.
By Lemma \ref{dimension isotropic}, the dimension of $V$ is equal to both $\dim W+\dim W^\perp$ and $\dim W^\perp+\dim(W^\perp)^\perp$, and therefore $\dim W=\dim(W^\perp)^\perp$. It follows that $W=(W^\perp)^\perp$.
\end{proof}

\begin{lemma}\label{sum of isotropic}
Let $V$ and $Z$ be finite-dimensional vector spaces over a field $K$ and let $\phi:V\times V\rightarrow Z$ be a non-degenerate alternating map. 
Assume that $Z$ has dimension $1$. 
Let moreover $X$ be a maximal isotropic subspace of $V$. Then there exists a maximal isotropic subspace $Y$ of $V$ such that 
$V=X\oplus Y$.
\end{lemma}

\begin{proof}
Let $Y$ be maximal among the isotropic subspaces of $V$ that intersect $X$ trivially. We will show that $Y$ is maximal isotropic and that $V=X+Y$.
Lemma \ref{isotropic in orthogonal} guarantees $Y\subseteq Y^\perp$ and $X=X^\perp$. We now claim that $Y^\perp$ is contained in $X+Y$. Indeed, if $v\in Y^\perp$, then $Y+Kv$ is an isotropic subspace containing $Y$. From the maximality of $Y$, it follows that $(Y+Kv)\cap X$ is non-trivial. Since $Y\cap X=\graffe{0}$, the element $v$ belongs to $X+Y$, as claimed. 
Now, by Lemma \ref{doppio perpendicolare}, the subspaces $(Y^\perp)^\perp$ and $Y$ are the same and, $Y^\perp$ being contained in $X+Y$, it follows that 
$$Y=(Y^\perp)^\perp\supseteq (X+Y)^\perp\supseteq X^\perp\cap Y^\perp=X\cap Y^\perp.$$ 
In particular, $X\cap Y^\perp$ is contained in $X\cap Y$, which is trivial by definition of $Y$. 
As a consequence of Lemma \ref{dimension isotropic}, we get that
$$2\dim X=\dim V\geq\dim(X\oplus Y^\perp)=\dim X+\dim Y^\perp=$$
$$\dim X+\dim V-\dim Y=3\dim X-\dim Y\geq 2\dim X,$$
and therefore $\dim Y^\perp=\dim X=\dim Y$.
As a result, $Y=Y^\perp$, and thus $V=X\oplus Y$.
The subspace $Y$ is maximal isotropic, by Lemma \ref{isotropic in orthogonal}, and the proof is complete. 
\end{proof}

\section{Commutators and the lower central series}\label{section commutators}

Let $G$ be a group. The \emph{commutator map} \index{commutator map} on $G$ is the map
$G\times G\rightarrow G$ that is defined by
\[
(x,y)\mapsto [x,y]=xyx^{-1}y^{-1}.
\]
Given two subgroups $H$ and $K$ of $G$, we define 
$[H,K]$ to be the subgroup of $G$ that is generated by all elements $[h,k]$, where $h\in H$, $k\in K$. The groups $[H,K]$ and $[K,H]$ are equal, because, for all $(h,k)\in H\times K$, the inverse of $[h,k]$ is $[k,h]$.

\begin{lemma}\label{commutators normalizer}
Let $G$ be a group and let $H$ be a subgroup of $G$. 
Let $g\in G$ and denote $[g,H]=\graffe{[g,h]\ :\ h\in H}$.
Then $g\in\nor_G(H)$ if and only if $[g,H]$ is contained in $H$.
\end{lemma}

\begin{proof}
Let $h\in H$. Then $ghg^{-1}=[g,h]h$ and $ghg^{-1}$ is in $H$ if and only if $[g,h]$ is in $H$.
\end{proof}

\begin{definition}\label{LCS}
Let $G$ be a group. 
The \emph{lower central series} \index{lower central series} of $G$ is the series $(G_i)_{i\geq 1}$ that is obtained by defining recursively, for all $i\in\Z_{\geq 1}$, the subgroups $G_1=G$ and $G_{i+1}=[G,G_i]$. One calls $G_2$ the \emph{commutator subgroup} \index{commutator subgroup} of $G$.
\end{definition}

\noindent
Unless otherwise specified, we will stick to the notation from Definition \ref{LCS} to refer to the lower central series of a group (see also the List of Symbols). We remark that, in Chapter \ref{chapter bilbao}, we will define the lower central series of a profinite group $G$, by taking the closures of the elements of the lower central series of $G$ as an abstract group. In the case of finite groups, the two notions coincide. 
\vspace{8pt}\\
\noindent
A group $G$ is said to be \emph{nilpotent}
\index{nilpotent group} if there exists $i\in\Z_{\geq 0}$ for which $G_{i+1}=1$ and, in the latter case, one calls  
\index{class (nilpotency class)} $\cl(G)=\min\{i\in\Z_{\geq 0}\ :\ G_{i+1}=1\}$ the (\emph{nilpotency}) \emph{class} of $G$. The class of a nilpotent group is, in other words, the number of elements of the lower central series that are distinct from 
$\graffe{1}$.
Another way of deciding whether a group is nilpotent is by looking at its upper central series.

\begin{definition}
Let $G$ be a group. 
The \indexx{upper central series} of $G$ is the series $(Z_i)_{i\geq 0}$ that is obtained by defining recursively, for all $i\in\Z_{\geq 0}$, the subgroups $Z_0=1$ and $Z_{i+1}/Z_{i}=\ZG(G/Z_{i})$.
\end{definition}

\noindent
It is a general result (see for example Chapter $4$ in \cite{isaacs}) that a group is nilpotent if and only if its upper central series stabilizes at the group itself. 
In other words, if $G$ is a finite group and $(Z_i)_{i\geq 0}$ is its upper central series, then $G$ is nilpotent if and only if there exists $r\in\Z_{\geq 0}$ such that $Z_r=G$. Moreover, one can show (see for example in \cite[\S $1$D]{isaacs}) that  if the group $G$ is nilpotent, then its class is equal to 
$\min\{r\in\Z_{\geq 0}\ :\ Z_r=G\}$. 

\begin{lemma}\label{nilpotent iff}
Let $G$ be a finite group. Then $G$ is nilpotent if and only if $G$ is equal to the direct product of its Sylow $p$-subgroups.
\end{lemma}

\begin{proof}
This is a weaker version of Hauptsatz $2.3$ from \cite[Ch.~$\mathrm{III}$]{huppert}.
\end{proof}


\begin{definition}\label{def multiplication}
Let $G$ be a group and let $m,n$ be integers with $m\leq n$.
Let moreover $\graffe{x_i}_{i=m}^n$ be a subset of $G$. Then 
\[\prod_{i=m}^nx_i = x_mx_{m+1}\ldots x_{n-1}x_n. \]
\end{definition}

\begin{lemma}[Multiplication formulas]\label{multiplication formulas commutators}
Let $G$ be a group and let $x,y,z,t$ be elements of $G$. Let moreover $n$ be a non-negative integer.
Then the following hold.
\begin{itemize}
 \item[$1$.] $[x,yz]=[x,y]y[x,z]y^{-1}$.
 \item[$2$.] $[xt,y]=x[t,y]x^{-1}[x,y]$.
 \item[$3$.] $[x^n,y][x,y]^{-n}=\prod_{s=1}^{n-1}[x,[x^{n-s},y]]$.
 \item[$4$.] $[x,y]^{-n}[x,y^n]=\prod_{r=1}^{n-1}[[y^r,x],y]$.
\end{itemize}
\end{lemma}

\begin{proof}
Easy exercise.
\end{proof}

\begin{lemma}[Three-subgroups Lemma]\label{three subgroups lemma}
Let $G$ be a group and let $X$, $Y$, $Z$, and $N$ be subgroups of $G$ such that $N$ is normal.
Assume moreover that both $[X,[Y,Z]]$ and $[Y,[Z,X]]$ are contained in $N$. Then 
$[Z,[X,Y]]$ is contained in $N$.
\end{lemma}

\begin{proof}
See for example \cite[Corollary $4.10$]{isaacs}.
\end{proof}

\begin{lemma}\label{commutator indices}
Let $G$ be a group and let $(G_i)_{i\geq 1}$ be the lower central series of $G$. 
Then, for all $h,k\in\Z_{\geq 1}$, one has $[G_h,G_k]\subseteq G_{h+k}$.
\end{lemma}

\begin{proof}
We work by induction on $h$. If $h=1$, we are done by definition of the lower central series. Let us now assume that $h>1$ and, for all $k\in\Z_{>0}$, that 
$[G_{h-1},G_k]\subseteq G_{h+k-1}$. 
It follows that $[G,[G_{h-1},G_k]]\subseteq [G,G_{h+k-1}]\subseteq G_{h+k}$ and $[G_{h-1},[G_k,G]]\subseteq[G_{h-1},G_{k+1}]\subseteq G_{h+k}$. By Lemma \ref{three subgroups lemma}, also $[G_h,G_k]=[[G,G_{h-1}],G_k]$ is contained in $G_{h+k}$. 
\end{proof}

\noindent
Let $H$, $K$, and $L$ be groups. Let moreover $\phi:H\times K\rightarrow L$. The map $\phi$ is \emph{bilinear}\index{bilinear map} if, for all $h\in H$, the map ${_{h}}\phi:K\rightarrow L$, defined by $x\mapsto\phi(h,x)$, is a homomorphism and, for all 
$k\in K$, the map $\phi_k:H\rightarrow L$, defined by 
$x\mapsto\phi(x,k)$, is also a homomorphisms.
If $\phi$ is bilinear, then the \emph{left kernel} and the \emph{right kernel} of $\phi$, \index{left kernel} \index{right kernel}
are defined as 
$$\ker^{\mathrm{left}}\phi=\bigcap_{h\in H}\ker{_{h}}\phi,\ \ \text{and} \ \  
\ker^{\mathrm{right}}\phi=\bigcap_{k\in K}\ker\phi_k.$$
Assume $H=K$. Then $\phi$ is \emph{alternating}\index{alternating map} if it is bilinear and, for all $x\in H$, one has $\phi(x,x)=1$.

\begin{definition}
Let $H$, $K$, and $L$ be groups. Let moreover $\phi:H\times K\rightarrow L$.
Then $\phi$ is 
\emph{non-degenerate} if it is bilinear and both maps \index{non-degenerate map}
$H\rightarrow\Hom(K,L)$, defined by $h\mapsto{_{h}}\phi$, and 
$K\rightarrow\Hom(H,L)$, defined by $k\mapsto\phi_k$, are injective.
\end{definition}

\begin{lemma}\label{tgt}
Let $G$ be a group and let $H,L$ be subgroups of $G$. Let moreover $\phi:H\times L\rightarrow G$ be defined by $\phi(x,y)=[x,y]$.
Then $\phi$ is bilinear if and only if $[H,L]$ is contained in the centre of the subgroup generated by $H$ and $L$. 
\end{lemma}

\begin{proof}
This follows directly from the multiplication formulas from Lemma \ref{multiplication formulas commutators}. 
\end{proof}

\begin{lemma}\label{bilinear LCS hk}
Let $G$ be a group. Then for every $h,k\in\Z_{\geq 1}$ the commutator map
induces a bilinear map 
$G_h/G_{h+1}\times G_k/G_{k+1}\rightarrow G_{h+k}/G_{h+k+1}$.
\end{lemma}

\begin{proof}
The integers $h$ and $k$ being positive, it follows from Lemma \ref{commutator indices} that both $[G_h,G_{h+k}]$ and 
$[G_k,G_{h+k}]$ are subgroups of $G_{h+k+1}$. Then $G_{h+k}/G_{h+k+1}$ is contained in the centre of $\gen{G_h,G_k}/G_{h+k+1}$ and, by Lemma \ref{tgt}, the commutator map $G_h\times G_k\rightarrow G_{h+k}/G_{h+k+1}$ is bilinear.
Again by Lemma \ref{commutator indices}, both subgroups $[G_h,G_{k+1}]$ and $[G_{h+1},G_k]$ are contained in $G_{h+k+1}$. The commutator map induces hence a bilinear map 
$G_h/G_{h+1}\times G_k/G_{k+1}\rightarrow G_{h+k}/G_{h+k+1}$. 
\end{proof}

\begin{lemma}\label{bilinear LCS}
Let $G$ be a group. Then for every $i\in\Z_{\geq 1}$ the commutator map
induces a bilinear map 
$G/G_2\times G_i/G_{i+1}\rightarrow G_{i+1}/G_{i+2}$ whose image generates $G_{i+1}/G_{i+2}$.
\end{lemma}

\begin{proof}
The commutator map
induces a bilinear map 
$\gamma:G/G_2\times G_i/G_{i+1}\rightarrow G_{i+1}/G_{i+2}$ by Lemma \ref{bilinear LCS hk}. The image of $\gamma$ generates 
$G_{i+1}/G_{i+2}$ by definition of the lower central series of a group.
\end{proof}

\noindent
In Lemma \ref{tensor LCS} and throughout the whole manuscript, we write $\otimes$ instead of $\otimes_\Z$ (in concordance with the List of Symbols). 

\begin{lemma}\label{tensor LCS}
Let $G$ be a group and let $(G_i)_{i\geq 1}$ be its lower central series. Then for every $i\in\Z_{\geq 1}$ the commutator map
induces a surjective homomorphism of groups 
$G/G_2\otimes G_i/G_{i+1}\rightarrow G_{i+1}/G_{i+2}$.
\end{lemma}

\begin{proof}
This follows from Lemma \ref{bilinear LCS} and the universal property of tensor products.
\end{proof}

\begin{lemma}\label{class 2 bilinear map}\label{non-degenerate extraspecial}
Let $G$ be a group of class at most $2$. Then the commutator map induces a non-degenerate alternating map 
$G/\ZG(G) \times G/\ZG(G)\rightarrow G_2$ whose image generates $G_2$.
\end{lemma}

\begin{proof}
By Lemma \ref{bilinear LCS}, the commutator map induces a bilinear map $\gamma:G/G_2\times G/G_2\rightarrow G_2$ whose image generates $G_2$. Moreover, $\gamma$ is alternating because each element of $G$ commutes with itself.
The class of $G$ being at most $2$, the subgroup $G_2$ is central and $\ZG(G)/G_2$ is equal to both the right and the left kernel of $\gamma$. Then $\gamma$ factors as a non-degenerate map $G/\ZG(G)\times G/\ZG(G)\rightarrow G_2$. 
\end{proof}

\begin{lemma}\label{quotient by centre not cyclic}
Let $G$ be a group and assume that $G/\ZG(G)$ is cyclic. Then $G$ is abelian.
\end{lemma}

\begin{proof}
Since the quotient $G/\ZG(G)$ is cyclic, the commutator subgroup of $G$ is contained in $\ZG(G)$. In particular, $G$ has class at most $2$, and therefore, thanks to Lemma \ref{class 2 bilinear map}, the commutator map induces a non-degenerate alternating map 
$\gamma:G/\ZG(G) \times G/\ZG(G)\rightarrow G_2$ whose image generates $G_2$.
The image of $\gamma$ is trivial, because $G/\ZG(G)$ is cyclic, and so $G=\ZG(G)$.
\end{proof}

\begin{lemma}\label{cyclic quotient commutators}
Let $G$ be a group and let $N$ be a normal subgroup of $G$. Assume that $G/N$ is cyclic. Then $G_2=[G,N]$.
\end{lemma}

\begin{proof}
Since $N$ is normal, the subgroup $[G,N]$ is normal in $G$. 
We denote by $\overline{G}=G/[G,N]$ and we use the bar notation for the subgroups of $\overline{G}$. By definition of $\overline{G}$, the subgroup $\overline{N}$ is contained in $\ZG(\overline{G})$, and so $\overline{G}/\ZG(\overline{G})$ is cyclic. It follows from Lemma \ref{quotient by centre not cyclic} that $\overline{G}$ is abelian, and therefore $G_2=[G,N]$. 
\end{proof}

\section{About $p$-groups}\label{section p-gps}

\noindent
Let $p$ be a prime number. A finite group $G$ is a \emph{$p$-group} \index{$p$-group} if the order of $G$ is a power of $p$. 
The trivial group is a $p$-group for each prime $p$. Moreover, as a direct consequence of Lemma \ref{nilpotent iff}, every finite $p$-group is nilpotent.

\begin{lemma}\label{normal intersection centre trivial}
Let $p$ be a prime number and let $G$ be a finite $p$-group. 
Let $N$ be a normal subgroup of $G$ such that $N\cap \ZG(G)=\graffe{1}$.
Then $N=\graffe{1}$.
\end{lemma}

\begin{proof}
This is Satz $7.2$(a) from \cite[Ch.~$\mathrm{III}$]{huppert}.
\end{proof}



\begin{lemma}\label{normali schiacchiati generale}
Let $p$ be a prime number and let $G$ be a finite $p$-group of class $c$.
Let moreover $N$ be a subgroup of $G$.
Assume that, for all ${i\in\graffe{1,\ldots,c}}$, if $H$ is a quotient of $G$ of class $i$, then $\ZG(H)=H_i$.
Then $N$ is normal if and only if there exists $i\in\Z_{>0}$ such that $G_{i+1}\subseteq N\subseteq G_i$.
\end{lemma}

\begin{proof}
($\Leftarrow$) Assume that $G_{i+1}\subseteq N\subseteq G_i$. Then the quotient $N/G_{i+1}$ is contained in $\ZG(G/G_{i+1})$, and so $N$ is normal modulo $G_{i+1}$. In particular, $N$ is normal in $G$. 
($\Rightarrow$) If $N=\graffe{1}$, the result is clear, because $G$ is nilpotent. We assume $N$ is a non-trivial normal subgroup of $G$ and we let $i\in\Z_{>0}$ be the minimum index such that $G_{i+1}\subseteq N$ and $G_{i+1}\neq N$. We claim that $N$ is contained in $G_i$. First assume that $G_i$ is contained in $N$. By the minimality of $i$, the subgroup $G_i$ is equal to $N$, so we are done with this case. We assume now that $G_i$ is not contained in $N$. 
Then the group $\overline{G}=G/(N\cap G_i)$ has class $i$ and, by assumption, the center of $\overline{G}$ is equal to 
$\overline{G_i}$. 
On the other hand, $\overline{N}$ is a normal subgroup of $\overline{G}$ that has trivial intersection with $\overline{G_i}$. Lemma \ref{normal intersection centre trivial} yields $\overline{N}=\graffe{1}$, and thus $N\subseteq G_i$.
\end{proof}

\begin{lemma}\label{index G'}
Let $p$ be a prime number and let $G$ be a finite $p$-group. Assume that $G/G_2$ is cyclic. Then $G$ is abelian.
\end{lemma}

\begin{proof}
This is a weaker version of Hilfssatz $7.1$(b) from \cite[Ch.~$\mathrm{III}.7$]{huppert}.
\end{proof}

\begin{definition}
Let $G$ be a group. Then, for all $n\in\Z$, define
\[ G^n=\gen{x^n\ :\ x\in G}\ \ \ \text{and}\ \ \ \mu_n(G)=\gen{x\in G\ :\ x^n=1}.\]
\end{definition}

\noindent
Let $p$ be a prime number. The \emph{Frattini subgroup} \index{Frattini subgroup} $\Phi(G)$ of a finite $p$-group $G$ is the  unique normal subgroup of $G$ minimal with the property that $G/\Phi(G)$ is elementary abelian: in other words $\Phi(G)=G^p[G,G]$.

\begin{lemma}\label{subgroup times frattini}
Let $p$ be a prima number and let $G$ be a finite $p$-group.
Let $H$ be a subgroup of $G$ such that $G=H\Phi(G)$. Then $H=G$.
\end{lemma}

\begin{proof}
This is a weaker reformulation of Satz $3.2$(a) from \cite[Ch.~$\mathrm{III}$]{huppert}.
\end{proof}

\begin{lemma}\label{kernel of reduction frattini}
Let $p$ be a prime number and let $G$ be a finite $p$-group. Then the map $\phi:\Aut(G)\rightarrow\Aut(G/\Phi(G))$ given by $\alpha\mapsto(x\Phi(G)\mapsto\alpha(x)\Phi(G))$ is a well-defined homomorphism. Moreover, the kernel of $\phi$ is a $p$-group.  
\end{lemma}

\begin{proof}
This is a reformulation of Satz $3.18$ from \cite[Ch.~$\mathrm{III}$]{huppert}.
\end{proof}

\begin{lemma}\label{normal index p}
Let $p$ be a prime number and let $H$ be a finite $p$-group. Let moreover $K$ and $N$ be normal subgroups of $H$ such that $K\subseteq N$ and $K\neq N$. Then there exists a normal subgroup $M$ of $H$ such that $K\subseteq M\subseteq N$ and 
$|N:M|=p$.
\end{lemma}

\begin{proof}
See \cite[Lemma $1.23$]{isaacs}.
\end{proof}

\begin{lemma}\label{frattini comm}
Let $p$ be a prime number and let $G$ be a finite $p$-group. Assume that 
$|G:G_2|=p^2$. Then one of the following holds.
\begin{itemize}
 \item[$1$.] The group $G$ is abelian.
 \item[$2$.] The group $G_2$ is equal to $\Phi(G)$.
\end{itemize}
\end{lemma}

\begin{proof}
Assume that $G$ is not abelian. It follows from Lemma \ref{index G'} that $G/G_2$ has exponent $p$ and so $G^p$ is contained in $G_2$. In particular, we have that $G_2=G_2G^p=\Phi(G)$.
\end{proof}

\begin{definition}
Let $G$ be a group and let $p$ be a prime number. 
The \emph{$p$-central series} \index{$p$-central series} of $G$ is the series $(P_i(G))_{i\geq 1}$ that is obtained by defining recursively, for all $i\in\Z_{\geq 1}$, the subgroups $P_1(G)=G$ and $P_{i+1}(G)=[G,P_i(G)]P_i(G)^p$.
\end{definition}

\noindent
We remark that, if $p$ is a prime number and $G$ is a finite $p$-group, then saying that the lower central series of $G$ coincides with its $p$-central series is equivalent to saying that all quotients of consecutive elements of the lower central series have exponent dividing $p$.

\section{Extraspecial $p$-groups}\label{section extraspecial}

In Section \ref{section extraspecial} we explore the world of extraspecial $p$-groups, a class of groups that have been widely studied and whose structure is very well-understood (see for example \cite[Ch.~$\mathrm{III}.13$]{huppert}).
Given a group $G$, we recall that $(G_i)_{i\geq 1}$ denotes its lower central series (see Section \ref{section commutators}).

\begin{definition}\label{def extraspecial}
Let $p$ be a prime number and let $G$ be a finite $p$-group. Then $G$ is \emph{extraspecial} \index{$p$-group!extraspecial} if 
$G_2$ is central and $\ZG(G)$ is cyclic of order $p$.
\end{definition} 

\begin{lemma}\label{extraspecial non-abelian}
Let $p$ be a prime number and let $G$ be a finite extraspecial $p$-group.
If $G$ is non-abelian, then $\ZG(G)$ and $G_2$ are equal and they both have order $p$.
\end{lemma}

\begin{proof}
Straightforward.
\end{proof}

\begin{lemma}\label{class 2 frattini}
Let $p$ be a prime number and let $G$ be a finite $p$-group.
Assume that $G$ has class at most $2$ and that the exponent of $G_2$ divides $p$. 
Then $\Phi(G)$ is contained in $\ZG(G)$.
\end{lemma}

\begin{proof}
By Lemma \ref{class 2 bilinear map}, the map
$G/\ZG(G)\times G/\ZG(G)\rightarrow G_2$
that is induced by the commutator map is bilinear.
Let now $g,x\in G$. Then $[g^p,x]=[g,x]^p=[g,x^p]$ and $[g,x]^p=1$, since the exponent of $G_2$ divides $p$. It follows that $G/\ZG(G)$ is annihilated by $p$ and thus $\Phi(G)\subseteq\ZG(G)$.
\end{proof}

\begin{lemma}\label{order extraspecial}
Let $p$ be a prime number and let $G$ be an extraspecial $p$-group. Then there exists $n\in\Z_{\geq0}$ such that $|G|=p^{2n+1}$. 
\end{lemma}

\begin{proof}
This is a reformulation of Satz $13.7$(c) from \cite[Ch.~$\mathrm{III}$]{huppert}.
\end{proof}


\begin{lemma}
Let $p$ be a prime number and let $X,Y,Z$ be finite-dimensional vector spaces over $\F_p$ such that $\dim_{\F_p}Z=1$. 
Let moreover $\theta:X\times Y\rightarrow Z$ be a non-degenerate map. 
Call $G=G(Z,Y,X,\theta)$ the set $Z\times Y\times X$ together with the multiplication defined by 
$(z,y,x)(z',y',x')=(z+z'+\theta(x,y'),y+y',x+x')$. 
Then the following hold.
\begin{itemize}
 \item[$1$.] One has that $G$ is an extraspecial $p$-group and, if $p$ is odd, then $G$ has exponent $p$.
 \item[$2$.] The centre of $G$ is $Z\times\graffe{0}\times\graffe{0}$. 
 \item[$3$.] The commutator map $G\times G\rightarrow G$ is given by $$((z,y,x),(z',y',x'))\mapsto[(z,y,x),(z',y',x')]=(\theta(x,y')-\theta(x',y),0,0).$$
\end{itemize}
\end{lemma}

\begin{proof}
Straightforward.
\end{proof}

\begin{lemma}\label{get extraspecial}
Let $p$ be a prime number and let $G$ be an extraspecial $p$-group of exponent $p$. 
Then there exist finite-dimensional vector spaces $X,Y,Z$ over $\F_p$, with $Z$ of dimension $1$, and 
a non-degenerate map $\theta:X\times Y\rightarrow Z$ such that $G\cong G(Z,Y,X,\theta)$. 
\end{lemma}

\begin{proof}
If $G$ is abelian, we take $X=Y=0$, $Z=G$, and $\theta$ to be the zero map. Since every group of exponent $2$ is abelian, we are done when $p=2$.
Assume now that $p$ is odd and that $G$ has class $2$. In this case $\ZG(G)=G_2$ and $V=G/\ZG(G)$ is a finite-dimensional vector space over $\F_p$, as a consequence of Lemma \ref{class 2 frattini}. Write $Z=\ZG(G)$ and let $\pi:G\rightarrow V$ denote the canonical projection. 
By Lemma \ref{non-degenerate extraspecial}, the commutator map on $G$ induces a non-degenerate map $\phi:V\times V\rightarrow Z$. 
Let now $X$ be a maximal isotropic subspace of $V$. 
By Lemma \ref{sum of isotropic} there exists an isotropic subspace $Y$ of $V$ such that 
$V=X\oplus Y$. It follows that the map $\phi_{|X\times Y}:X\times Y\rightarrow Z$ is non-degenerate. Now, we have that $0=\phi(X\times X)=[\pi^{-1}(X),\pi^{-1}(X)]$, and so 
$\pi^{-1}(X)$ is abelian of exponent $p$. As a result, the sequence 
${0\rightarrow Z\rightarrow \pi^{-1}(X)\rightarrow X\rightarrow 0}$ of $\F_p$-modules is split and there is a homomorphism $s:X\rightarrow\pi^{-1}(X)$ such that $\pi\circ s=\id_X$. The same argument applies to $Y$ and there exists therefore a homomorphism $t:Y\rightarrow\pi^{-1}(Y)$ such that $\pi\circ t=\id_Y$. 
To conclude, we denote $\theta=\phi_{|X\times Y}$ and we define $\psi:G(Z,Y,X,\theta)\rightarrow G$ by 
$(z,y,x)\mapsto zt(y)s(x)$. It is not difficult at this point to check that $\psi$ is an isomorphism.
\end{proof}

\begin{proposition}\label{homom extraspecial}
Let $p$ be a prime number. Let moreover $X,Y,Z,A,B,C$ be finite-dimensional vector spaces over $\F_p$ with $\dim Z=\dim C=1$. 
Let $\theta:X\times Y\rightarrow Z$ and $\psi:A\times B\rightarrow C$ be non-degenerate maps.
Let $f,g,h$ respectively belong to $\Hom(X,A),\Hom(Y,B),\Hom(Z,C)$ and assume that the following diagram is commutative. 
\[
\begin{diagram}
X    \times   Y           &  \rTo^{\theta} & Z  \\
\dTo^{f} \ \  \dTo_{g}    &                & \dTo_{h}\\
A    \times   B           &  \rTo^{\psi}   & C  \\
\end{diagram}
\]
Then $(h,g,f):G(Z,Y,X,\theta)\rightarrow G(C,B,A,\psi)$ is a homomorphism of groups and, if $f,g,h$ are isomorphisms, then  
$(h,g,f)$ is an isomorphism.
\end{proposition}

\begin{proof}
Straightforward.
\end{proof}

\begin{lemma}\label{maps1-1}
Let $T$ be a group and let $S$ be a central subgroup of $T$. Let moreover $\Delta$ denote the subgroup of $\Aut(T)$ consisting of all those elements $\delta$ such that $\delta(S)=S$ and such that $\delta$ induces the identity on both $S$ and $T/S$. Then the map 
\[
\Delta\rightarrow\Hom(T/S,S)
\] 
that is defined by 
\[
\delta \mapsto (xS \mapsto \delta(x)x^{-1})
\]
is bijective.
\end{lemma}

\begin{proof}
Let $\phi:\Delta\rightarrow\Hom(T/S,S)$ denote the map $\delta \mapsto (xS \mapsto \delta(x)x^{-1})$, which is well-defined because $S$ is central in $T$. The map $\phi$ is clearly injective and it is surjective because, given each homomorphism $f\in\Hom(T,S)$ with $S\subseteq\ker(f)$, the map $x\mapsto xf(x)$ belongs to $\Delta$.
\end{proof}

\begin{lemma}\label{innerextra}
Let $p$ be a prime number and let $G$ be an extraspecial $p$-group. Let 
$\Delta$ denote the subgroup of $\Aut(G)$ consisting of those automorphisms of $G$ that induce the identity on $G/G_2$. Then $\Delta=\Inn(G)$.
\end{lemma}

\begin{proof}
If $G$ is abelian, then $\Inn(G)$ is trivial and we are done. Assume now that $G$ is non-abelian. Then, by Lemma 39, the subgroups $\ZG(G)$ and $G_2$ are equal and they both have order $p$. It follows from Lemma 26 that the commutator map induces a non-degenerate map $G/G_2\times G/G_2\rightarrow G_2$ and so the homomorphism $G/G_2\rightarrow\Hom(G/G_2,G_2)$, defined by 
$t \mapsto (x \mapsto [t,x])$, is injective. Thanks to Lemma 40, the quotient $G/G_2$ is elementary abelian and thus $G/G_2\rightarrow\Hom(G/G_2,G_2)$ is an isomorphism. 
Now, by Lemma 60, each element $\delta$ of $\Delta$ restricts to the identity on $G_2$ and so we derive from Lemma \ref{maps1-1} that, for each element $\delta\in\Delta$, there exists $t\in G$ such that, for all $x\in G$, one has $\delta(x)=[t,x]x=txt^{-1}$. In particular, $\Delta$ is contained in $\Inn(G)$. The inclusion $\Inn(G)\subseteq\Delta$ is clear and so the proof is complete.
\end{proof}

\section{Regular $p$-groups}\label{section regular}

\noindent
Most of the results from this section are taken from \cite{huppert}, an excellent reference for getting acquainted with regular $p$-groups. 
We recall here briefly the \indexx{Hall-Petrescu formula} and we refer to Appendix \textit{A} from \cite{analytic} for more detail. We also refer to Definition \ref{def multiplication} for a clear interpretation of the Hall-Petrescu formula.

\begin{lemma}[Hall-Petrescu formula]\label{hall-petrescu lemma}
Let $G$ be a group and let $(G_i)_{i\geq 1}$ denote its lower central series. Let moreover $x$ and $y$ be elements of $G$.
Then, for all $n\in\Z_{>0}$, there exists $(c_k)_{k=2}^n\in\prod_{k=2}^nG_k$ such that 
\[(xy)^n=x^ny^n\prod_{k=2}^nc_k^{\binom{n}{k}}.\]
\end{lemma}

\begin{proof}
See \cite[Appendix A]{analytic}.
\end{proof}

\begin{corollary}\label{p map petrescu}
Let $p$ be a prime number and let $G$ be a group. Denote by $(G_i)_{i\geq 1}$ the lower central series of $G$. 
Then, for all $x,y\in G$, one has $(xy)^p\equiv x^py^p\bmod G_2^pG_{p}$.
\end{corollary}

\begin{proof}
This follows directly from Lemma \ref{hall-petrescu lemma}.
\end{proof}


\begin{definition}
Let $p$ be a prime number. A finite $p$-group $G$ is \emph{regular}
\index{$p$-group!regular}
if, for all $x,y\in G$, there exists $\gamma\in[\gen{x,y},\gen{x,y}]^p$ such that $(xy)^p=x^py^p\gamma$.
\end{definition}

\begin{lemma}\label{class at most p-1 regular}
Let $p$ be a prime number and let $G$ be a finite $p$-group of nilpotency class at most $p-1$.
Then $G$ is regular.
\end{lemma}

\begin{proof}
The class of each subgroup of $G$ is at most that of $G$. The result now follows directly from Corollary \ref{p map petrescu}.
\end{proof}

\noindent
Let $p$ be a prime number and let $G$ be a $p$-group. 
We will denote by $\rho$ the map $G\rightarrow G$ that is defined by $x\mapsto x^{p}$. We remark that the map $\rho$ is in general not a homomorphism and that $G^{p^k}=\gen{\rho^k(G)}$, for any integer $k$. We stick to the notation from the List of Symbols.

\begin{lemma}\label{exponent G2=p, endomorphism}
Let $p$ be a prime number and let $G$ be a finite $p$-group. 
Assume that $G$ is regular and that $G_2$ has exponent dividing $p$. Then $\rho$ is an endomorphism of $G$.
\end{lemma}

\begin{proof}
This follows directly from the definition of regularity.
\end{proof}

\begin{lemma}\label{regular implies power abelian}
Let $p$ be a prime number and let $G$ be a finite $p$-group. Assume that $G$ is regular. 
Then for all $k\in\Z_{\geq 0}$, the following hold.
\begin{itemize}
 \item[$1$.] One has $G^{p^k}=\rho^k(G)$.
 \item[$2$.] One has $\mu_{p^k}(G)=\graffe{x\in G\ :\ \rho^k(x)=1}$.
 \item[$3$.] One has $|\mu_{p^k}(G)|=|G:G^{p^k}|$.
\end{itemize} 
\end{lemma} 

\begin{proof}
The lemma is a combination of Satz $10.5$ and Satz $10.7$(a) from \cite{huppert}, Chapter $3$.
\end{proof}

\noindent
We remark that $p$-groups satisfying conditions $1$--$3$ from Lemma \ref{regular implies power abelian} are often referred to as \emph{power abelian}.

\begin{lemma}\label{regular if omega small}
Let $p$ be a prime number and let $G$ be a finite $p$-group. 
If $|G:G^p|<p^p$, then $G$ is regular.
\end{lemma}

\begin{proof}
The lemma is a simplified version of Satz $10.13$ from \cite{huppert}, Chapter $3$.
\end{proof}

\begin{lemma}\label{regular technical}\label{regular bilinear}
Let $p$ be a prime number and let $G$ be a finite regular $p$-group. Let $M,N$ be normal subgroups of $G$ and let $r,s$ be non-negative integers. 
Then $[\rho^r(M),\rho^s(N)]=\rho^{r+s}([M,N])$.
\end{lemma}

\begin{proof}
See \cite[Satz $10.8$(a) from Ch. $3$]{huppert}.
\end{proof}

\begin{lemma}\label{regular 3-gp 2gen}
Assume $G$ is a finite $3$-group that can be generated by $2$ elements. 
If $G$ is regular, then $G_2$ is cyclic.
\end{lemma}

\begin{proof}
See Satz $10.3$(b) from \cite{huppert}, Chapter $3$.
\end{proof}


\chapter{Coprime actions}\label{chapter actions}

The aim of this chapter is to create tools for later use, giving them however their own chance to shine. 
In Section \ref{section characters}, we define \emph{actions through characters} and prove a fundamental result, Theorem \ref{lambda mu}, in the context of intense automorphisms of groups. In Section \ref{section involutions}, we prove some elementary, yet quite entertaining, results concerning involutions of groups of odd order. The results from Section \ref{section involutions} will spark throughout the thesis, starting with Chapter \ref{chapter class 3}. The last section of this chapter, Section \ref{section jumps}, is dedicated to the theory of \emph{jumps}. In some sense, through jumps (and their width), we are able to recover structural information about subgroups of a given finite $p$-group. This theory will be heavily used when dealing with $p$-obelisks (from Chapter \ref{chapter obelisks} onwards).

\section{Actions through characters}\label{section characters}

\noindent
Until the end of Section \ref{section characters}, let $p$ be a prime number.
Every finite abelian $p$-group $G$ is naturally a $\Z_p$-module, with scalar multiplication
$\Z_p\rightarrow\End(G)$ defined by
\[
m\mapsto [x \mapsto (m\bmod|G|)\,x]\,.
\]
It follows directly from this definition that every homomorphism between abelian $p$-groups is $\Z_p$-linear, a fact that we will make hidden use of in several proofs from Chapter \ref{chapter actions}.
To conclude, we remark that we have here adopted the additive notation for the abelian group $G$, but this will sadly not be the case through the whole thesis. We will indeed often deal, instead of abelian groups, with abelian quotients of non-abelian groups (for which the multiplicative notation will be used). The first time we adopt the multiplicative notation in this context is in the proof of Lemma \ref{abelian>minus quotients}.

\begin{definition}
Let $A$ and $G$ be groups. An \emph{action of} $A$ \emph{on} $G$ is a homomorphism 
$A\rightarrow\Aut(G)$. \index{action on a group}
\end{definition}

\begin{definition}
Let $A$ be a group acting on a set $X$. A subset $Y$ of $X$ is \emph{$A$-stable} (or \emph{stable under the action of $A$}) if the action of $A$ on $X$ restricts to an action of $A$ on $Y$.
\end{definition}

\begin{definition}
Let $A$ be a group and let $\Z A$ denote its group ring over $\Z$. An \emph{$A$-module} is a module over $\Z A$. \index{module}
\end{definition}

\noindent
With respect to the last definition, any finite abelian $p$-group is naturally a $\Z_p^*$-module. We stress that, if $A$ is a group, then each $A$-module is, in particular, an abelian group.

\begin{definition}
Let $A$ be a group acting on two sets $T$ and $Z$. A map 
$\phi:T\rightarrow Z$ is said to \emph{respect the action of $A$} if, for all $t\in T$, $a\in A$, one has $\phi(at)=a\phi(t)$.
\end{definition}

\begin{definition}
Let $A$ be a group and let $G$ be a finite $p$-group that is also an $A$-module.
Let $\chi:A\rightarrow\Z_p^*$ be a homomorphism. 
Then \emph{$A$ acts on $G$ through~$\chi$} \index{action through a character}
if, for all $a\in A$ and $x\in G$, one has $ax=\chi(a)x$.
\end{definition}

\noindent
We want to emphasize the fact that $\Hom(A,\Z_p^*)$ is a group under multiplication (induced by that in $\Z_p^*$). We will refer to the elements of $\Hom(A,\Z_p^*)$ as \emph{characters} of $A$.

\begin{lemma}\label{lemma product of characters}
Let $X$, $Y$, and $Z$ be finite abelian $p$-groups. Let $A$ be a group acting on $X$, $Y$, and $Z$ and let $\phi:X\times Y\rightarrow Z$ be a bilinear map respecting the action of $A$. Let moreover, $\chi$ and $\psi$ be group homomorphisms $A\rightarrow\Z_p^*$ such that $A$ acts on $X$ and $Y$ respectively through $\chi$ and $\psi$. 
Then $A$ acts on $\gen{\phi(X\times Y)}$ through $\chi\psi$. 
\end{lemma}

\begin{proof}
Let $(x,y)\in X\times Y$ and $a\in A$. Then one has $$a\phi(x,y)=\phi(ax,ay)=
\phi(\chi(a)x,\psi(a)y)=\chi(a)\psi(a)\phi(x,y)=
(\chi\psi)(a)\phi(x,y).$$ Since $\chi$ and $\psi$ are homomorphisms, the action of $A$ on $\phi(X\times Y)$ is through $\chi\psi$.
\end{proof}

\begin{lemma}\label{action chi^i general}
Let $p$ be a prime number and let $G$ be a finite $p$-group. 
Let moreover $A$ be a finite group acting on $G$ and let $\chi:A\rightarrow\Z_p^*$ be a homomorphism.
Denote by $(G_i)_{i\geq 1}$ the lower central series of $G$ and assume that the induced action of $A$ on $G/G_2$ is through $\chi$. 
Then, for all $i\in\Z_{\geq 1}$, the induced action of $A$ on $G_i/G_{i+1}$ is through $\chi^i$.
\end{lemma}

\noindent
The elements of the lower central series of a group are characteristic subgroups and, 
for each $i\in\Z_{\geq 1}$, the quotient $G_i/G_{i+1}$ is abelian. Lemma \ref{action chi^i general} is thus well-stated.

\begin{proof}
We will work by induction on $i$. If $i=1$, we are done by hypothesis. Suppose now that $i>1$ and that the result holds for all indices smaller than $i$. By Lemma \ref{bilinear LCS} the commutator map induces a bilinear map $G/G_2\times G_{i-1}/G_i\rightarrow G_i/G_{i+1}$ whose image generates $G_i/G_{i+1}$. 
By the induction hypothesis, the induced action of $A$ on $G_{i-1}/G_i$ is through $\chi^{i-1}$ and, by Lemma \ref{lemma product of characters}, the group $A$ acts on $G_i/G_{i+1}$ through $\chi\chi^{i-1}=\chi^i$.
\end{proof}

\begin{lemma}\label{p-power characters}\label{intersection characters}
Let $A$ be a group and let $G$ and $H$ be finite $p$-groups that are also $A$-modules.
Let moreover $\phi:G\rightarrow H$ and $\chi:A\rightarrow\Z_p^*$ be group homomorphisms. 
Assume that the action of $A$ on $G$ is through $\chi$.
If $\phi$ is surjective and $\phi$ respects the action of $A$, then $A$ acts on $H$ through $\chi$.
\end{lemma}

\begin{proof}
Let $a\in A$. If $\phi$ is surjective, then, for each $h\in H$ there exists $g\in G$ such that $\phi(g)=h$. If, moreover, the action of $A$ is respected by $\phi$, then
$ah=a\phi(g)=\phi(ag)=\phi(\chi(a)g)=\chi(a)\phi(g)=\chi(a)h$.
\end{proof}

\begin{lemma}\label{sesZ_p}
The short exact sequence of abelian groups
\[1\longrightarrow 1+p\Z_p\longrightarrow\Z_p^*\longrightarrow\F_p^*\longrightarrow 1\]
has a unique section $\omega:\F_p^*\rightarrow\Z_p^*$.
\end{lemma}

\noindent
The last is a classical result, which can be found for example in \cite[\S $4.3$]{cohen}.
The homomorphism $\omega:\F_p^*\rightarrow\Z_p^*$ is called the \emph{Teichm\"{u}ller character} \index{Teichm\"{u}ller character} at $p$ and its image is contained in the torsion subgroup of $\Z_p^*$. (For a reminder of the notation, see the List of Symbols.)
Moreover, if $p$ is odd, then $\omega(\F_p^*)$ is in fact equal to the torsion subgroup of $\Z_p^*$; for more information see for example Section $4.3$ from \cite{cohen}.
\vspace{8pt}\\
\noindent
We remark that, if $V$ is a vector space over $\F_p$, then, for each $v\in V$ and for each $a\in\F_p^*$, one has $av=\omega(a)v$. It follows that the natural action of 
$\F_p^*$ on a vector space over $\F_p$ is through the Teichm\"{u}ller character.

\begin{lemma}\label{biscotto}
Let $A$ be a finite group and let $\lambda,\mu:A\rightarrow \Z_p^*$ be distinct group homomorphisms.
Assume that $p$ is odd. Then there exists $a\in A$ such that the element $\lambda(a)-\mu(a)$ belongs to $\Z_p^*$.
\end{lemma}

\begin{proof}
Let $\pi:\Z_p\rightarrow\F_p$ denote the canonical projection and let $\omega:\F_p^*\rightarrow\Z_p^*$ be the Teichm\"{u}ller character.
The group $A$ being finite, the images of $\lambda$ and $\mu$ live in the torsion of $\Z_p^*$, which is equal to $\omega(\F_p^*)$. 
Let now $a\in A$ be such that $\lambda(a)\neq\mu(a)$.
As a consequence of Lemma \ref{sesZ_p}, each element of $\omega(\F_p^*)$ is uniquely determined by its image modulo $p$ and, the characters being distinct, 
$\pi(\chi(a)-\psi(a))\in\F_p^*$. It follows that $\chi(a)-\psi(a)$ is invertible in $\Z_p$.
\end{proof}

\begin{lemma}\label{distinct characters on same gp}
Let $A$ be a finite group and let $G$ be a finite $p$-group that is also an $A$-module.
Let moreover $\lambda,\mu:A\rightarrow \Z_p^*$ be distinct group homomorphisms.
Assume that $p$ is odd and that $A$ acts on $G$ through both $\lambda$ and $\mu$.
Then $G=\graffe{0}$.
\end{lemma}

\begin{proof}
Let $x\in G$ and let $a\in A$ be as in Lemma \ref{biscotto}. Then $\lambda(a)x=ax=\mu(a)x$ and $(\lambda(a)-\mu(a))x=0$.  
The element $\lambda(a)-\mu(a)$ being invertible in $\Z_p$, it follows that $x=0$. As the choice of $x$ was arbitrary, we get 
$G=\graffe{0}$.
\end{proof}

\begin{definition}
Let $G$ be a group and let $N$ be a normal subgroup of $G$. A subgroup $H$ of $G$ is a \emph{complement of} $N$ \emph{in} $G$ \index{complement}
if $N\cap H=\graffe{1}$ and $NH=G$.
\end{definition}

\begin{theorem}\label{lambda mu}
Assume that $p$ is odd. Let $A$ be a finite abelian group and let 
$$0\longrightarrow N \overset{\iota}{\longrightarrow} G \overset{\pi}{\longrightarrow} G/N \longrightarrow 0$$ be a short exact sequence of $A$-modules. 
Let moreover $\lambda,\mu:A\rightarrow\Z_p^*$ be two distinct group homomorphisms and assume that the following hold.  
\begin{itemize}
 \item[$1$.] The group $G$ is a finite $p$-group.
 \item[$2$.] The group $A$ acts on $N$ through $\lambda$.
 \item[$3$.] The group $A$ acts on $G/N$ through $\mu$.  
\end{itemize}
Then $\iota(N)$ has a unique $A$-stable complement in $G$.
\end{theorem}

\noindent
We will devote the remaining part of Section \ref{section characters} to the proof of Theorem \ref{lambda mu}. 
For this purpose, let $R=\Z_pA$ be the group algebra of $A$ over $\Z_p$ and let $\sigma_{\lambda}$ and $\sigma_{\mu}$ be the homomorphisms of $\Z_p$-algebras $R\rightarrow\Z_p$ that are respectively induced, via linear extension, by $\lambda$ and $\mu$.
We define $I_{\lambda}=\ker\sigma_{\lambda}$ and $I_{\mu}=\ker\sigma_{\mu}$.



\begin{lemma}\label{coprime ideals}
One has $R=I_{\lambda}+I_{\mu}$.
\end{lemma}

\begin{proof}
We will construct an invertible element in $I_{\lambda}+I_{\mu}$. Let $a\in A$ be as in Lemma \ref{biscotto}. 
The element $\lambda(a)-\mu(a)=-(a-\lambda(a))+(a-\mu(a))$ belongs to $I_{\lambda}+I_{\mu}$, because
$a-\lambda(a)\in I_{\lambda}$ and $a-\mu(a)\in I_{\mu}$, and $\lambda(a)-\mu(a)$ is invertible because of the choice of $a$. 
\end{proof}

\begin{lemma}\label{bibidi}
The subgroup $\iota(N)$ has an $A$-stable complement.
\end{lemma}

\begin{proof}
Let $(e,f)\in I_\lambda\times I_\mu$ be such that $e+f=1$ in $R$; the pair $(e,f)$ exists thanks to Lemma \ref{coprime ideals}. As a direct consequence of the definition of $I_\mu$, the group $G/N$ is annihilated by $f$ and $f(G)\subseteq\iota(N)$. 
From the fact that $f\equiv 1\bmod I_\lambda$, it follows that $f(G)=\iota(N)$. With a similar argument, one shows that 
$e(G)$ is isomorphic to $e(G/N)=G/N$. We now have that 
\[G=(e+f)G=e(G)+f(G)=e(G)+\iota(N)\]   
and, the cardinalities of $e(G)$ and $\iota(N)$ being respectively $|G:N|$ and $|N|$, it follows that $G=e(G)\oplus\iota(N)$.
The subgroup $e(G)$ is thus a complement of $\iota(N)$ in $G$. Furthermore, the ring $R$ being commutative, for all $a\in A$, one has that 
$ae(G)=ea(G)$ is contained in $e(G)$ and therefore $e(G)$ is $A$-stable.
\end{proof}

\begin{lemma}\label{acquetta}
There exists a unique $A$-stable complement of $\iota(N)$.
\end{lemma}

\begin{proof}
The subgroup $\iota(N)$ has an $A$-stable complement in $G$, by Lemma \ref{bibidi};
assume it has two. Then there exist maps $f,f':G\rightarrow N$ respecting the action of $A$ such that 
$f\circ\iota=f'\circ\iota=\id_N$. We fix such $f,f'$ and write $r=f-f'$; we will show that $r=0$. Since $f\circ\iota=f'\circ\iota$, the subgroup $\iota(N)$ is contained in the kernel of $r$. It follows that $r\in\Hom(G/\iota(N),N)$, and so, thanks to Lemma \ref{p-power characters}, the group $A$ acts on the image of $r$ through $\mu$. On the other hand, the image of $r$ is contained in $N$ and hence the action of $A$ on $r(G)$ is also through $\lambda$. It follows from Lemma \ref{distinct characters on same gp} that $r=0$, as claimed. In particular, $f=f'$, and so $\iota(N)$ has a unique $A$-stable complement in 
$G$.
\end{proof}

\noindent
In view of Lemma \ref{acquetta}, Theorem \ref{lambda mu} is proven.

\section{Involutions}\label{section involutions}

Let $G$ be a finite group of odd order and let $A=\gen{\alpha}$ be a multiplicative group of order $2$. 
It follows that the orders of $G$ and $A$ are coprime. 
Assume that $A$ acts on $G$ and define
$$G^+=\graffe{g\in G : \alpha(g)=g} \ \ \text{and} \ \ G^-=\graffe{g\in G : \alpha(g)=g^{-1}}\, .$$
We will keep the notation we just introduced until the end of Section \ref{section involutions}. We remind the reader about the List of Symbols, at the beginning of this thesis.

\begin{lemma}\label{trivial intersection}
One has $G^+\cap G^-=\graffe{1}$.
\end{lemma}

\begin{proof}
Let $a\in G^+\cap G^-$. Then we have $a=\alpha(a)=a^{-1}$, so $a^2=1$. The order of $G$ being odd, it follows that $a=1$.
\end{proof}

\begin{proposition}\label{G+ acts on G-}
The set $G^+$ is a group. Moreover, $G^+$ acts by conjugation on the set $G^-$.
\end{proposition}

\begin{proof}
The subset $G^+$ is a group, because $\alpha$ is a homomorphism.
Let now $g$ be in $G^+$ and $a$ in $G^-$. Then we have that 
$$\alpha(gag^{-1})=\alpha(g)\alpha(a)\alpha(g)^{-1}=ga^{-1}g^{-1}=
(gag^{-1})^{-1},$$ and so $gag^{-1}$ belongs to $G^-$.
\end{proof}

\begin{lemma}\label{cardinality +-}
The map $G/G^+\rightarrow G^-$ that is defined by $xG^+\mapsto x\alpha(x)^{-1}$ is a bijection. Moreover, $|G|=|G^+||G^-|$.
\end{lemma}

\begin{proof}
Denote by $\phi$ the map $G/G^+\rightarrow G^-$ that is defined by $xG^+\mapsto x\alpha(x)^{-1}$.
To show that $\phi$ is injective is a straightforward exercise. To prove that it is surjective, we take $b\in G^-$. Since the order of $G$ is odd, there exists a unique $a$ in $G$ such that $a^2=b$. 
The element $a$ belongs to $\gen{b}$, ando so $\alpha(a)=a^{-1}$. 
As a consequence, we have that $a\alpha(a)^{-1}=a^2=b$. We have proven that $\phi$ is a bijection, from which it follows that 
$|G|/|G^+|=|G^-|$.
\end{proof}

\begin{lemma}\label{bijection +-}
The map $G^+\times G^-\rightarrow G$, defined by $(x,y)\mapsto xy$, is a bijection. 
\end{lemma}

\begin{proof}
Let $(x,y)$ and $(z,t)$ be elements of $G^+\times G^-$ satisfying $xy=zt$. Then 
$ty^{-1}=z^{-1}x$. By Lemma \ref{G+ acts on G-}, the set $G^+$ is a subgroup of $G$, so $z^{-1}x\in G^+$. 
As a consequence, we get that $ty^{-1}=\alpha(ty^{-1})=t^{-1}y$. It follows that $t^2=y^2$ so, the order of $G$ being odd, the elements $t$ and $y$ coincide. Consequently, $(x,y)=(z,t)$ and the map is injective.
Now by Lemma \ref{cardinality +-}, the cardinalities of $G^+\times G^-$ and $G$ are the same and the given multiplication is also surjective.
\end{proof}

\begin{corollary}\label{abelian sum +-}
Assume $G$ is abelian. Then $G=G^+\oplus G^-$.
\end{corollary}

\begin{proof}
The sets $G^+$ and $G^-$ are both subgroups of $G$, because $G$ is abelian, and $G^+\cap G^-=\graffe{1}$, by Lemma \ref{trivial intersection}. It follows from Lemma \ref{bijection +-} that $G=G^+\oplus G^-$.
\end{proof}

\begin{lemma}\label{abelian>minus quotients}
Let $N$ be a normal $A$-stable subgroup of $G$ such that the restriction of $\alpha$ to $N$ equals the map $x\mapsto x^{-1}$. 
Assume moreover that the automorphism of $G/N$ that is induced by $\alpha$ is equal to the inversion map $g\mapsto g^{-1}$.
Then $G=G^-$ and $G$ is abelian.
\end{lemma}

\begin{proof}
From the assumptions it follows that, for each $x\in G$, the element $\alpha(x)x$ belongs to $N$. Since the order of $G$ is odd, if $x\in G^+$, then $x$ is actually an element of $N$. It follows that $G^+$ is contained in $N$, but $N$ is contained in $G^-$ by assumption. 
As a consequence of Lemma \ref{trivial intersection}, the subgroup $G^+$ is trivial so, as a consequence of Lemma \ref{bijection +-}, the group $G$ is equal to $G^-$. 
Let now $x$ and $y$ be elements of $G$. Then we have that  
$$y^{-1}x^{-1}=(xy)^{-1}=\alpha(xy)=\alpha(x)\alpha(y)=x^{-1}y^{-1},$$ 
and therefore $[x,y]=1$. The choice of $x$ and $y$ being arbitrary, the group $G$ is abelian.
\end{proof}

\begin{lemma}\label{plus quotients}
Let $N$ be a normal $A$-stable subgroup of $G$. Assume that the action of $A$ on $N$ and the induced action of $A$ on $G/N$ are both trivial. Then $G=G^+$.
\end{lemma}

\begin{proof}
Let $x\in G$. Then $\alpha(x)x^{-1}$ is an element of $N$. If $x$ belongs to $G^-$, then $x^{-2}$ belongs to $N$ so, the order of $G$ being odd, $x$ itself is an element of $N$. It follows that $G^-$ is a subset of $N$. The group $N$ is however contained in $G^+$ and, as a consequence of Lemma \ref{trivial intersection}, one gets $G^-=\graffe{1}$. It follows from Lemma \ref{bijection +-} that $G=G^+$.
\end{proof}

\begin{lemma}\label{ses+}
Let $1\rightarrow N\overset{f}{\longrightarrow}G\overset{g}{\longrightarrow} \Gamma\rightarrow 1 $ is a short exact sequence of $A$-groups. Denote by $f'$ and $g'$ the restrictions of $f$ and $g$ respectively to $N^+$ and $G^+$.
Then $1\rightarrow N^+\overset{f'}{\longrightarrow}G^+ \overset{g'}{\longrightarrow}\Gamma^+\rightarrow 1 $ is a short exact sequence of $A$-groups.
\end{lemma}

\begin{proof}
Since $f$ and $g$ are homomorphisms respecting the action of $A$, it suffices to prove the surjectivity of $g'$.
Let $\gamma\in\Gamma^+$. Then there exists $x\in G$ such that $g(x)=\gamma$ and, by Lemma \ref{bijection +-}, there exists $(a,b)\in G^+\times G^-$ such that $x=ab$. Now, $(g(a),g(b))\in\Gamma^+\times\Gamma^-$, because $g$ respects the action of $A$, and $\gamma=g(x)=g(ab)=g(a)g(b)$. It follows that $g(b)=g(a)^{-1}\gamma$, and so $g(b)\in\Gamma^+\cap\Gamma^-$. Thanks to Lemma \ref{trivial intersection}, we get that $g(b)=1$, so $\gamma$ is equal to $g(a)$.
\end{proof}



\begin{lemma}\label{order +- zonder jumps}
Let $p$ be an odd prime number and assume that $G$ is a finite $p$-group. 
Let $(G_i)_{i\geq 1}$ denote the lower central series of $G$ and assume that the automorphism of $G/G_2$
that is induced by $\alpha$ equals the inversion map $x\mapsto x^{-1}$.
Let $H$ be an $A$-stable subgroup of $G$ and let $\cor{O}$ and $\cor{E}$ be the collections of respectively all odd and all even positive integers. 
Then the following hold.
\begin{itemize}
 \item[$1$.] Let $j\in\Z_{\geq 1}$. Then $H^+\cap G_j\neq H^+\cap G_{j+1}$ if and only if $H\cap G_j\neq H\cap G_{j+1}$ and $j\in\cor{E}$.
 \item[$2$.] One has $|H^+|=\prod_{j\in\cor{E}}|H\cap G_j:H\cap G_{j+1}|$. 
 \item[$3$.] One has $|H^-|=\prod_{j\in\cor{O}}|H\cap G_j:H\cap G_{j+1}|$.
\end{itemize} 
\end{lemma}

\begin{proof}
For all $j\in\Z_{\geq 1}$, we define $V_j=(H\cap G_j)/(H\cap G_{j+1})$ and
we consider the short exact sequence
\[
1\rightarrow H\cap G_{j+1}\rightarrow H\cap G_j\rightarrow V_j\rightarrow 1
\]
of $A$-groups. We first prove ($1$) and ($2$) together. 
If $j\in\Z_{\geq 1}$, we note that $(H\cap G_j)^+=H^+\cap G_j$, so 
Lemma \ref{ses+} yields
\[|H^+\cap G_j:H^+\cap G_{j+1}|=|V_j^+|.\]
From Lemma \ref{action chi^i general} it follows that $A$ acts on $G_j/G_{j+1}$ by scalar multiplication by $(-1)^j$, and so, whenever $j$ is an odd positive integer, Lemma \ref{trivial intersection} yields that the cardinality of $(G_j/G_{j+1})^+$ is equal to $1$. Since, for all $j\in\Z_{\geq 1}$, the groups $V_j$ and $(H\cap G_j)G_{j+1}/G_{j+1}$ are isomorphic, it follows that, for each odd $j$, the cardinality of $V_j^+$ is equal to $1$. Moreover, if $j$ is even, then $V_j$ is equal to $V_j^+$.
The cardinality of $H^+$ being equal to $\prod_{j\geq 1}|H^+\cap G_j:H^+\cap G_{j+1}|$, we get that 
\[|H^+|=\prod_{j\geq 1}|V_j^+|=\prod_{j\in\cor{E}}|V_j^+|=\prod_{j\in\cor{E}}|V_j|.\]
This proves ($1$) and ($2$). 
We now prove ($3$). By Lemma \ref{cardinality +-}, the cardinality of $|H^-|$ is equal to $|H|/|H^+|$. Since 
$|H|=\prod_{j\geq 1}|V_j|$, it follows from ($2$) that 
\[|H^-|=\frac{\prod_{j\geq 1}|V_j|}{\prod_{j\in\cor{E}}|V_j|}=\prod_{j\in\cor{O}}|V_j|.\]
\end{proof}

\begin{lemma}\label{conjugate stable iff in nor-times-G+}
Let $H$ be an $A$-stable subgroup of $G$ and let $g$ be an element of $G$.
Then the following are equivalent.
\begin{itemize}
 \item[$1$.] The subgroup $gHg^{-1}$ is $A$-stable.
 \item[$2$.] The element $g$ belongs to $G^+\nor_G(H)$.
\end{itemize}  
\end{lemma}

\begin{proof}
Let $I=\nor_G(H)$; then $I$ is $A$-stable, because $H$ is.
We first prove that ($1$) implies ($2$).  
Assume that the subgroup $gHg^{-1}$ is $A$-stable, so $\alpha(gHg^{-1})=gHg^{-1}$.
In particular, the element $g^{-1}\alpha(g)$ belongs to $I$, and thus $\alpha(gI)=\alpha(g)I=gI$.
Since the cardinality of $I$ is odd and $A$ has order $2$, 
there is an element $x$ in $I$ such that $\alpha(gx)=gx$. 
For such an element $x$, we then have that $gx\in G^+$, so $g=gx\cdot x^{-1}\in G^+I$.
Assume now ($2$) is satisfied; we prove ($1$). Since $g$ belongs to $G^+\nor_G(H)$, 
there exists $(\gamma,n)\in G^+\times\nor_G(H)$ such that $g=\gamma n$. 
For such pair $(\gamma,n)$, we have that $gHg^{-1}=\gamma H\gamma^{-1}$, and therefore
$\alpha(gHg^{-1})=gHg^{-1}$. This proves ($1$).
\end{proof}

\section{Jumps and width}\label{section jumps}

Let $p$ be a prime number and let $G$ be a finite $p$-group.
Denote by $(G_i)_{i\geq 1}$ the lower central series of $G$ (see Section \ref{section commutators}). 
If $x$ is a non-trivial element of $G$, then there exists a positive integer $d$ such that $x\in G_d\setminus G_{d+1}$. The number $d$ is called the \indexx{depth} of $x$ (in $G$) and it is denoted by $\dpt_G(x)$. 
Let now $H$ be a subgroup of $G$ and let $j$ be a positive integer. 
The \emph{$j$-th width of $H$ in $G$} is\index{width}
\[\wt_H^G(j)=\log_p|H\cap G_j:H\cap G_{j+1}|.\]
We observe that, if $\pi_j:G_j\rightarrow G_j/G_{j+1}$ denotes the canonical projection, then $\pi_j(H\cap G_i)$ has cardinality  $p^{\wt_H^G(j)}$.
An index $j$ is a \indexx{jump} \emph{of $H$ in $G$} if $\wt_H^G(j)\neq 0$ and, whenever it will be clear that $j$ is a jump of $H$ in $G$, we will refer to $\wt_H^G(j)$ as the \emph{width of $j$ in $G$}.
If $G=H$, we denote the \emph{$j$-th width of $G$} by $\wt_G(j)$ instead of $\wt_G^G(j)$ and, in several results, we will lighten the notation even further by writing $w_j=\wt_G(j)$.
The \emph{width of $G$} is defined as \index{width}
$\wt(G)=\max_{i\geq 1}\wt_G(i)$; for a generalization to general pro-$p$-groups, see \cite{leedham}.   

\begin{lemma}\label{jumps and depth}
Let $p$ be a prime number and let $j$ be a positive integer. 
Let moreover $G$ be a finite $p$-group and let $H$ be a subgroup of $G$.
Then $j$ is a jump of $H$ in $G$ if and only if $H$ contains an element of depth $j$ in $G$.
\end{lemma}

\begin{proof}
Straightforward.
\end{proof}

\begin{lemma}\label{same jumps conjugates}
Let $p$ be a prime number. 
Let moreover $G$ be a finite $p$-group and let $H$ be a subgroup of $G$. 
Then, for all $\alpha\in\Aut(G)$, the groups $H$ and $\alpha(H)$ have the same jumps in $G$.
\end{lemma}

\begin{proof}
Let $\alpha$ be an automorpshism of $G$ and let $j$ be a positive integer.
By Lemma \ref{jumps and depth}, the integer $j$ is a jump of $H$ if and only if there exists $x\in H$ such that $\dpt_G(x)=j$. As the elements of the lower central series are characteristic in $G$, we have that $\dpt_G(\alpha(x))=\dpt_G(x)$, and thus we are done.
\end{proof}

\begin{lemma}\label{order product orders jumps}
Let $p$ be a prime number and let $G$ be a finite $p$-group. Let moreover $H$ be a subgroup of $G$ and call $\cor{J}$ the collection of all jumps of $H$ in $G$. 
Then $|H|=\prod_{j\in\cor{J}}p^{\wt_H^G(j)}$. 
\end{lemma}

\begin{proof}
It follows directly from the definitions of jumps.
\end{proof}

\noindent
In the next proposition, we use the notation introduced in Section \ref{section involutions}.

\begin{lemma}\label{order +- jumps}
Let $p$ be an odd prime number and let $G$ be a finite $p$-group. Let $A=\gen{\alpha}$ be a multiplicative group of order $2$ acting on $G$.
Let 
$\chi:A\rightarrow\graffe{\pm 1}$ be an isomorphism. 
Let $(G_i)_{i\geq 1}$ denote the lower central series of $G$ and assume that the automorphism of $G/G_2$
that is induced by $\alpha$ equals the inversion map $x\mapsto x^{-1}$.
Let $H$ be an $A$-stable subgroup of $G$ and let $\cor{O}$ and $\cor{E}$ be the collections of respectively all odd and all even jumps of $H$ in $G$. Assume that the induced action of $A$ on $G/G_2$ is through $\chi$. 
Then the following hold.
\begin{itemize}
 \item[$1$.] One has $|H^+|=\prod_{j\in\cor{E}}p^{\wt_H^G(j)}$ and $\cor{E}$ is the set of jumps of $H^+$ in $G$.
 \item[$2$.] One has $|H^-|=\prod_{j\in\cor{O}}p^{\wt_H^G(j)}$.
\end{itemize} 
\end{lemma}

\begin{proof}
Apply the new dictionary to Lemma \ref{order +- zonder jumps}.
\end{proof}


\chapter{Intense automorphisms}\label{CH formulation}

Let $G$ be a group. 
An automorphism $\alpha$ of $G$ is \emph{intense}
\index{intense automorphism} if for every subgroup $H$ of $G$ there 
exists $g\in G$ such that $\alpha(H)=gHg^{-1}$.
We denote by $\Int(G)$ the collection of all intense automorphisms of $G$, which is easily seen to be a normal subgroup of $\Aut(G)$. 
\vspace{8pt} \\
\noindent
In this chapter we will prove some basic properties of intense automorphisms and formulate the main research question of this 
thesis. Among others, we will prove the following result.

\begin{theorem}\label{theorem abelian}
Let $p$ be a prime number and let $G$ be a finite $p$-group. Then $\Int(G)$ is isomorphic to $S\rtimes C$, where $S$ is a Sylow $p$-subgroup of $\Int(G)$ and $C$ is a subgroup of $\F_p^*$. Moreover, if $G$ is non-trivial abelian, then $C=\F_p^*$.
\end{theorem}

\section{Basic properties}\label{section intense properties}

\noindent
Section \ref{section intense properties} is devoted to basic properties of intense automorphisms. Most of the notation used appears in the List of Symbols, at the beginning of this thesis.

\begin{proposition}
Let $G$ be a group.
Then $\Inn(G)\subseteq \Int(G)\subseteq\Aut(G)$ and $\Int(G)$ is normal in $\Aut(G)$.
\end{proposition}

\begin{proof}
Straightforward application of the definitions. 
\end{proof}

\noindent
We recall that, if $A$ is a group acting on a set $X$, a subset $Y$ of $X$ is $A$-stable if the action of $A$ on $X$ restricts to an action of $A$ on $Y$ (see Section \ref{section characters}).

\begin{lemma}\label{intense properties}
Let $G$ be a group and let $N$ be a normal subgroup of $G$. Then the following hold.
\begin{itemize}
 \item[$1$.] The subgroup $N$ is $\Int(G)$-stable.
 \item[$2$.] The natural projection $G\rightarrow G/N$ induces a well-defined homomorphism $\Int(G)\rightarrow\Int(G/N)$, by means 
 			 of $\alpha\mapsto (xN\mapsto \alpha(x)N)$.
 \item[$3$.] Assume $N$ is contained in $\ZG(G)$. Then the image of the homomorphism
             $\Int(G)\rightarrow\Aut(N)$, sending $\alpha$ to
             $\alpha_{|N}$, is contained in $\Int(N)$.                       
\end{itemize}
\end{lemma}

\begin{proof}
($1$) Intense automorphisms send every subgroup to a conjugate and therefore each normal subgroup to itself. 
($2$) The map is well-defined as a consequence of ($1$) and it is a homomorphism by construction.
($3$) This follows from the fact that conjugation in $G$ restricts to the trivial map on each subgroup of $\ZG(G)$.
\end{proof}

\noindent
In the following proposition, let $\omega:\F_p^*\rightarrow\Z_p^*$ be the 
Teichm\"{u}ller character at $p$ as defined in Section \ref{section characters}.

\begin{lemma}\label{intense vector space}
Let $p$ be a prime number and let $V$ be a vector space over $\F_p$. 
Then there exists a unique injective homomorphism $\lambda:\Int(V)\rightarrow\Z_p^*$ such that 
the following hold.
\begin{itemize}
 \item[$1$.] The group $\Int(V)$ acts on $V$ through $\lambda$.
 \item[$2$.] If $V\neq0$, then $\lambda(\Int(V))=\omega(\F_p^*)$.
\end{itemize}
\end{lemma}

\begin{proof}
If $V=0$, define $\lambda:\id_V\mapsto 1$.
Assume $V\neq 0$. Since $V$ is abelian, every one-dimensional subspace of $V$ is stable under the action of $\Int(V)$. It follows that, for all 
$v\in V\setminus\graffe{0}$ and $\alpha\in\Int(V)$, there exists (a unique) $\mu(\alpha,v)\in\F_p^*$ such that $\alpha(v)=\mu(\alpha,v) v$. We will show that $\mu(\alpha,v)$ is independent of the choice of $v$.
To this end, let $\alpha\in\Int(V)$ and let $v$ and $w$ be elements of $V\setminus\{0\}$. If $v$ and $w$ are linearly dependent, then $\mu(\alpha,v)=\mu(\alpha,w)$. We assume that $v$ and $w$ are linearly independent. From the linearity of $\alpha$, it follows that $\mu(\alpha, v+w)(v+w)=\mu(\alpha,v)v+\mu(\alpha, w)w$. The vectors $v$ and $w$ are linearly independent, so 
$\mu(\alpha,v)=\mu(\alpha,v+w)=\mu(\alpha,w)$, as required. We fix $v\in V\setminus\graffe{0}$ and define $\mu:\Int(V)\rightarrow\F_p^*$ by $\alpha\mapsto\mu(\alpha,v)$. 
The map $\mu$ is well-defined and it is an injective homomorphism of groups by construction. Moreover, $\mu$ is surjective, because scalar multiplication by any element of $\F_p^*$ is an intense automorphism of $V$.
We define $\lambda=\omega\circ\mu$. 
Then $\Int(V)$ acts on $V$ through $\lambda$ and the image of $\lambda$ is equal to $\omega(\F_p^*)$ by construction.
The uniqueness of $\lambda$ follows from Lemma \ref{distinct characters on same gp}.
\end{proof}

\noindent
We recall that, an action of a group $C$ on a group $B$ is a homomorphism $C\rightarrow\Aut(B)$ (see Section \ref{section characters}).

\begin{definition}
Let $C$ be a finite group acting on a finite group $B$. Assume moreover that both $B$ and $C$ act on a set $X$. 
The actions are said to be \emph{compatible}
\index{compatible action} if for all $x\in X$, $b\in B$, and $c\in C$, one has $c(bx)=(cb)(cx)$.
\end{definition}

\begin{lemma}[Glauberman's lemma]\label{glauberman lemma}
Let $G$ and $A$ be finite groups of coprime orders. Assume that at least one of $A$ and $G$ is solvable. Assume $A$ acts on $G$ and that each of them acts on some 
set $X$, where the action of $G$ is transitive. Finally, assume the three actions are compatible. 
Then there exists an $A$-stable element in $X$.    
\end{lemma}

\begin{proof}
See \cite[Lemma $3.24$]{isaacs}.
\end{proof}

\begin{lemma}[An equivalent condition]\label{equivalent intense coprime}
Let $G$ be a finite group and let $\alpha\in\Aut(G)$ be of order coprime to the order of $G$. Let $H$ and $N$ be subgroups of $G$ and assume that $\alpha(N)=N$. Then the following are equivalent.
\begin{itemize}
 \item[$1$.] There exists $a\in N$ such that $\alpha(H)=aHa^{-1}$.
 \item[$2$.] There exists $b\in N$ such that $bHb^{-1}$ is $\gen{\alpha}$-stable.
\end{itemize}
\end{lemma}

\begin{proof}
$(2)\Rightarrow(1)$ By assumption there exists an element $b\in N$ such that 
$bHb^{-1}=\alpha(bHb^{-1})=\alpha(b)\alpha(H)\alpha(b)^{-1}$. Define $a=\alpha(b)^{-1}b$.
$(1)\Rightarrow(2)$ Write $X=\{gHg^{-1}\, :\, g\in N\}$. Then $N$ acts on $X$ by conjugation and $\gen{\alpha}$ acts on $X$ by assumption. The actions are compatible and the action of $N$ is transitive. By 
Lemma \ref{glauberman lemma}, there exists an element of $X$ that is fixed by $\alpha$. 
\end{proof}

\begin{lemma}\label{equivalent intense coprime-pgrps}
Let $G$ be a finite group and let $\alpha$ be an automorphism of $G$ of order coprime to $|G|$. Then $\alpha\in\Int(G)$ if and only if each subgroup of $G$ has an $\gen{\alpha}$-stable conjugate.
\end{lemma}

\begin{proof}
Take $N=G$ in Lemma \ref{equivalent intense coprime}.
\end{proof}

\begin{lemma}\label{orbits}
Let $G$ be a finite group and let $\alpha\in\Int(G)$ be of order coprime to the order of $G$.
Let $X$ be a collection of subgroups of $G$ on which $G$ acts by conjugation and let $X^+=\graffe{H\in X\, :\, \alpha(H)=H}$. 
Then $$|X|\leq\sum_{H\in X^+} |G:\nor_G(H)|.$$ Equality holds if and only if the elements of $X^+$ are pairwise non-conjugate in $G$.
\end{lemma}

\begin{proof}
Let $\cor{C}$ be the collection of orbits of $X$ under $G$. By Lemma \ref{equivalent intense coprime-pgrps}, there exists a subset $\cor{R}$ of $X^+$ whose elements are representatives for the elements of $\cor{C}$. It follows that 
$$|X|=\sum_{C\in\cor{C}}|C|=\sum_{H\in\cor{R}}|G:\nor_G(H)|\leq
\sum_{H\in X^+}|G:\nor_G(H)|.$$
Equality holds if and only if $\cor{R}=X^+$.
\end{proof}

\section{The main question}\label{section main}

\noindent
Let $p$ be a prime number and let $G$ be a finite $p$-group. We recall that $\Phi(G)=[G,G]G^p$ and so $G/\Phi(G)$ is a vector space over $\F_p$ (see Section \ref{section p-gps}). In Section \ref{section main} we build the foundation for our theory and we give the dictionary that we will use throughout the whole thesis. We will also prove the following result. 

\begin{proposition}\label{proposition 2gps}
Let $G$ be a finite $2$-group. Then $\Int(G)$ is a finite $2$-group.
\end{proposition}

\begin{definition}
Let $p$ be a prime number and let $G$ be a finite $p$-group. 
The \emph{intense character} \index{intense character} of $G$ is the homomorphism $\chi_G:\Int(G)\rightarrow\Z_p^*$ that is gotten from composition of the following.
\begin{itemize}
 \item[$\circ$] The homomorphism $\Int(G)\rightarrow\Int(G/\Phi(G))$ from Lemma \ref{intense properties}($2$).
 \item[$\circ$] The homomorphism $\lambda:\Int(G/\Phi(G))\rightarrow\Z_p^*$ from Lemma \ref{intense vector space}.
\end{itemize}
\end{definition}

\vspace{5pt}

\begin{lemma}\label{formulation}
\pgp\ Let moreover $\chi_G:\Int(G)\rightarrow\Z_p^*$ be the intense character of $G$.
Then the group $\ker\chi_G$ is the unique Sylow $p$-subgroup of $\Int(G)$.
\end{lemma}

\begin{proof}
If $G$ is the trivial group, then $\Int(G)=\ker\chi_G=\graffe{1}$ and $\graffe{1}$ is a Sylow $p$-subgroup of $\Int(G)$. 
Assume now $G$ is non-trivial and set $V=G/\Phi(G)$. By Lemma \ref{intense vector space}, the map $\lambda: V\rightarrow\Int(V)$ is injective and so, with the notation of Lemma \ref{kernel of reduction frattini}, the subgroup $\ker\chi_G$ is equal to $\Int(G)\cap\ker\phi$.
As a consequence of Lemma \ref{kernel of reduction frattini}, the kernel of $\chi_G$ is a normal Sylow $p$-subgroup of $\Int(G)$. Since $\Int(G)$ acts on the collection of its Sylow $p$-subgroup in a transitive manner, $\Int(G)$ has a unique Sylow $p$-subgroup, namely $\ker\chi_G$.
\end{proof}

\begin{definition}
Let $p$ be a prime number and let $G$ be a finite $p$-group. Let $\chi_G:\Int(G)\rightarrow\Z_p^*$ be the intense character of 
$G$. The \emph{intensity} of $G$ \index{intensity} is the value $|\Int(G):\ker\chi_G|$ and it is 
denoted by $\inte(G)$.
\end{definition}

\noindent
We remark that the intensity of a $p$-group $G$ is equal to the size of the image of the intense character $\chi_G$ 
inside $\omega(\F_p^*)$. In particular, if $G$ is a $2$-group, then its intensity is always $1$.

\begin{lemma}\label{intense complement}
Let $p$ be a prime number and let $G$ be a finite $p$-group. Let moreover $\chi_G:\Int(G)\rightarrow\Z_p^*$ be the intense character
of $G$.
Then $\inte(G)$ divides $p-1$ and $\ker\chi_G$ has a cyclic complement in $\Int(G)$ of order $\inte(G)$.
\end{lemma}

\begin{proof}
The image of $\chi_G$ is a subgroup of $\omega(\F_p^*)$, which has order $p-1$. It follows that $\inte(G)$ divides $p-1$. 
Now, by Proposition \ref{formulation}, the kernel of $\chi_G$ is the unique Sylow $p$-subgroup of $G$ and it is therefore normal.
The group $\ker\chi_G$ has a complement $C$ in $\Int(G)$, by the Schur-Zassenhaus theorem, and $C$ is cyclic because it is 
isomorphic to a subgroup of $\F_p^*$.
\end{proof}

\noindent
Proposition \ref{proposition 2gps} now follows from Lemmas \ref{formulation} and \ref{intense complement}. 
\vspace{8pt} \\
\noindent
The major goal of this thesis if classifying all pairs $(p,G)$ where $p$ is a prime number and $G$ is a finite $p$-group with $\inte(G)>1$. 
Starting from the next chapter, we will therefore often be working with odd primes. 
Explicit assumptions will be made at the beginning of each section.

\section{The abelian case}\label{section abelian}

The main result of Section \ref{section abelian} is the following.

\begin{proposition}\label{proposition abelian}
Let $p$ be a prime number and let $G$ be a finite non-trivial abelian $p$-group. 
Then $\inte(G)=p-1$.
\end{proposition}

\begin{lemma}\label{intensity of quotients}
Let $p$ be a prime number and let $G$ be a finite $p$-group.
Let $N$ be a normal subgroup of $G$. If $N\neq G$, then $\inte(G)$ divides $\inte(G/N)$.
\end{lemma}

\begin{proof}
Assume $N\neq G$; then $G$ is non-trivial.
Let $\phi:\Int(G)\rightarrow\Int(G/N)$ be as in Lemma \ref{intense properties}($2$). 
The subgroup $N\Phi(G)$ is different from $G$, thanks to Lemma \ref{subgroup times frattini}, and therefore $\Phi(G)N/N\neq G/N$.
The groups $(G/N)/\Phi(G/N)$ and $G/(\Phi(G)N)$ being isomorphic (and non-trivial!), it follows that
$\chi_G=\chi_{G/N}\circ\phi$.
The image of $\chi_G$ is a subgroup of the image of $\chi_{G/N}$ and thus $\inte(G)$ divides $\inte(G/N)$.
\end{proof}

\noindent
We recall that a group $A$ acts through a character on a finite abelian $p$-group $G$ if there exists a homomorphism $\chi:A\rightarrow\Z_p^*$ such that, for all $a\in A$ and $x\in G$, one has $ax=\chi(a)x$. For more details, see Section \ref{section characters}.

\begin{lemma}\label{abelian case}
Let $p$ be a prime number and let $G$ be a finite abelian $p$-group.
Let $\alpha$ be intense of order dividing $\inte(G)$ and write $\chi={\chi_G}_{|\gen{\alpha}}$.
Then $\gen{\alpha}$ acts on $G$ through $\chi$ and, if $G$ is non-trivial, then
$\inte(G)=p-1$. 
\end{lemma}

\begin{proof}
Write $A=\gen{\alpha}$.
If $G$ is the trivial group, then the only automorphism of $G$ is the identity, which is intense. Assume now $G$ is non-trivial. 
The group $\omega(\F_p^*)$ acts on $G$ (as described at the beginning of Section \ref{section characters}) via intense automorphisms 
and it induces scalar multiplication by elements of $\F_p^*$ on $G/\Phi(G)$. The image of the intense character of $G$ is thus $\omega(\F_p^*)$, and so, $\inte(G)=p-1$.
Let now $\Omega$ denote the image of $\omega(\F_p^*)\rightarrow\Int(G)$ and write $\Omega=\gen{\beta}$. Then 
$\Int(G)=\ker\chi_G\rtimes\Omega$, and, as a consequence of Schur-Zassenhaus, there exist $m\in\Z_{\geq 0}$ and
$\gamma\in\ker\chi_G$ such that $\alpha=\gamma\beta^m\gamma^{-1}$. 
We get
$$\chi(\alpha)=\chi_G(\alpha)=\chi_G(\gamma\beta^m\gamma^{-1})=\chi_G(\beta^m).$$
Since each homomorphism of abelian groups is $\Z_p$-linear and $\Omega$ acts on $G$ through ${\chi_G}_{|\Omega}$, the group $A$ acts on $G$ through $\chi$.
\end{proof}

\noindent
We remark that Proposition \ref{proposition abelian} is a special case of Lemma \ref{abelian case}. Moreover,
Theorem \ref{theorem abelian} is proven by combining Lemmas \ref{formulation}, \ref{intense complement}, and Proposition \ref{proposition abelian}.

\begin{corollary}\label{centre one character}
\pgp\
Let moreover $\alpha$ be an intense automorphism of $G$ of order dividing $\inte(G)$.
Then $\gen{\alpha}$ acts on the centre of $G$ through a character $\gen{\alpha}\rightarrow\Z_p^*$.
\end{corollary}

\begin{proof}
Let $\zeta:\Int(G)\rightarrow\Int(\ZG(G))$ be the map from Lemma \ref{intense properties}($3$) and 
define $\sigma={\chi_{\ZG(G)}}_{|\gen{\zeta(\alpha)}}\circ\zeta_{|\gen{\alpha}}$. Lemma \ref{formulation} yields that $\gen{\alpha}$ acts on $\ZG(G)$ through $\sigma$.
\end{proof}

\begin{lemma}\label{action chi^i}
Let $p$ be a prime number and let $G$ be a finite $p$-group. 
Let $\alpha$ be intense of order dividing $\inte(G)$ and write $\chi={\chi_G}_{|\gen{\alpha}}$.
Denote by $(G_i)_{i\geq 1}$ the lower central series of $G$.
Then, for all $i\in\Z_{\geq 1}$, the induced action of $\gen{\alpha}$ on $G_i/G_{i+1}$ is through $\chi^i$.
\end{lemma}

\begin{proof}
Denote $A=\gen{\alpha}$.
As a consequence of Proposition \ref{intense properties}($2$), the action of $A$ on $G$ induces an action of $A$ on $G/G_2$. By Proposition \ref{abelian case}, the action of $A$ on $G/G_2$ is through $\chi$. We now apply Lemma \ref{action chi^i general}.
\end{proof}


\chapter{Intensity of groups of class 2}\label{CH class 2}

\noindent
The main goal of this thesis, as stated in Section \ref{section main}, is to classify all finite $p$-groups whose group of intense automorphisms is not itself a $p$-group. 
We will proceed to a classification by separating the cases according to the class of the $p$-groups. We remind the reader that a finite $p$-group is always nilpotent and that its (nilpotency) class is defined to be the number of non-trivial successive quotients of the lower central series (see Section \ref{section commutators}). If the class is $0$, the group is trivial and the intensity is $1$. For the case in which the class is $1$ (non-trivial abelian case) we refer to Chapter \ref{CH formulation}. In this chapter we study the case in which the class is equal to $2$. We prove the following main result.

\begin{theorem}\label{theorem class2 complete}
Let $p$ be a prime number and let $G$ be a finite $p$-group of class $2$. 
Then the following are equivalent.
\begin{itemize}
 \item[$1$.] One has $\inte(G)\neq 1$.
 \item[$2$.] The group $G$ is extraspecial of exponent $p$.
 \item[$3$.] The prime $p$ is odd and $\inte(G)=p-1$.
\end{itemize}
\end{theorem}


\section{Small commutator subgroup}\label{section small commutator}

\noindent
Let $p$ be a prime number.
We recall that a group $A$ acts on a finite abelian $p$-group $G$ through a character if there exists a homomorphism 
$\chi:A\rightarrow\Z_p^*$ such that, for all $x\in G$, $a\in A$, one has $ax=\chi(a)x$. For more detail about actions through characters see Section \ref{section characters}. 
\vspace{8pt} \\
\noindent
Until the end of Section \ref{section small commutator}, the following assumptions will be valid.
Let $p$ be a prime number and let $G$ denote a finite $p$-group of nilpotency class $2$ (see Section \ref{section commutators}). Let moreover $\alpha$ be intense of order $\inte(G)$. Write $A=\gen{\alpha}$ and $\chi={\chi_G}_{|A}$. 
Assume that the intensity of $G$ is greater than $1$. It follows that $G$ is non-trivial and $p$ is odd (see Sections \ref{section main} and \ref{section abelian}). 
We will keep this notation until the end of this section, together with the one from the List of Symbols.

\begin{lemma}\label{chi^2}
Assume $G_2$ has exponent $p$.
Then $\Phi(G)$ is central and $A$ acts on $G_2$ is through $\chi^2$.
\end{lemma}

\begin{proof}
The Frattini subgroup of $G$ is central by Lemma \ref{class 2 frattini} and
$A$ acts on $G_2$ through $\chi^2$ by Lemma \ref{action chi^i}.
\end{proof}

\begin{lemma}\label{chi neq chi2}
The homomorphisms $\chi,\chi^2:A\rightarrow\Z_p^*$ are distinct.
\end{lemma}

\begin{proof}
Assume $\chi=\chi^2$. Then $\chi(\alpha)=\chi(\alpha)^2$ and $\chi(\alpha)=1$. It follows that the intensity of $G$ is equal to $1$. Contradiction.
\end{proof}

\begin{lemma}\label{centre frattini lemma}\label{centre commutator lemma}
Assume $G_2$ has exponent $p$. Then $\ZG(G)=\Phi(G)=G_2$ and $A$ acts on $\ZG(G)$ through $\chi^2$.
\end{lemma}

\begin{proof}
The group $G_2$ is a non-trivial subgroup of $\ZG(G)$ and, by Lemma \ref{chi^2}, the group $A$ acts on $G_2$ through $\chi^2$. By Corollary 
\ref{centre one character}, the group $A$ acts on $\ZG(G)$ through a character and, as a consequence of Lemma \ref{distinct characters on same gp}, the action of $A$ on the centre is through $\chi^2$. On the other hand, by Lemma \ref{action chi^i}, the induced action of $A$ on $G/G_2$ is through $\chi$. The group $A$ acts hence on  $\ZG(G)/G_2$ both through $\chi$ and $\chi^2$. The characters $\chi$ and $\chi^2$ being distinct, Lemma \ref{distinct characters on same gp} yields $\ZG(G)=G_2$. 
By Lemma \ref{chi^2} the subgroup $\Phi(G)$ is central, and thus $G_2=\Phi(G)=\ZG(G)$.
\end{proof}

\begin{lemma}\label{extraspecial exponent}
Assume $G_2$ has order $p$. Then $G$ is an extraspecial group of exponent $p$.
\end{lemma}

\begin{proof}
Thanks to Lemma \ref{centre commutator lemma} we are only left with showing that $G$ has exponent $p$. Assume by contradiction there exists $g\in G$ of order $p^2$ and write $H=\gen{g}$. Then $H^p$ has order $p$. Now, $H^p$ is contained in $\Phi(G)$ and, as a consequence of Lemma \ref{centre commutator lemma}, the Frattini subgroup of $G$ has itself order $p$. It follows that $H^p=\Phi(G)$ and, in particular, $H$ contains $\Phi(G)$. The group $G_2$ is equal to $\Phi(G)$, by Lemma \ref{centre frattini lemma}, so the group $H$ is normal. By Lemma \ref{intense properties}($1$), the subgroup $H$ is $A$-stable. As a consequence of Lemma \ref{action chi^i}, the actions of $A$ on $H/G_2$ and $G_2$ are respectively through $\chi$ and $\chi^2$ and, by Lemma \ref{chi neq chi2}, the characters $\chi$ and $\chi^2$ are distinct.
From Theorem \ref{lambda mu} it follows that the groups
$H$ and $(H/G_2)\oplus G_2$ are isomorphic. Contradiction.
\end{proof}

\begin{lemma}\label{elementary abelian}
Let $Q$ be a finite $p$-group of both class and intensity greater than $1$. Denote by $(Q_i)_{i\geq 1}$ the lower central series of $Q$. 
Then, for all $i\in\Z_{\geq 1}$, the exponent of $Q_i/Q_{i+1}$ divides $p$. 
\end{lemma}

\begin{proof}
We work by induction on $i$ and we start by assuming $i=1$. Let $M$ be a normal subgroup of $Q$ that is contained in $Q_2$ with index $p$; the group $M$ exists by Lemma \ref{normal index p}.
Thanks to the isomorphism theorems, the groups $Q/Q_2$ and $(Q/M)/(Q_2/M)$ are isomorphic. 
We write $\overline{Q}=Q/M$ and use the bar notation for the subgroups of $\overline{Q}$.
Then $\overline{Q_2}=[\overline{Q},\overline{Q}]$ has order $p$ and
$\overline{Q}$ has intensity greater than $1$, by Lemma \ref{intensity of quotients}.
From Lemma \ref{extraspecial exponent}, it follows that 
$\overline{Q}/\overline{Q_2}$ is elementary abelian and therefore so is $Q/Q_2$.
Assume now that $i$ is greater than $1$ and that the result holds for all indices smaller than $i$.
The property of being annihilated by $p$ is preserved by tensor products and surjective homomorphisms so, as a consequence of Lemma \ref{tensor LCS}, the exponent of $Q_i/Q_{i+1}$ divides $p$. 
\end{proof}

\begin{corollary}\label{centre=commutator class2}
Let $Q$ be a finite $p$-group of nilpotency class $2$. If $\inte(Q)>1$, then $\ZG(Q)=Q_2$.
\end{corollary}

\begin{proof}
By Lemma \ref{elementary abelian}, the commutator subgroup of $Q$ has exponent $p$. 
To conclude, apply Lemma \ref{centre commutator lemma}. 
\end{proof}

\section{More general setting}\label{section general class 2}

\noindent
Throughout this whole section (Section \ref{section general class 2}), let $p$ be a prime number and let $G$ be a finite $p$-group of class $2$ and intensity greater than $1$. 
It follows from the work done in Sections \ref{section main} and \ref{section abelian} that $G$ is not trivial and $p$ is odd. 
Let $\alpha$ be intense of order $\inte(G)$ and write $A=\gen{\alpha}$ and $\chi={\chi_G}_{|A}$. 
We denote by $V$ and $Z$ respectively $G/G_2$ and $G_2$ and by $\pi$ the canonical projection $G\rightarrow V$. From Lemma \ref{elementary abelian} it follows that
both $V$ and $Z$ are vector spaces over $\F_p$. 
By Corollary \ref{centre=commutator class2}, the non-trivial subgroup $Z$ is equal to $\ZG(G)$ and, as a consequence of Lemma \ref{class 2 bilinear map}, the map 
$\phi:V\times V\rightarrow Z$ that is induced by the commutator map is alternating.

\begin{lemma}\label{phi_H}
Let $H$ be a linear subspace of $Z$ of codimension $1$. Then the map 
$\phi_H:V\times V\rightarrow Z/H$, defined by $(x,y)\mapsto \phi(x,y)+H$, is non-degenerate.
\end{lemma}

\begin{proof}
The subgroup $H$ is contained in the centre $Z$ and is therefore a normal subgroup of $G$. It follows from Lemma \ref{intensity of quotients} that $\inte(G/H)>1$.
As a consequence of Lemma \ref{extraspecial exponent}, the group $G/H$ is extraspecial, and so,
thanks to Lemma \ref{non-degenerate extraspecial}, the map $\phi_H:V\times V\rightarrow Z/H=[G/H,G/H]$ is non-degenerate.
\end{proof}

\begin{corollary}\label{dimV}
There exists $n\in\Z_{>0}$ such that $\dim V=2n$.
\end{corollary}

\begin{proof}
Let $H$ be a linear subspace of $Z$ of codimension $1$ and let $\phi_H$ be as in Lemma \ref{phi_H}. Then $\phi_H$ is non-degenerate, and so, by Lemma \ref{dimension isotropic}, the dimension of $V$ is even. The dimension is positive, because $G$ has class $2$.
\end{proof}

\begin{lemma}\label{graph+complements} 
Let $G$ be a group, let $N$ be a central subgroup, and let $H$ be a complement of $N$ in $G$. 
Let moreover $\cor{C}_N$ be the collection of complements of $N$ in $G$ and, for all $f\in\Hom(H,N)$, call $\cor{G}_f=\{f(h)h \, :\, h\in H\}$. 
Then the map $\Hom(H,N)\rightarrow\cor{C}_N$, given by $f\mapsto\cor{G}_f$, is well-defined and bijective. 
\end{lemma}

\begin{proof}
Straightforward.
\end{proof}

\vspace{8pt}
\noindent 
We recall that, as defined in Section \ref{section linear algebra}, an isotropic subspace of $V$ is a linear subspace $T$ of $V$ such that $\phi(T\times T)=0$. 
\vspace{8pt}

\begin{lemma}\label{isotropic abelian}
Let $T$ be a linear subspace of $V$. Then 
$T$ is isotropic if and only if $\pi^{-1}(T)$ is abelian. 
\end{lemma}

\begin{proof}
The subspace $T$ is isotropic if and only if $\phi(T\times T)=0$, which happens if and only if $[\pi^{-1}(T),\pi^{-1}(T)]=1$.
\end{proof}

\vspace{8pt}
\noindent
In the next lemma, we use the same notation as in Section \ref{section linear algebra}. The map $\phi_T$ is given in Definition \ref{def induced map on quotient isotropic}.
\vspace{8pt}

\begin{lemma}\label{isotropic surjective}
Let $T$ be an isotropic subspace of $V$. Then the map $\phi_T:V/T\rightarrow\Hom(T,Z)$, defined by $v+T\mapsto(t\mapsto\phi(v,t))$, is surjective.
\end{lemma}

\begin{proof}
Let $T$ be an isotropic subspace of $V$. 
The subgroup $\pi^{-1}(T)$ is abelian, by Lemma \ref{isotropic abelian}, and it contains $Z$. It follows that $\pi^{-1}(T)$ is normal, and so, by Proposition \ref{intense properties}($1$), it is $A$-stable. By Lemma \ref{action chi^i}, the actions of $A$ on $\pi^{-1}(T)/Z$ and on $Z$ are respectively through $\chi$ and $\chi^2$, which are distinct by Lemma \ref{chi neq chi2}. 
By Theorem \ref{lambda mu} the subgroup $Z$ has a unique $A$-stable complement $H$ in $\pi^{-1}(T)$, which is isomorphic to $T$ via $\pi$.
We now show that $\phi_T$ is surjective. 
For this purpose, let $f\in\Hom(T,Z)$ and note that $\Hom(T,Z)$ and $\Hom(H,Z)$ are naturally isomorphic. We identify $f$ with its image in $\Hom(H,Z)$.
By Lemma \ref{graph+complements}, the set
$L=\{f(t)t\ |\ t\in H\}$ is a complement of $Z$ in $\pi^{-1}(T)$ and so, being $H$ the unique $A$-stable complement of $Z$,
Lemma \ref{equivalent intense coprime-pgrps} guarantees that there exists $g\in G$ such that $L=gHg^{-1}$. Fix such an element $g$. 
Then, for each $h\in H$, there exists $t\in H$ such that $[g,h]h=ghg^{-1}=f(t)t$. It follows that 
$ht^{-1}=[h,g]f(t)$ belongs to both $H$ and $Z$, but $H$ and $Z$ intersect trivially, so we get $h=t$.
We have proven that $f$ is the map $t\mapsto[g,t]$. 
It follows from Definition \ref{def induced map on quotient isotropic} that $\phi_T$ is surjective.
\end{proof}

\begin{corollary}\label{max iso bijection}
Let $T$ be an isotropic subspace of $V$. Then $T$ is maximal isotropic if and only if the map $\phi_T:V/T\rightarrow\Hom(T,Z)$, defined by $v+T\mapsto(t\mapsto\phi(v,t))$, is a bijection.
\end{corollary}

\begin{proof}
The map $\phi_T$ is surjective by Lemma \ref{isotropic surjective} and it is injective by Lemma \ref{max isotropic}.
\end{proof}

\begin{lemma}\label{dimW 2}
The dimension of $Z$ is different from $2$. 
\end{lemma}

\begin{proof}
Assume by contradiction that $Z$ has dimension $2$. Let $T$ be an isotropic subspace of $V$ of maximal dimension $t$ and let $d=\dim V$, which is positive. From Corollary \ref{max iso bijection}, it follows that $d=3t$ and in particular that $t>0$. 
Let $L$ be a subspace of $T$ of codimension $1$, which is itself isotropic.
Let moreover $\phi_L:V/L\rightarrow\Hom(L,Z)$ be defined by $v+L\mapsto(l\mapsto\phi(v,l))$. The linear map $\phi_L$ is surjective
by Lemma \ref{isotropic surjective}. Let $U$ be the kernel of $\phi_L$ and let 
$\phi_U: U\times U\rightarrow Z$ be induced by $\phi$. Then $\dim U=d-3(t-1)=3$ and $\phi_U$ is alternating.
By the universal property of wedge products, there exists a unique linear map 
$\psi:\bigwedge^2 U \rightarrow Z$ that, composed with the canonical map 
$U \times U \rightarrow \bigwedge^2 U$, gives $\phi_U$.
The dimension of $\bigwedge^2 U$ being $3$, the dimension of $\ker\psi$ is positive and, as a consequence of Lemma \ref{wedge}, there are linearly independent elements 
$s,r\in U$ such that $\psi(s\wedge r)=0$. 
Set $R=L\,\oplus\,\F_ps\,\oplus\,\F_pr$. By construction, $R$ is an isotropic subspace of $V$ of dimension $t+1$. Contradiction to the maximality of $t$.
\end{proof}

\begin{corollary}\label{class 2 extraspecial}
The group $G$ is extraspecial of exponent $p$.  
\end{corollary}

\begin{proof}
The commutator subgroup of $G$ is non-trivial. If $G_2$ has order $p$, then $G$ is extraspecial of exponent $p$, by Lemma \ref{extraspecial exponent}.
We claim that the order of $G_2$ is in fact $p$. 
Assume by contradiction that $G_2$ has order larger than $p$. Then, by Lemma \ref{normal index p}, there exists a normal subgroup $M$ of $G$ that is contained in $G_2$ with index $p^2$. The group $G/M$ has class $2$ and, by Lemma \ref{intensity of quotients}, its intensity is greater than $1$. This is a contradiction to Lemma \ref{dimW 2}, with $G_2/M$ in the role of $Z$. 
\end{proof}

\noindent
We remark that Corollary \ref{class 2 extraspecial} gives $(1)\Rightarrow(2)$ in Theorem \ref{theorem class2 complete}. We complete the proof in the next section.

\section{The extraspecial case}\label{section class 2 extraspecial}

\noindent
In Section \ref{section class 2 extraspecial} we will see how the structure of extraspecial groups of exponent $p$ (see Section \ref{section extraspecial}) is particularly suitable for explicit construction of intense automorphisms of order coprime to $p$.
In this section, we conclude the proof of Theorem \ref{theorem class2 complete}.

\begin{lemma}\label{tosse}
Let $p$ be a prime number and let $G$ be a non-abelian extraspecial group of exponent $p$.
Let moreover $H$ be a subgroup of $G$ that trivially intersects $G_2$. 
Then $|G:\nor_G(H)|=|\Hom(H,G_2)|$.
\end{lemma}

\begin{proof}
The group $G$ being non-abelian, Lemma \ref{extraspecial non-abelian} yields that 
$\ZG(G)=G_2$ and $G_2$ has order $p$. 
Since $H\cap G_2$ is trivial, we have $\nor_G(H)=\Cyc_G(H)$ and $H$ is abelian. 
By Lemma \ref{tgt}, the commutator map $G\times G\rightarrow G_2$ is bilinear, and moreover, since $H\cap \ZG(G)$ is trivial, it induces a non-degenerate map $G/\Cyc_G(H)\times H\rightarrow G_2$.
Now, both $G/\Cyc_G(H)$ and $H$ are $\F_p$-vector spaces and $G_2$ has order $p$. 
It follows from Lemma \ref{non-degenerate dim 1} that 
$|G:\nor_G(H)|=|G:\Cyc_G(H)|=|H|=|\Hom(H,G_2)|$.
\end{proof}

\begin{lemma}\label{extraspecial action on first quotient}
Let $p$ be a prime number and let $G$ be a non-abelian extraspecial group of exponent $p$. Let $\alpha$ be an automorphism of $G$ such that $\gen{\alpha}$ acts on $G/G_2$ through a character. Then $\alpha\in\Int(G)$.
\end{lemma}

\begin{proof}
Let $H$ be a subgroup of $G$ and write $A=\gen{\alpha}$. We want to show that $H$ and $\alpha(H)$ are conjugate in $G$. As $G$ is non-abelian, Lemma \ref{extraspecial non-abelian} yields that $G_2=\ZG(G)$ and $G_2$ has order $p$.
It follows that either $H$ contains $G_2$ or the intersection of $H$ with $G_2$ is trivial. In the first case, $H/G_2$ is a linear subspace of $G/G_2$, and is therefore $A$-stable; in particular, also $H$ is $A$-stable.
We now consider the case in which ${H\cap G_2=\{1\}}$. In this case, $H$ is abelian and the group $T=H\oplus G_2$ is $A$-stable. The group $G_2$ being $A$-stable, $\alpha(H)$ is a complement of $G_2$ in $T$. Also each $G$-conjugate of $H$ is a complement of $G_2$ in $T$, because $G_2$ and $T$ are both normal. 
By Lemma \ref{graph+complements}, the number of complements of $G_2$ in $T$ equals the cardinality of $\Hom(H,G_2)$, which is equal to $|G:\nor_G(H)|$ by Lemma \ref{tosse}. It follows that the number of complements of $G_2$ in $T$ is equal to the number of conjugates of $H$ in $G$.
As all conjugates of $H$ are themselves complements of $G_2$ in $T$, we get that every complement of $G_2$ in $T$ is conjugate to $H$ in $G$. 
In particular, $H$ and $\alpha(H)$ are conjugate in $G$. The choice of $H$ being arbitrary, it follows that $\alpha\in \Int(G)$. 
\end{proof}

\begin{lemma}\label{extraspecial has intensity p-1}
Let $p$ be a prime number and let $G$ be a non-abelian extraspecial $p$-group of exponent $p$. Then $p$ is odd and $\inte(G)=p-1$.
\end{lemma}

\begin{proof}
The prime $p$ is odd, because all groups of exponent $2$ are abelian.
By Proposition \ref{get extraspecial},
we can write $G$ in the form $G(Z,Y,X,\theta)$, where $X$, $Y$, and $Z$ are vector spaces over $\F_p$ and $\theta:X\times Y\rightarrow Z$ is bilinear.
Now, the group $\F_p^*$ acts on $X$, $Y$, and $Z$, as described in Section \ref{section characters}, and so each $m\in\F_p^*$ gives rise to an automorphism of each of the three vector spaces. For each $m\in\F_p^*$, the following diagram is commutative because $\theta$ is bilinear.
\[
\begin{diagram}
X    \times   Y           &  \rTo^{\theta} & Z  \\
\dTo^{m} \ \  \dTo_{m}    &              & \dTo_{m^2}\\
X    \times   Y           &  \rTo^{\theta} & Z  \\
\end{diagram}
\]
By Proposition \ref{homom extraspecial}, for each $m\in\F_p^*$ there exists an automorphism $a_m$ of $G$ such that the maps induced by $a_m$ respectively on $X\times Y$ and $Z$ are scalar multiplications by $m$ and $m^2$.
The set $A=\{a_m\ |\ m\in\F_p^*\}$ is a subgroup of $\Aut(G)$ that is isomorphic to $\F_p^*$.
Thanks to Lemma \ref{extraspecial action on first quotient}, the subgroup $A$ is contained in $\Int(G)$ and therefore $\inte(G)=p-1$.
\end{proof}

\noindent
We remark that Lemma \ref{extraspecial has intensity p-1} is the same as $(2)\Rightarrow(3)$ in Theorem \ref{theorem class2 complete}. Since the implication $(3)\Rightarrow(1)$ is clear and $(1)\Rightarrow(2)$ is given by Corollary \ref{class 2 extraspecial}, Theorem \ref{theorem class2 complete} is finally proven.

\begin{proposition}\label{p-power jumps 2}
\pgp\
Denote by $(G_i)_{i\geq 1}$ the lower central series of $G$. 
Assume both the class and the intensity of $G$ are greater than $1$. 
Then, for all $i\in\Z_{\geq 1}$, the exponent of $G_i/G_{i+2}$ divides $p$. 
\end{proposition}

\begin{proof}
Let $\alpha$ be intense of order $\inte(G)$ and write $\chi={\chi_G}_{|\gen{\alpha}}$. Let moreover $i$ be a positive integer. 
The case in which $i=1$ is given by the combination of Lemma \ref{intensity of quotients} and Theorem \ref{theorem class2 complete}; we assume that $i>1$. As a consequence of Lemma \ref{commutator indices}, the quotient $G_i/G_{i+2}$ is abelian.
By Lemma \ref{action chi^i}, the action of $\gen{\alpha}$ on $G_i/G_{i+1}$ and $G_{i+1}/G_{i+2}$ is respectively through $\chi^i$ and $\chi^{i+1}$, which are distinct because $\inte(G)\neq1$. It follows from Theorem \ref{lambda mu} that the groups $G_i/G_{i+2}$ and $G_i/G_{i+1}\oplus G_{i+1}/G_{i+2}$ are isomorphic. The exponent of $G_i/G_{i+2}$ divides $p$ as a consequence of Lemma \ref{elementary abelian}.
\end{proof}


\chapter{Intensity of groups of class 3}\label{chapter class 3}

\noindent
The purpose of this chapter is giving a complete overview of the case in which the class is $3$. We will prove the following theorems.

\begin{theorem}\label{theorem intensity class 3}\label{theorem index commutator class 3}
Let $p$ be a prime number and let $G$ be a finite $p$-group of class $3$. Then the following are equivalent.
\begin{itemize}
 \item[$1$.] One has $\inte(G)> 1$.
 \item[$2$.] The prime $p$ is odd and $\inte(G)=2$.
 \item[$3$.] The prime $p$ is odd and $|G:G_2|=p^2$.
\end{itemize}
\end{theorem}

\noindent
We remind the reader that, if $G$ is a finite $p$-group and $j$ is a positive integer, then $\wt_G(j)=\log_p|G_j:G_{j+1}|$. For more detail, see Section \ref{section jumps}.

\begin{theorem}\label{theorem class at least 3}
Let $p$ be a prime number and let $G$ be a finite $p$-group of class at least $3$.
Assume that $\inte(G)>1$. For each positive integer $j$, set moreover $w_j=\wt_G(j)$.
Then the following hold.
\begin{itemize}
 \item[$1$.] One has $\inte(G)=2$.
 \item[$2$.] One has $(w_1,w_2,w_3)=(2,1,f)$, where $f\in\graffe{1,2}$.
\end{itemize} 
\end{theorem}

\section{Low intensity}\label{section class 3}

In Section \ref{section class 3} we derive some restrictions on the structure of finite $p$-groups of class at least $3$ and intensity greater than $1$. We will prove the following main result. 

{\begin{proposition}\label{mezzo teorema class 3}
Let $p$ be a prime number and let $G$ be a finite $p$-group of class at least $3$. Assume that $\inte(G)>1$.
Then the following hold.
\begin{itemize}
 \item[$1$.] The prime $p$ is odd.
 \item[$2$.] One has $\inte(G)=2$.
 \item[$3$.] One has $|G:G_2|=p^2$. 
\end{itemize}
\end{proposition}}

\noindent
Our main goal for this section being the proof of Proposition \ref{mezzo teorema class 3}, we will work under the following assumptions until the end of Section \ref{section class 3}.
Let $p$ be a prime number and let $G$ be a finite $p$-group of class at least $3$. 
Assume that $\inte(G)>1$ and let $\alpha$ be intense of order $\inte(G).$
Write $A=\gen{\alpha}$ and $\chi={\chi_G}_{|A}$, where $\chi_G$ denotes the intense character of $G$ (see Section \ref{section main}). 
For the rest of the notation we refer to the List of Symbols.
We remark that, $\inte(G)$ being greater than $1$, the prime $p$ is odd and $G$ is non-trivial. For more detail see Chapter \ref{CH formulation}.

\begin{lemma}\label{related to extra}
Assume that $G$ has class $3$.
Then the following hold.
\begin{itemize}
 \item[$1$.] One has $G^p\subseteq G_3$.
 \item[$2$.] One has $|G_2:G_3|=p$.
 \item[$3$.] One has $\ZG(G)=G_3$.
\end{itemize}
\end{lemma}

\begin{proof}
The subgroup $G_3$ is contained in $\ZG(G)$ because $G$ has class $3$. 
By Lemma \ref{intensity of quotients} the intensity of $G/G_3$ is greater than $1$, and thus,
by Theorem \ref{theorem class2 complete}, the group $G/G_3$ is extraspecial of exponent $p$. 
It follows that $G_2/G_3$ has size $p$ and that $G^p$ is contained in $G_3$. 
Moreover, one has $\ZG(G/G_3)=G_2/G_3$ and, since $\ZG(G)/G_3$ is contained in $\ZG(G/G_3)$, we get 
$G_3\subseteq\ZG(G)\subseteq G_2$. As the class of $G$ is $3$, the subgroup $\ZG(G)$ is different from $G_2$, and so
$\ZG(G)=G_3$.
\end{proof}

\begin{lemma}\label{senza N di ordine p}
Assume $G$ has class $3$.
Then the following hold.
\begin{itemize}
 \item[$1$.] The group $G_2$ is elementary abelian.
 \item[$2$.] The group $\Cyc_G(G_2)$ is abelian and $A$-stable.
\end{itemize}
\end{lemma}

\begin{proof}
($1$) The group $G_4$ is trivial and, as a consequence of Lemma \ref{commutator indices}, the subgroup $G_2$ is abelian.
The exponent of $G_2$ is equal to $p$, by Proposition \ref{p-power jumps 2}. 
This proves ($1$). We now prove ($2$).
To lighten the notation, let $C=\Cyc_G(G_2)$.
The subgroup $G_2$ is central in $C$, by definition of $C$, and, as a consequence of Lemma \ref{class 2 bilinear map}, 
the commutator map induces a bilinear map $\phi:C/G_2\times C/G_2\rightarrow [C,C]$. The subgroups $C$ and $[C,C]$ are 
characteristic in $G$ and thus they are $A$-stable. 
Thanks to Lemma \ref{action chi^i}, the group $A$ acts on $C/G_2$ through $\chi$, and, by Lemma \ref{lemma product of characters}, 
it acts on $[C,C]$ through $\chi^2$. 
By Lemma \ref{action chi^i}, the action of $A$ on $G_3$ is through $\chi^3$. The character $\chi$ not being trivial, 
one has $\chi^2\neq\chi^3$, and Lemma \ref{distinct characters on same gp} yields $[C,C]\cap G_3=\graffe{1}$.
By Lemma \ref{related to extra}($3$), the group $G_3$ is equal to $\ZG(G)$ so the group $[C,C]$ is a normal subgroup of $G$ that trivially intersects $\ZG(G)$. It follows from Lemma \ref{normal intersection centre trivial} that $[C,C]=\graffe{1}$.
\end{proof}

\begin{lemma}\label{con N ordine p}
Assume that $G_3$ has order $p$. Then the following hold.
\begin{itemize}
 \item[$1$.] One has $|G:\Cyc_G(G_2)|=p$.
 \item[$2$.] One has $|G:G_2|=p^2$.
 \item[$3$.] One has $|\Cyc_G(G_2)|=p^3$.
\end{itemize}
\end{lemma}

\begin{proof}
The group $G_3$ having order $p$, it follows from Lemma \ref{normal intersection centre trivial} that $G_3$ is central, and so $G$ has class $3$.
To lighten the notation, let $C=\Cyc_G(G_2)$. 
Let moreover $V=G/G_2$, $Z=G_2/G_3$, and $T=C/G_2$. The groups $V$, $Z$, and $T$ are vector spaces over $\F_p$, as a consequence of Lemma \ref{elementary abelian}.
We prove ($1$). 
By Lemma \ref{bilinear LCS}, the commutator map induces a bilinear map
$\psi:V\times Z\rightarrow G_3$ whose left kernel is $T$. 
The centre of $G$ is equal to $G_3$, by Lemma \ref{related to extra}($3$), so the right kernel of $\psi$ is trivial.
The map $\psi_C:V/T\times Z\rightarrow G_3$ that is induced from $\psi$ is thus non-degenerate. 
The dimension of $Z$ is equal to $1$, by Lemma \ref{related to extra}($2$), and Lemma \ref{non-degenerate dim 1} yields 
$\dim V/T=1$. This proves ($1$).
We prove ($2$) and ($3$) together. Let $\phi:V\times V\rightarrow Z$ be the bilinear map from Lemma \ref{bilinear LCS}. The map $\phi$ is induced from the commutator map and, by Lemma 
\ref{senza N di ordine p}($2$), the group $C$ is abelian. It follows that $T$ is isotropic. 
As a consequence of ($1$) the space $T$ has codimension $1$ in $V$ and $T$ is maximal isotropic, because $\phi$ is not the zero map.
From Corollary \ref{max iso bijection}, it follows that 
$1=\dim(V/T)=\dim\Hom(T,Z)=\dim T$ and so $\dim V=2$. 
To conclude, we compute $|G:G_2|=p^{\dim V}=p^2$ and $|C|=|C:G_2||G_2:G_3||G_3|=p^{\dim T}p^{\dim Z}p=p^3$.
\end{proof}

\begin{lemma}\label{C elem abelian class 3}
Assume that $\chi^2\neq 1$ and that $G$ has class $3$. Then $\Cyc_G(G_2)$ is elementary abelian.
\end{lemma}

\begin{proof}
Let $C=\Cyc_G(G_2)$. The group $C$ is abelian and $A$-stable by Lemma \ref{senza N di ordine p}($2$). We will show that $C$ has exponent $p$.
By Lemma \ref{senza N di ordine p}($1$) the group $G_2$ is elementary abelian and $G_2\subseteq C$.
The group $A$ acts on $C/G_2$ through $\chi$, as a consequence of Lemma \ref{action chi^i}, and, by Lemma \ref{p-power characters}, it acts on $C^p$ also through $\chi$. It follows from Lemma \ref{related to extra}($1$) that $C^p\subseteq G_3$. The action of $A$ on $G_3$ is through $\chi^3$, by Lemma \ref{senza N di ordine p}($2$), and thus $A$ acts on $C^p$ both through $\chi$ and $\chi^3$. Since $\chi^2\neq 1$, the characters $\chi$ and $\chi^3$ are distinct and, as a consequence of Lemma \ref{distinct characters on same gp}, the group $C$ has exponent $p$.  
\end{proof}

\begin{lemma}\label{c12}
Assume that $\chi^2\neq 1$ and that $G_3$ has order $p$. Then $\Cyc_G(G_2)$ is a vector space over $\F_p$ and there exist unique $A$-stable subspaces $C_1$ and $C_2$ of 
$\Cyc_G(G_2)$, of dimension $1$, such that $\Cyc_G(G_2)=C_1\oplus C_2\oplus G_3$.
\end{lemma}

\begin{proof}
To lighten the notation, let $C=\Cyc_G(G_2)$.
Since $G_3$ has order $p$, it follows from Lemma \ref{normal intersection centre trivial} that $G_3$ is central so $G$ has class $3$. As a consequence of Lemmas \ref{con N ordine p}($3$) and \ref{C elem abelian class 3}, the group $C$ is a vector space over 
$\F_p$ of dimension $3$. The group $C$ is $A$-stable, by Lemma \ref{senza N di ordine p}($2$), and, by Lemma \ref{action chi^i}, the action of $A$ on $C/G_2$, $G_2/G_3$, and $G_3$ is respectively through $\chi$, $\chi^2$, and $\chi^3$. The three characters are pairwise distinct because $\chi^2\neq 1$. We first apply Theorem \ref{lambda mu} to $C/G_3$. Then there exists a unique $A$-stable complement $D_1/G_3$ of $G_2/G_3$. It follows that $D_1\cap G_2=G_3$. We now apply Theorem \ref{lambda mu} to both $D_1$ and $G_2$ to get unique $A$-stable subspaces $C_1$ and $C_2$ of $C$ satisfying $D_1=C_1\oplus G_3$ and $G_2=C_2\oplus G_3$. 
As a consequence of Lemma \ref{related to extra}($2$), the subspace $G_2$ has dimension $2$, so both $C_1$ and $C_2$ have dimension $1$.
Moreover, the intersection of $D_1$ with $G_2$ being equal to $G_3$, it follows that $C=C_1\oplus C_2\oplus G_3$.
\end{proof}

\begin{lemma}\label{alpha has order at most 2}
Assume that $G_3$ has order $p$. Then $\inte(G)=2$.
\end{lemma}

\begin{proof}
If $\chi^2=1$, then $1<\inte(G)\leq 2$ and we are done. We assume now that $\chi^2\neq 1$ and we will derive a contradiction.
Let now $C_1$ and $C_2$ be as in Lemma \ref{c12} and denote by $X$ be the collection of subspaces of dimension $1$ of $C$. Since $C$ is normal, the group $G$ acts on $X$ by conjugation. By Lemma \ref{con N ordine p}($1$), the index of $C$ in $G$ is equal to $p$ and the size of each orbit of $X$ under $G$ is at most $p$. Moreover, the elements of $X$ that are stable under the action of $A$ are precisely $C_1$, $C_2$, and $G_3$. 
Lemma \ref{orbits} yields 
$$p^2+p+1=|X|\leq |G:\nor_G(C_1)|+|G:\nor_G(C_2)|+|G:\nor_G(G_3)|\leq 3p,$$ which is satisfied if and only if 
$(p-1)^2\leq 0$. Contradiction. 
\end{proof}

\noindent
We can finally give the proof of Proposition \ref{mezzo teorema class 3}. The prime $p$ is odd as a consequence of Proposition \ref{proposition 2gps}.
Since $G$ has class at least $3$, the group $G_3$ is non-trivial, so, by Lemma \ref{normal index p}, there exists a normal subgroup $M$ of $G$ that is contained in $G_3$ with index $p$. By Lemma \ref{intensity of quotients}, the group $G/M$ has intensity greater than $1$ and, 
as a consequence of Lemma \ref{alpha has order at most 2}, the intensity of $G/M$ is equal to $2$. From Lemma \ref{intensity of quotients} it follows that $1<\inte(G)\leq\inte(G/M)=2$. 
Moreover, Lemma \ref{con N ordine p}($2$) yields $$|G:G_2|=|G/M:[G/M,G/M]|=p^2.$$ 
We remark that Proposition \ref{mezzo teorema class 3} gives $(1)\Leftrightarrow(2)$ and $(1)\Rightarrow(3)$ in Theorem \ref{theorem intensity class 3} and Theorem \ref{theorem class at least 3}($1$). We give the full proof of Theorem \ref{theorem intensity class 3} in Section \ref{section construction} and the full proof of Theorem \ref{theorem class at least 3} in Section \ref{section gps class 3}.

\begin{proposition}\label{proposition divisible by 2}
Let $p$ be a prime number and let $G$ be a finite $p$-group. 
Then the following are equivalent.
\begin{itemize}
 \item[$1$.] One has $\inte(G)>1$.
 \item[$2$.] The prime $p$ is odd and $\inte(G)$ is even.
\end{itemize}
\end{proposition}

\begin{proof}
The implication $(2)\Rightarrow(1)$ is clear. Assume now $(1)$. Then $G$ is non-trivial and $p$ is odd, by Proposition \ref{proposition 2gps}. Moreover, if $G$ has class at least $3$, then Proposition \ref{mezzo teorema class 3} yields $\inte(G)=2$. On the other hand, if $G$ has class at most $2$, then we know from Theorems \ref{theorem abelian} and \ref{theorem class2 complete} that $\inte(G)=p-1$, which is even because $p$ is odd.
\end{proof}

\begin{proposition}\label{proposition -1^i}
Let $p$ be a prime number and let $G$ be a finite $p$-group of class at least $3$. 
Let $\alpha$ be an intense automorphism of $G$ of order $\inte(G)$ and assume that $\inte(G)>1$. 
Then $\alpha$ has order $2$ and, for all $i\geq 1$, it induces scalar multiplication by $(-1)^i$ on $G_i/G_{i+1}$.
\end{proposition}

\begin{proof}
Let $\chi$ denote the restriction of $\chi_G$ to $\gen{\alpha}$. By Proposition \ref{mezzo teorema class 3}, the intensity of $G$ is $2$ and $p$ is odd. In particular, $\chi(\alpha)$ has order $2$ in $\omega(\F_p^*)$, so $\chi(\alpha)=-1$.
Lemma \ref{action chi^i} draws the conclusion.
\end{proof}

\section{Groups of class $3$}\label{section gps class 3}

\noindent
This section is devoted to understanding the structure of $p$-groups $G$ of class $3$ with the property that $|G:G_2|=p^2$. We will see in the next section that all these groups have intensity $2$. We also give the proof of Theorem \ref{theorem class at least 3}.

\begin{lemma}\label{class 3 G3}
Let $p$ be a prime number and let $G$ be a finite $p$-group of class $3$. Assume that 
$|G:G_2|=p^2$. Then the following hold.
\begin{itemize}
 \item[$1$.] One has $|G_2:G_3|=p$.
 \item[$2$.] One has $|G_3|\in\graffe{p,p^2}$.
 \item[$3$.] The subgroup $G_3$ is elementary abelian.
\end{itemize}
\end{lemma}

\begin{proof}
The group $G$ is non-abelian and, as a consequence of Lemma \ref{frattini comm}, the group $G/G_2$ is a vector space over $\F_p$ of dimension $2$. By Lemma \ref{tensor LCS}, the commutator map induces a surjective homomorphism $G/G_2\otimes G/G_2\rightarrow G_2/G_3$ which factors as a surjective homomorphism $\bigwedge^2(G/G_2)\rightarrow G_2/G_3$ by the universal property of wedge products.
The space $\bigwedge^2(G/G_2)$ has dimension $1$ and so $|G_2:G_3|=p$. Again applying Lemma \ref{tensor LCS}, we derive that $G_3$ is isomorphic to a quotient of $G/G_2\otimes G_2/G_3$. In particular, one has $$1<|G_3|\leq|G/G_2\otimes G_2/G_3|=p^2$$ and $G_3$ is elementary abelian. 
\end{proof}

\noindent
Thanks to Lemma \ref{class 3 G3}, we can finally give the proof of Theorem \ref{theorem class at least 3}. Indeed, if $p$ is a prime number and $G$ is a finite $p$-group of class at least $3$ with $\inte(G)>1$, then Proposition \ref{mezzo teorema class 3} yields $\inte(G)=2$ and $w_1=2$. We now apply Lemma \ref{class 3 G3}, with $G/G_4$ in place of $G$, to get $w_2=1$ and $w_3\in\graffe{1,2}$. The proof of Theorem \ref{theorem class at least 3} is now complete.

\begin{lemma}\label{isomorphism modulo centralizer}
Let $p$ be an prime number and let $G$ be a finite $p$-group of class $3$. Assume that $|G:G_2|=p^2$.
Then the following hold.
\begin{itemize}
 \item[$1$.] One has $G_2\subseteq \Cyc_G(G_2)$.
 \item[$2$.] The commutator map induces an isomorphism $G/\Cyc_G(G_2)\otimes G_2/G_3\rightarrow G_3$.
 \item[$3$.] One has $|G:\Cyc_G(G_2)|=|G_3|$.
\end{itemize}
\end{lemma}

\begin{proof}
($1$) As a consequence of Lemma \ref{commutator indices}, the group $[G_2,G_2]$ is contained in $G_4=\graffe{1}$ and $G_2$ centralizes itself. This proves ($1$). We now prove ($2$) and ($3$) together. 
The subgroup $G_3$ is central, because the class of $G$ is $3$, so, by Lemma \ref{tgt}, the commutator map $\gamma:G\times G_2 \rightarrow G_3$ is bilinear. The right kernel of $\gamma$ is equal to $G_2\cap\ZG(G)$, which is equal to $G_3$ as a consequence of Lemma \ref{class 3 G3}($1$). The left kernel of $\gamma$ coincides with $\Cyc_G(G_2)$ and, in particular, $\gamma$ induces a non-degenerate map $\gamma_1:G/\Cyc_G(G_2)\times G_2/G_3\rightarrow G_3$ whose image generates $G_3$. Thanks to ($1$) the quotient $G/\Cyc_G(G_2)$ is abelian and, thanks to Lemma \ref{frattini comm}, it has exponent $p$.
By the universal property of tensor products, $\gamma_1$ induces a surjective homomorphism $G/\Cyc_G(G_2)\otimes G_2/G_3\rightarrow G_3$, which is also an isomorphism since the index $|G_2:G_3|$ is equal to $p$. Moreover, we have 
$|G:\Cyc_G(G_2)|=|G_3|$.
\end{proof}

\begin{lemma}\label{class 3 centralizer G2}
Let $p$ be a prime number and let $G$ be a finite $p$-group of class $3$. Assume that $|G:G_2|=p^2$. Then the following hold.
\begin{itemize}
 \item[$1$.] One has $|G_3|=p$ if and only if $|\Cyc_G(G_2):G_2|=p$.
 \item[$2$.] One has $|G_3|=p^2$ if and only if $\Cyc_G(G_2)=G_2$.
\end{itemize}
\end{lemma}

\begin{proof}
By Lemma \ref{isomorphism modulo centralizer}($3$), we have 
$|G:\Cyc_G(G_2)|=|G_3|$. 
By Lemma \ref{isomorphism modulo centralizer}($1$), the subgroup $G_2$ is contained in $\Cyc_G(G_2)$, so ($1$) and ($2$) follow from the fact that $|G:G_2|=p^2$.
\end{proof}

\begin{lemma}\label{class 3 cent comm}
Let $p$ be a prime number and let $G$ be a finite $p$-group of class $3$. Assume that $|G:G_2|=p^2$. Then $\Cyc_G(G_2)$ is abelian.
\end{lemma}

\begin{proof}
Write $C=\Cyc_G(G_2)$.
As a consequence of Lemma \ref{class 3 centralizer G2}, the group $C/G_2$ is cyclic. It follows from Lemma \ref{cyclic quotient commutators} that $[C,C]=[C,G_2]=\graffe{1}$.
\end{proof}

\begin{lemma}\label{class 3 G/G3 extraspecial of exp p}
Let $p$ be an odd prime number and let $G$ be a finite $p$-group of class $3$. Assume that $|G:G_2|=p^2$. Then $G/G_3$ is extraspecial of exponent $p$.
\end{lemma}

\begin{proof}
We write $\overline{G}=G/G_3$ and we use the bar notation for the subgroups of $\overline{G}$.
The group $\overline{G}$ has class $2$ and $\overline{G_2}$ is contained in $\ZG(\overline{G})$. By Lemma \ref{class 3 G3}($1$), the order of $\overline{G_2}$ is equal to $p$ and, as a consequence of Lemma \ref{quotient by centre not cyclic}, the groups $\overline{G_2}$ and $\ZG(\overline{G})$ coincide. In particular, $\overline{G}$ is extraspecial.
We now show that $\overline{G}$ has exponent $p$. 
Define $C=\Cyc_G(G_2)$ and $D=\graffe{x\in G\ :\ x^p\in G_3}$. 
Then $C\neq G$, because $G_2$ is not central, and $D$ is a group, thanks to Corollary \ref{p map petrescu}.
Let now $x\in G\setminus C$.
As a consequence of Lemma \ref{frattini comm}, the element $x^p$ belongs to $G_2$. 
Moreover, by Lemma \ref{isomorphism modulo centralizer}($2$), the commutator map induces an isomorphism 
${G/C\otimes G_2/G_3\rightarrow G_3}$, so,
since $x$ is not in the centralizer of $G_2$, the element $x^p$ belongs to $G_3$.
It follows that $x\in D$ and, in particular, we have proven that $G=C\cup D$. 
The group $C$ is different from $G$, thus the groups $D$ and $G$ are the same. 
It follows that $\overline{G}=\overline{D}$ and so $\overline{G}$ has exponent $p$.
\end{proof}

\begin{lemma}\label{class 3 G3=Z}
Let $p$ be an odd prime number and let $G$ be a finite $p$-group of class $3$. Assume that $|G:G_2|=p^2$. Then $G_3=\ZG(G)$.
\end{lemma}

\begin{proof}
The subgroup $G_3$ is contained in $\ZG(G)$, since $G$ has class $3$ and, as a consequence of Lemma \ref{class 3 G/G3 extraspecial of exp p}, the centre of $G/G_3$ is equal to $G_2/G_3$. 
It follows that $\ZG(G)/G_3\subseteq G_2/G_3$ and $\ZG(G)\subseteq G_2$. Moreover, the group $\ZG(G)$ does not contain $G_2$, because $G$ has class $3$.
The group $G_2/G_3$ having order $p$, one gets $G_3=\ZG(G)$. 
\end{proof}



\begin{lemma}\label{class 3 G2 elem abelian}
Let $p$ be an odd prime number and let $G$ be a finite $p$-group of class $3$. Assume that $|G:G_2|=p^2$. Then $G_2$ is elementary abelian.
\end{lemma}

\begin{proof}
The group $G_2$ is abelian as a consequence of Lemma \ref{isomorphism modulo centralizer}($1$). We prove that it has exponent $p$.
Let $M$ be a maximal subgroup of $G_3$; then $M$ has index $p$ in $G_3$ and it is normal, because $G_3$ is central. 
We write $\overline{G}=G/M$ and use the bar notation for the subgroups of $\overline{G}$. The subgroup $\overline{G_3}$ has order $p$ and $|\overline{G}:\overline{G_2}|=|G:G_2|=p^2$. It follows from Lemma \ref{class 3 cent comm} that 
$\Cyc_{\overline{G}}(\overline{G_2})$ is abelian and, from Lemma \ref{class 3 centralizer G2}($1$), that it contains $\overline{G_2}$ with index $p$. Write $\overline{C}=\Cyc_{\overline{G}}(\overline{G_2})$. As a consequence of Lemma \ref{class 3 G/G3 extraspecial of exp p}, the subgroup $\overline{C}^{\,p}$ is contained in $\overline{G_3}$, so  $\mu_p(\overline{C})$ is a normal subgroup of $\overline{G}$ of order at least $p^2$. 
Moreover, $\overline{G_3}$ is contained in $\mu_p(\overline{C})$, so $\mu_p(\overline{C})/\overline{G_3}$ is a non-trivial normal subgroup of $G/G_3$. The quotient $G/G_3$ is extraspecial, by Lemma \ref{class 3 G/G3 extraspecial of exp p}, so $G_2/G_3$ is equal to $\ZG(G/G_3)$. As a consequence of Lemma \ref{class 3 G3}($1$), the quotient $G_2/G_3$ has order $p$, so Lemma \ref{normal intersection centre trivial} yields $\overline{G_2}\subseteq\mu_p(\overline{C})$.
In particular, one has $G_2^p\subseteq M$. If $M=\graffe{1}$ we are done, otherwise let $N$ be another maximal subgroup of $G_3$. In this case, $G_3$ is elementary abelian of order $p^2$, by Lemma \ref{class 3 G3}($2$-$3$), and $G_2^p$ is contained in $N\cap M=\graffe{1}$, by the previous arguments. The exponent of $G_2$ is thus $p$. 
\end{proof}

\section{Intensity given the automorphism}\label{section class 3 with character}

We recall that, for any group $G$, the lower central series of $G$ is denoted $(G_i)_{i\geq 1}$ and it consists of characteristic subgroups of $G$. For more detail see Section \ref{section commutators}.
The main result of this section is the following. 

\begin{proposition}\label{p^4 has cp}
Let $p$ be an odd prime number and let $G$ be a finite $p$-group of class $3$ such that $|G:G_2|=p^2$. 
Let moreover $\alpha$ be an automorphism of $G$ of order $2$ that induces the inversion map $x\mapsto x^{-1}$ on $G/G_2$. Then 
$\alpha$ is intense and $\inte(G)=2$.
\end{proposition}

\noindent
The following assumptions will be valid until the end of Section \ref{section class 3 with character}.
Let $p$ be an odd prime number and let $G$ be a finite $p$-group of class $3$ such that $|G:G_2|=p^2$. Let $\alpha$ be an automorphism of $G$ of order $2$ and write $A=\gen{\alpha}$. Let moreover $\chi:A\rightarrow\graffe{\pm 1}$ be an isomorphism of groups and assume that the induced action of $A$ on $G/G_2$ is through $\chi$. 
We will prove that $\alpha$ is intense.

\begin{lemma}\label{pinocchio}
Every subgroup of $G$ that contains $G_3$ has an $A$-stable conjugate in $G$.
\end{lemma}

\begin{proof}
Let $H$ be a subgroup of $G$ that contains $G_3$. 
By Lemma \ref{class 3 G/G3 extraspecial of exp p}, the group $G/G_3$ is extraspecial of exponent $p$ and by assumption $A$ acts on $G/G_2$ through $\chi$. As a consequence of Lemmas \ref{extraspecial action on first quotient} and Lemma \ref{equivalent intense coprime-pgrps}, there exists $g\in G$ such that $\alpha(gHg^{-1})/G_3=(gHg^{-1})/G_3$ and, $G_3$ being $A$-stable, $\alpha(gHg^{-1})=gHg^{-1}$.
\end{proof}

\noindent
We remind the reader that, if $H$ is a subgroup of $G$, then a positive integer $j$ is a jump of $H$ in $G$ if 
$H\cap G_j\neq H\cap G_{j+1}$. The $j$-th width of $H$ in $G$ is $\wt_H^G(j)=\log_p|H\cap G_j:H\cap G_{j+1}|$.
For more information about jumps and width see Section \ref{section jumps}.

\begin{lemma}\label{class 3 jumps}
Let $H$ be a subgroup of $G$ that trivially intersects $G_3$. 
Then the following hold.
\begin{itemize}
 \item[$1$.] If $1$ is a jump of $H$ in $G$, then $\wt_H^G(1)=1$.
 \item[$2$.] If $2$ is a jump of $H$ in $G$, then $H\subseteq\Cyc_G(G_2)$.
\end{itemize}
\end{lemma}

\begin{proof}
($1$) Assume that $1$ is a jump of $H$ in $G$. By Lemma \ref{frattini comm}, the Frattini subgroup of $G$ is equal to $G_2$. The subgroup $H$ does not contain $G_3$ and thus $H\neq G$. By Lemma \ref{subgroup times frattini}, we have $H\Phi(G)\neq G$ so $H\Phi(G)/\Phi(G)=HG_2/G_2$ has order $p$.
Since $HG_2/G_2$ is isomorphic to $(H\cap G_1)/(H\cap G_2)$, this proves ($1$).
We now prove ($2$). Assume that $2$ is a jump of $H$ in $G$. Then by Lemma \ref{jumps and depth} there exists an element $x\in (H\cap G_2)\setminus G_3$. Fix $x$. As a consequence of Lemma \ref{class 3 G3}($1$), the group $G_2$ is equal to $\gen{x,G_3}$. The group $G_3$ being central, it follows that $[H,G_2]=[H,\gen{x}]$. The subgroup $[H,\gen{x}]$ is contained in 
$H\cap [G,G_2]=H\cap G_3$, which is trivial by assumption. In particular, $H$ centralizes $G_2$. 
\end{proof}

\begin{lemma}\label{grilloparlante}
Let $H$ be a subgroup of $G$ that trivially intersects $G_2$. 
Then $H$ has an $A$-stable conjugate in $G$.
\end{lemma}

\begin{proof}
The group $H$ is abelian, because $[H,H]\subseteq H\cap [G,G]=\graffe{1}$. 
By Lemma \ref{class 3 G3=Z}, the groups $G_3$ and $\ZG(G)$ are equal, so the group $T=H\oplus G_3$ is abelian. 
By Lemma \ref{pinocchio}, there exists $g\in G$ such that $gTg^{-1}$ is $A$-stable and, the group $G_3$ being characteristic, $gTg^{-1}=gHg^{-1}\oplus G_3$. We fix such element $g$ and note that $gTg^{-1}\cap G_2=G_3$. It follows from Lemma \ref{intersection characters} that the induced action of $A$ on $gTg^{-1}/G_3$ is through $\chi$. Moreover, by Lemma \ref{action chi^i general}, the group $A$ acts on $G_3$ through $\chi^3=\chi$. 
From Lemma \ref{abelian>minus quotients}, it follows that $\alpha$ sends each element of $gTg^{-1}$ to its inverse, so each subgroup of $gTg^{-1}$ is $A$-stable.
In particular, $gHg^{-1}$ is $A$-stable. 
\end{proof}

\begin{lemma}\label{bizzarro}
Let $H$ be a subgroup of $G$ such that $G_2=H\oplus G_3$. 
Then $H$ has an $A$-stable conjugate in $G$.
\end{lemma}

\begin{proof}
By Lemma \ref{class 3 G3}($1$), the index $|G_2:G_3|$ is equal to $p$, so $H$ has order $p$.
By Lemma \ref{action chi^i general}, the induced action of $A$ on $G_2/G_3$ and $G_3$ is respectively through $\chi^2$ and $\chi^3=\chi$. By assumption, the characters $\chi$ and $\chi^2$ are distinct. Moreover, by Lemma \ref{isomorphism modulo centralizer}($1$), the group $G_2$ is abelian and so, by Theorem \ref{lambda mu}, there exists a unique $A$-stable complement $K$ of $G_3$ in $G_2$.
We want to show that $H$ and $K$ are conjugate in $G$. 
The groups $G_3$ and $\ZG(G)$ coincide, by Lemma \ref{class 3 G3=Z}, thus $\Cyc_G(H)=\Cyc_G(G_2)$. Moreover, we have that $H\cap[H,G]\subseteq H\cap G_3=\graffe{1}$, so $\Cyc_G(H)=\nor_G(H)$. 
Let $X$ be the collection of complements of $G_3$ in $G_2$. Then $K$ and all conjugates of $H$ in $G$ are in $X$. 
By Lemma \ref{isomorphism modulo centralizer}($3$), we have 
$|G:\Cyc_G(G_2)|=|G_3|$.
By Lemma \ref{class 3 G3}($3$), the subgroup $G_3$ is elementary abelian and, by Lemma \ref{graph+complements}, the cardinality of $X$ is equal to the cardinality of $\Hom(H,G_3)$, which coincides with $|G_3|$ because $H$ has order $p$.
It follows that $|X|=|G:\Cyc_G(G_2)|=|G:\nor_G(H)|$ and, every conjugate of $H$ being in $X$, every complement of $G_3$ in $G_2$ is conjugate to $H$. In particular, $K$ and $H$ are conjugate in $G$.
\end{proof}

\begin{lemma}\label{geppetto}
Let $H$ be a subgroup of $G$ that is not contained in $\Cyc_G(G_2)$ and that has trivial intersection with $G_3$. 
Then $H$ has a conjugate that is $A$-stable.
\end{lemma}

\begin{proof}
As a consequence of Lemma \ref{class 3 jumps}($2$), the subgroup $H$ has trivial intersection with $G_2$.
We now apply Lemma \ref{grilloparlante}.
\end{proof}

\begin{lemma}\label{lucignolo}
Let $H$ be a subgroup of $\Cyc_G(G_2)$ of order $p$ that has trivial intersection with $G_3$. 
Then $H$ has a conjugate that is $A$-stable.
\end{lemma}

\begin{proof}
Let us call $T=H\oplus G_3$. If $T=G_2$, then $H$ has an $A$-stable conjugate by Lemma \ref{bizzarro}. Assume now that $T\cap G_2=G_3$. 
Then $H\cap G_2=H\cap T\cap G_2=H\cap G_3=\graffe{1}$, so we conclude applying Lemma \ref{grilloparlante}.
\end{proof}

\noindent
We denote $G^+=\graffe{x\in G : \alpha(x)=x}$ and 
$G^-=\graffe{x\in G : \alpha(x)=x^{-1}}$, in concordance with the notation from Section \ref{section involutions}. In the context of Section \ref{section class 3 with character}, we will use this notation in Lemmas \ref{bonn} and \ref{sputafuoco}.

\begin{lemma}\label{bonn}
Let $H$ be a subgroup of $\Cyc_G(G_2)$ such that $H\cap G_3=\graffe{1}$. Then the following hold.
\begin{itemize}
 \item[$1$.] The subgroup $H$ is elementary abelian. 
 \item[$2$.] One has $G^+\nor_G(H)=\nor_G(H)$.
\end{itemize}

\end{lemma}

\begin{proof}
($1$) The subgroup $\Cyc_G(G_2)$ is abelian, by Lemma \ref{class 3 cent comm}, and therefore $H$ is abelian. Moreover, as a consequence of Lemma \ref{class 3 G/G3 extraspecial of exp p}, the subgroup $H^p$ is contained in $H\cap G_3=\{1\}$, so $H$ is elementary abelian. This proves ($1$). We now prove ($2$). The subgroup $G^+$ is contained in $G_2$, thanks to Lemma \ref{order +- jumps}, and $G_2$ centralizes $C$, by definition of $C$. It follows that $G^+\nor_G(H)\subseteq G_2\nor_G(H)=\nor_G(H)$. Since $\nor_G(H)$ is contained in $G^+\nor_G(H)$, the proof is complete.
\end{proof}

\begin{lemma}\label{sputafuoco}
Let $H$ be a subgroup of $G$ such that $\Cyc_G(G_2)=H\oplus G_3$. 
Then $H$ has a conjugate that is $A$-stable.
\end{lemma}

\begin{proof}
To lighten the notation, write $C=\Cyc_G(G_2)$.
If $C=G_2$, then we are done by Lemma \ref{bizzarro}. Assume now that $C\neq G_2$. As a consequence of Lemma \ref{class 3 centralizer G2}($1$), the group $C$ contains $G_2$ with index $p$ and $G_3$ has order $p$.
We define $X$ to be the collection of subgroups $K$ of $G$ such that $C=K\oplus G_3$ and denote 
$X^+=\graffe{K\in X\ |\ \alpha(K)=K}$. The centre of $G$ is equal to $G_3$, by Lemma \ref{class 3 G3=Z}, and, as a consequence of Lemma \ref{normal intersection centre trivial}, all elements of $X$ are non-normal subgroups of $G$. In particular, for any $K\in X$, one has $|G:\nor_G(K)|\geq p$.
Now, by Lemma \ref{bonn}($1$), the subgroup $H$ is elementary abelian, and, $G_3$ being central of order $p$, it follows that $C$ is naturally an $\F_p$-vector space. Combining Lemmas \ref{class 3 G3}($1$) and \ref{class 3 centralizer G2}($1$), we get that $\dim C=3$.
Write $C^+=\graffe{x\in C\ :\ \alpha(x)=x}$ and $C^-=\graffe{x\in C\ :\ \alpha(x)=x^{-1}}$.
Then $C=C^+\oplus C^-$, thanks to Corollary \ref{abelian sum +-} and, as a consequence of Lemma \ref{order +- jumps}, one has 
$|C^-|=p^2$ and $|C^+|=p$. Moreover, $C$ being abelian, both $C^+$ and $C^-$ are linear subspaces of $C$.
It is not difficult to show at this point that 
$$X^+=\graffe{C^+\oplus\ell : \ell \subseteq G^-,\, \ell\cap G_3=\graffe{1},\, \dim(\ell)=1}.$$
It follows that $X^+$ has cardinality $p$, while the cardinality of $X$ is $p^2$. 
Moreover, the combination of Lemmas \ref{conjugate stable iff in nor-times-G+} and \ref{bonn}($2$), ensures that no two elements of $X^+$ are conjugate in $G$.
It follows from Lemma \ref{orbits} that
$$p^2=|X|\geq\sum_{K\in X^+}|G:\nor_G(K)|\geq\sum_{K\in X^+}p=|X^+|p=p^2,$$ 
and therefore every element of $X$ is conjugate to an element of $X^+$. In particular, $H$ has an $A$-stable conjugate.
\end{proof}

\begin{lemma}\label{blue monday}
Every subgroup of $G$ that trivially intersects $G_3$ has an $A$-stable conjugate in $G$.
\end{lemma}

\begin{proof}
Let $H$ be a subgroup of $G$ such that $H\cap G_3=\graffe{1}$. If $H$ is contained in $\Cyc_G(G_2)$ and has order $p$, then we are done by Lemma \ref{lucignolo}. Assume now that $H$ is contained in $\Cyc_G(G_2)$ and that $H$ has order $p^2$. The group $\Cyc_G(G_2)$ is abelian, by Lemma \ref{class 3 cent comm}, and thus, as a consequence of Lemmas \ref{class 3 G3}($1$) and \ref{class 3 centralizer G2}, one has $\Cyc_G(G_2)=H\oplus G_3$. The group $H$ has an $A$-stable conjugate by Lemma \ref{sputafuoco}. We conclude by Lemma \ref{geppetto}, in case $H$ is not contained in $\Cyc_G(G_2)$.
\end{proof}

\begin{lemma}\label{oggi}
Let $H$ be a subgroup of $G$ such that $H\cap G_3\neq\graffe{1}$. 
Then $H$ has a conjugate that is $A$-stable.
\end{lemma}

\begin{proof}
Lemma \ref{pinocchio} covers the case in which $H$ contains $G_3$. Assume now that the group $H\cap G_3$ is different from both $\graffe{1}$ and $G_3$.
As a consequence of Lemma \ref{class 3 G3}($2$), the cardinality of $G_3$ is $p^2$, so $H\cap G_3$ has order $p$.
We write $\overline{G}=G/(H\cap G_3)$ and use the bar notation for the subgroups of $\overline{G}$. The group $\overline{G}$ has class $3$ and $|\overline{G}:\overline{G_2}|=p^2$. Moreover, $\overline{H}\cap\overline{G_3}=\graffe{1}$. Thanks to Lemma \ref{blue monday}, the subgroup $\overline{H}$ has an $A$-stable conjugate, and therefore so does $H$.
\end{proof}

\begin{lemma}\label{fotocopiatrice}
The automorphism $\alpha$ is intense and $\inte(G)=2$. 
\end{lemma}

\begin{proof}
We will show that $\alpha\in\Int(G)$. Thanks to Lemma \ref{equivalent intense coprime-pgrps}, it suffices to show that every subgroup of $G$ has an $A$-stable conjugate. Let $H$ be a subgroup of $G$. If $H\cap G_3=\graffe{1}$, we are done by Lemma \ref{blue monday}, otherwise apply Lemma \ref{oggi}.
\end{proof}

\noindent
Thanks to Lemma \ref{fotocopiatrice}, Proposition \ref{p^4 has cp} is proven.

\section{Constructing intense automorphisms}\label{section construction}

\noindent
The aim of Section \ref{section construction} is giving the proof of Theorem \ref{theorem index commutator class 3}. We will prove the following essential result. 

\begin{proposition}\label{proposition construction class 3}
Let $p$ be an odd prime number and let $G$ be a finite $p$-group of class $3$ such that $|G:G_2|=p^2$. 
Then there exists an automorphism $\alpha$ of $G$ of order $2$ that induces the inversion map $x\mapsto x^{-1}$ on $G/G_2$. 
\end{proposition}

\noindent
In order to prove Proposition \ref{proposition construction class 3}, let $p$ be an odd prime number and let $G$ be a finite $p$-group of class $3$. 
Let moreover $(G_i)_{i\geq 1}$ denote the lower central series of $G$ and assume that $|G:G_2|=p^2$. We will keep these assumptions and notation until the end of Section \ref{section construction}.
We will work to construct an automorphism $\alpha$ of $G$ and an isomorphism 
$\chi:\gen{\alpha}\rightarrow\graffe{\pm 1}$ in order to apply the results achieved in the previous section. 

\vspace{8pt}
\noindent
Let $F$ be the free group on the set $S=\graffe{a,b}$ and let $\iota:S\rightarrow G$ be a map such that $G=\gen{\iota(S)}$. 
By the universal property of free groups, there exists a unique homomorphism $\theta: F\rightarrow G$ such that 
$\theta(a)=\iota(a)$ and $\theta(b)=\iota(b)$. 
In particular, the map $\theta$ is surjective. 
We denote by $(F_i)_{i\geq 1}$ the $p$-central series of $F$, which is recursively defined as \label{series free}
\[
F_1=F \ \ \ \text{and}\ \ \ F_{i+1}=[F,F_i]F_i^p.\] 
We want to stress the fact that the notation we use here for the $p$-central series of $F$ clashes with the notation we have 
adopted so far (see the section ``Exceptions'' from the List of Symbols). 
Define additionally 
$$L=F_3F^p \ \ \text{and} \ \ E=[F,L]F_2^p.$$ 
The notation we just introduced will be valid until the end of Section \ref{section construction}. We will introduce some extra notation between Lemma \ref{e in n} and Lemma \ref{free 2}. 
We refer to the diagram given at the end of the present section for a visualization of the proof of Proposition \ref{proposition construction class 3}.

\begin{lemma}\label{construction indices}
One has $\theta^{-1}(G_2)=F_2$.
\end{lemma}

\begin{proof}
Since $\theta$ is a surjective homomorphism, one has $\theta(F_2)=\theta([F,F]F^p)=G_2G^p=\Phi(G)$ and so, as a consequence of Lemma \ref{frattini comm}, we get $\theta(F_2)=G_2$. In other words, $F_2\subseteq\theta^{-1}(G_2)$. The group $F$ being $2$-generated, we have that 
$|F:F_2|=p^2$, and so $F_2=\theta^{-1}(G_2)$.
\end{proof}

\begin{lemma}\label{free map}
The commutator map induces an alternating map $F/F_2\times F/F_2\rightarrow F_2/L$ whose image generates 
$F_2/L$. Furthermore, the index $|F_2:L|$ is at most $p$.
\end{lemma}

\begin{proof}
We write $\overline{F}=F/L$ and we will use the bar notation for the subgroups of $\overline{F}$.
The subgroup $[F,[F,F]]$ is contained in $F_3$ and $[\overline{F},\overline{F}]$ is central. 
Moreover, $\overline{F}$ being annihilated by $p$, the subgroup $\overline{F_2}$ coincides with $[\overline{F},\overline{F}]$. 
As a consequence of Lemma \ref{tgt}, the commutator map induces a bilinear map 
$\phi:\overline{F}/\overline{F_2}\times\overline{F}/\overline{F_2}\rightarrow\overline{F_2}$ whose image generates 
$\overline{F_2}=[\overline{F},\overline{F}]$. The map $\phi$ is alternating 
because every element of a group commutes with itself. 
By the universal property of the exterior square, 
$\phi$ factors as a surjective homomorphism
$\bigwedge^2(\overline{F}/\overline{F_2})\rightarrow\overline{F_2}$.
As a consequence of Lemma \ref{frattini comm}, the quotient 
$\overline{F}/\overline{F_2}$ is a $2$-dimensional vector space over $\F_p$ and 
$\bigwedge^2(\overline{F}/\overline{F_2})$ has dimension $1$. 
It follows that $\overline{F_2}$ has order at most $p$. 
\end{proof}

\begin{lemma}\label{description L}
One has $\theta^{-1}(G_3)=L$ and $|F_2:L|=p$.
\end{lemma}

\begin{proof} 
As a consequence of Lemma \ref{class 3 G/G3 extraspecial of exp p}, the group $G_3$ contains $G^p$, from which it follows that 
$\theta(L)=G_3$. In particular, the subgroup $L$ is contained in $\theta^{-1}(G_3)$.
As a consequence of Lemma \ref{construction indices}, the subgroup $\theta^{-1}(G_3)$ is contained in 
$F_2$ and $\theta^{-1}(G_3)\neq F_2$, because $G_3\neq G_2$. In particular, $F_2$ is different from $L$.
By Lemma \ref{free map}, the index $|F_2:L|$ is at most $p$, and so we get that $|F_2:L|=p$ and
$\theta^{-1}(G_3)=L$.
\end{proof}

\begin{lemma}\label{e in n}
One has $E\subseteq\ker\theta\cap F_3$.
\end{lemma}

\begin{proof}
The group $E$ is contained in $[F,F_2]F_2^p=F_3$, by definition of the $p$-central series of $F$.
As a consequence of Lemma \ref{class 3 G/G3 extraspecial of exp p}, the image of $L$ under $\theta$ is equal to $G_3$ and, as a consequence of 
Lemma \ref{class 3 G2 elem abelian}, the subgroup $F_2^p$ is contained in the kernel of $\theta$.
It follows that $\theta(E)=\theta([F,L]F_2^p)=[G,G_3]=G_4=\graffe{1}$.
\end{proof}

\noindent
Let $\beta$ be the endomorphism of $F$ sending $a$ to $a^{-1}$ and $b$ to $b^{-1}$, and note that $\beta$ exists by the universal property of free pro-$p$-groups. 
Then $\beta^2$ is equal to $\id_F$, and thus $\beta$ is an automorphism of $F$.
Write $B=\gen{\beta}$ and define the homomorphism $\sigma:B\rightarrow\graffe{\pm 1}$ by $\beta\mapsto -1$. 
We will respect this notation until the end of Section \ref{section construction}.

\begin{lemma}\label{free 2}
The induced action of $B$ on $F/F_2$ and $F_2/L$ is respectively through $\sigma$ and $\sigma^2$.
\end{lemma}

\begin{proof}
By definition of $\beta$, every element of $F$ is inverted by $\beta$ modulo $F_2$; in other words, the action of $B$ on $F/F_2$ is through $\sigma$.
By Lemma \ref{free map}, the commutator map induces a bilinear map 
$\phi:F/F_2\times F/F_2\rightarrow F_2/L$ whose image generates $F_2/L$. The group $B$ acts on $F/F_2$ through 
$\sigma$ and, by Lemma \ref{lemma product of characters}, the action of $B$ on $F_2/L$ is through $\sigma^2$.
\end{proof}

\begin{lemma}\label{free 3}
The induced action of $B$ on $L/F_3$ is through $\sigma$.
\end{lemma}

\begin{proof}
We write $\overline{F}=F/F_3$ and we use the bar notation for its subgroups. 
Then $\overline{L}$ is equal to $\overline{F}^{\, p}$ and, since $[F,[F,F]]$ is contained in $F_3$, the group $\overline{F}$ has class at most $2$. Moreover, we have that
$[F,F]^p\subseteq F_2^p\subseteq F_3$, so $[\overline{F},\overline{F}]$ is annihilated by $p$.
It follows from Lemma 
\ref{p map petrescu} that the $p$-power map is an endomorphism of $\overline{F}$, and therefore $\overline{L}$ is an epimorphic 
image of $\overline{F}/\overline{F_2}$. By Lemma \ref{p-power characters}, the induced action of $B$ on $\overline{L}$ is through $\sigma$. 
\end{proof}

\begin{lemma}\label{free 4}
The induced action of $B$ on $F_3/E$ is through $\sigma$.
\end{lemma}

\begin{proof}
The group $F_2/L$ is cyclic, thanks to Lemma \ref{description L}, so Lemma \ref{cyclic quotient commutators} yields 
$[F_2,F_2]=[F_2,L]$.
It follows that $[F_2,F_2]$ is contained in $E$. The group $[F,F_3]$ is also contained in $E$ and, as a consequence of Lemma \ref{tgt}, 
the commutator map induces a bilinear map 
$\phi:F/F_2\times F_2/L\rightarrow F_3/E$. By definition of $F_3$, the image of $\phi$ generates $F_3/E$. 
By Lemma \ref{free 2}, the induced actions of $B$ on $F/F_2$ and $F_2/L$ are respectively through $\sigma$ and $\sigma^2$
and, by Lemma 
\ref{lemma product of characters}, the action of $B$ on $F_3/E$ is through $\sigma^3=\sigma$. 
\end{proof}

\begin{lemma}\label{beta layers}
The induced action of $B$ on $L/E$ is through $\sigma$.
Moreover, the kernel of $\theta$ is $B$-stable.
\end{lemma}

\begin{proof}
As a consequence of Lemmas \ref{free 3} and \ref{free 4}, the induced actions of $B$ on $L/F_3$ and $F_3/E$ are both through $\sigma$.
It follows from Lemma \ref{abelian>minus quotients} that the action of $B$ on 
$L/E$ is through 
$\sigma$. As a consequence of Lemmas \ref{description L} and \ref{e in n}, one has $E\subseteq \ker\theta\subseteq L$, and, in particular, the action of $B$ on $L/E$ restricts to an action of $B$ on $\ker\theta/E$. It follows that $\ker\theta$ is $B$-stable and the proof is complete.
\end{proof}

\begin{lemma}\label{raccoon}
Given any two generators $x$ and $y$ of $G$, there exists an intense automorphism of $G$ such that $\alpha(x)=x^{-1}$ and 
$\alpha(y)=y^{-1}$. 
\end{lemma}

\begin{proof}
Let $x$ and $y$ be generators of $G$. Without loss of generality, we assume that $\iota(a)=x$ and $\iota(b)=y$.
Let moreover $\bar{\theta}:F/\ker\theta\rightarrow G$ be the isomorphism that is induced from $\theta$. 
By Lemma \ref{beta layers}, the subgroup $\ker\theta$ of $F$ is $B$-stable, and therefore $\beta$ induces an automorphism $\bar{\beta}$ of $F/\ker\theta$.
Define $\alpha: G\rightarrow G$ by $\alpha=\bar{\theta}\circ\bar{\beta}\circ\bar{\theta}^{-1}$. 
Then $\alpha$ is an automorphism $G$ of order $2$ that inverts the generators $x$ and $y$. Proposition \ref{p^4 has cp} yields that $\alpha$ is intense.
\end{proof}

\noindent
We remark that Proposition \ref{proposition construction class 3} follows directly from Lemma \ref{raccoon}.
Moreover, we are also finally ready to give the proof of Theorem \ref{theorem intensity class 3}. 
Proposition \ref{mezzo teorema class 3} gives $(1)\Leftrightarrow(2)$ and $(1)\Rightarrow(3)$. On the other hand, the implication 
$(3)\Rightarrow(2)$ is given by the combination of Lemma \ref{raccoon} and Proposition \ref{p^4 has cp}. The proof of Theorem \ref{theorem intensity class 3} is finally complete.
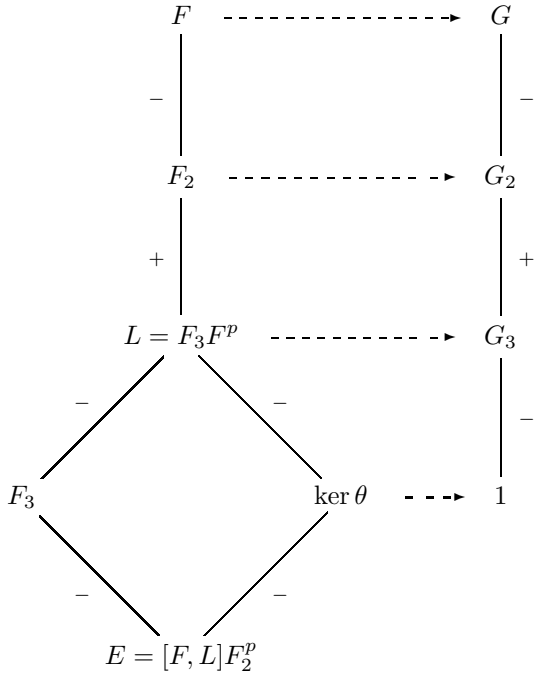
\vspace{10pt}
\[
\begin{diagram}[nohug]
       &          &  F  &  &  & \ \ \rDashto \ \    &  G  \\
       &          & \dLine^{- \ }    &  & &  &    \dLine_{\ -}  \\
       &          &  F_2    &   &   & \ \ \rDashto \ \   & G_2    \\
       &          &  \dLine^{+ \ }   &  & &     &  \dLine_{\ +} \\
       &          &  L=F_3F^p  &      &   & \ \ \rDashto \ \ & G_3    \\
       &  \ldLine^{- \ }  &   & \rdLine^{\ -} & &  &    \dLine_{\ -}   \\
F_3    &               &    &     &       \ker\theta  & \ \ \rDashto \ \ & 1    \\
       &   \rdLine_{- \ } &    &    \ldLine_{\ -} & &   &   \\
       &         &  E=[F,L]F_2^p           &   &   &     &
\end{diagram}
\]


\chapter{Some structural restrictions}\label{chapter break}

In this chapter we will see how the structure of finite $p$-groups whose intensity is greater than $1$ starts becoming more and more rigid, as soon as the class is at least $4$. We recall that, if $(G_i)_{i\geq 1}$ denotes the lower central series of $G$, then the class of $G$ is the number of indices $i$ for which $G_i\neq 1$. Recall moreover that, for each positive integer 
$i$, the $i$-th width of $G$ is $\wt_G(i)=\log_p|G_i:G_{i+1}|$ (see Section \ref{section jumps}).
The main results from Chapter \ref{chapter break} are the following.

\begin{theorem}\label{theorem dimensions}
Let $p$ be a prime number and let $G$ be a finite $p$-group of class at least $4$.
For all $i\in\graffe{1,2,3,4}$, write $w_i=\wt_G(i)$. 
If $\inte(G)>1$, then $(w_1,w_2,w_3,w_4)=(2,1,2,1)$.
\end{theorem}

\begin{theorem}\label{theorem consecutive layers}
Let $p$ be a prime number and let $G$ be a finite $p$-group of class at least $3$.
For all $i\in\Z_{\geq 1}$, write $w_i=\wt_G(i)$. 
Assume that $\inte(G)>1$. Then, for all $i\in\Z_{\geq 1}$, one has $w_iw_{i+1}\leq 2$.
\end{theorem}

\section{Normal subgroups}\label{section higher properties}

We devote Section \ref{section higher properties} to understanding the normal subgroup structure of a finite $p$-group of intensity greater than $1$. We prove the following result.

\begin{proposition}\label{normal squeezed}
Let $p$ be a prime number and let $G$ be a finite $p$-group with $\inte(G)>1$.
Let $N$ be a subgroup of $G$. 
Then $N$ is normal if and only if there exists $i\in\Z_{\geq 1}$ such that $G_{i+1}\subseteq N\subseteq G_i$.
\end{proposition}

\noindent
The following assumptions will be satisfied until the end of Section \ref{section higher properties}.
Let $p$ be a prime number and let $G$ be a finite $p$-group of intensity greater than $1$.
It follows that $p$ is odd and that $G$ is non-trivial (see Section \ref{section main}).
Denote by $(G_i)_{i\geq 1}$ the lower central series of $G$ and, for each positive integer $i$, write $w_i=\wt_G(i)$ for the $i$-th width of $G$. 
Let $\alpha$ be intense of order $2$ and write $A=\gen{\alpha}$. Denote $\chi={\chi_G}_{|A}$, the restriction of the intense character of $G$ to $A$ (once again, we refer to Section \ref{section main}).
In concordance with the notation from Section \ref{section involutions}, let $G^+=\graffe{x\in G\ |\ \alpha(x)=x}$ and $G^-=\graffe{x\in G\ |\ \alpha(x)=x^{-1}}$. For a subgroup $H$ of $G$ we will write $H^+=H\cap G^+$ and $H^-=H\cap G^-$.

\begin{lemma}\label{cyclic subgroup intense}
Let $H$ be an $A$-stable subgroup of $G$. If $H$ is cyclic, then $H\subseteq G^+$ or $H\subseteq G^-$.
\end{lemma}

\begin{proof}
Assume that $H$ is cyclic.
As a consequence of Corollary \ref{abelian sum +-}, the subgroup $H$ decomposes as $H=H^+\oplus H^-$. Then $H$ is cyclic if and only if one of $H^+$ and $H^-$ is trivial. This concludes the proof.
\end{proof}

\noindent
We recall that, as defined in Section \ref{section jumps}, if $x$ is a non-trivial element of $G$, then the depth $\dpt_G(x)$ of $x$ is the unique positive integer $d$ for which $x\in G_d\setminus G_{d+1}$.

\begin{lemma}\label{conjugate of x}
Let $x\in G\setminus\graffe{1}$. Then the following hold. 
\begin{itemize}
 \item[$1$.] The depth of $x$ is even if and only if there exists $g\in G$ such that $gxg^{-1}$ belongs to $G^+$.
 \item[$2$.] The depth of $x$ is odd if and only if there exists $g\in G$ such that $gxg^{-1}$ belongs to $G^-$.
\end{itemize}
\end{lemma}

\begin{proof}
The automorphism $\alpha$ being intense, it follows from Lemma \ref{equivalent intense coprime-pgrps} that there exists $g\in G$ such that $\gen{gxg^{-1}}$ is $A$-stable. Write $d=\dpt_G(x)=\dpt_G(gxg^{-1})$ and $H=\gen{gxg^{-1}}$.
By Lemma \ref{cyclic subgroup intense}, the subgroup $H$ is contained either in $G^+$ or in $G^-$. By Lemma \ref{action chi^i}, the action of $A$ on $(HG_{d+1})/G_{d+1}$ is through $\chi^d$ and the choice between $G^+$ and $G^-$ only depends from the parity of $d$. If $d$ is even, then $\chi^d=1$ and $H$ is contained in $G^+$. Otherwise, $H$ is contained in $G^-$.
\end{proof}

\noindent
We recall that, if $H$ is a subgroup of $G$, then a jump of $H$ in $G$ is a positive integer $j$ 
such that $H\cap G_j\neq H\cap G_{j+1}$. 

\begin{lemma}\label{jumps cyclic sbgs all same parity}
All jumps of a cyclic subgroup of $G$ have the same parity.
\end{lemma}

\begin{proof}
Let $H$ be a cyclic subgroup of $G$.
The automorphism $\alpha$ being intense, it follows from Lemma \ref{equivalent intense coprime-pgrps} that there exists $g\in G$ such that $gHg^{-1}$ is $A$-stable. By Lemma \ref{cyclic subgroup intense}, the subgroup $gHg^{-1}$ is contained in $G^+$ or in $G^-$. We conclude by applying Lemma \ref{conjugate of x}.
\end{proof}

\begin{lemma}\label{centre character sign}
Let $c\in\Z_{\geq 1}$ denote the class of $G$. Then the following hold. 
\begin{itemize}
 \item[$1$.] The induced action of $A$ on $\ZG(G)$ is through $\chi^c$.
 \item[$2$.] If $c$ is even, then $\ZG(G)\subseteq G^+$.
 \item[$3$.] If $c$ is odd, then $\ZG(G)\subseteq G^-$.
\end{itemize}
\end{lemma}

\begin{proof}
The subgroup $G_c$ is contained in $\ZG(G)$ and, by Lemma \ref{action chi^i}, the group $A$ acts on $G_c$ through $\chi^c$. From the combination of Corollary \ref{centre one character} with Lemma \ref{distinct characters on same gp}, it follows that $A$ acts on $\ZG(G)$ through $\chi^c$. If $c$ is even, then $\chi^c=1$ and $\ZG(G)$ is contained in $G^+$. Otherwise, $\chi^c=\chi$ and $\ZG(G)$ is a subset of $G^-$.
\end{proof}

\begin{lemma}\label{centre at the bottom}
Let $c\in\Z_{\geq 1}$ be the class of $G$. 
Then, for all $i\in\graffe{1,\ldots,c}$, if $H$ is a quotient of $G$ of class $i$, then $\ZG(H)=H_i$.
\end{lemma}

\begin{proof}
If $i=1$ the result is clear; we assume that $i$ is at least $2$ and that the result holds for $i-1$.
Let $H$ be a quotient of $G$ of class $i$, which has, thanks to Lemma \ref{intensity of quotients}, intensity greater than $1$.
Let $\beta$ be intense of order $2$ and let $B=\gen{\beta}$ and $\psi={\chi_H}_{|B}$. 
The subgroup $H_i$ is central in $H$, so $\ZG(H)/H_i$ is isomorphic to a subgroup of $\ZG(H/H_i)$. 
By the induction hypothesis $\ZG(H/H_i)=H_{i-1}/H_i$ and it follows that 
$H_i\subseteq\ZG(H)\subseteq H_{i-1}$. By Lemma  
\ref{action chi^i}, the group $B$ acts on $H_{i-1}/H_i$ and $H_i$, respectively through $\psi^{i-1}$ and $\psi^i$, which are distinct characters since $\psi\neq 1$. Moreover, the induced action of $B$ on $\ZG(H)$ is through $\psi^i$, by Lemma \ref{centre character sign}($1$). Lemma \ref{distinct characters on same gp} yields $\ZG(H)=H_i$.
\end{proof}

\noindent
We remark that, to prove Proposition \ref{normal squeezed}, it now suffices to combine Lemma \ref{centre at the bottom} with Lemma \ref{normali schiacchiati generale}.

\section{About the third width}\label{section example}

Let $p$ be a prime number and let $G$ be a finite $p$-group.
If $i$ is a positive integer, we recall that the $i$-th width of $G$ is defined to be $\wt_G(i)=\log_p|G_i:G_{i+1}|$, 
where $(G_i)_{i\geq 1}$ denotes the lower central series of $G$. Thanks to Theorem \ref{theorem class at least 3}($2$), we know that, if $G$ has class at least $3$ and $\inte(G)>1$, then $(\wt_G(1),\wt_G(2))=(2,1)$ and $\wt_G(3)$ is either $1$ or $2$. In the case in which the class of $G$ equals $3$, then both situations $\wt_G(3)=1$ and $\wt_G(3)=2$ occur. What about higher nilpotency classes? We prove the following.

\begin{proposition}\label{p^5, class 4, intensity 1}
Let $p$ be a prime number  and let $G$ be a finite $p$-group of class at least $4$. 
For each positive integer $i$, denote $w_i=\wt_G(i)$. Assume that $\inte(G)>1$.
Then $(w_1,w_2,w_3)=(2,1,2)$.
\end{proposition}

\noindent
Until the end of Section \ref{section example}, the following assumptions will hold.
Let $p$ be a prime number and let $G$ be a finite $p$-group. Let $(G_i)_{i\geq 1}$ denote the lower central series of $G$ and, for each positive integer $i$, denote $w_i=\wt_G(i)$.
We assume that $G$ has class $4$ and that $(w_1,w_2,w_3,w_4)=(2,1,1,1)$. 
We will show that $\inte(G)=1$.

\begin{lemma}\label{max class G4=ZG}
Assume that $p$ is odd. Then $\ZG(G)=G_4$.
\end{lemma}

\begin{proof}
The class of $G$ being $4$, the subgroup $G_4$ is contained in $\ZG(G)$. Now, the group $\ZG(G)/G_4$ is contained in $\ZG(G/G_4)$ and, as a consequence of Lemma \ref{class 3 G3=Z}, the centre of $G/G_4$ is equal to $G_3/G_4$. It follows that 
$G_4\subseteq\ZG(G)\subseteq G_3$. The groups $\ZG(G)$ and $G_3$ are distinct, because the class of $G$ is $4$, so, the index $|G_3:G_4|$ being $p$, we get that $G_4=\ZG(G)$.
\end{proof}

\begin{lemma}\label{max class G2 abelian}
The subgroup $G_2$ is abelian.
\end{lemma}

\begin{proof}
The group $G_2/G_3$ is cyclic, because $w_2=1$, so $[G_2,G_2]=[G_2,G_3]$, by Lemma \ref{cyclic quotient commutators}. It follows from Lemma \ref{commutator indices} that $[G_2,G_2]\subseteq G_5=\graffe{1}$, and thus $G_2$ is abelian.
\end{proof}

\begin{lemma}\label{max class C p}
Assume that $p$ is odd. Then the following hold.
\begin{itemize}
 \item[$1$.] The subgroup $[\Cyc_G(G_3),G_2]$ is contained in $G_4$.
 \item[$2$.] One has $|\Cyc_G(G_3):G_2|=p$.
\end{itemize}
\end{lemma}

\begin{proof} 
To lighten the notation, write $C=\Cyc_G(G_3)$.
We first prove ($1$). The group $[G_2,[G,C]]$ is contained in $[G_2,G_2]$ so it follows from Lemma \ref{max class G2 abelian} that $[G_2,[G,C]]=\graffe{1}$. Moreover, $[C,[G_2,G]]=[C,G_3]=\graffe{1}$, by definition of $C$. Thanks to Lemma \ref{three subgroups lemma}, the subgroup $[G,[C,G_2]]$ is trivial, and thus $[C,G_2]$ is contained in $\ZG(G)$. The centre of $G$ is equal to $G_4$, by Lemma \ref{max class G4=ZG}, and ($1$) is proven. We prove ($2$). Let $D$ be the unique subgroup of $G$ containing $G_4$ such that $D/G_4=\Cyc_{G/G_4}(G_2/G_4)$.
Then one has $[[G,D],G_2]\subseteq [G_2,G_2]=\graffe{1}$ and $[[D,G_2],G]\subseteq [G_4,G]=\graffe{1}$. Lemma \ref{three subgroups lemma} yields that $[D,G_3]=[D,[G,G_2]]=\graffe{1}$ and therefore $D$ is contained in $C$.
It follows from Lemma \ref{class 3 centralizer G2}($1$) that $$p\leq |C:G_2|\leq |G:G_2|=p^2.$$ The group $G_3$ is not central, and so 
$|C:G_2|=p$.
\end{proof}

\begin{lemma}\label{max class C abelian}
If $\inte(G)>1$, then $\Cyc_G(G_3)$ is abelian.
\end{lemma}

\begin{proof}
Assume that $\inte(G)>1$. As a consequence of Proposition \ref{proposition 2gps}, the prime $p$ is odd.
Let $\alpha$ be an intense automorphism of $G$ of order $2$ and write $A=\gen{\alpha}$ and $\chi={\chi_G}_{|A}$. 
To lighten the notation, write $C=\Cyc_G(G_3)$. 
By Lemma \ref{max class C p}($2$), the index of $G_2$ in $C$ is $p$, so it follows from Lemma \ref{cyclic quotient commutators} that $[C,C]=[C,G_2]$. Moreover, $[C,G_2]$ is contained in $G_4$, by Lemma \ref{max class C p}($1$), and $G_4=\ZG(G)$ by Lemma \ref{max class G4=ZG}. By Lemma \ref{tgt}, the commutator map 
$C\times G_2\rightarrow G_4$ is bilinear and, as a consequence of Lemma \ref{max class G2 abelian}, it factors as $\phi:C/G_2\times G_2/G_3\rightarrow G_4$. By Lemma \ref{action chi^i}, the group $A$ acts on $C/G_2$ and $G_2/G_3$ respectively through $\chi$ and $\chi^2$, so, as a consequence of Lemma \ref{lemma product of characters}, the group $A$ acts on $[C,G_2]$ through $\chi^3=\chi$. By Lemma \ref{action chi^i}, the group $A$ acts on $G_4$ through $\chi^4=1$. Since $\chi\neq 1$, it follows from Lemma \ref{distinct characters on same gp} that $[C,C]$ is trivial, and therefore $C$ is abelian.
\end{proof}

\noindent
We recall here that, if $A=\gen{\alpha}$ is a multiplicative group of order $2$ acting on a finite group $B$ of odd order, then one defines $B^+=\graffe{x\in B : \alpha(x)=x}$  and 
$B^-=\graffe{x\in B : \alpha(x)=x^{-1}}$. (See Section \ref{section involutions}.)

\begin{lemma}\label{abelian p^4 plus-minus}
Assume that $\inte(G)>1$ and let $\alpha$ be an intense automorphism of $G$ of order $2$.
Write $C=\Cyc_G(G_3)$.
Then $C=C^+\oplus C^-$ and $|C^+|=|C^-|=p^2$. 
\end{lemma}

\begin{proof}
The group $C$ is $A$-stable and it is abelian by Lemma \ref{max class C abelian}.
From Corollary \ref{abelian sum +-}, it follows that $C=C^+\oplus C^-$. 
The cardinalities of $C^+$ and $C^-$ are both equal to $p^2$, as a consequence of Lemma \ref{order +- jumps}.
\end{proof}

\begin{lemma}\label{M+ EA iff M- EA}
Assume that $\inte(G)>1$ and let $\alpha$ be an intense automorphism of $G$ of order $2$.
Write $C=\Cyc_G(G_3)$.
Then $C^+$ is cyclic if and only if $C^-$ is cyclic.
\end{lemma}

\begin{proof}
The subgroup $C$ is $A$-stable and $C=C^+\oplus C^-$ , by Lemma \ref{abelian p^4 plus-minus}. Moreover, both $C^+$ and $C^-$ have cardinality $p^2$. The subgroup $C^p$ is characteristic in $G$, so $C^p$ is $A$-stable.  
The group $G$ has class $4$ and, as a consequence of Lemma \ref{centre character sign}($2$), the subgroup $\ZG(G)$ is contained in $G^+$.
We first prove the implication from right to left. Assume that $C^-$ is cyclic. Then $C^p$ is a non-trivial subgroup of the 
$p$-group $G$. From Lemma \ref{normal intersection centre trivial}, it follows that $C^p$ has non-trivial intersection with the centre of $G$ and, in particular, 
$C^p\cap G^+\neq\graffe{1}$. The group $C^+$ is cyclic, because it has order $p^2$ and exponent different from $p$.
To prove the implication from left to right, assume that $C^+$ is cyclic. 
As a consequence of Lemma \ref{order +- jumps}, the group $C^+$ is contained in $G_2$. We claim that there exists an element $x\in C\setminus G_2$ of order $p^2$. 
If not, then it means that $C$ is equal to the union of two of its proper subgroups, namely $C\cap G_2$ with $\mu_p(C)$, which is impossible. It follows that there exists $x\in C$ of order $p^2$,  
with $\dpt_G(x)=1$. Fix such $x$. By Lemma \ref{conjugate of x}, there exists $g\in G$ such that $gxg^{-1}$ belongs to 
$G^-$. Since both $C^-$ and $\gen{x}$ have order $p^2$, the group $C^-$ is cyclic.
\end{proof}

\begin{lemma}\label{max class conjs in C}
Let $H$ be a subgroup of $\Cyc_G(G_3)$. If $\inte(G)>1$, then $H$ has at most $p$ conjugates in $G$.
\end{lemma}

\begin{proof}
Assume that $\inte(G)>1$. The group $\Cyc_G(G_3)$ is abelian, by Lemma \ref{max class C abelian}, and therefore $\Cyc_G(G_3)$ 
normalizes $H$. It follows that 
$|G:\nor_G(H)|$ is at most $|G:\Cyc_G(G_3)|$. By Lemma \ref{max class C p}($2$) the index $|G:\Cyc_G(G_3)|$ is equal to $p$, and thus $H$ has at most $p$ conjugates in $G$.
\end{proof}

\begin{lemma}\label{pollo}
Assume that $\inte(G)>1$ and let $\alpha$ be an intense automorphism of $G$ of order $2$.
Write $C=\Cyc_G(G_3)$.
Then $C^+$ is cyclic.
\end{lemma}

\begin{proof}
Assume the contrary. Then, as a consequence of Lemma \ref{M+ EA iff M- EA}, both $C^+$ and $C^-$ are elementary abelian.
From Lemma \ref{abelian p^4 plus-minus} it follows that $C$ is an $\F_p$-vector space of dimension $4$. Let $X$ be the collection of $1$-dimensional subspaces of $C$; then we have 
$$|X|=\frac{p^4-1}{p-1}=p^3+p^2+p+1.$$ 
Let moreover $X^+=\graffe{H\in X \ :\ \alpha(H)=H}$. 
As a consequence of Lemma \ref{cyclic subgroup intense}, the set $X^+$ consists of the $1$-dimensional subspaces of $C$ that are contained in $C^+\cup C^-$. Then $|X^+|=2(p+1)$. 
By Lemma \ref{max class conjs in C}, each element of $X^+$ has at most $p$ conjugates in $G$, so it follows from Lemma \ref{orbits} that 
$$2p(p+1)=p|X^+|\geq\sum_{H\in X^+}|G:\nor_G(H)|\geq |X|=p^3+p^2+p+1.$$
Contradiction.
\end{proof}

\begin{lemma}\label{belva}
The intensity of $G$ is equal to $1$.
\end{lemma}

\begin{proof}
Assume by contradiction that $\inte(G)>1$ and let $\alpha$ be an intense automorphism of $G$ of order $2$. 
Write $C=\Cyc_G(G_3)$. 
The group $C$ is abelian, by Lemma \ref{max class C abelian}, and $C=C^+\oplus C^-$, by Lemma \ref{abelian p^4 plus-minus}. Moreover, $C^+$ and $C^-$ have both cardinality $p^2$.
By Lemma \ref{pollo}, the subgroup $C^+$ is cyclic so, by Lemma \ref{M+ EA iff M- EA}, the subgroup $C^-$ is also cyclic.
Let $X$ be the collection of cyclic subgroups of $C$ of order $p^2$
and let $X^+$ be the subset of $X$ consisting of the $A$-stable ones. 
It follows from Lemma \ref{cyclic subgroup intense} that $X^+=\graffe{C^+, C^-}$ and
the cardinality of $X^+$ is thus $2$. On the other hand, the cardinality of $X$ is equal to 
$$|X|=\frac{p^4-p^2}{p^2-p}=p(p+1).$$ By Lemma \ref{max class conjs in C}, each element of $X^+$ has at most $p$ conjugates, so it follows from Lemma \ref{orbits} that $$2p\geq p|X^+|\geq|X|=p^2+p.$$ Contradiction.
\end{proof}

\noindent
We conclude Section \ref{section example} by giving the proof of Proposition \ref{p^5, class 4, intensity 1}. 
Let $Q$ be a finite $p$-group of class at least $4$ with $\inte(Q)>1$. 
The class of $Q$ being $4$, the subgroup $Q_4$ is non-trivial and, by Lemma \ref{normal index p}, there exists a normal subgroup $M$ of $Q$ that is contained in $Q_4$ with index $p$. Fix $M$ and denote $\overline{Q}=Q/M$. 
Thanks to Lemma \ref{intensity of quotients}, the intensity of $\overline{Q}$ is greater than $1$, so it follows from Theorem \ref{theorem class at least 3}($2$) that 
$(\wt_{\overline{Q}}(1),\wt_{\overline{Q}}(2),\wt_{\overline{Q}}(3),\wt_{\overline{Q}}(4))=(2,1,f,1)$, where $f\in\graffe{1,2}$. 
Lemma \ref{belva} yields $f=2$ and the proof of Proposition \ref{p^5, class 4, intensity 1} is complete.

\section{A bound on the width}\label{section bound dim}

From Section \ref{section jumps}, we recall that, given a finite $p$-group $G$
and a positive integer $i$, the $i$-th width of $G$ is defined to be $\wt_G(i)=\log_p|G_i:G_{i+1}|$. 
The unique purpose of Section \ref{section bound dim} is to prove the following result.

\begin{proposition}\label{consecutive layers}
Let $p$ be a prime number and let $G$ be a finite $p$-group.
Let $c$ denote the class of $G$. Assume $c\geq 3$ and $\inte(G)>1$.
Then, for each $i$ in $\graffe{1,2, \ldots ,c-1}$, one has $\wt_G(i)\wt_G(i+1)\leq 2$. 
\end{proposition}

\noindent
We devote the remaining part of this section to the proof of Proposition \ref{consecutive layers}. Until the end of Section \ref{section bound dim} we work thus under the assumptions of Proposition \ref{consecutive layers}. As a consequence of Proposition \ref{proposition 2gps}, the prime $p$ is odd.

\begin{lemma}\label{dim 1}
One has $\wt_G(1)\wt_G(2)=2$. 
\end{lemma}

\begin{proof}
This follows directly from Theorem \ref{theorem class at least 3}($2$). 
\end{proof}

\begin{lemma}\label{dim phi}
Let $\phi:G/G_2\rightarrow\Hom(G_{c-1}/G_c,G_c)$ be the function defined by $xG_2\mapsto(aG_{c}\mapsto[x,a])$. Then $\phi$ is a homomorphism. 
\end{lemma}

\begin{proof}
The map $\phi$ is induced from the surjective homomorphism from Lemma \ref{tensor LCS} 
and $\phi$ is thus itself a homomorphism of groups.  
\end{proof}

\begin{lemma}\label{dim abelian}
The group $G_{c-1}$ is abelian.
\end{lemma}

\begin{proof}
By Lemma \ref{commutator indices}, the group $[G_{c-1},G_{c-1}]$ is contained in $G_{2c-2}$, which is itself contained in $G_{c+1}$ because $c\geq 3$. Since $G_{c+1}=\graffe{1}$, the group $G_{c-1}$ is abelian.
\end{proof}

\begin{lemma}\label{dim complement}
Let $\alpha$ be an intense automorphism of $G$ of order $\inte(G)$ and write $A=\gen{\alpha}$.
Then $G_c$ has a unique $A$-stable complement in $G_{c-1}$.
\end{lemma}

\begin{proof}
Denote by $\chi$ the restriction of the intense character of $G$ to $A$ (see Section \ref{section main}).
The group $G_{c-1}$ is abelian, by Lemma \ref{dim abelian}, and, by Lemma \ref{action chi^i}, the group $A$ acts on $G_{c-1}/G_c$ and $G_c$ respectively through $\chi^{c-1}$ and $\chi^c$.
The characters $\chi^{c-1}$ and $\chi^c$ are distinct, because $\chi\neq 1$, so, by Theorem \ref{lambda mu}, the subgroup $G_c$ has a unique $A$-stable complement in $G_{c-1}$. 
\end{proof}

\begin{lemma}\label{dim surjective}
The homomorphism $\phi$ from Lemma \ref{dim phi} is surjective.
\end{lemma}

\begin{proof}
Let $\alpha$ be an intense automorphism of $G$ of order $\inte(G)$ and write $A=\gen{\alpha}$.
By Lemma \ref{dim abelian}, the group $G_{c-1}$ is abelian and, by Lemma \ref{dim complement}, there exists a unique $A$-stable complement $M$ of $G_c$ in $G_{c-1}$. Now the group $A$ acts in a natural way on the set of complements of $G_c$ in $G_{c-1}$ and, the automorphism $\alpha$ being intense, it follows from Lemma \ref{equivalent intense coprime-pgrps} that all complements of $G_c$ in $G_{c-1}$ are conjugate to $M$ in $G$. 
On the other hand, by Lemma \ref{graph+complements}, the set of complements of $G_c$ in $G_{c-1}$ consists of the elements $\graffe{mf(m)\ :\ m\in M}$ as $f$ varies in $\Hom(M,G_c)$.
It follows that, for each $f\in\Hom(M,G_c)$, there exists $x\in G$, such that 
$\graffe{mf(m)\ :\ m\in M}=xMx^{-1}$. Fix the pair $(f,x)$. Then, for all $m\in M$, there exists $n\in M$ such that $mf(m)=xnx^{-1}=[x,n]n$. It follows that $n^{-1}m=[x,n]f(m)^{-1}$ belongs to $M\cap G_c=\graffe{1}$, so $m=n$. We have proven that $f(m)=[x,m]$. Now, the groups $\Hom(G_{c-1}/G_c,G_c)$ and $\Hom(M,G_c)$ are isomorphic and, the choice of $f$ being arbitrary,  each homomorphism $f:G_{c-1}/G_c\rightarrow G_c$ is of the form $mG_c\mapsto [x,m]$, for some $x\in G$. In other words, we have proven surjectivity of $\phi$. 
\end{proof}

\begin{lemma}\label{dim c}
One has $\wt_G(c-1)\wt_G(c)\leq 2$.
\end{lemma}

\begin{proof}
The groups $G_{c-1}/G_c$ and $G_c$ are $\F_p$-vector spaces, as a consequence of Lemma \ref{elementary abelian}. It follows that the dimension of $\Hom(G_{c-1}/G_c, G_c)$ is equal to $\wt_G(c-1)\wt_G(c)$. Thanks to Lemma \ref{dim surjective}, the product 
$\wt_G(c-1)\wt_G(c)$ is at most $\wt_G(1)$, which is equal to $2$, by Theorem \ref{theorem class at least 3}($2$). We get thus that $\wt_G(c-1)\wt_G(c)\leq 2$. 
\end{proof}

\noindent
The proof of Proposition \ref{consecutive layers} is now an easy exercise, which we spell out here. If $i=1$, then we are done by Lemma \ref{dim 1}. Assume that $i>1$.
For all indices $j\leq i$, the quotient $G/G_{j+1}$ has class $j$ so, without loss of generality, we assume that $c=i+1$. We conclude by applying Lemma \ref{dim c}.
\vspace{8pt} \\
\noindent
We remark that Theorem \ref{theorem consecutive layers} is the same as Proposition \ref{consecutive layers}. Moreover, Theorem \ref{theorem dimensions} is given by the combination of Propositions \ref{p^5, class 4, intensity 1} and \ref{consecutive layers}.

\chapter{Higher nilpotency classes}\label{chapter higher}

The aim of this chapter is to gain better control of the $p$-power map on finite $p$-groups of intensity greater than $1$. 
We remind the reader that, if $n$ is a positive integer and $G$ is a group, then $G^n$ is equal to the subgroup of $G$ that is generated by the $n$-th powers of the elements of $G$, i.e. $G^n=\gen{x^n : x\in G}$ (see List of Symbols).
One of the most important results we achieve in Chapter \ref{chapter higher} is the following.

\begin{theorem}\label{theorem exp not p}
\pgp\
Assume that the class of $G$ is at least $4$ and that $\inte(G)>1$. 
Then $G^p=G_3$.
\end{theorem}

\noindent
We remark that, whenever $p$ is larger than $3$, Theorem \ref{theorem exp not p} cannot be extended to groups of class $3$. There are indeed examples, for $p>3$, of finite $p$-groups of class $3$, intensity greater than $1$, and exponent $p$. We deal extensively with the case of $3$-groups in Chapter \ref{chapter 3gps}.

\section{Groups of class 4}\label{section warm up 4}

In Section \ref{section warm up 4} we start the preparation for the proof of Theorem \ref{theorem exp not p}, which will be given in Section \ref{section class 4i}. This very section will thus just consist of structural lemmas, which will be later of use.
\vspace{8pt} \\
\noindent
The following assumptions will be valid until the end of Section \ref{section warm up 4}.
Let $p$ be a prime number. Let moreover $G$ be a finite $p$-group of class $4$ and denote by $(G_i)_{i\geq 1}$ the lower central series of $G$. For $i\in\graffe{1,2,3,4}$, we define $w_i$ to be $\wt_G(i)=\log_p|G_i:G_{i+1}|$ (see Section \ref{section jumps}). 

\begin{lemma}\label{class 4 comm map non-deg}
Assume that $(w_1,w_2,w_3,w_4)=(2,1,2,1)$ and $\ZG(G)=G_4$.
Then the commutator map induces a non-degenerate map 
$G/G_2\times G_3/G_4\rightarrow G_4$. 
\end{lemma}

\begin{proof}
The commutator map induces a bilinear map $\gamma:G/G_2\times G_3/G_4\rightarrow G_4$ whose image generates $G_4$, by Lemma \ref{bilinear LCS}. The subgroup $G_4$ has dimension $1$ as an $\F_p$-vector space, while $w_1=w_3=2$. 
The quotient $G/G_2$ has exponent $p$, thanks to Lemma \ref{frattini comm}, and the exponent of $G_3/G_4$ is equal to $p$, as a consequence of Lemma \ref{tensor LCS} (the property of being elementary abelian is preserved by surjective homomorphisms and tensor products). 
Moreover, the centre of $G$ being $G_4$, the right kernel of $\gamma$ is trivial. It follows from Lemma \ref{non-degenerate dim 1} that the left kernel of $\gamma$ has also dimension $0$, so $\gamma$ is non-degenerate.
\end{proof}

\begin{lemma}\label{max sub contains G2}
Assume that $(w_1,w_2,w_3,w_4)=(2,1,2,1)$.
Let $C$ be a maximal subgroup of $G$. Then $C$ contains $G_2$ and 
$|G:C|=|C:G_2|=p$.
\end{lemma}

\begin{proof}
The subgroup $C$ being maximal, it has index $p$ in $G$ and it contains the Frattini subgroup of $G$. Since $G_2$ is contained in $\Phi(G)$ and $|G:C|=p$, we get
$|C:G_2|=p$.
\end{proof}

\begin{lemma}\label{class 4 structure}
Assume that $\inte(G)>1$.
Then $(w_1,w_2,w_3,w_4)=(2,1,2,1)$ and $G$ has order $p^6$.
Moreover, one has $\ZG(G)=G_4$.
\end{lemma} 

\begin{proof}
The quadruple $(w_1,w_2,w_3,w_4)$ is equal to $(2,1,2,1)$ by Theorem \ref{theorem dimensions} and the order of $G$ is equal to $p^6$, as a consequence of Lemma \ref{order product orders jumps}. The centre of $G$ is equal to $G_4$ thanks to Lemma \ref{centre at the bottom}.
\end{proof}

\begin{lemma}\label{centre of max subgp class 4}
Let $C$ be a maximal subgroup of $G$ and assume that $\inte(G)>1$. Then $G_4\subseteq\ZG(C)\subseteq G_3$ and $|G_3:\ZG(C)|=|\ZG(C):G_4|=p$.
\end{lemma}

\begin{proof}
We assume that $\inte(G)>1$.
By Lemma \ref{class 4 structure}, the subgroups $G_4$ and $\ZG(G)$ are the same and $(w_1,w_2,w_3,w_4)=(2,1,2,1)$. Moreover, by Lemma \ref{class 4 comm map non-deg}, the commutator map induces a non-degenerate map $G/G_2\times G_3/G_4\rightarrow G_4$. 
It follows from Lemma \ref{non-degenerate dim 1} that $G_3\cap\ZG(C)$ has index $p$ in $G_3$.
Now the subgroup $\ZG(C)$ is normal in $G$, because it is characteristic in the normal subgroup $C$, and therefore Proposition \ref{normal squeezed} yields $\ZG(C)\subseteq G_3$. We get that $|G_3:\ZG(C)|=|\ZG(C):G_4|=p$.
\end{proof}

\begin{lemma}\label{binocolo}
Let $M$ be a normal subgroup of $G$ that contains $G_4$ with index $p$. If $\inte(G)>1$, then $\Cyc_G(M)$ is a maximal subgroup of $G$.
\end{lemma}

\begin{proof}
We assume that $\inte(G)>1$. Then, by Lemma \ref{class 4 structure}, we have that $\ZG(G)=G_4$ and $(w_1,w_2,w_3,w_4)=(2,1,2,1)$.
Let now $M$ be a normal subgroup of $G$ that contains $G_4$ with index $p$. As a consequence of Proposition \ref{normal squeezed}, the subgroup $M$ is contained in $G_3$.
The commutator map from Lemma \ref{class 4 comm map non-deg} being non-degenerate, it follows from Lemma \ref{non-degenerate dim 1} that $\Cyc_G(M)$ is maximal in $G$.
\end{proof}

\begin{lemma}\label{bijection normal subgroups layers}
Assume that $\inte(G)>1$.
Let $\cor{M}$ be the collection of maximal subgroups of $G$ and let $\cor{N}$ be the collection of normal subgroups of $G$ that contain $G_4$ with index $p$.
Let $f:\mathcal{M}\rightarrow\mathcal{N}$ be defined by $N\mapsto\ZG(N)$. Then $f$ is a bijection with inverse $f^{-1}:M\mapsto\Cyc_G(M)$.
\end{lemma}

\begin{proof}
The map $f$ is well-defined, as a consequence of Lemma \ref{centre of max subgp class 4}. 
Also the map $g:\mathcal{N}\rightarrow\mathcal{M}$, sending $M$ to $\Cyc_G(M)$,
is well-defined thanks to Lemma \ref{binocolo}. It is now easy to show that $f$ and $g$ are inverses to each other.
\end{proof}

\section{Class 4 and $p$-th powers}\label{section class 4p}

The goal of this section is to prove some technical lemmas regarding the $p$-th powering on finite $p$-groups of intensity greater than $1$. We will use such lemmas in Section \ref{section class 4i}, where Theorem \ref{theorem exp not p} is proven.
\vspace{8pt} \\
\noindent
Through the whole of Section \ref{section class 4p}, the following assumptions will be valid.
Let $p$ be a prime number. Let moreover $G$ be a finite $p$-group of class $4$ and denote by $(G_i)_{i\geq 1}$ the lower central series of $G$. Let $\rho:G\rightarrow G$ be defined by $x\mapsto x^p$; this is the first time we introduce this notation, which can be also found in the List of Symbols.
For each $i\in\graffe{1,2,3,4}$, we will write $w_i=\wt_G(i)$ for the $i$-th width of $G$ (see Section \ref{section jumps}). Assume that $\inte(G)>1$. It follows from Proposition \ref{proposition 2gps} that $G$ is non-trivial and $p$ is odd.

\begin{lemma}\label{induced powering on abelianization}
The map $\rho$ induces a map $\overline{\rho}:G/G_2\rightarrow G_3/G_4$.
\end{lemma}

\begin{proof}
For each index $i\in\Z_{\geq 1}$, the subset $\rho(G_i)$ is contained in $G_{i+2}$, as a consequence of Proposition 
\ref{p-power jumps 2}. In particular, $\rho(G_2)$ is contained in $G_4$.
By Lemma \ref{class 4 structure}, the subgroup $\ZG(G)$ is equal to $G_4$ and $(w_1,w_2,w_3,w_4)=(2,1,2,1)$.
Let $x$ be an element of $G$ and define $C=\gen{x,G_2}$; denote by $(C_i)_{i\geq 1}$ the lower central series of $C$. The quotient $C/G_2$ is cyclic, so, thanks to Lemma \ref{cyclic quotient commutators}, the subgroups $C_2$ and $[C,G_2]$ are equal. 
It follows that $C_3=[C,C_2]$ is contained in $G_4$, and, the prime $p$ being odd, we get that $C_2^pC_p$ is contained in $G_4$.
Let now $y\in G_2$. By Lemma \ref{p map petrescu}, we have that
$\rho(xy)\equiv\rho(x)\rho(y)\bmod C_2^pC_p$, and therefore 
$\rho(xy)\equiv\rho(x)\rho(y)\bmod G_4$. Since the element $\rho(y)$ belongs to $G_4$ and the choices of $x$ and $y$ were arbitrary, the map $\overline{\rho}$ is well-defined. 
\end{proof}

\begin{lemma}\label{ciccia}
Let $C$ be a maximal subgroup of $G$ and assume that $\rho(C\setminus G_2)\cap G_4$ is not empty.
Then $\rho(C)\subseteq G_4$.
\end{lemma}

\begin{proof}
By Lemma \ref{induced powering on abelianization}, the 
map $\rho$ induces a function
$\overline{\rho}:G/G_2\rightarrow G_3/G_4$, which then becomes a homomorphism whenever we restrict it to a cyclic subgroup of $G/G_2$. Since $\rho^{-1}(G_4)\cap (C\setminus G_2)$ is non-empty, it generates $C$ modulo $G_2$, and thus
$\overline{\rho}(C/G_2)\subseteq G_4$. It follows that $\rho(C)\subseteq G_4$.
\end{proof}

\begin{lemma}\label{ilovewilliepeyote}
Let $C$ be a maximal subgroup of $G$ and assume that $\rho(C\setminus G_2)\cap G_4$ is not empty. Then $\rho(C\setminus G_2)=\graffe{1}$.
\end{lemma}

\begin{proof}
Let $\alpha$ be an intense automorphism of $G$ of order $2$ and set $A=\gen{\alpha}$.
Let $H$ be a cyclic subgroup of $C$, not contained in $G_2$, and such that $\rho(H)\subseteq G_4$. 
Without loss of generality we assume that $H$ is $A$-stable (otherwise we can take a conjugate of $H$ that is $A$-stable, thanks to Lemma \ref{equivalent intense coprime-pgrps}).
As a consequence of Proposition \ref{proposition -1^i}, the automorphism $\alpha$ induces scalar multiplication by $-1$ on $H/(H\cap G_2)$, so, thanks to Lemma \ref{p-power characters}, the restriction of $\alpha$ to $H^p$ coincides with scalar multiplication by $-1$. However, the subgroup $H^p$ being contained in $G_4$, it follows from Proposition \ref{proposition -1^i} that $\alpha$ coincides with the identity map on $H^p$. Lemma \ref{distinct characters on same gp} yields $H^p=\graffe{1}$, and, the choiche of $H$ being arbitrary, we get $\rho(C\setminus G_2)=\graffe{1}$.
\end{proof}

\section{Class 4 and intensity}\label{section class 4i}

The unique purpose of Section \ref{section class 4i} is to give the proof of the following proposition.

\begin{proposition}\label{olè}
Let $p$ be a prime number and let $G$ be a finite $p$-group of class at least $4$. 
Denote by $(G_i)_{i\geq 1}$ the lower central series of $G$.
If $\inte(G)>1$, then $G^p=G_3$. 
\end{proposition}

\noindent 
Until the end of Section \ref{section class 4i}, the following assumptions will hold. Let $p$ be a prime number and let $G$ be a finite $p$-group of class $4$. Let $\rho:G\rightarrow G$ be defined by $x\mapsto x^p$ (see also the List of Symbols).
Assume that $\inte(G)>1$. It follows that $p$ is odd and the group $G$ is non-trivial (see Section \ref{section main}).
Let $\alpha$ denote an intense automorphism of $G$ of order $2$ and write $A=\gen{\alpha}$. 
Set
$G^+=\graffe{g\in G\ :\ \alpha(g)=g}$ and $G^-=\graffe{g\in G\ :\ \alpha(g)=g^{-1}}$. 
For each maximal subgroup $C$ of $G$, define moreover $Y_C$ to be the collection of abelian subgroups of $G$ that can be written as $\gen{x}\oplus\gen{y}$, with $x\in C\setminus G_2$ and $y\in\ZG(C)\setminus G_4$. 
We will call $Y_C^+$ the set consisting of the $A$-stable elements of $Y_C$.

\begin{lemma}\label{class 4 exp p}
Let $C$ be a maximal subgroup of $G$ and assume that $\rho(C\setminus G_2)\cap G_4$ is not empty.
Let $H$ be an element of $Y_C$. Then $H$ has exponent $p$ and $H\cap G_4=\graffe{1}$.
\end{lemma}

\begin{proof}
Let $H=\gen{x}\oplus\gen{y}$ be an element of $Y_C$, where $x\in C\setminus G_2$ and $y\in\ZG(C)\setminus G_4$.
The group $\ZG(C)$ is normal in $G$, because $\ZG(C)$ is characteristic in the normal subgroup $C$, and, by Lemma \ref{centre of max subgp class 4}, the group $G_3$ contains $\ZG(C)$.
From Proposition \ref{p-power jumps 2}, it follows that $\ZG(C)$ has exponent $p$, and thus $y^p=1$. 
The element $x^p$ is $1$, by Lemma \ref{ilovewilliepeyote}, and so $H^p=\graffe{1}$.
To conclude, assume that $x^ay^b\in H\cap G_4$. Then $x^a=(x^ay^b)y^{-b}$ belongs to $H\cap G_3$, so $a\equiv 0\bmod p$. From the fact that $\gen{y}\cap G_4=\graffe{1}$, we conclude that
$H\cap G_4=\graffe{1}$.
\end{proof}

\begin{lemma}\label{coez}
Let $C$ be a maximal subgroup of $G$ and assume that $\rho(C\setminus G_2)\cap G_4$ is not empty.
If $H\in Y_C^+$, then $H\subseteq G^-$.
\end{lemma}

\begin{proof}
Let $H=\gen{x}\oplus\gen{y}$ be an element of $Y_C^+$, where $x\in C\setminus G_2$ and $y\in\ZG(C)\setminus G_4$. By Lemma \ref{class 4 exp p}, the group $H$ has exponent $p$ and so the order of $H$ is $p^2$.
The element $x$ has depth $1$ and the depth of $y$ is $3$, as a consequence of Lemma \ref{centre of max subgp class 4}. 
It follows from Lemma \ref{order +- jumps} that 
$$|H|\geq|H\cap G^-|=p^{\wt_H^G(1)}p^{\wt_H^G(3)}\geq p^2=|H|.$$ 
All inequalities are in fact equalities and $H\cap G^-=H$.
\end{proof}

\begin{lemma}\label{k7}
Let $C$ be a maximal subgroup of $G$ and assume that $\rho(C\setminus G_2)\cap G_4$ is not empty.
Then the cardinality of $Y_C^+$ is equal to $p$.
\end{lemma}

\begin{proof}
Write $C^-= C\cap G^-$ and $\ZG(C)^-=\ZG(C)\cap G^-$.
Thanks to Lemma \ref{coez} we are reduced to count the subgroups of the form $\gen{x}\oplus\gen{y}$, with $x\in C^-\setminus G_2$ and $y\in\ZG(C)^-\setminus G_4$. By Lemma \ref{centre of max subgp class 4}, the subgroup $\ZG(C)$ contains $G_4$. By Lemma \ref{class 4 structure}, the quadruple 
$(\wt_G(1),\wt_G(2),\wt_G(3),\wt_G(4))$ is equal to $(2,1,2,1)$ so, as a consequence of Lemma \ref{order +- jumps}, the cardinalities of $C^-\setminus G_2$ and $\ZG(C)^-\setminus G_4$ are respectively $p^3-p^2$ and $p-1$. 
Fix now a basis $(x,y)$ for a subgroup $H$, where $x\in C^-\setminus G_2$ and $y\in\ZG(C)^-\setminus G_4$.
Thanks to Lemma \ref{class 4 exp p}, the set of equivalent bases for $H$ is 
$B=\graffe{(x^ay^b,y^c) : a,c\in\F_p^*,\, b\in\F_p}$, and thus $|B|=p(p-1)^2$. The cardinality of $Y_C^+$ is 
$$|Y_C^+|=\frac{|C^-\setminus G_2|\,|\ZG(C^-)\setminus G_4|}{|B|}=
\frac{(p^3-p^2)(p-1)}{p(p-1)^2}=p.$$
\end{proof}

\begin{lemma}\label{comebabycome}
Let $C$ be a maximal subgroup of $G$ and assume that $\rho(C\setminus G_2)\cap G_4$ is not empty.
Then the cardinality of $Y_C$ is equal to $p^4$.
\end{lemma}

\begin{proof}
We want to count the subgroups of the form $\gen{x}\oplus\gen{y}$, with $x\in C\setminus G_2$ and $y\in\ZG(C)\setminus G_4$. 
The quadruples $(\wt_G(1),\wt_G(2),\wt_G(3),\wt_G(4))$ and $(2,1,2,1)$ are the same, by Lemma \ref{class 4 structure}, and so $|C|-|G_2|=p^5-p^4$. Moreover, thanks to Lemma \ref{centre of max subgp class 4}, the set $\ZG(C)\setminus G_4$ has cardinality $p^2-p$. 
Fix now $(x,y)$ a basis for an element $H\in Y_C$, such that $x\in C\setminus G_2$ and $y\in\ZG(C)\setminus G_4$. As a consequence of Lemma \ref{class 4 exp p}, the set of equivalent bases for $H$ is 
$B=\graffe{(x^ay^b,y^c) : a,c\in\F_p^*,\, b\in\F_p}$, and so $B$ has cardinality $p(p-1)^2$.
It follows that 
$$|Y_C|=\frac{|C\setminus G_2|\,|\ZG(C)\setminus G_4|}{|B|}=
\frac{(p^5-p^4)(p^2-p)}{p(p-1)^2}=p^4.$$
\end{proof}

\begin{lemma}\label{transitivium}
Let $C$ be a maximal subgroup of $G$ and assume that $\rho(C\setminus G_2)\cap G_4$ is not empty.
Let $H$ be an element of $Y_C^+$. Then one has $\nor_G(H)=HG_4$ and $|G:\nor_G(H)|\leq p^3$.
\end{lemma}

\begin{proof}
Let $H$ be an arbitrary element of $Y_C^+$.
By Lemma \ref{class 4 structure}, the subgroup $G_4$ is central of order $p$ so, as a consequence of Lemma \ref{class 4 exp p}, the cardinality of $HG_4$ is at least $p^3$. Moreover, by Lemma \ref{class 4 structure}, the cardinality of $G$ is equal to $p^6$.
The subgroup $HG_4$ is contained in the normalizer of $H$ and, in particular,
$|G:\nor_G(H)|\leq|G:HG_4|\leq p^3$. 
Assume by contradiction that there exists $K\in Y_C^+$ such that $\nor_G(K)\neq KG_4$. Then $|G:\nor_G(K)|<p^3$, and thus
it follows from Lemma \ref{orbits} that
$$|Y_C|\leq \sum_{H\in Y_C^+}|G:\nor_G(H)|<|Y_C^+|\, p^3.$$ 
By Lemma \ref{k7}, the cardinality of $Y_C^+$ is equal to $p$, so we get a contradiction to Lemma \ref{comebabycome}.
\end{proof}



\begin{lemma}\label{no torsion in first layer}
One has $\rho^{-1}(G_4)\subseteq G_2$.
\end{lemma}

\begin{proof}
Assume by contradiction that there exists a maximal subgroup $C$ of $G$ such that 
$\rho(C\setminus G_2)\cap G_4$ is not empty. Thanks to Lemma \ref{transitivium}, the normalizer of each element $H$ of $Y_C^+$ is equal to $HG_4$. It follows from the definition of $Y_C$ that, given any $H\in Y_C^+$, the $A$-stable subgroup $\nor_G(H)$ does not contain $G_2$. As a consequence of Lemma \ref{order +- jumps}, the subgroup $G^+$ is not contained in $\nor_G(H)$. From the combination of Lemmas \ref{conjugate stable iff in nor-times-G+} and \ref{orbits}, we get that 
$|Y_C|<\sum_{H\in Y_C^+}|G:\nor_G(H)|$. By Lemma \ref{transitivium}, the normalizer of each element of $Y_C^+$ has index at most $p^3$ in $G$, so, together with Lemmas \ref{k7} and \ref{comebabycome}, we obtain $p^4=|Y_C|<|Y_C^+|\,p^3=p^4$. Contradiction.
\end{proof}

\begin{lemma}\label{corollary bijection rho}
Let $\overline{\rho}$ be as in Lemma \ref{induced powering on abelianization}. 
Let moreover $C$ be a maximal subgroup of $G$.
Then the following hold.
\begin{itemize}
 \item[$1$.] The map $\overline{\rho}$ is a bijection. 
 \item[$2$.] One has $\ZG(C)=C^p$.
\end{itemize}
\end{lemma}

\begin{proof}
The restriction of $\overline{\rho}$ to any cyclic subgroup of $G/G_2$ is a homomorphism, in particular the restriction to $C/G_2$. As a consequence of Lemma \ref{no torsion in first layer}, the subgroup $\overline{\rho}(C/G_2)$ has size $p$, and so $C^p$ is not contained in $G_4$. The subgroup $C^p$ is characteristic in the normal subgroup $C$, and therefore $C^p$ is normal in $G$. It follows from Lemma \ref{normal squeezed} that $C^p$ contains $G_4$, and so, if $x\in C\setminus G_2$, 
then $C^p=\gen{x^p,G_4}$. 
By Lemma \ref{class 4 comm map non-deg}, the commutator map induces a non-degenerate map $\gamma:G/G_2\times G_3/G_4\rightarrow G_4$ and, if $x\in C$, then $\gamma(xG_2,x^pG_4)=1$. It follows that $\gamma(C/G_2,\overline{\rho}(C/G_2))=1$ and so $C^p\subseteq\ZG(C)$. Since $\gamma$ is non-degenerate, we get $C^p=\ZG(C)$, and thus ($2$) is proven. We now prove ($1$).
Denote by $\cor{M}$ the collection of maximal subgroups of $G$.
As a consequence of Lemma \ref{bijection normal subgroups layers}, the quotient $G_3/G_4$ is equal to 
$\bigcup_{N\in\mathcal{M}}\ZG(N)/G_4=\bigcup_{N\in\mathcal{M}}\overline{\rho}(N/G_2)$ and so $\overline{\rho}$ is surjective. By Lemma \ref{class 4 structure}, the indices $|G_1:G_2|$ and $|G_3:G_4|$ are equal, so the map $\overline{\rho}$ is a bijection.
\end{proof}

\noindent
We remark that Theorem \ref{theorem exp not p} is the same as Proposition \ref{olè}, which we now prove. 
Let $Q$ be a finite $p$-group of class at least $4$. Assume that $\inte(G)>1$. As a consequence of Lemma \ref{intensity of quotients}, the group $Q/Q_5$ has intensity greater than $1$, so Lemma \ref{corollary bijection rho} yields 
$Q_3=Q^pQ_5$. The subgroup $Q^p$ being normal in $Q$, it follows from Lemma \ref{normal squeezed} that $Q^p=Q_3$. The proof of Proposition \ref{olè} is now complete.

\section{Groups of class 5}\label{section warm up 5}

In analogy with Sections \ref{section warm up 4} and \ref{section class 4i}, this section serves as foundations for the results in Section \ref{section class 5i}. 
\vspace{8pt} \\
\noindent
Until the end of Section \ref{section warm up 5}, the following assumptions will hold.
Let $p$ be a prime number and let $G$ be a finite $p$-group. Let $\rho:G\rightarrow G$ denote the $p$-the powering on $G$, i.e. the map $x\mapsto x^p$.
Denote by $(G_i)_{i\geq 1}$ the lower central series of $G$ and, for each positive integer $i$, write $w_i=\wt_G(i)$. Assume that $|G_5|=p$, so $G$ has class $5$. Furthermore, assume that $\inte(G)>1$, so $p$ is odd. 
Let $\alpha$ be an intense automorphism of $G$ of order $2$ and write $A=\gen{\alpha}$.

\begin{lemma}\label{class 5 structure, mini}
One has $(w_1,w_2,w_3,w_4,w_5)=(2,1,2,1,1)$ and the order of $G$ is $p^7$.
Moreover, one has $\ZG(G)=G_5$.
\end{lemma}

\begin{proof}
As a consequence of Lemma \ref{intensity of quotients}, the intensity of $G/G_5$ is greater than $1$. The group $G/G_5$ has class $4$, so from Lemma \ref{class 4 structure} it follows that $(w_1,w_2,w_3,w_4)=(2,1,2,1)$ and that the order of $G/G_5$ is $p^6$. 
Since $G_5$ has order $p$, the order of $G$ is equal to $p^7$.
The centre of $G$ is equal to $G_5$ by Lemma \ref{centre at the bottom}. 
\end{proof}

\begin{lemma}\label{sumsumsum41}
The subgroup $G_3$ is abelian and 
$G_2\subseteq\Cyc_G(G_4)$. 
\end{lemma}

\begin{proof}
The group $G_6$ being trivial, both claims follow from Lemma \ref{commutator indices}.
\end{proof}

\begin{lemma}\label{centralino}
One has $|G_3:\Cyc_{G_3}(G_2)|\leq p$.
\end{lemma}

\begin{proof}
To lighten the notation, let $C=\Cyc_{G_3}(G_2)$.
By Lemma \ref{bilinear LCS hk}, the commutator map induces a bilinear map $\gamma:G_2/G_3\times G_3/G_4\rightarrow G_5$ whose right kernel is equal to $C/G_4$. It follows that $\gamma$ induces 
an injective homomorphism $G_3/C\rightarrow\Hom(G_2/G_3,G_5)$.
By Lemma \ref{class 5 structure, mini}, both $w_2$ and $w_5$ are equal to $1$, so $\Hom(G_2/G_3,G_5)$ has order $p$.
In particular, we get $|G_3:C|\leq p$. 
\end{proof}

\begin{lemma}\label{G2 endomorphism in class 5}
The restriction of $\rho$ to $G_2$ is an endomorphism of $G_2$.
\end{lemma}

\begin{proof}
Thanks to Lemma \ref{class 5 structure, mini}, the group $G_2/G_3$ is cyclic and so $[G_2,G_2]=[G_2,G_3]$. From Lemma \ref{commutator indices}, it follows that $[G_2,G_2]$ is contained in $G_5$, which is equal to the centre of $G$ by Lemma \ref{class 5 structure, mini}. In particular, the class of $G_2$ is at most $2$ and, $p$ being odd, Lemma \ref{class at most p-1 regular} yields that $G_2$ is regular. Now the commutator subgroup of $G_2$ is contained in $G_5$, whose order is $p$, so, by Lemma \ref{exponent G2=p, endomorphism}, the restriction of $\rho$ to $G_2$ is an endomorphism of $G_2$.
\end{proof}



\section{Class 5 and intensity}\label{section class 5i}

We recall that, if $G$ is a finite group, we denote by $(G_i)_{i\geq 1}$ the lower central series of $G$ (see List of Symbols). In this section, we prove the following result. 

\begin{proposition}\label{class at least 5, G2^p=G4}
Let $p$ be a prime number and let $G$ be a finite $p$-group of class at least $5$. 
If $\inte(G)>1$, then $G_2^p=G_4$.
\end{proposition}

\noindent
We will keep the following assumptions until the end of Section \ref{section class 5i}. Let $p$ be a prime number and let $G$ be a finite $p$-group. For any positive integer $i$, write $w_i=\wt_G(i)$ and assume that $w_5=1$. Then the class of $G$ is $5$. 
Assume moreover that $\inte(G)>1$, so, thanks to Proposition \ref{proposition 2gps}, the prime $p$ is odd. 
Let $\alpha$ be an intense automorphism of $G$ of order $2$ and write $A=\gen{\alpha}$.
In concordance with the notation from Section \ref{section involutions}, write $G^+=\graffe{x\in G : \alpha(x)=x}$. 
In conclusion, define $X$ to be the collection of all subgroups of $G$ whose jumps in $G$ are exactly $2$ and $4$ and denote $X^+=\graffe{H\in X : \alpha(H)=H}$. In this section, the List of Symbols will be fully respected.

\begin{lemma}\label{class 5 p2}
The elements of $X$ have order $p^2$.
\end{lemma}

\begin{proof}
Let $H$ be an element of $X$.
As a consequence of Lemma \ref{class 5 structure, mini}, 
both widths $\wt_H^G(2)$ and $\wt_H^G(4)$ are equal to $1$. Now apply Lemma \ref{order product orders jumps}.
\end{proof}

\begin{lemma}\label{H as sum of x and y, class 5}
Assume that $G_2$ has exponent $p$. Let $H$ be a subgroup of $G$.
Then $H\in X$ if and only if there exist $x\in G_2\setminus G_3$ and $y\in G_4\setminus G_5$ such that $H=\gen{x}\oplus\gen{y}$.
\end{lemma} 

\begin{proof}
If $H=\gen{x}\oplus\gen{y}$, with $x\in G_2\setminus G_3$ and $y\in G_4\setminus G_5$, then $H$ belongs to $X$, thanks to Lemma \ref{jumps and depth}. We prove the converse.
The subgroup $H$ has order $p^2$, by Lemma \ref{class 5 p2}, and $H$ cannot be cyclic, because it is contained in $G_2$. 
The jumps of $H$ in $G$ being $2$ and $4$, it follows from Lemma $\ref{jumps and depth}$ that there exist elements $x$ and $y$ in $H$ of depths respectively $2$ and $4$ in $G$. As a consequence of Lemma \ref{sumsumsum41}, the subgroup $H$ decomposes as $H=\gen{x}\oplus\gen{y}$.
\end{proof}

\begin{lemma}\label{fat lip}
Assume that $G_2$ has exponent $p$. Then $|X|=p^4$.
\end{lemma}

\begin{proof}
Thanks to Lemma \ref{H as sum of x and y, class 5}, all elements $H$ of $X$ are of the form $H=\gen{x}\oplus\gen{y}$, with $x\in G_2\setminus G_3$ and $y\in G_4\setminus G_5$.
Let $(x,y)\in (G_2\setminus G_3)\times(G_4\setminus G_5)$ and let $H$ be the $\F_p$-vector space that is spanned by $x$ and $y$. The collection of equivalent bases for $H$ is 
$B=\graffe{(x^ay^b,y^c)\ :\ a,c\in\F_p^*,\, b\in\F_p}$ and so $B$ has cardinality $p(p-1)^2$.
From Lemma \ref{class 5 structure, mini} it follows that the cardinalities of $G_2\setminus G_3$ and $G_4\setminus G_5$ are respectively $p^5-p^4$ and $p^2-p$. We conclude by computing 
$$|X|=\frac{|G_2\setminus G_3|\,|G_4\setminus G_5|}{|B|}=
\frac{(p^5-p^4)(p^2-p)}{p(p-1)^2}=p^4.$$
\end{proof}

\begin{lemma}\label{concerto a febbraio!}
One has $X^+=\graffe{G^+}$.
\end{lemma}

\begin{proof}
From Lemmas \ref{order +- jumps} and \ref{class 5 structure, mini}, it follows that $G^+$ has order $p^2$. 
Let now $H$ be an element of $X^+$. It follows from Lemma \ref{order +- jumps} that $H\cap G^+$ has cardinality $p^2$ and, thanks to Lemma \ref{class 5 p2}, the subgroups $H$ and $G^+$ are the same. 
In particular, the only element of $X^+$ is $G^+$.
\end{proof}

\begin{lemma}\label{foreign beggars}
The exponent of $G_2$ is different from $p$.
\end{lemma}

\begin{proof}
We work by contradiction, assuming that the exponent of $G_2$ is $p$.
To lighten the notation, write $C=\Cyc_{G_3}(G_2)$ and $N=CG^+$. The group $C$ is characteristic in $G$, so $N$ is a subgroup of $G$. Moreover, $G^+$ is contained in $G_2$, thanks to Lemma \ref{order +- jumps}, so $N$ is a subgroup of $\nor_G(G^+)$.
By Lemma \ref{centralino}, the group $C$ is contained in $G_3$ with index at most $p$ and, by Lemma \ref{class 5 structure, mini}, the quadruple $(w_2,w_3,w_4,w_5)$ is equal to $(1,2,1,1)$.  It follows from Lemma \ref{order product orders jumps} 
that the order of $N$ is at least $p^4$. 
The order of $G$ is $p^7$, by Lemma \ref{class 5 structure, mini}, and thus $|G:N|\leq p^3$. By Lemma \ref{concerto a febbraio!}, the set $X^+$ has only one element, namely $G^+$, so Lemma \ref{orbits} yields
$$|X|\leq|G:\nor_G(G^+)|\leq |G:N|\leq p^3.$$ 
Contradiction to Lemma \ref{fat lip}.
\end{proof}

\begin{lemma}\label{rho}
One has $\rho(G_2)=G_4$.
\end{lemma}

\begin{proof}
As a consequence of Lemma \ref{G2 endomorphism in class 5}, the set $\rho(G_2)$ is a characteristic subgroup of $G$ and, by Lemma \ref{foreign beggars}, it is non-trivial. By Lemma \ref{class 5 structure, mini}, the centre of $G$ is equal to $G_5$ so, as a consequence of Lemma \ref{normal intersection centre trivial}, the intersection $G_5\cap\rho(G_2)$ is non-trivial. 
The order of $G_5$ being $p$, the subgroup $\rho(G_2)$ contains $G_5$.
Thanks to Proposition \ref{p-power jumps 2}, the quotient $G_2/G_4$ is elementary abelian and so $G_5\subseteq\rho(G_2)\subseteq G_4$. By Lemma \ref{class 5 structure, mini}, the dimension of $G_4/G_5$ is $1$ and therefore there are only two possibilities: either $\rho(G_2)=G_4$ or $\rho(G_2)=G_5$. In the first case we are done, so assume by contradiction the second. 
Then, by Lemma \ref{jumps cyclic sbgs all same parity}, each element of $G_2\setminus G_3$ has order $p$. It follows that 
$G_2$ is equal to the union of two proper subgroups, namely $\ker\rho_{|G_2}$ and $G_3$. Contradiction.
\end{proof}

\noindent
We are finally ready to prove Proposition \ref{class at least 5, G2^p=G4}. To this end, let $Q$ be a finite $p$-group of class at least $5$ with $\inte(Q)>1$.
Then the group $Q/Q_6$ has class $5$ and, as a consequence of Lemma \ref{intensity of quotients}, the intensity of $Q/Q_6$ is greater than $1$. By Lemma \ref{rho}, the subgroups 
$(Q_2/Q_6)^p$ and $Q_4/Q_6$ are equal, and so $Q_2^pQ_6=Q_4$. The subgroup $Q_2^p$ being normal in $G$, it follows from Lemma \ref{normal squeezed} that $Q_2^p=Q_4$. This concludes the proof of Proposition \ref{class at least 5, G2^p=G4}.
\vspace{8pt} \\
\noindent
We remark that Proposition \ref{class at least 5, G2^p=G4} can be easily derived, when $p$ is greater than $3$, from Theorem \ref{theorem exp not p}. We will show a way of doing so in Section \ref{section different primes}.

\chapter{A disparity between the primes}\label{chapter different primes}

The main result of Chapter \ref{chapter different primes} is Theorem \ref{theorem dimension layers}. We recall that, if $G$ is a finite $p$-group and $i$ is a positive integer, then $\wt_G(i)=\log_p|G_i:G_{i+1}|$, where $(G_i)_{i\geq 1}$ denotes the lower central series of $G$.

\begin{theorem}\label{theorem dimension layers}
Let $p>3$ be a prime number and let $G$ be a finite $p$-group with $\inte(G)>1$. 
Let $c$ denote the class of $G$ and assume that $c\geq 3$. 
If $i$ is a positive integer such that $\wt_G(i)\wt_G(i+1)=1$, then $i=c-1$. 
\end{theorem}

\noindent
An equivalent way of formulating Theorem \ref{theorem dimension layers} is that of saying that, if $G$ satisfies the assumptions of Theorem \ref{theorem dimension layers} and we writw $w_i=\wt_G(i)$, then 
\[(w_i)_{i\geq 1}=(2,1,2,1,\ldots,2,1,f,0,0,\ldots) \ \ \text{where} \ \ f\in\graffe{0,1,2}.\]
The restriction to primes greater than $3$ in Theorem \ref{theorem dimension layers} is superfluous; it is however not worth the effort proving the result in general, since, as we will see in the next chapter, $3$-groups of intensity greater than $1$ have class at most $4$ and we know from Theorems \ref{theorem class at least 3}($2$) and \ref{theorem dimensions} that Theorem \ref{theorem dimension layers} is valid when $c$ is $3$ or $4$.

\section{Regularity}\label{section different primes}

In Section \ref{section different primes} we make a distinction, for the first time, among the odd primes: namely we separate the cases $p=3$ and $p>3$. 
The main result of this section is
Proposition \ref{regular proposition equivalence}. We refer to Section \ref{section regular}, for an overview of regular $p$-groups.

\begin{proposition}\label{regular proposition equivalence}
Let $p$ be a prime number and let $G$ be a finite $p$-group. Assume that 
$\inte(G)>1$. 
Then the following are equivalent.
\begin{itemize}
 \item[$1$.] The group $G$ is not regular.
 \item[$2$.] The class of $G$ is larger than $2$ and $p=3$.
\end{itemize}
\end{proposition}

\noindent
We will give the proof of Proposition \ref{regular proposition equivalence} at the end of Section \ref{section different primes}.

\begin{lemma}\label{regular and pp=p3}
Let $p>3$ be a prime number and let $G$ be a finite $p$-group. 
Assume that $\inte(G)>1$. Then the following hold.
\begin{itemize}
 \item[$1$.] The group $G$ is regular.
 \item[$2$.] If the class of $G$ is at least $4$, then $G_3=\rho(G)$.
\end{itemize}
\end{lemma}

\begin{proof}
If the class of $G$ is at most $4$, the group $G$ is regular by Lemma \ref{class at most p-1 regular}. 
We assume that $G$ has class at least $4$. 
It follows from Lemma \ref{intensity of quotients} that $\inte(G/G_5)$ is larger than $1$, and so, thanks to Lemma \ref{class 4 structure}, the index 
$|G:G_3|$ is equal to $p^3$.
From Theorem \ref{theorem exp not p}, we get that $G^p=G_3$, and therefore $|G:G^p|<p^p$. 
The group $G$ is regular, by Lemma \ref{regular if omega small}, so, thanks to Lemma \ref{regular implies power abelian}, the subgroup $G^p$ coincides with $\rho(G)$. The proof is now complete. 
\end{proof}

\noindent
We would like to stress that, for $p>3$, Proposition \ref{regular and pp=p3}($2$) is a stronger version of Theorem \ref{theorem exp not p}. In fact, not only $G_3=G^p=\gen{\rho(G)}$ but $G_3$ coincides with the set of $p$-th powers of elements of $G$.

\begin{lemma}\label{not regular for p=3}
Let $G$ be a finite $3$-group with $\inte(G)>1$. 
Then $G$ is regular if and only if $G$ has class at most $2$.
\end{lemma}

\begin{proof}
If $G$ has class at most $2$, then $G$ is regular by Lemma \ref{class at most p-1 regular}. Assume by contradiction that $G$ is regular of class at least $3$. As a consequence of Theorem \ref{theorem class at least 3}($2$), the group $G$ is $2$-generated, and so, by Lemma \ref{regular 3-gp 2gen}, the subgroup $G_2$ is cyclic.
Proposition \ref{p-power jumps 2} yields that $G_3=\graffe{1}$. Contradiction.
\end{proof}

\noindent
We now give the proof of Proposition \ref{regular proposition equivalence}. To this end, let $p$ be a prime number and let
$G$ be a finite $p$-group with $\inte(G)>1$.
The intensity of $G$ being greater than $1$, it follows from Proposition \ref{proposition 2gps} that $p$ is odd. 
The implication $(2)\Rightarrow (1)$ is given by Lemma \ref{not regular for p=3}. We prove $(1)\Rightarrow(2)$.
Assume that $G$ is not regular. Then Lemma \ref{regular and pp=p3} yields $p=3$.
Moreover, the class of $G$ is larger than $2$, as a consequence of Lemma \ref{class at most p-1 regular}.
The proof of Proposition \ref{regular proposition equivalence} is complete.

\section{Rank}\label{rank}

\noindent
The \indexx{rank} of a finite group $G$ is the smallest integer $r$ with the property that each subgroup of $G$ can be generated by $r$ elements. We denote the rank of $G$ by $\rk(G)$. We will prove the following.

\begin{proposition}\label{rank 3}
Let $p>3$ be a prime number and let $G$ be a finite $p$-group of class at least $4$.
If $\inte(G)>1$, then $\rk(G)=3$.
\end{proposition}

\noindent
We recall that, if $G$ is a group and $n$ is a positive integer, then the subgroup $\mu_n(G)$ is defined to be $\gen{x\in G : x^n=1}$.

\begin{lemma}\label{lemma laffey}
Let $p$ be a prime number and let $G$ be a non-trivial finite $p$-group. 
Then $\rk(G)\leq\log_p|\mu_p(G)|$.
\end{lemma}

\begin{proof}
This is is Corollary $2$ from \cite{laffey}.
\end{proof}

\begin{lemma}\label{rk at most 3}
Let $p>3$ be a prime number and let $G$ be a finite $p$-group of class at least $4$.
If $\inte(G)>1$, then $\rk(G)\leq 3$.
\end{lemma}

\begin{proof}
Assume that $\inte(G)>1$.
By Theorem \ref{theorem exp not p}, the subgroup $G^p$ is equal to $G_3$ so, by Lemma \ref{regular implies power abelian}($3$), the order of $\mu_p(G)$ is equal to $|G:G^p|=|G:G_3|$. As a consequence of Theorem \ref{theorem dimensions}, the index 
$|G:G_3|$ is equal to $p^3$, and thus Lemma \ref{lemma laffey} yields
$\rk(G)\leq \log_p|G:G_3|=3$.
\end{proof}

\noindent
We can now finally prove Proposition \ref{rank 3}. In order to do this, let $p$ be a prime number and let $G$ be a finite $p$-group of class at least $4$, with $\inte(G)>1$.
Thanks to Lemma \ref{rk at most 3}, it suffices to present a subgroup of $G$ whose minimum number of generators is at least $3$. 
The group $G/G_5$ has class $4$ and, thanks to Lemma \ref{intensity of quotients}, it has intensity greater than $1$.
As a consequence of Theorem \ref{theorem dimensions}, the index 
$|G_2:G_4|$ is equal to $p^3$ and, thanks to Proposition \ref{p-power jumps 2}, the quotient $G_2/G_4$ is elementary abelian. It follows that $\Phi(G_2)\subseteq G_4$ and the minimum number of generators for $G_2$ is at least 
$\log_p(|G_2:G_4|)=3$. Proposition \ref{rank 3} is now proven.
\vspace{8pt} \\
\noindent
We would like to remark that, if $p=3$, then Proposition \ref{rank 3} is not valid. We will see indeed in the next chapter that finite $3$-groups of class $4$ and intensity larger than $1$ have a commutator subgroup that is elementary abelian of order $p^4$, so the rank of such groups is at least $4$.

\section{A sharper bound on the width}\label{section limitation}

The aim of  Section \ref{section limitation} is to give the proof of Proposition \ref{finalmente!}, which is the same as Theorem \ref{theorem dimension layers}.

\begin{proposition}\label{finalmente!}
Let $p>3$ be a prime number and let $G$ be a finite $p$-group with $\inte(G)>1$. 
Let $c$ denote the class of $G$ and assume that that $c\geq 3$. 
If $i$ is a positive integer such that $\wt_G(i)\wt_G(i+1)=1$, then $i=c-1$.
\end{proposition}

\noindent
We list here a number of assumptions that will hold until the end of Section \ref{section limitation}. 
Let $p>3$ be a prime number and let $G$ be a finite $p$-group with lower central series $(G_i)_{i\geq 1}$. 
Let $c$ denote the class of $G$ and, for each positive integer $i$, write $w_i=\wt_G(i)$.
Assume that $\inte(G)>1$. Then, as a consequence of Proposition \ref{proposition 2gps}, the prime $p$ is odd and $G$ is non-trivial.
Let $\alpha$ be an intense automorphism of $G$ of order $2$ and write $A=\gen{\alpha}$.

\begin{lemma}\label{i>1}
Let $i\in\Z_{\geq 1}$ be such that $w_iw_{i+1}=1$. 
Then $i>1$.
\end{lemma}

\begin{proof}
The subgroup $G_{i+1}$ being non-trivial, Lemma \ref{index G'} yields $i>1$.
\end{proof}

\begin{lemma}\label{hosonno}
Assume that $w_2w_3=1$.
Then $c=3$.
\end{lemma}

\begin{proof}
The widths $w_2$ and $w_3$ are both equal to $1$, so Theorem \ref{theorem dimensions} yields $c=3$.
\end{proof}

\begin{lemma}\label{cinquina}
Let $i\in\Z_{\geq 1}$ be such that $w_iw_{i+1}=1$. 
If $c>3$, then $i\geq 4$.
\end{lemma}

\begin{proof}
Assume that the class of $G$ is at least $4$. Then,  by Lemma \ref{intensity of quotients}, the group $G/G_5$ has intensity greater than $1$. Theorem \ref{theorem dimensions} yields $(w_1,w_2,w_3,w_4)=(2,1,2,1)$ and therefore $i\geq 4$.
\end{proof}

\begin{lemma}\label{prezzemolo}
Let $i\in\Z_{\geq 1}$ be minimal such that $w_iw_{i+1}=1$. 
If $c>3$, then $i$ is even and $w_{i-1}=2$.
\end{lemma}

\begin{proof}
Assume $c>3$. Then Lemma \ref{cinquina} yields $i-1>1$.
The width $w_{i-1}$ is at most $2$, as a consequence of Theorem \ref{theorem consecutive layers}, and thus, $i$ being minimal with the property that $w_iw_{i+1}=1$, it follows that $w_{i-1}=2$.
Another consequence of the minimality of $i$ is that $i$ is even. Indeed, thanks to Theorem \ref{theorem consecutive layers} and the minimality of $i$, whenever $j<i$, the product $w_jw_{j+1}$ is equal to $2$. Moreover, by Theorem \ref{theorem class at least 3}($2$), we have that $w_1=2$, so $i$ is even. 
\end{proof}

\begin{lemma}\label{ozonosfera}
Let $i\in\Z_{\geq 1}$ be minimal such that $w_iw_{i+1}=1$. 
Assume that $c>3$ and that $w_{i+2}=1$. 
Then $G_{i-1}/G_{i+3}$ has exponent $p$.
\end{lemma}

\begin{proof}
We write $\overline{G}=G/G_{i+3}$ and we will use the bar notation for the subgroups of $\overline{G}$. The intensity of $\overline{G}$ is larger than $1$ thanks to Lemma \ref{intensity of quotients}.
The group $[G_{i-1},G_{i-1}]$ is contained in $G_{2i-2}$, by Lemma \ref{commutator indices}, and, by Lemma \ref{cinquina}, the index $i$ is larger than $3$. 
It follows that
$[G_{i-1},G_{i-1}]\subseteq G_{2i-2}\subseteq G_{i+2}$, and therefore 
$\overline{G}_{i-1}$ has class at most $2$ and $[\overline{G}_{i-1},\overline{G}_{i-1}]^p=\graffe{1}$. 
As a consequence of Corollary \ref{p map petrescu}, the 
$p$-power map is an endomorphism of $\overline{G}_{i-1}$. Now, thanks to Proposition \ref{p-power jumps 2}, the subgroup $\overline{G}_{i-1}^{\,p}$ is contained in $\overline{G}_{i+1}$ and, from Lemmas \ref{regular implies power abelian}($3$) and \ref{prezzemolo}, it follows that 
$|\mu_p(\overline{G}_{i-1})|=|\overline{G}_{i-1}:\overline{G}_{i-1}^{\,p}|\geq 
|\overline{G}_{i-1}:\overline{G}_{i+1}|=p^3$. 
Also the order of $\overline{G}_{i}$ is equal to $p^3$ and, as a consequence of Proposition \ref{normal squeezed}, the subgroup 
$\overline{G}_{i}$ is contained in $\mu_p(\overline{G}_{i-1})$. Hence the $p$-power map factors thus as a homomorphism 
$\overline{G}_{i-1}/\overline{G}_{i}\rightarrow \overline{G}_{i+1}$. 
By Lemma \ref{prezzemolo}, the index $i$ is even, and so, by Proposition \ref{proposition -1^i}, the automorphism of 
$\overline{G}_{i-1}/\overline{G}_{i}$ that is induced by $\alpha$ is equal to the inversion map. It follows from Lemma \ref{p-power characters} that $\alpha$ restricts to the inversion map on $\overline{G}_{i-1}^{\, p}$. 
Moreover, again by Proposition \ref{proposition -1^i}, the action of $A$ on $\overline{G}_{i+2}$ is trivial. 
It follows from Lemma \ref{distinct characters on same gp} that 
$\overline{G}_{i-1}^{\, p}\cap\overline{G}_{i+2}=\graffe{1}$. 
The subgroup $\overline{G}_{i-1}^p$ is clearly characteristic in $\overline{G}$, while the subgroup $\overline{G}_{i+2}$ is equal to the centre of $\overline{G}$, by Lemma \ref{centre at the bottom}. Lemma \ref{normal intersection centre trivial} yields $\overline{G}_{i-1}^{\, p}=\graffe{1}$.
\end{proof}

\noindent
We conclude Section \ref{section limitation} with the proof of Proposition \ref{finalmente!}.
By Lemma \ref{i>1}, the integer $i$ is larger than $1$. If $i=2$, then Lemma \ref{hosonno} yields $c=3=i+1$. 
We assume that $i$ is greater than $2$, so $c>3$, and,
without loss of generality, that $i$ is minimal with the property that $w_iw_{i+1}=1$. In particular, the subgroup $G_{i+1}$ is non-trivial. If $G_{i+2}=\graffe{1}$, then the class of $G$ is equal to $i+1$, and so $i=c-1$.
Assume now by contradiction that $G_{i+2}$ is non-trivial. 
By Lemma \ref{normal index p}, there exists a normal subgroup $N$ of $G$ that is contained in $G_{i+2}$ with index $p$. We fix $N$ and denote the quotient $G/N$ by $\overline{G}$. Lemma \ref{intensity of quotients} yields $\inte(\overline{G})>1$. 
By Lemma \ref{prezzemolo}, the width 
$\wt_{\overline{G}}(i-1)=w_{i-1}$ is equal to $2$ so, by Lemma \ref{order product orders jumps}, the order of $\overline{G}_{i-1}$ is equal to $p^5$.  By Lemma \ref{cinquina}, the index $i$ is at least $4$, and thus, as a consequence of Lemma \ref{commutator indices}, the subgroup 
$[\overline{G}_{i-1},\overline{G}_{i}]$ is contained in $\overline{G}_{i+3}=\graffe{1}$. It follows that $\overline{G}_{i-1}$ and $\overline{G}_{i}$ centralize each other. Let now $M$ be a maximal subgroup of $\overline{G}_{i-1}$ that contains 
$\overline{G}_{i}$. The index $|M:\overline{G}_i|$ is equal to $p$, because $w_{i-1}=2$, and so Lemma \ref{cyclic quotient commutators} gives $[M,M]=[M,\overline{G}_i]=\graffe{1}$.
Moreover, the order of $M$ is equal to $p^4$ and $M$ has exponent $p$, because of Lemma \ref{ozonosfera}. In particular, $M$ is a $4$-dimensional vector space over $\F_p$. Contradiction to proposition \ref{rank 3}.


\chapter{The special case of 3-groups}\label{chapter 3gps}

\noindent
Let $R=\F_3[\epsilon]$ be of cardinality $9$, with $\epsilon^2=0$.
Denote by $\A$ the quaternion algebra 
\[\A=R+R\mathrm{i}+R\mathrm{j}+R\mathrm{k}\] with defining relations
$\mathrm{i}^2=\mathrm{j}^2=\epsilon$ and $\mathrm{k}=\mathrm{ji}=-\mathrm{ij}$. Let the \emph{bar map} on $\A$ be defined by 
\[x=a+b\mathrm{i}+c\mathrm{j}+d\mathrm{k}\ \mapsto \
\overline{x}=a-b\mathrm{i}-c\mathrm{j}-d\mathrm{k}.\]
We write $\mathfrak{m}=\A\mathrm{i}+\A\mathrm{j}$, which is a $2$-sided nilpotent ideal of $\A$, and we define
$\yo$ to be the subgroup of $1+\mathfrak{m}$ consisting of those elements $x$ satisfying $\overline{x}=x^{-1}$. 
The main result of this chapter is the following.

\begin{theorem}\label{theorem 3-groups}
Let $G$ be a finite $3$-group. Then the following are equivalent.
\begin{itemize}
 \item[$1$.] The group $G$ has class at least $4$ and $\inte(G)>1$.
 \item[$2$.] The group $G$ has class $4$, order $729$, and $\inte(G)=2$.
 \item[$3$.] The group $G$ is isomorphic to $\yo$.
\end{itemize}
\end{theorem}

\noindent
A considerable part of the present chapter is devoted to the proof of Theorem \ref{theorem 3-groups}, which is given in Section \ref{section 3-gps sufficient}. An essential contribution to it is given by the theory of ``$\kappa$-groups'' we develop.

\begin{definition}\label{definition kappa}
A $\kappa$-group \index{$\kappa$-group} is a finite $3$-group $G$ such that 
$|G:G_2|=9$ and with the property that the cubing map on $G$ induces a bijective map $\kappa: G/G_2\rightarrow G_3/G_4$.
\end{definition}

\noindent
Our interest in $\kappa$-groups arises from Lemma \ref{corollary bijection rho}($1$), which asserts that, if $p$ is an odd prime number and $G$ is a finite $p$-group of class at least $4$ with $\inte(G)>1$, then the map $x\mapsto x^p$ induces a bijection $\overline{\rho}:G/G_2\rightarrow G_3/G_4$. As a consequence of Theorem \ref{theorem class at least 3}, each finite $3$-group of class at least $4$ and intensity greater than $1$ is a $\kappa$-group, where $\kappa$ coincides with $\overline{\rho}$. 
The reason why, in this chapter, we work exclusively with $3$-groups is that they are more ``difficult to deal with'': several techniques that apply to the case in which $p$ is a prime larger than $3$ do not apply to the case of $3$-groups of higher class, as the results from Chapter \ref{chapter different primes} suggest. For example, it is not difficult to show, using results from Section \ref{section regular}, that, whenever $p>3$ and $G$ is a finite $p$-group, the map $\overline{\rho}:G/G_2\rightarrow G_3/G_4$ from Lemma \ref{corollary bijection rho} is an isomorphism of groups, while, if $G$ is a $\kappa$-group, then, given any two elements $x,y\in G/G_2$, one has
\[ \kappa(xy)\equiv \kappa(x)\kappa(y)[xy^{-1},[x,y]]\bmod G_4,\]
as we show in Lemma \ref{formula cubing}.
What plays in our favour is that a finite $3$-group $G$ is a $\kappa$-group if and only if $G/G_4$ is a $\kappa$-group: to detect $\kappa$-groups it is thus sufficient to be able to detect $\kappa$-groups among the finite $3$-groups of class $3$. We will prove the following result.

\begin{theorem}\label{theorem unique kappa}
Let $G$ be a finite $3$-group of class $3$. 
Then $G$ is a $\kappa$-group if and only if $G$ is isomorphic to $\yo/\yo_4$.
\end{theorem}

\noindent
In Section \ref{section strutture free}, we prove Theorem \ref{theorem unique kappa} by building $\kappa$-groups as quotients of a free group: we give a sketch of the proof here.
Let $F$ be the free group on $2$ generators and let $(F_i)_{i\geq 1}$ be defined recursively by $F_1=F$ and $F_{i+1}=[F,F_i]F_i^3$.
Then $V=F/F_2$ is a vector space over $\F_3$ of dimension $2$. Let moreover 
$L=F_3F^3$ and set $\overline{F}=F/([F,L]F_2^3)$; we use the bar notation for the subgroups of $\overline{F}$. We will show that the cubing map on $F$ induces a map $V\rightarrow\overline{L}$, which we denote by $c$, and we will construct, in Sections \ref{section strutture} and \ref{section strutture free}, isomorphisms of the following {$\Aut(F)$-sets}, all having cardinality $3$.
\begin{center}
$\cor{I}_V=\graffe{k\subseteq\End(V) \ \text{subfield} : |k|=9}$ \\
\vspace{5pt}
$\downarrow$ \\
\vspace{5pt}
$\cor{K}_V=\{\kappa:V\rightarrow V\otimes\bigwedge^2(V) \ \text{bijective} :  \text{for all} \ x,y\in V, \ \text{one has} $ \\
$\kappa(x+y)=\kappa(x)+\kappa(y)+(x-y)\otimes(x\wedge y)\}$ \\
\vspace{5pt}
$\downarrow$ \\
\vspace{5pt}
$\cor{P}=\graffe{\pi\in\Hom(\overline{L},\overline{F_3}) : \pi\circ c \ \text{is bijective},\, \pi_{|\bar{F_3}}=\id_{\bar{F_3}}}$ \\
\vspace{5pt}
$\downarrow$ \\
\vspace{5pt}
$\cor{N}_3=\graffe{N\subseteq F \ \text{normal subgroup} : F/N \ \text{is a $\kappa$-group of class} \ 3}.$
\vspace{5pt} \\
\end{center}
We will then prove that the natural action of $\Aut(F)$ on $\cor{I}_V$ is transitive and so it will follow that $\Aut(F)$ acts transitively on $\cor{N}_3$, leading to the fact that all $\kappa$-groups of class $3$ are isomorphic to the $\kappa$-group $\yo/\yo_4$. 
To extend the investigation of $\kappa$-groups to class $4$, we consider the ``smallest possible case'' and look at extensions of $\yo/\yo_4$ by a group of order $3$. In Section \ref{section extension}, we prove the following result. 

\begin{theorem}\label{theorem kappa G2}
Let $G$ be a $\kappa$-group such that $G_4$ has order $3$. Then $G_2$ is elementary abelian.
\end{theorem}

\noindent
It would be interesting to explore the world of $\kappa$-groups more extensively, however Theorems \ref{theorem unique kappa} and \ref{theorem kappa G2} provide us with sufficient information on the structure of $\kappa$-groups  to be able to go into the proof of Theorem \ref{theorem 3-groups}.  
Let $G$ be a finite $3$-group of class at least $4$. We have seen that a necessary condition for $\inte(G)$ to be greater than $1$ is that of being a $\kappa$-group, however we can only hope to construct an intense automorphism of $G$ of order $2$ if 
\begin{center}
$(\ast)$ there exists an automorphism of $G$ of order $2$ that inverts all elements of $G$ modulo $G_2$. 
\end{center}
We proved in Section \ref{section construction} that such an automorphism can always be constructed for $G/G_4$, so we want to understand which conditions we need to impose on the structure of $G$ to be able to lift such an automorphism from $G/G_4$ to $G$.
For this purpose, we define
\begin{center}
$\cor{N}_4=\{N\subseteq F \ \text{normal subgroup} : F/N \ \text{is a $\kappa$-group of class $4$ with}$ \\
$\wt_{F/N}(4)=1 \ \text{and satisfying} \ (\ast)\}.$
\end{center}
Via building a bijection $\cor{N}_4\rightarrow\cor{N}_3$, we will be able to prove that the natural action of $\Aut(F)$ on $\cor{N}_4$ is transitive and so that, given $M$ and $N$ in $\cor{N}_4$, the quotients $F/M$ and $F/N$ are isomorphic. The group $\yo$ being a $\kappa$-group of class $4$ with $\wt_{\yo}(4)=1$ and $(\ast)$, it follows that each quotient $F/N$, with $N\in\cor{N}_4$, is isomorphic to $\yo$.
Since $\yo$ has an elementary abelian commutator subgroup, Proposition \ref{class at least 5, G2^p=G4} yields that each finite $3$-group of intensity greater than $1$ has class at most $4$.

\section{The cubing map}\label{section cubing}

In this section we prove some structural properties about $\kappa$-groups of class $4$. 
We remind the reader that, if $G$ is a finite $3$-group and $i$ is a positive integer, then the $i$-th width of $G$ is defined to be $\wt_G(i)=\log_3|G_i:G_{i+1}|$ (see Section \ref{section jumps}). We warn the reader that we will make a set of assumptions, which will hold until the end of Section \ref{section cubing}, right after Lemma \ref{quotient kappa}.

\begin{lemma}\label{3-group exponent bigger than 3}
Let $G$ be a group of order $81$ and class $3$. Then the exponent of $G$ is different from $3$.
\end{lemma}

\begin{proof}
The class of $G$ is $3$ so, thanks to Lemma \ref{index G'}, the quotient $G/G_2$ is non-cyclic. The order of $G$ being $81$, it follows that 
$(\wt_G(1),\wt_G(2),\wt_G(3))=(2,1,1)$.
Let now $C=\Cyc_G(G_2)$. Then, by Lemma \ref{class 3 centralizer G2}($1$),  
the subgroup $C$ contains $G_2$ with index $3$ and, by Lemma \ref{class 3 G3=Z}, the centre of $G$ is equal to $G_3$.
Let $(a,b)\in G\times C$ be such that $\graffe{a,b}$ generates $G$ and define $c=[a,b]$, which is an element of $G_2\setminus G_3$. Let moreover $d$ be a generator for $G_3$. 
Assume by contradiction that the exponent of $G$ is $3$.
As a consequence of Lemma \ref{class 3 cent comm}, the subgroup $C$ is elementary abelian and, in particular, $G_2=\gen{c}\oplus\gen{d}$. Since $G_2$ is central modulo $G_3$ and $d$ generates $G_3$, there exists an integer $k$ such that $aca^{-1}=cd^k$. The element $d^k$ is  not equal to the identity element, because $a$ and $c$ do not centralize each other. Keeping in mind that $C$ is abelian, we compute
$$
1 = (ba)^3 = bababa = b(cba)(cba)a = b^2cacba^2 = $$
$$ b^2c^2d^kaba^2 = b^2c^2d^kcba^3 = b^3c^3d^k =d^k. 
$$
Contradiction.  
\end{proof}

\noindent
We recall here that, if $G$ is a group and $n$ is a positive integer, then $G^n$ is defined to be 
$G^n=\gen{x^n : x\in G}$. 

\begin{lemma}\label{3-group 3-power class 3}
Let $G$ be a finite $3$-group of class $3$ such that $|G:G_2|=9$. 
Then $G_3=G^3$.
\end{lemma}

\begin{proof}
The class of $G$ is equal to $3$ and so $G_3$ is central in $G$. By Lemma \ref{class 3 G3}($1$), the index $|G_2:G_3|$ is equal to $3$ and, by Lemma \ref{class 3 G3}($2$), 
the order of $G_3$ is either $3$ or $9$.
As a consequence of Lemma \ref{class 3 G/G3 extraspecial of exp p}, moreover, the subgroup $G^3$ is contained in $G_3$.
Assume by contradiction that $G_3\neq G^3$. Then, by Lemma \ref{normal index p}, there exists a normal subgroup $M$ of $G$ such that $G^3\subseteq M\subseteq G_3$ and $|G_3:M|=3$. Fix such $M$.
Then the quotient $G/M$ has class $3$ and order $81$. Moreover, the exponent of $G/M$ is equal to $3$. Contradiction to Lemma \ref{3-group exponent bigger than 3}.
\end{proof}

\begin{lemma}\label{formula cubing}
Let $G$ be a group of class at most $3$ and assume that $G_2$ has exponent dividing $3$. Then, for all $x,y\in G$, one has 
$(xy)^3=x^3y^3[xy^{-1},[x,y]]$. 
\end{lemma}

\begin{proof}
The class of $G$ is at most $3$, so the subgroup $G_3$ is central. Fix $x$ and $y$ in $G$. Then we have
\begin{align*}
(xy)^3 & = xyxyxy \\
       & = xyx[y,x]xy^2 \\
       & = xyx[[y,x],x]x[y,x]y^2 \\ 
       & = xyxx[y,x]y^2[[y,x],x] \\
       & = x[y,x]xyx[y,x]y^2[[y,x],x] \\ 
       & = x[[y,x],x]x[y,x]yx[y,x]y^2[[y,x],x] \\
       & = x^2[y,x]yx[y,x]y^2[[y,x],x]^2 \\
       & = x^2[y,x]^2xy[y,x]y^2[[y,x],x]^2 \\
       & = x^2[y,x]^2[xy,[y,x]][y,x]xy^3[[y,x],x]^2 \\
       & = x^2[y,x]^3xy^3[[y,x],x]^2[xy,[y,x]].
\end{align*}
The commutator subgroup of $G$ being annihilated by $3$, the element $[y,x]^3$ is trivial. Moreover, thanks to Lemma \ref{bilinear LCS}, the commutator map induces a bilinear map $G/G_2\times G_2/G_3\rightarrow G_3$. It follows that 
\begin{align*}
(xy)^3 & = x^3y^3[[y,x],x]^2[xy,[y,x]] \\
       & = x^3y^3[[y,x],x^2][[y,x],(xy)^{-1}] \\
       & = x^3y^3[[y,x],x^2y^{-1}x^{-1}] \\
       & = x^3y^3[[y,x],xy^{-1}] \\
       & = x^3y^3[xy^{-1},[x,y]].
\end{align*}
The proof is now complete.
\end{proof}

\begin{lemma}\label{cubing induced}
Let $G$ be a finite $3$-group of class at least $3$ and assume that ${|G:G_2|}=9$.
Then the cubing map induces a map $\kappa: G/G_2\rightarrow G_3/G_4$.
\end{lemma}

\begin{proof}
We assume without loss of generality that $G_4=\graffe{1}$. As a consequence of Lemma \ref{3-group 3-power class 3}, the image of the cubing map is contained in $G_3$ and, by Lemma \ref{class 3 G2 elem abelian}, the commutator subgroup of $G$ has exponent $3$. We now prove that the map $\kappa:G/G_2\rightarrow G_3$, given by $\kappa(xG_2)=x^3$, is well-defined. To this end, let $(x,y)\in G\times G_2$. Then $y^3=1$ and $[y,x]$ belongs to $G_3$, a central subgroup. From Lemma \ref{formula cubing}, we get $$(xy)^3=x^3y^3[[y,x],xy^{-1}]=x^3y^3=x^3$$
so every element of $xG_2$ has the same cube $x^3$ in $G$, as claimed. 
\end{proof}

\noindent
We remark that, in concordance with Definition \ref{definition kappa}, 
the real requirement for a $3$-group $G$ satisfying $|G:G_2|=9$ to be a $\kappa$-group is that the map from Lemma \ref{cubing induced} is a bijection. 
The reason why we are interested in $\kappa$-groups is given by the following lemma.

\begin{lemma}\label{quotient kappa}
Let $G$ be a finite $3$-group of class $4$ and denote by $(G_i)_{i\geq 1}$ the lower central series of $G$. 
Assume that $\inte(G)>1$.
Then $G$ is a $\kappa$-group.
\end{lemma}

\begin{proof}
Take $p=3$ in Lemma \ref{corollary bijection rho}($1$).
\end{proof}

\noindent
In the remaining part of this section, we will prove some structural results about $\kappa$-groups. 
Until the end of Section \ref{section cubing}, let thus $G$ be a finite $3$-group of class $4$.
Let $(G_i)_{i\geq 1}$ denote the lower central series of $G$ and, for each $i\in\Z_{\geq 1}$, denote $w_i=\wt_G(i)$. 
Assume that $(w_1,w_2,w_3,w_4)=(2,1,2,1)$ and, to conclude, let $\kappa:G/G_2\rightarrow G_3/G_4$ be the map from Lemma \ref{cubing induced}.

\begin{lemma}\label{G2 piccoletto}
The group $G_2$ is abelian.
\end{lemma}

\begin{proof}
The quotient $G_2/G_3$ being cyclic, it follows from Lemma \ref{cyclic quotient commutators} that $[G_2,G_2]=[G_2,G_3]$, so, thanks to Lemma \ref{commutator indices}, we get $[G_2,G_3]\subseteq G_5$. The class of $G$ is $4$, so $G_5=\graffe{1}$ and $G_2$ is abelian.
\end{proof}

\begin{lemma}\label{commumino}
The commutator map $G\times G_2\rightarrow G_3$ induces an isomorphism $G/G_2\otimes G_2/G_3\rightarrow G_3/G_4$.
\end{lemma}

\begin{proof}
By Lemma \ref{tensor LCS}, the commutator map induces a surjective homomorphism 
$G/G_2\otimes G_2/G_3\rightarrow G_3/G_4$. The induced map is bijective because 
$|G/G_2\otimes G_2/G_3|=3^{w_1w_2}=9=3^{w_3}=|G_3:G_4|$.
\end{proof}

\noindent
We recall that, if $C$ is a group and $n$ is a positive integer, then $C^n$ and $\mu_n(C)$ are respectively defined as
$C^n=\gen{x^n : x\in C}$ and $\mu_n(C)=\gen{x\in C : x^n=1}$.

\begin{lemma}\label{quasicentri dei massimali}
Let $C$ be a maximal subgroup of $G$. 
Then $G_4C^3\subseteq\ZG(C)$.
\end{lemma}

\begin{proof}
The subgroup $G_4$ is central in $G$, because the class of $G$ is $4$, so $G_4$ is contained in $\ZG(C)$.
By Lemma \ref{tensor LCS}, the commutator map induces a homomorphism $\gamma:G/G_2\otimes G_3/G_4\rightarrow G_4$ and, $C/G_2$ being cyclic, the subgroup $\gamma(C/G_2\otimes\kappa(C/G_2))$ is trivial. 
The quotient $G_4C^3/G_4$ being equal to $\kappa(C/G_2)$, it follows that
$G_4C^3$ is contained in the centre of $C$.
\end{proof}

\begin{lemma}\label{centralizer G3}
There exists at most one maximal subgroup $C$ of $G$ such that 
$G_3\subseteq \ZG(C)$.
\end{lemma}

\begin{proof}
Let $C$ and $D$ be maximal subgroups of $G$ such that $G_3$ is contained in $\ZG(C)\cap\ZG(D)$. Then $CD$ centralizes $G_3$ and, the class of $G$ being equal to $4$, the subgroup $CD$ is different from $G$. It follows that $C=D$.
\end{proof}

\begin{lemma}\label{centro 3}
Assume that $G$ is a $\kappa$-group.
Then $\ZG(G)=G_4$. 
\end{lemma}

\begin{proof}
We first claim that $G_4\subseteq\ZG(G)\subsetneq G_3$. 
The subgroup $G_4$ is contained in $\ZG(G)$ and, as a consequence of Lemma \ref{class 3 G3=Z}, one has $\ZG(G)/G_4\subseteq\ZG(G/G_4)= G_3/G_4$. Since the class of $G$ is $4$, the inclusion $\ZG(G)\subseteq G_3$ is not an equality so the claim is proven. 
Now, by Lemma \ref{frattini comm}, the subgroup $G_2$ is equal to $\Phi(G)$ and so, the dimension $w_1$ being equal to $2$, the group $G$ has precisely $4$ maximal subgroups. Thanks to Lemma \ref{centralizer G3}, there exist two distinct maximal subgroups $C$ and $D$ of $G$ such that both $\ZG(C)$ and $\ZG(D)$ do not contain $G_3$. Fix such $C$ and $D$.
Since $\kappa$ is a bijection and $w_3=2$, Lemma \ref{quasicentri dei massimali} yields $\ZG(C)\cap G_3=C^3G_4$ and $\ZG(D)\cap G_3=D^3G_4$. 
Now, the subgroup $\ZG(G)$ contains $G_4$ and is contained in $\ZG(C)\cap\ZG(D)\cap G_3=C^3G_4\cap D^3G_4$. The map $\kappa$ being a bijection,
the subgroup $C^3G_4\cap D^3G_4$ is equal to $G_4$ and therefore
$\ZG(G)=G_4$.
\end{proof}

\begin{lemma}\label{properties C}
Let $C$ be a maximal subgroup of $G$. 
Assume moreover that $G$ is a $\kappa$-group and that $G_2$ has exponent $3$.
Then $[C,C]\cap\ZG(C)=G_4$.
\end{lemma}

\begin{proof}
The quotient $C/G_2$ is cyclic of order $3$ so, by Lemma \ref{cyclic quotient commutators}, the subgroups $[C,C]$ and $[C,G_2]$ are equal. It follows that $[C,C]$ is contained in $G_3$ and, from Lemma \ref{commumino}, that the index of $([C,C]G_4)/G_4$ in $G_3/G_4$ is equal to $|G:C|=3$. In particular, $[C,C]$ is non-trivial. Now, the subgroup $[C,C]$ is characteristic in the normal subgroup $C$ and therefore it is itself normal in $G$; Lemma \ref{normal intersection centre trivial} yields $[C,C]\cap\ZG(G)\neq 1$.
By Lemma \ref{centro 3}, the centre of $G$ is equal to $G_4$ and thus $[C,C]$ contains $G_4$. As a result, $[C,C]$ is equal to $[C,C]G_4$ and it has thus cardinality $9$. In an analogous way, since $G$ is a $\kappa$-group, the normal subgroup $C^3$ is non-trivial and it contains therefore $G_4$. However, $C^3$ is different from $G_4$ because $G$ is a $\kappa$-group.
We have proven that $G_4$ is contained in $[C,C]\cap\ZG(C)$. 
We assume now by contradiction that $[C,C]\cap\ZG(C)$ is different from $G_4$.
It follows that $[C,C]\cap\ZG(C)$ has cardinality at least $9$, which is the same as the cardinality of $[C,C]$. We get that $[C,C]$ is contained in $\ZG(C)$ and so, as a consequence of Lemmas \ref{class at most p-1 regular} and \ref{exponent G2=p, endomorphism}, the cubing map is an endomorphism of $C$.
By assumption, the exponent of $G_2$ is $3$, and so 
it follows that $$|C^3|=|C:\mu_3(C)|\leq |C:G_2|=3.$$ 
Since $C^3$ contains $G_4$, we get that $C^3=G_4$. Contradiction.
\end{proof}

\begin{lemma}\label{g3 elem abelian}
Assume that $G$ is a $\kappa$-group. Then $G_3$ has exponent $3$.
\end{lemma}

\begin{proof}
The subgroup $G_2$ is abelian, by Lemma \ref{G2 piccoletto}, and $G_2^3$ is contained in $G_4$, as a consequence of Lemma \ref{cubing induced}. It follows that $\mu_3(G_2)$ has cardinality at least $|G_2:G_4|=27$. 
Set $N=\mu_3(G_2)\cap G_3$. We denote $\overline{G}=G/N$ and use the bar notation for the subgroups of $\overline{G}$. 
If $\overline{G_3}=\graffe{1}$, then $G_3$ is contained in $\mu_3(G_2)$ and we are done. Assume by contradiction that $\overline{G_3}$ is non-trivial. Then $\overline{G_3}$ has cardinality at least $3$ so, $\mu_3(G_2)$ consisting of at least $27$ elements, it follows that $\overline{\mu_3(G_2)}$ is non-trivial. However, $\overline{\mu_3(G_2)}$ has trivial intersection with $\overline{G_3}$, which is equal to $\ZG(\overline{G})$, thanks to Lemma \ref{class 3 G3=Z}. Contradiction to Lemma \ref{normal intersection centre trivial}.
\end{proof}

\section{A specific example}\label{section yo}

This section is entirely devoted to understanding the structure of the group $\yo$, which is defined at the beginning of the present chapter. The name $\yo$ refers to the fact that $\yo$ turns out to be an example of maximal class among the finite $3$-groups of intensity greater than $1$. Moreover, as stated in Theorem \ref{theorem 3-groups}, given any finite $3$-group $G$ of class at least $4$, either $\inte(G)=1$ or $G$ is isomorphic to $\yo$.
We recall the definition of $\yo$. 
\vspace{8pt} \\
\noindent
Let $R=\F_3[\epsilon]$ be of cardinality $9$, with $\epsilon^2=0$, and let $\A$ denote the quaternion algebra $\big(\frac{\epsilon, \epsilon}{R}\big)$. In other words, $\A$ is given by
\[\A=R+R\mathrm{i}+R\mathrm{j}+R\mathrm{k}\] with defining relations
$\mathrm{i}^2=\mathrm{j}^2=\epsilon$ and $\mathrm{k}=\mathrm{ji}=-\mathrm{ij}$.
The ring $\A$ has a unique (left/right/$2$-sided) maximal ideal 
$\mathfrak{m}=\A\mathrm{i}+\A\mathrm{j}$ and the residue field $k=\A/\mathfrak{m}$ is equal to $\F_3$.
The algebra $\A$ is also equipped with a natural anti-automorphism of order $2$, which is defined by 
\[x=s+t\mathrm{i}+u\mathrm{j}+v\mathrm{k}\ \mapsto \
\overline{x}=s-t\mathrm{i}-u\mathrm{j}-v\mathrm{k}.\]
We define
$\yo$ to be the subgroup of $1+\mathfrak{m}$ consisting of those elements $x$ satisfying $\overline{x}=x^{-1}$. We denote by $(\yo_i)_{i\geq 1}$ the lower central series of $\yo$ and, for each $i\in\Z_{\geq 1}$, we define  $M_i=(1+\mathfrak{m}^i)\cap G$. One easily shows that 
$(M_i)_{i\geq 1}$ is central and that, for each $i\geq 1$, the commutator map induces a map $M_1/M_2\times M_i/M_{i+1}\rightarrow M_{i+1}/M_{i+2}$ whose image generates $M_{i+1}/M_{i+2}$. For each $i\geq 1$, it follows that $M_{i+1}=[M_1,M_{i}]$ and, since $M_1=G$, we have that
\[\yo_i=\yo\cap (1+\mathfrak{m}^i).\]
The rest of the present section is devoted to the proof of some technical Lemmas that we will use in the proof of Theorem \ref{theorem 3-groups}.

\begin{lemma}\label{yo class}
The group $\yo$ has class $4$ and order $729$.
\end{lemma}

\begin{proof}
We start by proving that $\yo$ has order $729$. The cardinality of $R$ is equal to $9$ and therefore the cardinality of $\A$ is $9^4$. Since $\A/\mathfrak{m}$ is isomorphic to $\F_3$, the cardinality of $\mathfrak{m}$ is equal to $(9^4/3)=3^7$ and therefore also $1+\mathfrak{m}$ has cardinality $3^7$. 
Now, asking for an element $x\in 1+\mathfrak{m}$ to satisfy $x\overline{x}=1$ lowers our freedom in the choice of coordinates of $x$ by $1$ and therefore $G$ has cardinality $3^6=729$.
To conclude the proof, we note that $\yo_5$ is trivial, because $\mathfrak{m}^5=\graffe{0}$, while $1+\epsilon\mathrm{k}$ is a non-trivial element of $\yo_4$. It follows that $\yo$ has class $4$.
\end{proof}

\begin{lemma}\label{yo structure}
Set $G=\yo$ and, for each $i\in\Z_{\geq 1}$, denote $w_i=\wt_G(i)$. Then the following hold.
\begin{itemize}
 \item[$1$.] One has $(w_1,w_2,w_3,w_4)=(2,1,2,1)$.
 \item[$2$.] There exist generators $a$ and $b$ of $G$ such that 
 			 $a^3\equiv [b,[a,b]]^{-1}\bmod G_4$ and 
 			 $b^3\equiv[a,[a,b]]\bmod G_4$.
\end{itemize}
\end{lemma}

\begin{proof}
($1$) 
Let $i\in\graffe{1,2,3,4}$. Then the function $G\rightarrow \mathfrak{m}$ that is defined by $x\mapsto x-1$ induces an injective homomorphism
$d_i:G_i/G_{i+1}\rightarrow \mathfrak{m}^i/\mathfrak{m}^{i+1}$, which commutes with the bar map of $\A$. Now, for each element $x\in G_i$, one has that $\overline{x-1}+(x-1)$ belongs to $\mathfrak{m}^{i+1}$ and therefore the image of $d_i$ is contained in 
$D_i=\graffe{y + \mathfrak{m}^{i+1} : y\in\mathfrak{m}^i, \overline{y}+y \in\mathfrak{m}^{i+1}}$.
With an easy computation, one shows that $D_i$ coincides with the image of $d_i$ and, consequently, that $(w_1,w_2,w_3,w_4)=(2,1,2,1)$.
To prove ($2$), define 
$$a=1-\epsilon+\mathrm{i}\ \ \text{and} \ \ b=1-\epsilon+\mathrm{j}$$ 
and note that $a$ and $b$ belong to $G$.
Since $w_1=2$ and $a$ and $b$ are linearly independent modulo $G_2=G\cap(1+\mathfrak{m}^2)$, the group $G$ is generated by $a$ and $b$.
Using the defining properties of $\A$, we compute $a^3=1+\epsilon\mathrm{i}$ 
and $b^3=1+\epsilon\mathrm{j}$. Define
$c=[a,b]$, $d=[a,c]$, and $e=[b,c]$. 
Then, working modulo $G_3$, we get 
\begin{align*}
c = ab\overline{a}\overline{b} & \equiv (1-\epsilon +\mathrm{i})(1-\epsilon +\mathrm{j})(1-\epsilon -\mathrm{i})(1-\epsilon -\mathrm{j}) \\
  & \equiv (1+\epsilon +\mathrm{i}+\mathrm{j}+\mathrm{k})(1+\epsilon -\mathrm{i}-\mathrm{j}+\mathrm{k}) \\
  & \equiv 1-\mathrm{k} \bmod G_3.  
\end{align*}
Thanks to Lemma \ref{bilinear LCS}, one has $d\equiv [a,1+\mathrm{k}]\bmod G_4$ and $e\equiv [b,1+\mathrm{k}]\bmod G_4$ and it is now easy to compute
$d\equiv 1+\epsilon\mathrm{j}\bmod G_4$ and $e\equiv 1-\epsilon\mathrm{i}\bmod G_4$.
It follows that both $ea^3$ and $d^{-1}b^3$ belong to $G_4$ and so the proof is complete. 
\end{proof}

\noindent
We remind the reader that, in concordance with Definition \ref{definition kappa}, a $\kappa$-group is a finite $3$-group $G$ such that $|G:G_2|=9$ and such that the cubing map on $G$ induces a bijection $G/G_2\rightarrow G_3/G_4$.

\begin{lemma}\label{yo kappa}
The group $\yo$ is a $\kappa$-group.
\end{lemma}

\begin{proof}
Write $G=\yo$ and, for each $i\in\Z_{\geq 1}$, denote $w_i=\wt_G(i)$.
By Lemma \ref{yo class}, the group $G$ has class $4$ and, by Lemma \ref{yo structure}($1$), one has $(w_1,w_2,w_3,w_4)=(2,1,2,1)$. Let 
$\kappa:G/G_2\rightarrow G_3/G_4$ be as in Lemma \ref{cubing induced}; we want to show that $\kappa$ is a bijection. Let $a$ and $b$ be as in Lemma \ref{yo structure}($2$) and define $d=[a,[a,b]]$ and $e=[b,[a,b]]$. 
Then $\kappa(a)\equiv e^{-1}\bmod G_4$ and $\kappa(b)\equiv d\bmod G_4$. Moreover, since $w_2=1$, it follows from Lemma \ref{commumino} that $d$ and $e$ generate $G_3$ modulo $G_4$.
We claim that $\kappa$ is surjective. Let $r,s$ be integers and let $y=d^se^r$. 
If $r=0$ or $s=0$, then $\kappa(b^s)\equiv y\bmod G_4$ or 
$\kappa(a^{-r})\equiv e^r\bmod G_4$. The quotient $G_3/G_4$ being elementary abelian, we may now assume that $r$ and $s$ are both non-zero modulo $3$ and they satisfy therefore $r^2\equiv s^2\equiv 1 \bmod 3$. Define 
$x=a^rb^{-s}$. 
Working modulo $G_4$, we get from Lemma \ref{formula cubing} that
\begin{align*}
\kappa(x) & \equiv a^{3r}b^{-3s}[a^rb^s,[a^r,b^{-s}]] \\
					 & \equiv e^{-r}d^{-s}[a,[a,b]]^{-r^2s}[b,[a,b]]^{-rs^2} \\
					 & \equiv e^{-r-rs^2}d^{-s-r^2s} \\
					 & \equiv e^{-2r}d^{-2s} \\
					 & \equiv y \bmod G_4.
\end{align*}
We have proven that $\kappa$ is surjective so, the widths $w_1$ and $w_3$ being the same, it follows that $\kappa$ is a bijection.
\end{proof}

\begin{lemma}\label{yo inversion}
Define $\alpha:\yo\rightarrow \yo$ by
$$x=s+t\mathrm{i}+u\mathrm{j}+v\mathrm{k}\ \mapsto \
\alpha(x)=s-t\mathrm{i}-u\mathrm{j}+v\mathrm{k}.$$
Then $\alpha$ is an automorphism of order $2$ of $\yo$. Moreover, $\alpha$ induces the inversion map on $\yo/\yo_2$.
\end{lemma}

\begin{proof}
Set $G=\yo$.
It is easy to check that $\alpha$ is an automorphism of order $2$ of $G$, so we prove that $\alpha$ induces the inversion map on $G/G_2$.
The subgroup $G_2$ is equal to $G\cap(1+\mathfrak{m}^2)$ and, thanks to Lemma
\ref{yo structure}($1$), the order of $G/G_2$ is $9$.
It follows from Lemma \ref{frattini comm}($2$) that $G/G_2$
is elementary abelian. 
We define $a=1-\epsilon+\mathrm{i}$ and $b=1-\epsilon+\mathrm{j}$. Then $a$ and $b$ span $G$ modulo $G_2$ and 
$$\alpha(a)=\bar{a}=a^{-1}\ \ \text{and} \ \ \alpha(b)=\bar{b}=b^{-1}.$$
The quotient $G/G_2$ being commutative, the map $G/G_2\rightarrow G/G_2$ that is induced by $\alpha$ is equal to the inversion map $x\mapsto x^{-1}$.
\end{proof}

\noindent
We conclude Section \ref{section yo} by remarking that another characterization of $\yo$ has been provided by Derek Holt and Frieder Ladisch; this characterization was found using computer algebra systems.
The group $\yo$ turns out to be isomorphic to a Sylow $3$-subgroup of the Schur cover $3.\mathrm{J}_3$ of the simple Janko-$3$ group $\mathrm{J}_3$.
If $S$ is a Sylow $3$-subgroup of $3.\mathrm{J}_3$ and $N$ denotes the normalizer of $S$ in $3.\mathrm{J}_3$, then conjugation under any element of order $2$ of $N$ restricts to an automorphism of order $2$ of $S$ that induces the inversion map on the abelianization.
The isomorphism class of $\yo$ is denoted by $[729,57]$ in the GAP system.

\section{Structures on vector spaces}\label{section strutture}

Until the end of Section \ref{section strutture}, the following notation will be adopted. Let $V$ be a $2$-dimensional vector space over $\F_3$. 
A \indexx{$\kappa$-structure} on $V$ is a bijective map 
$\kappa:V\rightarrow V\otimes\bigwedge^2V$ such that, for each $x,y\in V$, one has
\begin{equation}\tag{$\mathrm{A}1$}
\kappa(x+y)=\kappa(x)+\kappa(y)+(x-y)\otimes(x\wedge y).
\end{equation}
We denote by $\cor{K}_V$ the collection of $\kappa$-structures of $V$ and by $\cor{I}_V$ the collection of subfields of $\End(V)$ of cardinality $9$. 
We remark that, for each element $k$ of $\cor{I}_V$, there exists $i\in\End(V)$ such that $i^2=-1$ and $k=\F_3[i]$. Moreover, $V$ is naturally a vector space of dimension $1$ over each of the elements of $\cor{I}_V$.
The rest of Section \ref{section strutture} will be devoted to the proof of the following result. Until the end of Section \ref{section strutture}, all tensor and wedge products will be defined over $\F_3$.

\begin{proposition}\label{proposition sV}
Let $V$ be a $2$-dimensional vector space over $\F_3$ and let the map $s_V:\cor{I}_V\longrightarrow\cor{K}_V$ be defined by 
\[k=\F_3[i]\ \mapsto\ (x\mapsto ix\otimes(ix\wedge x)).\]
Then $s_V$ is a bijection. Moreover, the cardinality of $\cor{K}_V$ is equal to $3$.
\end{proposition}

\noindent
As the goal of this section is to prove Proposition \ref{proposition sV}, we will respect the notation of the very same proposition until the end of Section \ref{section strutture}.
\vspace{8pt} \\ 
\noindent
We put a field structure on $V$, via an $\F_3$-linear isomorphism with $\F_9$. We define then $\Lambda$ to be the collection of bijective maps 
$\lambda:V\rightarrow V$ such that, for all $x,y\in V$, one has
\begin{equation}\tag{$\mathrm{A}2$}
\lambda(x+y)=\lambda(x)+\lambda(y)+(x-y)(xy^3-x^3y). 
\end{equation}
We let moreover $\sigma_V:\cor{I}_V\rightarrow\Lambda_V$ be defined by 
\[
k=\F_3[i]\ \mapsto\ (x\mapsto ix((ix)x^3 - (ix)^3x)).
\]

\begin{lemma}\label{x5}
The map $V\rightarrow V$, defined by $x\mapsto x^5$, is an element of $\Lambda$.
\end{lemma}

\begin{proof}
The group of units of $V$ has order $8$ and, since $8$ and $5$ are coprime, the map $x\mapsto x^5$ is a bijection $V^*\rightarrow V^*$ which extends to a bijection $V\rightarrow V$. Let now $x$ and $y$ be elements of $V$. Keeping in mind that $V$ has characteristic $3$, one computes 
\begin{align*}
(x+y)^5 & = \sum_{k=0}^5\binom{5}{k}x^ky^{5-k} 
		 = x^5 - xy^4 + x^2y^3 + x^3y^2 -x^4y +y^5 \\
		& = x^5 + y^5 + (x-y)(xy^3-x^3y)
\end{align*} 
and therefore $x\mapsto x^5$ satisfies $(\mathrm{A}2)$.
\end{proof}

\begin{lemma}
The map $\sigma_V$ is well-defined.
\end{lemma}

\begin{proof}
Let $k=\F_3[i]$ be an element of $\cor{I}_V$. The group $k^*$ is cyclic of order $8$ and there are therefore exactly two square roots of $-1$ in $k$, namely $i$ and $-i$. 
Now, for each element $x$ of $V$, we have 
$$ix((ix)x^3 - (ix)^3x))=-ix((-ix)x^3 - (-ix)^3x))$$ and thus $k$ gives a map 
$V\rightarrow V$.
Let now $k\rightarrow V$ denote an isomorphism of fields and identify $i$ with its image in $V$. Then, for each $x\in V$, we have 
\[
ix((ix)x^3 - (ix)^3x) = x(ix)^2(x^2-(ix)^2) = -x^3(x^2+x^2)=x^5
\]
and so, as a consequence of Lemma \ref{x5}, the map $\sigma_V$ is well-defined.
\end{proof}

\begin{lemma}\label{V4}
Let $\mathbb{P}V$ denote the collection of $1$-dimensional subspaces of $V$. Then the natural homomorphism $\Aut(V)\rightarrow\Sym(\mathbb{P}V)$ induces an isomorphism $\Aut(V)/\F_3^*\rightarrow\Sym(\mathbb{P}V)$. 
\end{lemma}

\begin{proof}
The natural homomorphism $\Aut(V)\rightarrow\Sym(\mathbb{P}V)$ factors as an injective homomorphism  $\Aut(V)/\F_3^*\rightarrow\Sym(\mathbb{P}V)$, 
which is in fact also surjective, because 
$|\Aut(V):\F_3^*|=48/2=24=|\Sn_4|=|\Sym(\mathbb{P}V)|$.
\end{proof}

\begin{lemma}\label{cardinality IV}
The set $\cor{I}_V$ has cardinality $3$. Moreover, the action by conjugation of $\Aut(V)$ on $\cor{I}_V$ is transitive.
\end{lemma}

\begin{proof}
Let $f:\cor{I}_V\rightarrow\Aut(V)/\F_3^*$ be defined by 
$k=\F_3[i]\mapsto i\F_3^*$ and observe that, since $\F_3[i]=\F_3[-i]$, the map $f$ is well-defined. Moreover, since each element of $\cor{I}_V$ is uniquely determined, modulo $\F_3^{*}$, by a square root of $-1$, the map $f$ is injective. 
Let 
$\mathbb{P}V$ denote the collection of $1$-dimensional subspaces of $V$ and let $\epsilon:\Aut(V)/\F_3*\rightarrow\Sn_4$ be the composition of the isomorphism $\Aut(V)/\F_3^*\rightarrow\Sym(\mathbb{P}V)$ from Lemma \ref{V4} with a given isomorphism $\Sym(\mathbb{P}V)\rightarrow\Sn_4$. Then $(\epsilon\circ f)(\cor{I}_V)$ consists of elements of order $2$. Now, each element $k$ of $\cor{I}_V$ can be written as $k=\F_3[i]$, with $i^2=-1$, and this suffices to show that $(\epsilon\circ f)(\cor{I}_V)$ is in fact contained in the Klein subgroup $\mathrm{V}_4$ of $\Sn_4$. The set $V_4\setminus\graffe{1}$ forms a unique conjugacy class in $\Sn_4$ and thus the elements of $\cor{I}_V$ form a unique orbit under the action by conjugation of $\Aut(V)$. Since the set $V_4\setminus\graffe{1}$ has cardinality $3$, the cardinality of $\cor{I}_V$
is also equal to $3$.
\end{proof}

\begin{lemma}\label{mezza theta}
Write $V=\F_3[i]$, with $i^2=-1$. 
Then the map $\bigwedge^2V\rightarrow\F_3i$ that is defined by 
$x\wedge y\mapsto xy^3-x^3y$ is an isomorphism of vector spaces.
\end{lemma}

\begin{proof}
Let $\phi:V\times V\rightarrow V$ be defined by $(x,y)\mapsto xy^3-x^3y$. It is easy to show that $\phi$ is alternating and that $\phi(V\times V)$ is contained in $\F_3i$, the eigenspace of the Frobenius homomorphism that is associated to $-1$. Moreover, the map $\phi$ is non-zero. It follows that $\phi$ induces a linear homomorphism $\phi':\bigwedge^2 V\rightarrow \F_3i$ that is non-trivial. Since both $\bigwedge^2V$ and $\F_3i$ have dimension $1$ over $\F_3$, the map $\phi'$ is an isomorphism.
\end{proof}

\begin{lemma}\label{multiplication F9}
Write $V=\F_3[i]$, with $i^2=-1$. 
Then the map $\mu:V\otimes\F_3i\rightarrow V$ that is defined by 
$x\otimes y\mapsto xy$ is an isomorphism of vector spaces.
\end{lemma}

\begin{proof}
The map $V\times\F_3i\rightarrow V$ that is defined by $(x,y)\mapsto xy$ is a bilinear surjective map. By the universal property of tensor products, it factors as the surjective homomorphism $\mu:V\otimes\F_3i\rightarrow V$. The dimensions of $V\otimes\F_3i$ and $V$ being the same, $\mu$ is an isomorphism. 
\end{proof}

\noindent
Write $V=\F_3[i]$ with $i^2=-1$.
Define
$\theta:V\otimes\bigwedge^2V\rightarrow V\otimes\F_3i$ by 
$$\theta(a\otimes(x\wedge y))=a\otimes(xy^3-x^3y)$$ and 
note that $\theta$ is an isomorphism of vector spaces, as a consequence of Lemma \ref{mezza theta}.
Let $\mu$ be as in Lemma \ref{multiplication F9}. We keep this notation until the end of Section \ref{section strutture}.

\begin{lemma}\label{gamma27}
The map $l_V: \cor{K}_V\rightarrow \Lambda$ that is defined by $\kappa\mapsto \mu\circ\theta\circ\kappa$ is bijective.
\end{lemma} 

\begin{proof}
Let $\kappa$ be an element of $\cor{K}_V$. Then $l_V(\kappa)$ is bijective, because it is the composition of bijective maps, and, for each $x,y\in V$, one has
\begin{align*}
l_V(\kappa)(x+y) & = \mu\circ\theta\circ\kappa(x+y)\\
             & = \mu\circ\theta(\kappa(x)+\kappa(y)+(x-y)\otimes(x\wedge y))\\
             & = l_V(\kappa)(x)+l_V(\kappa)(y)+\mu\circ\theta((x-y)\otimes(x\wedge y))\\
             & = l_V(\kappa)(x)+l_V(\kappa)(y)+(x-y)(xy^3-x^3y).
\end{align*}
We have proven that $l_V(\kappa)$ belongs to $\Lambda$ and so $l_V$ is well-defined. Moreover, $l_V$ is bijective, because $\mu$ and $\theta$ are bijective.
\end{proof}

\begin{lemma}\label{sigma27}
Let $l_V$ be as in Lemma \ref{gamma27}.
Then $\sigma_V=l_V\circ s_V$ and $s_V$ is well-defined.
\end{lemma}

\begin{proof}
Let $k=\F_3[i]$ be an element of $\cor{I}_V$. Let moreover $\kappa$ and $\lambda$ respectively denote $s_V(k)$ and $\sigma_V(k)$. 
Then one has 
\begin{align*}
l_V(\kappa)(x) & = \mu\circ\theta(ix\otimes ix\wedge x) \\
				  & = \mu(ix\otimes ((ix)x^3-(ix)^3x)) \\
				  & = ix((ix)x^3-(ix)^3x) \\
				  & = \lambda(x)
\end{align*}
and so, the choices of $k$ and $x$ being arbitrary, $\sigma_V=l_V\circ s_V$.
As a consequence, the map $s_V$ is well-defined.
\end{proof}

\begin{lemma}\label{injective sV}
The map $s_V$ is injective. 
\end{lemma}

\begin{proof}
Let $k$ and $k'$ be elements of $\cor{I}_V$ and let $i,j\in\End(V)$ be such that $k=\F_3[i]$, $k'=\F_3[j]$, and $i^2=j^2=-1$. Assume moreover that 
$s_V(k)=s_V(k')$. For each $x\in V$, we have
$\F_3x+\F_3ix=V=\F_3x+\F_3jx$ and therefore there exists $\omega_x\in\graffe{\pm 1}$ such that $ix\equiv \omega_x jx\bmod\F_3x$. 
For each $x\in V$, it then follows that
\[
jx\otimes(jx\wedge x)=ix\otimes(ix\wedge x)=ix\otimes((\omega_xjx)\wedge x)=
\omega_x ix\otimes(jx\wedge x)
\]
and, $\mu\circ\theta$ being bijective, the elements $jx$ and $\omega_x ix$ are the same.
The choice of $x$ being arbitrary, we get
$$V=\graffe{x\in V : ix=jx}\cup\graffe{x\in V : ix=-jx}$$
and so, $V$ being equal to the union of two subgroups, either $i=j$ or $i=-j$.
In either case, $i$ and $j$ are linearly dependent over $\F_3$ and so $k=k'$. 
\end{proof}

\begin{lemma}\label{endomorph F9}
Let $\phi$ be an $\F_3$-linear endomorphism of $V$. Then there exist unique $a,b\in V$ such that, for each $x\in V$, one has $\phi(x)=ax^3+bx$. 
\end{lemma}

\begin{proof}
The characteristic of $V$ being $3$, for each pair $(a,b)$ in $V^2$, the map $x\mapsto ax^3+bx$ is an $\F_3$-linear endomorphism of $V$. 
The order of $\End(V)$ being equal to the order of $V^2$, it follows that each element $\psi$ of $\End(V)$ is of the form $x\mapsto ax^3+bx$, where $a,b\in V$ are uniquely determined by $\psi$. In particular, this holds for $\phi$. 
\end{proof}

\begin{lemma}\label{shape of lambda}
Let $\lambda\in\Lambda$. Then there exist $a,b\in V$ such that, for each $x\in V$, one has $\lambda(x)=x^5+ax^3+bx$.
\end{lemma}

\begin{proof}
Because of $(\mathrm{A}2)$, the difference of any two elements of $\Lambda$ belongs to $\End(V)$, so, thanks to Lemma \ref{x5}, we have 
$\lambda\in (x\mapsto x^5)+\End(V)$. It now follows from Lemma \ref{endomorph F9} that there exist $a,b\in V$ such that, for each $x\in V$, we have 
$\lambda(x)=x^5+ax^3+bx$.
\end{proof}

\begin{lemma}\label{maldischiena}
Let $m$ be a positive integer and let $q$ be a prime power. Then 
\[\sum_{x\in\F_q}x^m=
\left\{
\begin{array}{ll}
      -1 & \text{when} \ \ (q-1)|m \\    
      0  & \text{otherwise} \\
\end{array} 
\right
.\]
\end{lemma}

\begin{proof}
This is Lemma $2.5.1$ from \cite{cohen}.
\end{proof}

\begin{lemma}\label{shapino lambda}
Let $\lambda\in\Lambda$. Then there exists $b\in V$ such that, for each $x\in V$, one has $\lambda(x)=x^5+bx$.
\end{lemma}

\begin{proof}
Let $a,b\in V$ be as in Lemma \ref{shape of lambda}. By definition of $\Lambda$, the map $\lambda$ is bijective so each element of $V$ belongs to the image of $\lambda$. With $x$ replaced by $\lambda(x)$, Lemma \ref{maldischiena} yields
\[
0=\sum_{x\in V}\lambda(x)^2 = \sum_{x\in V}(x^5+ax^3+bx)^2 
                              = \sum_{x\in V}2ax^8 = -2a.
\]
It follows that $a=0$ and therefore, for each $x\in V$, one has 
$\lambda(x)=x^5+bx$.
\end{proof}

\begin{lemma}\label{cardinality lambda}
The cardinality of $\Lambda$ is at most $3$. 
\end{lemma}

\begin{proof}
Let $\lambda\in\Lambda$ and let $b\in V$ be as in Lemma \ref{shapino lambda}. 
The map $\lambda$ is bijective and so, with $x$ replaced by $\lambda(x)$, Lemma \ref{maldischiena} gives
\begin{align*}
0=\sum_{x\in V}\lambda(x)^4 & = \sum_{x\in V}(x^5+bx)^4 \\
                              & = \sum_{x\in V}(x^{10}-bx^6+b^2x^2)^2 \\
                              & = \sum_{x\in V}(bx^{16}+b^3x^8) \\
                              & = -b(1+b^2).
\end{align*}
It follows that there are at most $3$ choices for $b$ in $V$ and thus $\Lambda$ has cardinality at most $3$.
\end{proof}

\noindent
We conclude Section \ref{section strutture} by giving the proof of Proposition \ref{proposition sV}. The function $s_V:\cor{I}_V\rightarrow\cor{K}_V$ is injective by Lemma \ref{injective sV} and, by Lemma \ref{cardinality IV}, the cardinality of $\cor{I}_V$ is equal to $3$. It follows that $\cor{K}_{V}$ has at least $3$ elements. 
Now, as a consequence of Lemma \ref{gamma27}, the set $\Lambda$ has the same cardinality as $\cor{K}_V$ and thus, as a consequence of Lemma \ref{cardinality lambda}, the cardinality of $\cor{K}_V$ is equal to $3$. From its injectivity, it now follows that $s_V$ is bijective. The proof of Proposition \ref{proposition sV} is complete.

\section{Structures and free groups}\label{section strutture free}

We recall that a $\kappa$-group is a finite $3$-group $G$ such that $|G:G_2|=9$ and such that the cubing map on $G$ induces a bijection $G/G_2\rightarrow G_3/G_4$. In the present section, we consider $\kappa$-groups of class $3$ and we prove the following main result.

\begin{proposition}\label{proposition yo kappa}
Let $G$ be a $\kappa$-group of class $3$. Then $G$ is isomorphic to $\yo/\yo_4$.
\end{proposition}

\noindent
As a consequence of Lemma \ref{frattini comm}($2$), each $\kappa$-group is $2$-generated. Our strategy, for proving Proposition \ref{proposition yo kappa}, will be that of constructing all $\kappa$-groups of class $3$ as quotients of a free group.
To this end, the following assumptions will be valid until the end of Section \ref{section strutture free}.
Let $F$ be the free group on two generators and let $(F_i)_{i\geq 1}$ denote the lower $3$-series of $F$, which we recall from Section \ref{section construction} to be defined by
\[
F_1=F \ \ \text{and} \ \ F_{i+1}=[F,F_i]F_i^3.
\] 
We remark that the notation we use for the lower $3$-series is not concordant with our usual notation (see Exceptions in List of Symbols). 
We denote 
$$V=F/F_2,\ L=F_3F^3, \ \ \text{and} \ \ E=[F,L]F_2^3.$$
The group $V$ is a vector space of dimension $2$ over $\F_3$, by construction, so we let $\cor{K}_V$ be defined as in Section \ref{section strutture}.
We write moreover $\overline{F}=F/E$ and we use the bar notation for the subsets of $\overline{F}$. 
We define additionally $\cor{N}_3$ to be the collection of normal subgroups $N$ of $F$ with the property that $F/N$ is a $\kappa$-group of class $3$.

\begin{lemma}\label{LsuF3dim2}
The map $c_3:F\rightarrow L/F_3$, defined by $x\mapsto x^3F_3$, is surjective.
Moreover, $c_3$ induces an isomorphism $V\rightarrow L/F_3$ and $|L:F_3|=9$.
\end{lemma}

\begin{proof}
The map $c_3$ is well-defined, by definition of $L$, and $L/F_3=(F^3F_3)/F_3$.
As a consequence of Lemma \ref{p map petrescu}, the map $c_3$ is a surjective homomorphism, which, $F_2^3$ being contained in $F_3$, factors as a surjective homomorphism $c_2:F/F_2\rightarrow L/F_3$. 
Since $V=F/F_2$ has order $9$, the order of $L/F_3$ is at most $9$.
Let now $A=\Z/9\Z\times\Z/9\Z$ and let $\psi:F\rightarrow A$ be a surjective homomorphism. 
The group $A$ being abelian, we have that $F_3$ is contained in $\ker\psi$. 
Moreover, since $L=F^3F_3$, the group $\psi(L)$ is equal to $3\Z/9\Z\times 3\Z/9\Z$, which has order $9$. 
As a consequence, $L/F_3$ has cardinality at least $9$ and so $|L:F_3|=9$. In addition, the map $c_2$ is an isomorphism of groups. 
\end{proof}

\begin{lemma}\label{F3bar of dim at least 2}
One has $|F_3:E|\geq 9$.
\end{lemma}

\begin{proof}
Thanks to Lemma \ref{yo structure}, we have $|\yo:\yo_2|=9$ and therefore, as a consequence of Lemma \ref{frattini comm}($2$), the group $\yo$ is $2$-generated. We fix a surjective homomorphism $\phi:F\rightarrow \yo$ and we denote by $\pi$ the canonical projection $\pi:\yo\rightarrow \yo/\yo_4$. The homomorphism $\pi\circ\phi:F\rightarrow\yo/\yo_4$ is surjective and so, from Lemma 
 \ref{description L}, we get $L=(\pi\circ\phi)^{-1}(\yo_3/\yo_4)$. As a consequence, $L$ is equal to $\phi^{-1}(\yo_3)$ and thus $\phi(L)=\yo_3$. Moreover, thanks to Lemma \ref{class 3 G2 elem abelian}, we know that $\yo_2^3$ is contained in $\yo_4$
and therefore $\phi(F_2^3)\subseteq \yo_4$.
It follows that 
\[\phi(F_3)=\phi([F,F_2])\phi(F_2^3)=[\yo,\yo_2]\yo_2^3=\yo_3\]
and also that
\[\phi(E)=\phi([F,L])\phi(F_2^3)=[\yo,\yo_3]\yo_2^3=\yo_4.\]
As a result, the index $|F_3:E|$ is at least $|\yo_3:\yo_4|$, which is, by Lemma \ref{yo structure}($1$), equal to $9$.
\end{proof}

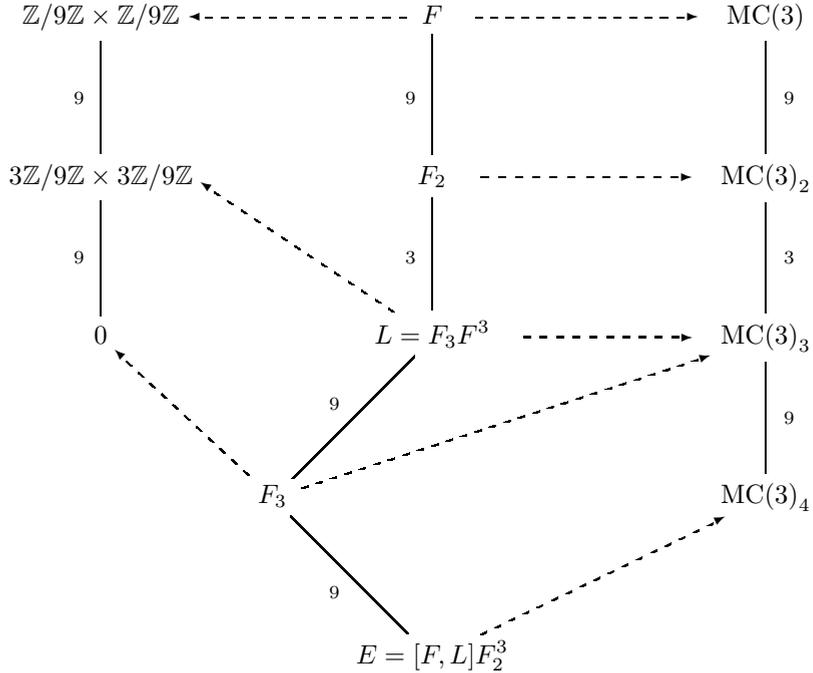
\[
\begin{diagram}[nohug]
\Z/9\Z\times\Z/9\Z &  & \lDashto  &   &  F  &  &  & \ \ \rDashto \ \    &  \yo  \\
\dLine^{9 \ } &   &    &   & \dLine^{9 \ }    &  & &  &    \dLine_{\ 9}  \\
3\Z/9\Z\times 3\Z/9\Z & &  &    &  F_2    &   &   & \ \ \rDashto \ \   & \yo_2    \\
\dLine^{9\ } &   & \luDashto(3,2)   &    &  \dLine^{3 \ }   &  & &     &  \dLine_{\ 3} \\
0  &       &     &    &  L=F_3F^3  &      &   & \ \ \rDashto \ \ & \yo_3    \\
 & \luDashto  &  &   \ldLine^{9 \ }  &   &  & & \ruDashto(6,2)  &    \dLine_{\ 9}   \\
&  & F_3    &     &    &     &      &       & \yo_4    \\
 &  &    &   \rdLine_{9 \ } & & & &     \ruDashto(4,2)   \\
&   &    &         &  E=[F,L]F_2^3          &   &   &     &
\end{diagram}
\]

\vspace{12pt}

\begin{lemma}\label{gammino35}\label{F3bar of dim 2}
The commutator map $F\times F_2\rightarrow F_3$ induces an isomorphism of groups
$\delta:F/F_2\otimes F_2/L\rightarrow\overline{F_3}$. Moreover, $|F_3:E|=9$.
\end{lemma}

\begin{proof}
The subgroup $\overline{F_3}$ is central in $\overline{F}$, by definition of $E$, and so, by Lemma \ref{tgt}, the commutator map $F\times F_2\rightarrow\overline{F_3}$ is bilinear. Moreover, the quotient $F_2/L$ is cyclic of order $3$, thanks to Lemma \ref{description L}, and so, by Lemma \ref{cyclic quotient commutators}, the subgroup $[F_2,F_2]$ is equal to $[F_2,L]$. The commutator map factors thus as a surjective homomorphism 
$\delta: F/F_2\otimes F_2/L\rightarrow\overline{F_3}$ and therefore 
$|F_3:E|\leq |F/F_2\otimes F_2/L|=9$.
Now, the group $\overline{F_3}$ has order at least $9$, by Lemma \ref{F3bar of dim at least 2}, and therefore 
$|F_3:E|=9$ and $\delta$ is an isomorphism. 
\end{proof}

\begin{lemma}\label{Lbar of dim 4}
The group $\overline{L}$ is an $\F_3$-vector space of dimension $4$.
\end{lemma}

\begin{proof}
By the definition of $E$, the group $\overline{L}$ is central in $\overline{F}$ and so it is abelian. Moreover $L^3$ is contained in $F_2^3$, which is itself contained in $E$. It follows that $\overline{L}$ is naturally a vector space over $\F_3$. The dimension of $\overline{L}$ is equal to $4$, thanks to the combination of Lemmas \ref{LsuF3dim2} and \ref{F3bar of dim 2}.
\end{proof}

\begin{lemma}\label{gammino3}
The commutator map $F\times F_2\rightarrow F_3$ induces an isomorphism of groups
$\gamma:V\otimes\bigwedge^2V\rightarrow\overline{F_3}$.
\end{lemma}

\begin{proof}
The subgroup $F_2$ is central modulo $L$ so, thanks to Lemma \ref{tgt}, the commutator map $F\times F\rightarrow F_2/L$ 
is bilinear. Since $[F,F_2]$ is contained in $L$, we get a bilinear map $V\times V\rightarrow F_2/L$, which is also alternating. By the universal property of wedge products, the last map factors as a homomorphism 
${\theta:\bigwedge^2 V\rightarrow F_2/L}$ mapping $x\wedge y$ to $[x,y]$. By Lemma \ref{description L}, the cardinality of $F_2/L$ is equal to $3$, which is the same as the cardinality of $\bigwedge^2V$ and so, $\theta$ being non-trivial, it is an isomorphism of groups. 
We conclude by defining $\gamma=\delta\circ(1\otimes\theta)$, where $\delta$ is as in Lemma \ref{gammino35}.
\end{proof}

\begin{lemma}\label{cappino1}
Let $\gamma$ be as in Lemma \ref{gammino3} and use the additive notation for the vector spaces $V$ and $\overline{L}$.
Then the cubing map on $\overline{F}$ induces a map 
$c:V\rightarrow\overline{L}$ such that, for each $x,y\in V$, one has 
$$c(x+y)=c(x)+c(y)+\gamma((x-y)\otimes(x\wedge y)).$$
\end{lemma}

\begin{proof}
The subgroup $[F,[F,F]]$ is contained in $L$, which is central modulo $E$, so $\overline{F}$ has class at most $3$. Moreover, $[F,F]^3$ is contained in $F_2^3$ and so $[\overline{F},\overline{F}]$ has exponent at most $3$. By Lemma \ref{formula cubing}, given any two elements $x,y$ of $\overline{F}$, one has $(xy)^3=x^3y^3[xy^{-1},[x,y]]$. Since both $F_2^3$ and $[F,[F,F_2]]$ are contained in $E$, cubing on $\overline{F}$ induces a map $c:V\rightarrow\overline{L}$.
Using the additive notation for the vector spaces $V$ and $\overline{L}$, it follows from the definition of $\gamma$ that, for each $x,y\in V$, one has 
$c(x+y)=c(x)+c(y)+\gamma((x-y)\otimes(x\wedge y))$.
\end{proof}

\begin{lemma}\label{youusedto}
Let $0\rightarrow A\overset{\iota}{\rightarrow} B\overset{\sigma}{\rightarrow} C\rightarrow 0$ be a short exact sequence of abelian groups. 
Let moreover $s:C\rightarrow B$ be a function such that $\sigma\circ s=\id_C$. Write 
${\cor{R}=\graffe{f\in\Hom(B,A) : f\circ\iota=\id_A}}$ and let $\cor{H}$ be the collection of maps $g:C\rightarrow A$ such that, for all $u,v\in C$, one has $$\iota(g(u+v)-g(u)-g(v))=s(u+v)-s(u)-s(v).$$ 
Then the function 
$\cor{R}\rightarrow\cor{H}$ that is defined by $f\mapsto f\circ s$ is bijective.
\end{lemma}

\begin{proof}
Let $\nu:\cor{R}\rightarrow\cor{H}$ be defined by $f\mapsto f\circ s$. We first prove that $\nu$ is well-defined. To this end, let $f\in\cor{R}$ and let 
$u,v\in C$. Since $\sigma\circ s=\id_C$, the element $s(u+v)-s(u)-s(v)$ belongs to $\ker\sigma=\iota(A)$. Since $f\circ\iota=\id_A$, we get that $\iota\circ f_{|\iota(A)}=\id_{|\iota(A)}$ and therefore
$$\iota((f\circ s)(u+v)-(f\circ s)(u)-(f\circ s)(v))=
\iota(f(s(u+v)-s(u)-s(v)))=$$
$$s(u+v)-s(u)-s(v).$$
We have proven that $\nu$ is well-defined. We now prove that $\nu$ is injective. Let $f,h\in\cor{R}$ be such that $\nu(f)=\nu(h)$. Since $f\circ\iota=h\circ\iota=\id_A$, the group $\iota(A)$ is contained in $\ker(f-h)$ and thus $f-h$ induces a homomorphism 
$B/\iota(A)\rightarrow A$. Now, $B/\iota(A)=\graffe{s(c)+\iota(A) : c\in C}$ and, the maps $f\circ s$ and $h\circ s$ being the same, we get $f-h=0$. The maps $f$ and $g$ are the same and $\nu$ is injective. To conclude, we prove that $\nu$ is surjective. Let $g\in\cor{H}$. Since each element $x$ of $B$ can be written uniquely as $x=\iota(a)+s(u)$, with $a\in A$ and $u\in C$, we define $f:B\rightarrow A$ by
$$x=\iota(a)+s(u)\mapsto f(x)=a+g(u).$$
For each $u\in C$, we have then $f\circ s(u)=g(u)$. We prove now that $f$ is a homomorphism. 
Let $x,y\in B$ and let $a,b\in A$ and $u,v\in C$ be such that $x=\iota(a)+s(u)$ and $y=\iota(b)+s(v)$. Keeping in mind that $g$ belongs to $\cor{H}$, we compute
\begin{align*}
f(x+y)-f(x)-f(y) & = f(\iota(a)+s(u)+\iota(b)+s(v))-f(\iota(a)+s(u))+ \\
				 &\ \ \ -f(\iota(b)+s(v)) \\
& = f(\iota(a)+\iota(b)-g(u+v)+g(u)+g(v)+s(u+v))+ \\
& \ \ \  -f(\iota(a)+s(u))-f(\iota(b)+s(v))
\end{align*}
and, since $g(C)$ is contained in $A$, we get
\begin{align*}
f(x+y)-f(x)-f(y) 
& = a+b-g(u+v)+g(u)+g(v)+g(u+v)+ \\
& \ \ \ -f(a+s(u))-f(b+s(v)) \\
& = a+b+g(u)+g(v)-a-g(u)-b-g(v) \\
& =0.
\end{align*}
We have proven that $f$ is a homomorphism and so $\nu$ is surjective.
\end{proof}

\begin{proposition}\label{proposition tV}
Let $c$ be as in Lemma \ref{cappino1} and let $\gamma$ be as in Lemma \ref{gammino3}. 
Set
$\cor{P}=\graffe{\pi\in\Hom(\overline{L},\overline{F_3}) : 
\pi_{|\bar{F_3}}=\id_{\bar{F_3}},\pi\circ c \ \text{bijective}}$
and let $t_V:\cor{P}\rightarrow\cor{K}_V$ be defined by
$\pi\mapsto \gamma^{-1}\circ\pi\circ c$.
Then $t_V$ is a bijection and $\cor{P}$ has cardinality $3$.
\end{proposition}

\begin{proof}
Let $c_2:V\rightarrow L/F_3$ be the isomorphism from Lemma \ref{LsuF3dim2}.
Composing the canonical projection $\overline{L}\rightarrow L/F_3$ with $c_2^{-1}$, we get the short exact sequence of abelian groups 
$0\rightarrow \overline{F_3}\rightarrow\overline{L}\rightarrow V\rightarrow 0$.
With $A=\overline{F_3}$, $B=\overline{L}$, $C=V$, and $s=c$, Lemma \ref{youusedto} applies. Let thus 
$\cor{R}=\graffe{\pi\in\Hom(\overline{L},\overline{F_3}) : 
\pi_{|\bar{F_3}}=\id_{\bar{F_3}}}$ and let 
$\cor{H}$ be the collection of maps $g:V\rightarrow\overline{F_3}$ such that, for all $x,y\in V$, one has $g(x+y)-g(x)-g(y)=c(x+y)-c(x)-c(y)$.
Then, thanks to Lemma \ref{youusedto}, each element of $\cor{H}$ is of the form $\pi\circ c$, where $\pi$ belongs to $\cor{R}$. 
In particular, the subset $\cor{P}$ of $\cor{R}$ is sent bijectively to the subset $\cor{H}_{\mathrm{bij}}$ of bijective elements of $\cor{H}$.
Now, by Lemma \ref{cappino1}, given any two elements $x,y\in V$, we have 
$c(x+y)-c(x)-c(y)=\gamma((x-y)\otimes(x\wedge y))$ and therefore each element 
$\kappa=\gamma^{-1}\circ\pi\circ c$, with $\pi\in\cor{P}$, belongs to $\cor{K}_V$. The map $\gamma$ being an isomorphism, $t_V$ is injective. Moreover, since $\gamma$ is bijective, Lemma \ref{cappino1} yields a well-defined injection $\cor{K}_V\rightarrow\cor{H}_{\mathrm{bij}}$, given by $\kappa\mapsto\gamma\circ\kappa$. It follows that 
$|\cor{P}|\leq |\cor{K}_V|\leq |\cor{H}_{\mathrm{bij}}|=|\cor{P}|$ 
and therefore $t_V$ is a bijection. Thanks to Proposition \ref{proposition sV}, the cardinality of $\cor{P}$ is $3$.
\end{proof}

\noindent
We remind the reader that $\cor{N}_3$ has been defined to be the collection of normal subgroups $N$ of $F$ such that $F/N$ is a $\kappa$-group of class $3$.

\begin{lemma}\label{n3}
There exists $M$ in $\cor{N}_3$ such that the quotient $F/M$ is isomorphic to $\yo/\yo_4$. Moreover, the set $\cor{N}_3$ is non-empty.
\end{lemma}

\begin{proof}
The group $\yo$ is a $\kappa$-group, by Lemma \ref{yo kappa}, and
it has class $4$, by Lemma \ref{yo class}.
It follows that there exists $M$ in $\cor{N}_3$ such that $F/M$ is isomorphic to $\yo/\yo_4$ and, in particular, $\cor{N}_3$ is non-empty. 
\end{proof}

\begin{lemma}\label{pi3}\label{LsuN}
Let $\cor{P}$ be as in Proposition \ref{proposition tV} and denote, for each $\pi\in\cor{P}$, by $K_{\pi}$ the unique normal subgroup of $F$ containing $E$ such that $\overline{K_{\pi}}=\ker\pi$. 
Then the map $r:\cor{P}\rightarrow\cor{N}_3$ that is defined by $\pi\mapsto K_{\pi}$ is a bijection.
Moreover, for each $N\in\cor{N}_3$, one has $|L:N|=9$.
\end{lemma}

\begin{proof}
We first show that $r$ is well-defined. To this end, let $\pi$ be an element of $\cor{P}$ and set $G=F/K_{\pi}$. 
Since $|F:F_2|=9$, Lemma \ref{frattini comm} yields that $G_2$ is equal to $F_2/K_{\pi}$.
Moreover, it easily follows from the definition of $\cor{P}$ that $\overline{L}$ decomposes as $\ker\pi\oplus\overline{F_3}=\overline{K_{\pi}}\oplus\overline{F_3}$. In particular, 
$L/K_{\pi}$ and $\overline{F_3}$ are naturally isomorphic and so, as a consequence of Lemma \ref{gammino35}, the subgroup $G_3$ coincides with $L/K_{\pi}$. The class of $G$ is equal to $3$, because $L$ is central modulo $E$. Now, the map $\pi\circ c$ being bijective, it follows that the cubing map induces a bijection 
$F/F_2\rightarrow \overline{F_3}$ and so, via the natural isomorphism $\overline{F_3}\rightarrow L/K_{\pi}$, the cubing map induces a bijection $G/G_2\rightarrow G_3$. As a result, we have that $|L:K_{\pi}|=|G_3|=|G:G_2|=|F:F_2|=9$ and $G$ is a $\kappa$-group. The choice of $\pi$ being arbitrary, we have proven that $r$ is well-defined. It is now easy to show that $r$ is bijective. From the surjectivity of $r$ one deduces that, for all $N\in\cor{N}_3$, the index $|L:N|$ is equal to $9$.
\end{proof}

\begin{proposition}\label{transitive action N}
The set $\cor{N}_3$ has cardinality $3$ and the natural action of $\Aut(F)$ on $\cor{N}_3$ is transitive.
\end{proposition}

\begin{proof}
Let $\cor{I}_V$ be defined as in Section \ref{section strutture}.
Define moreover $\psi:\cor{I}_V\rightarrow\cor{N}_3$ to be $\psi=r\circ t_V^{-1}\circ s_V$, where $s_V$, $t_V$, and $r$ are as in Propositions \ref{proposition sV} and \ref{proposition tV} and Lemma \ref{pi3}.
The combination of the just-mentioned results yields that $\psi$ is a bijection and, from its definition, it is easy to check that it respects the action of $\Aut(F)$.
Now, by Lemma \ref{cardinality IV}, the set $\cor{I}_V$ has cardinality $3$ and so $\cor{N}_3$ has cardinality $3$.
Again by Lemma \ref{cardinality IV}, the action of $\Aut(V)$ on $\cor{I}_V$ is transitive and thus the action of $\Aut(F)$ on $\cor{I}_V$ is transitive.
Since the map $\psi$ is an isomorphism of $\Aut(F)$-sets, the action of $\Aut(F)$ on $\cor{N}_3$ is transitive.
\end{proof}

\noindent
We are finally ready to give the proof of Proposition \ref{proposition yo kappa}. Let $G$ be a $\kappa$-group of class $3$. As a consequence of Proposition \ref{n3}, there exist $N$ and $M$ normal subgroups of $F$ such that $F/N$ and $F/M$ are respectively isomorphic to $G$ and $\yo/\yo_4$. Fix such $M$ and $N$. Then, thanks to Lemma \ref{transitive action N}, there exists an automorphism of $F$ mapping $M$ to $N$, which thus induces an isomorphism between $G$ and $\yo/\yo_4$. The choice of $G$ being arbitrary, the proof of Proposition \ref{proposition yo kappa} is complete.
\vspace{8pt}\\
\noindent
We conclude the present section by giving the proof of Theorem \ref{theorem unique kappa}. If $G$ is a $\kappa$-group of class $3$, then, by Proposition \ref{proposition yo kappa}, the group $G$ is isomorphic to $\yo/\yo_4$. On the other hand, the group $\yo$ has class $4$, by Lemma \ref{yo class}, and it is a $\kappa$-group, by Lemma \ref{yo kappa}. It follows that $\yo/\yo_4$ is a $\kappa$-group of class $3$. This proves Theorem \ref{theorem unique kappa}.

\section{Extensions of $\kappa$-groups}\label{section extension}

\noindent
We recall here that a $\kappa$-group is a finite $3$-group $G$ such that 
$|G:G_2|=9$ and such that cubing in $G$ induces a bijection $G/G_2\rightarrow G_3/G_4$. We remind the reader that we investigate $\kappa$-groups because we aim at classifying $3$-groups of class at least $4$ and intensity greater than $1$: those groups are all $\kappa$-groups, as a consequence of Lemma \ref{quotient kappa}.
The main purpose of the present section is that of proving the following proposition, which is the same as Theorem \ref{theorem kappa G2}.

\begin{proposition}\label{G2elemabelian}
Let $G$ be a $\kappa$-group such that $G_4$ has order $3$.
Then the subgroup $G_2$ is elementary abelian.
\end{proposition}

\noindent
Until the end of Section \ref{section extension}, we will work under the assumptions of Proposition \ref{G2elemabelian}. For each $i\in\Z_{\geq 1}$, we set $w_i=\wt_G(i)$. It follows from the assumptions, together with Lemma \ref{class 3 G3}($1$), that $(w_1,w_2,w_3,w_4)=(2,1,2,1)$.
Moreover, the group $G/G_4$ being a $\kappa$-group of class $3$, Proposition \ref{proposition yo kappa} yields that $G/G_4$ is isomorphic to $\yo/\yo_4$. It follows from Lemma \ref{yo structure}($2$) that there exist generators $a,b$ of $G$ satisfying 
$a^3\equiv [b,[a,b]]^{-1}\bmod G_4$ and $b^3\equiv [a,[a,b]]\bmod G_4$.
Call $c=[a,b]$, $d=[a,c]$, and $e=[b,c]$. Let moreover $f=[a,d]$. Then we have $a^3\equiv e^{-1}\bmod G_4$ and $b^3\equiv d\bmod G_4$.

\begin{lemma}\label{fbe}
The elements $d$ and $e$ generate $G_3$ modulo $G_4$.
\end{lemma}

\begin{proof}
The index $|G_2:G_3|$ is equal to $3$, so $c$ generates $G_2$ modulo $G_3$. 
By Lemma \ref{commumino}, the commutator map induces an isomorphism 
$G/G_2\otimes G_2/G_3\rightarrow G_3/G_4$, and so $d$ and $e$ span $G_3$ modulo $G_4$. 
\end{proof}

\begin{lemma}\label{greenday}
One has $G_4=\gen{f}=\gen{[b,e]}$.
\end{lemma}

\begin{proof}
By Lemma \ref{centro 3}, the centre of $G$ is equal to $G_4$ and, by Lemma \ref{class 4 comm map non-deg},
the commutator map 
$G/G_2\times G_3/G_4\rightarrow G_4$ is non-degenerate. 
The elements $d$ and $e$ generate $G_3$ modulo $G_4$, thanks to Lemma \ref{fbe}, and, by the choice of $a$ and $b$, we also have 
 $a^3\equiv e^{-1}\bmod G_4$ and $b^3\equiv d\bmod G_4$. From the non-degeneracy of the commutator map, it follows that both $f$ and $[b,e]$ are non-trivial elements of $G_4$, which, being cyclic of order $3$, then satisfies $G_4=\gen{f}=\gen{[b,e]}$.
\end{proof}

\begin{lemma}\label{utrs}
There exists a pair $(u,t)$ in $\graffe{\pm 1}\times\Z$ such that 
$[b,e]=f^u$ and $c^3=f^t$. Moreover, there exist $r,s\in\Z$ such that 
$a^3=e^{-1}f^r$ and $b^3=df^s$.
\end{lemma}

\begin{proof}
By assumption, the order of $G_4$ is $3$ and, by Lemma \ref{fbe}, both elements $f$ and $[b,e]$ generate $G_4$. There exists thus $u\in\graffe{\pm 1}$ such that $[b,e]=f^u$. Moreover, by the choice of $a$ and $b$, we know that $a^3\equiv e^{-1}\bmod G_4$ and $b^3\equiv d\bmod G_4$. There exist hence integers $r$ and $s$ such that $a^3=e^{-1}f^r$ and $b^3=df^s$. To conclude, thanks to Lemma \ref{class 3 G2 elem abelian}, the subgroup $G_2^3$ is contained in $G_4$ so there exists $t\in\Z$ such that $c^3=f^t$.
\end{proof}

\noindent 
We are now ready to give the proof of Proposition \ref{G2elemabelian}. 
To this end, let $u,t,r,s$ be as in Lemma \ref{utrs}.
By Lemma \ref{G2 piccoletto}, the subgroup $G_2$ is abelian and, by Lemma \ref{g3 elem abelian}, the exponent of $G_3$ is equal to $3$. 
It follows that 
\begin{align*}
ab^3 & = cbab^2 \\
	 & = cbcbab \\
	 & = cbcbcba \\
	 & = cecbecb^2a \\
	 & = ec^2f^uebcb^2a \\
	 & = f^ue^2c^2ecb^3a \\
	 & = f^ue^3c^3b^3a \\
	 & = f^uf^tb^3a \\
	 & =f^{u+t}b^3a
\end{align*}
from which we derive
\[
fdaf^s=adf^s=ab^3=f^{u+t}b^3a=f^{u+t}df^sa.
\]
The subgroup $G_4$ is central, thanks to Lemma \ref{centro 3}, and so one gets $$fda=f^{u+t}da.$$
Since the exponent of $G_3$ is equal to $3$, we have $u+t\equiv 1\bmod 3$ and so 
$$(u,t)\equiv(1,0)\bmod 3 \ \ \text{or}\ \ (u,t)\equiv(-1,-1)\bmod 3.$$
If $(u,t)\equiv(1,0)\bmod 3$, then we are done.
We assume by contradiction that $(u,t)\equiv(-1,-1)\bmod 3$. Then $c^3=f^{-1}$ and we compute
\begin{align*}
a^3b & = a^2cba \\
	 & = adcaba \\
	 & = fdac^2ba^2 \\
	 & = fd^2cacba^2 \\
	 & = fd^2cdcaba^2 \\
	 & = fd^3c^3ba^3 \\
	 & = ba^3.
\end{align*}
We have shown that $a^3$ centralizes $b$ in $G$. Call $C=\gen{\graffe{b} \cup G_2}$. Then $a^3$ belongs to $\ZG(C)$, which then, thanks to Lemma \ref{quasicentri dei massimali}, 
contains $\graffe{a^3,b^3}\cup G_4$. The group $G$ being a $\kappa$-group, it follows that $\ZG(C)$ contains $G_3$, and so $[b,e]=1$. Contradiction to Lemma \ref{greenday}. The proof of Proposition \ref{G2elemabelian}, and thus that of Theorem \ref{theorem kappa G2}, is now complete.

\begin{corollary}\label{yoG2}
The subgroup $\yo_2$ of $\yo$ is elementary abelian.
\end{corollary}

\begin{proof}
The group $\yo$ is a $\kappa$-group by Lemma \ref{yo kappa} and, thanks to Lemma \ref{yo structure}($1$), the subgroup $\yo_4$ has order $3$. It follows from Proposition \ref{G2elemabelian} that $\yo_2$ is elementary abelian.
\end{proof}

\begin{corollary}\label{corollary G2elemab}
Let $Q$ be a finite $3$-group of class $4$ and let $(Q_i)_{i\geq 1}$ denote the lower central series of $Q$. If $\inte(Q)>1$, then $Q_2$ is elementary abelian.
\end{corollary}

\begin{proof}
By Lemma \ref{quotient kappa}, the group $Q$ is a $\kappa$-group. 
Moreover, thanks to Theorem \ref{theorem dimensions}, the subgroup $Q_4$ has order $3$. 
It follows from Proposition \ref{G2elemabelian} that $Q_2$ is elementary abelian. 
\end{proof}

\begin{corollary}\label{3-gps class at most 4}
Let $Q$ be a finite $3$-group with $\inte(Q)>1$. Then $Q$ has nilpotency class at most $4$.
\end{corollary}

\begin{proof}
Assume that $Q$ has class at least $4$. 
Thanks to Lemma \ref{intensity of quotients}, the intensity of $Q/Q_5$ is greater than $1$, and so, as a consequence of Corollary \ref{corollary G2elemab}, the subgroup $Q_2^3$ is contained in $Q_5$. 
However, because of Proposition \ref{class at least 5, G2^p=G4}, each finite $3$-group $H$ of class at least $5$ with $\inte(H)>1$ satisfies $H_2^3=H_4$ and so it follows that $Q$ has class at most $4$.
\end{proof}

\section{Constructing automorphisms}\label{section unique with automorphism}

In this section we aim at understanding the structure of finite $3$-groups of class $4$ and intensity greater than $1$. We recall that a $\kappa$-group is a finite $3$-group $G$ such that $|G:G_2|=9$ and the cubing map on $G$ induces a bijection $G/G_2\rightarrow G_3/G_4$ (see Section \ref{section cubing} for a closer look at $\kappa$-groups). The reason why $\kappa$-groups are so special for us is Lemma \ref{quotient kappa}, which asserts that any finite $3$-group of class $4$ and intensity greater than $1$ is a $\kappa$-group. Moreover, we know from Proposition \ref{proposition -1^i}, that if we hope to construct a $3$-group $G$ of large class and intensity greater than $1$, then we need as well to construct an automorphism of order $2$ of $G$ that induces the inversion map on the abelianization of $G$. 
We will devote the present section to the proof of the following result.

\begin{proposition}\label{proposition unique with automorphism}
Let $G$ be a $\kappa$-group such that $G_4$ has order $3$. Assume that $G$ possesses an automorphism of order $2$ that induces the inversion map on $G/G_2$. Then $G$ is isomorphic to $\yo$.
\end{proposition}

\noindent
We will prove Proposition \ref{proposition unique with automorphism} at the end of the present section and so
the following assumptions will hold until the end of Section \ref{section unique with automorphism}.
Let $G$ be a $\kappa$-group such that $G_4$ has order $3$. Then the group $G$ has class $4$ and $(\wt_G(i))_{i=1}^4=(2,1,2,1)$.
Let $F$ be the free group on the set $S=\graffe{a,b}$ and
let $\iota:S\rightarrow G$ be such that $G=\gen{\iota(S)}$. 
By the universal property of free groups, there exists a unique homomorphism $\phi: F\rightarrow G$ such that 
$\phi(a)=\iota(a)$ and $\phi(b)=\iota(b)$. 
As a consequence of its definition, the map $\phi$ is surjective.  
Let $(F_i)_{i\geq 1}$ denote the lower $3$-series of $F$, 
which is defined recursively as
\[F_1=F \ \ \text{and} \ \ F_{i+1}=[F,F_i]F_i^3.\]
and, in addition, let
$$L=F^3F_3 \ \ \text{and}\ \ E=[F,L]F_2^3.$$ 
All $F_i$'s, $L$, and $E$ are stabilized by any endomorphism of $F$. For a visualization of such groups we refer to the end of Section \ref{section construction} or to the diagram before Lemma \ref{complementE}.
Let $\beta$ be the endomorphism of $F$ sending $a$ to $a^{-1}$ and $b$ to $b^{-1}$. Since $\beta^2=\id_F$, the map $\beta$ is an automorphism of $F$. We remind the reader that we have already worked with such an automorphism $\beta$ in Section \ref{section construction} and we will thus, in this section, often apply results achieved in Section \ref{section construction}. 
We conclude by defining two specific sets, consisting of normal subgroups of $F$.
Let $\cor{N}_3$ denote the collection of normal subgroups $N$ of $F$ such that $F/N$ is a $\kappa$-group of class $3$, as defined in Section \ref{section strutture free}. 
For each element $N$ of $\cor{N}_3$, we set
$$D_N=[F,N]F_2^3[F_2,F_2].$$ 
We define moreover $\cor{N}_4$ to be the collection of normal subgroups $M$ of $F$ such that $F/M$ is a $\kappa$-group of class $4$ with $\wt_{F/M}(4)=1$ and such that $F/M$ possesses an automorphism of order $2$ that induces the inversion map on the abelianization $(F/M)/(F/M)_2$ of $F/M$. 
We will keep this notation until the end of Section \ref{section unique with automorphism}.

\begin{lemma}
Let $N\in\cor{N}_3$. Then $D_N$ is contained in $E$.
\end{lemma}

\begin{proof}
As a consequence of Lemma \ref{description L}, the subgroup $L$ contains $N$.
Again by Lemma \ref{description L}, the index $|F_2:L|$ is equal to $3$ and so, thanks to Lemma \ref{cyclic quotient commutators}, one has $[F_2,F_2]=[F_2,L]$. We get
\[
[F,N]F_2^3[F_2,F_2]\subseteq[F,L]F_2^3[F_2,L]=[F,L]F_2^3=E
\]
and therefore $D_N$ is contained in $E$. 
\end{proof}

\begin{lemma}\label{kkkk maggiore di 5}
For each $k\in\Z_{\geq 5}$, one has $\phi(F_k)=\graffe{1}$.
\end{lemma}

\begin{proof}
Let $k\in\Z_{\geq 5}$ and recall that $F_k=[F,F_{k-1}]F_{k-1}^3$. By definition of $E$, one has 
$\phi([F,F_{k-1}])\subseteq\phi([F,F_4])\subseteq\phi([F,E])=[\phi(F),\phi(E)]$
and so, as a consequence of Lemma \ref{e in n}, we get
$\phi([F,F_{k-1}])\subseteq [G,G_4]=\graffe{1}$. 
It follows from Lemma \ref{frattini comm} that 
$\phi(F_k)=\phi(F_{k-1}^3)\subseteq \phi(F_2^3)\subseteq \Phi(G)^3=G_2^3$ and so
Proposition \ref{G2elemabelian} yields $\phi(F_k)=\graffe{1}$.
\end{proof}

\begin{lemma}\label{xy invertiti}
Let $\alpha$ be an automorphism of order $2$ of $G$ that induces the inversion map on $G/G_2$.
Then there exist generators $x$ and $y$ of $G$ such that $\alpha(x)=x^{-1}$ and $\alpha(y)=y^{-1}$.
\end{lemma}

\begin{proof}
Write $G^-=\graffe{g\in G : \alpha(g)=g^{-1}}$.
Since $(\wt_G(i))_{i=1}^4=(2,1,2,1)$, Lemma \ref{order +- jumps} yields that the map $G^-\rightarrow G/G_2$, defined by $g\mapsto gG_2$, is surjective. Thanks to Lemma \ref{frattini comm}, the subgroups $G_2$ and $\Phi(G)$ coincide and therefore there exist two elements $x$ and $y$ of $G^-$ that generate $G$. 
\end{proof}

\begin{proposition}\label{proposition extension alpha}
Let $\alpha$ be an automorphism of order $2$ of $G$ that induces the inversion map on $G/G_2$. Let moreover $k\in\Z_{\geq 5}$ and let $\phi_k:F/F_k\rightarrow G$ be the map induced by $\phi$. 
Then there exists $\epsilon\in\Aut(F/F_k)$ of order $2$ such that 
$\alpha\phi_k=\phi_k\epsilon$.
\end{proposition}

\begin{proof}
For each $k\in\Z_{\geq 5}$, the map $\phi_k:F/F_k\rightarrow G$ is well-defined, thanks to Lemma \ref{kkkk maggiore di 5}.
Let now $x$ and $y$ be as in Lemma \ref{xy invertiti} and let $c$ and $d$ be elements of $F$ such that $\phi(c)=x$ and $\phi(d)=y$, which exist because $\phi$ is surjective. As a consequence of Lemma \ref{construction indices}, the map $\phi$ induces an isomorphism $F/F_2\rightarrow G/G_2$ and therefore $c$ and $d$ generate $F$ modulo $F_2$. Let now $\psi:F\rightarrow F$ be the endomorphism of $F$ sending $a\mapsto c$ and $b\mapsto d$; such $\psi$ exists thanks to the universal property of free groups. Fix $k\in\Z_{\geq 5}$. The subgroup $F_k$ being being stabilized by any endomorphism of $F$, the map $\psi$ induces an endomorphism $\overline{\psi}$ of the $3$-group $\overline{F}=F/F_k$. However, since $\Phi(\overline{F})=\overline{F_2}$, the map $\overline{\psi}$ induces an automorphism of $\overline{F}/\Phi(\overline{F})$ and so $\overline{\psi}$ is in fact an automorphism of $\overline{F}$. 
Let $\overline{\beta}$ be the automorphism of $\overline{F}$ that is induced by $\beta$ and define $\epsilon=\overline{\psi}\overline{\beta}\overline{\psi^{-1}}$.
By construction, the following diagram is commutative.
\[
\begin{diagram}
F/F_k     &  \rTo^{\phi_k} & G  \\
\dTo^{\epsilon}     &                & \dTo_{\alpha}\\
F/F_k     &  \rTo^{\phi_k}   & G \\
\end{diagram}
\]
Moreover, $\epsilon$ has order $2$, because it is conjugate in $\Aut(\overline{F})$ to $\overline{\beta}$.
\end{proof}

\begin{lemma}\label{Dker}
Let $M$ be an element of $\cor{N}_4$. Then $N=ME$ belongs to $\cor{N}_3$ and $D_N$ is contained in $M$.
\end{lemma}

\begin{proof}
Let $H=F/M$ and let $\pi:F\rightarrow H$ be the canonical projection.
Then $\pi(N)=\pi(ME)=\pi(E)$ and so, as a consequence of Lemma \ref{e in n}, we get $\pi(N)\subseteq H_4$. The order of $H_4$ being  $3$, either $\pi(N)=H_4$ or $N\subseteq M$. 
Assume first that $\pi(N)=H_4$. Then we have $M\subseteq N\subseteq\pi^{-1}(H_4)$ and $M\neq N$. On the other hand, we know $|\pi^{-1}(H_4):M|=|H_4|=3$ and therefore $N=\pi^{-1}(H_4)$. 
As a result, $F/N$ is isomorphic to $H/H_4$ and so $N$ belongs to $\cor{N}_3$. 
We prove that $\pi(D_N)$ is trivial.
The image of $F_2$ under $\pi$ is equal to $H_2$, thanks to Lemma \ref{construction indices}. Moreover, the subgroup $H_4$ is central in $H$, because $H$ has class $4$, and the commutator subgroup of $H$ is elementary abelian, thanks to Proposition \ref{G2elemabelian}. We compute
$$\pi(D_N)=\pi([F,N])\pi(F_2^3)\pi([F_2,F_2])=[H,H_4]H_2^3[H_2,H_2]=\graffe{1},$$
and so $D_N$ is contained in $M$.
We now prove that $\pi(N)=H_4$. 
We work by contradiction, assuming that $N\subseteq M$. Since $N=ME$, we get that $E$ is contained in $M$. 
As a consequence, the group $F/E$ has class at least $4$. However, one has
\[
[F,[F,[F,F]]]\subseteq [F,[F,F_2]]\subseteq  [F,F_3]\subseteq [F,L]\subseteq E
\]
and therefore $F/E$ has class at most $3$. Contradiction.
\end{proof}

\begin{lemma}\label{mappette}
Let $N\in\cor{N}_3$. Denote moreover $H=\yo$. 
Then there exists a surjective homomorphism $\varphi:F\rightarrow G$ such that $N=\varphi^{-1}(G_4)$. Moreover, $\varphi$ induces isomorphisms $\varphi_1:F/F_2\rightarrow H/H_2$ and $\varphi_3:L/N\rightarrow H_3/H_4$ and a surjective homomorphism $\varphi_4:E/D_N\rightarrow H_4$.
\end{lemma}

\begin{proof}
Let $\psi:F\rightarrow H$ be a surjective homomorphism, which exists thanks to the universal property of free groups. Set $K=\psi^{-1}(H_4)$. Then $K$ belongs to $\cor{N}_3$, because $H$ is a $\kappa$-group, and so, thanks to Lemma \ref{transitive action N}, there exists an automorphism $r$ of $F$ such that $r(N)=K$. Define $\varphi=\psi\circ r$. Then $\varphi$ is a surjective homomorphism $F\rightarrow H$ such that $\varphi^{-1}(H_4)=N$. Moreover, $\varphi$ induces isomorphisms $\varphi_1:F/F_2\rightarrow H/H_2$ and $\varphi_3:L/N\rightarrow H_3/H_4$ as a consequence of Lemmas \ref{construction indices} and \ref{description L}. We conclude by showing that $\varphi$ induces a surjective homomorphism $E/D_N\rightarrow H_4$. Thanks to Lemma \ref{Dker}, the subgroup $D_N$ is contained in the kernel of $\varphi$. Moreover, since $\varphi(F_2)=H_2$ and $\varphi(L)=H_3$, we get
\[
\varphi(E)=\varphi([F,L]F_2^3)=[H,H_3]H_2^3=H_4H_2^3.
\]
Now, the group $H$ is a $\kappa$-group and hence $H_2^3\subseteq H_4$. It follows that $\varphi(E)=H_4$ and therefore $\varphi$ induces a surjective homomorphism $E/D_N\rightarrow H_4$.
\end{proof}

\begin{lemma}\label{nondeg with DN}
Let $N\in\cor{N}_3$. Then the commutator map induces a non-degenerate map $F/F_2\times L/N\rightarrow E/D_N$ whose image generates $E/D_N$. In addition, one has $E\neq D_N$.
\end{lemma}

\begin{proof}
Write $\overline{F}=F/D_N$ and use the bar notation for the subgroups of $\overline{F}$.
From the definition of $E$, one sees that 
$\overline{E}=[\overline{F},\overline{L}]$. Moreover, by Lemma \ref{e in n}, the subgroup $E$ is contained in $N$ and so $[F,E]\subseteq [F,N]\subseteq D_N$. In particular, $\overline{E}$ is central in $\overline{F}$ and so it follows from Lemma \ref{tgt} that the commutator map $F\times L\rightarrow \overline{E}$ is bilinear. Since $[F_2,L]$ and $[F,N]$ are both contained in $D_N$, the last map factors as a bilinear map 
$\gamma:F/F_2\times L/N\rightarrow\overline{E}$ whose image generates $\overline{E}$. Set now $H=\yo$. Then, as a consequence of Lemmas \ref{yo kappa}  and \ref{centro 3}, the centre of $H$ is equal to $H_4$ and thus Lemma \ref{class 4 comm map non-deg} yields that the commutator map induces a non-degenerate map 
$\nu:H/H_2 \times H_3/H_4 \rightarrow H_4$. 
With the notation from Lemma \ref{mappette}, the following diagram is commutative. 
\[
\begin{diagram}
F/F_2    \times   L/N           &  \rTo^{\gamma} & E/D_N  \\
\dTo^{\varphi_1} \ \  \dTo_{\varphi_3}    &                & \dTo_{\varphi_4}\\
H/H_2    \times   H_3/H_4           &  \rTo^{\nu}   & H_4  \\
\end{diagram}
\]
Since the map $\nu$ is non-degenerate and both $\varphi_1$ and $\varphi_3$ are isomorphisms, the map $\gamma$ is non-degenerate. It follows in particular that $E\neq D_N$.
\end{proof}

\begin{lemma}\label{cN}
Let $V$ and $W$ be $2$-dimensional vector spaces over $\F_3$ and let 
$\eta:V\rightarrow W$ be a bijective map with the property that, for each $\lambda\in\F_3$ and $v\in V$, one has $\eta(\lambda v)=\lambda\eta(v)$. 
Define $K=\gen{v\otimes\eta(v) : v\in V}$. 
Then the quotient $(V\otimes W)/K$ has dimension $1$ as a vector space over $\F_3$.
\end{lemma}

\begin{proof}
Without loss of generality we assume that $V=W$. Assume first that 
$\eta$ is an automorphism of $V$ and define the automorphism $\sigma$ of $V\otimes V$ by $x\otimes y\mapsto x\otimes\eta(y)$. Then the subspace $\Delta=\gen{v\otimes v :v \in V}$ is mapped isomorphically to $K$ via $\sigma$. It follows that $(V\otimes V)/K$ has the same dimension as $(V\otimes V)/\Delta=\bigwedge^2V$ and so 
$(V\otimes V)/K$ has dimension $1$. Let now $\eta$ be any map satisfying the hypotheses of Lemma \ref{cN}. Then $\eta$ induces a bijective map 
$\overline{\eta}:\mathbb{P}V\rightarrow\mathbb{P}V$, where $\mathbb{P}V$ denotes the collection of $1$-dimensional subspaces of $V$. 
As a consequence of Lemma \ref{V4}, there exists an automorphism $\tau$ of $V$ such that $\overline{\tau}=\overline{\eta}$ and, for each $v\in V$, one has $\F_3\tau(v)=\F_3\eta(v)$. As a consequence, 
we get $K=\gen{v\otimes \tau(v) : v\in V}$ and therefore 
$(V\otimes V)/K$ has dimension $1$ over $\F_3$.   
\end{proof}



\begin{lemma}\label{orderE}
Let $N\in\cor{N}_3$.
Then $|E:D_N|=3$.
\end{lemma}

\begin{proof}
The quotient $F/F_2$ is a $2$-dimensional vector space over $\F_3$, by definition of $F_2$, while $L/E$ is a $4$-dimensional vector space over $\F_3$, thanks to Lemma \ref{Lbar of dim 4}. Moreover, by Lemma \ref{e in n}, the subgroup $N$ contains $E$ and, as a consequence of Lemma \ref{LsuN}, the quotient $L/N$ is a vector space of dimension $2$ over $\F_3$.
Let $\gamma:F/F_2\otimes L/N\rightarrow E/D_N$ be the surjective homomorphism induced from the non-degenerate map of Lemma \ref{nondeg with DN}. 
Let moreover $c:F/F_2\rightarrow L/E$ be the map from Lemma \ref{cappino1} and let $\pi$ denote the canonical projection $L/E\rightarrow L/N$. Denote $c_N=\pi\circ c$ and note that, as a consequence of Lemma \ref{pi3}, the map $c_N:F/F_2\rightarrow L/N$ is a bijection between vector spaces of dimension $2$ over $\F_3$. From Lemma \ref{cappino1}, it is clear that $c$ commutes with scalar multiplication by elements of $\F_3$.
Define $K=\gen{x\otimes c_N(x) : x\in F/F_2}$.
As a consequence of the definition of $c$, each element $x\otimes c_N(x)$, with $x\in F/F_2$, belongs to the kernel of $\gamma$, and therefore $K$ is contained in $\ker\gamma$. 
It follows from Lemma \ref{cN} that $(F/F_2\otimes L/N)/K$ has dimension $1$ and therefore $E/D_N$ has dimension at most $1$ as a vector space over $\F_3$. 
By Lemma \ref{nondeg with DN}, the quotient $E/D_N$ is non-trivial and so $\gamma$ is not the trivial map. It follows that $\ker\gamma$ has dimension $3$ and thus $E/D_N$ has cardinality $3$.
\end{proof}

\noindent
The following lemmas pave the way to proving Proposition \ref{proposition beta a meno di interni}.

\begin{lemma}\label{innermodL}
Let $\eta$ be an automorphism of $F/L$ of order $2$ that induces the inversion map on $F/F_2$.
Then there exists $\varphi_L\in\Inn(F/L)$ such that, for each $x\in F$, one has $\beta(x)\equiv(\varphi_L\eta\varphi_L^{-1})(x)\bmod L$.
\end{lemma}

\begin{proof}
Write $H=F/L$. The group $F$ being $2$-generated, $|F:F_2|=9$ and so, as a consequence of Lemma \ref{description L}, the group $H$ has order $27$. Thanks to the definitions of $F_2$ and $L$, one easily sees that $H$ is non-abelian of exponent $3$ and that $H_2=F_2/L$. Lemma \ref{normal intersection centre trivial} yields that $H_2$ is central in $H$. Applying Lemma \ref{quotient by centre not cyclic}, we get that $H_2=\ZG(H)$ and therefore $H$ is extraspecial. Let now $\beta_L$ be the automorphism of $H$ that is induced by $\beta$. Then $\eta^{-1}\beta_L$ induces the identity on $H/H_2$ and so, thanks to Lemma \ref{innerextra}, one gets $\eta^{-1}\beta_L\in\Inn(H)$. The group $\Inn(H)$ being a normal $3$-subgroup of $\Aut(H)$, the Schur-Zassenhaus theorem applies to $\Inn(H)\rtimes\gen{\eta}$ and ensures that there exists $\varphi_L\in\Inn(H)$ with the property that $\beta_L=\varphi_L\eta\varphi_L^{-1}$. 
\end{proof}

\begin{lemma}\label{innermodE}
Let $\eta$ be an automorphism of $F/E$ of order $2$ that induces the inversion map on $F/F_2$.
Assume that $\beta$ coincides with $\eta$ modulo $L$. 
Then, for all $x\in F$, one has $\beta(x)\equiv\eta(x)\bmod E$.
\end{lemma}

\begin{proof}
Let $\Delta$ denote the subgroup of $\Aut(F/E)$ consisting of all those automorphisms of $F/E$ inducing the identity on both $F/L$ and $L/E$. Let $\beta_E$ be the automorphism that is induced on $F/E$ by $\beta$. As a consequence of Lemma \ref{beta layers}, the element $\psi=\eta^{-1}\beta_E$ belongs to $\Delta$ and so, thanks to Lemma \ref{maps1-1}, there exists a homomorphism $h:F/L\rightarrow L/E$ such that, for all $x\in F/E$, one has $\psi(x)= h(x)x$. 
The quotient $L/E$ being elementary abelian, the groups $\Hom(F/L,L/E)$ and $\Hom(F/F_2,L/E)$ are naturally isomorphic and so $h$ factors as a homomorphism $F/F_2\rightarrow L/E$.
Now $\eta$ coincides with $\beta$ on $F/F_2$ and so, thanks to Lemma \ref{beta layers}, it induces the inversion map on $L/E$. For each $x\in F/F_2$, it follows that 
$$(\eta h\eta^{-1})(x)=\eta(h(x^{-1}))=(h(x^{-1}))^{-1}=h(x).$$
However, the automorphisms $\eta$ and $\beta_E$ having order $2$, one also has 
$$\eta^2=1=\beta_E^2=\eta\psi\eta\psi$$
and therefore $\eta\psi\eta^{-1}=\psi^{-1}$.
For all $x\in F/E$, we compute 
\begin{align*}
\psi(x)x^{-1} = h(x) & = (\eta h\eta^{-1})(x)\\
			         & = \eta(h(\eta^{-1}(x))) \\
			         & = \eta\big(\psi(\eta^{-1}(x))(\eta^{-1}(x))^{-1}\big) \\
			         & = (\eta\psi\eta^{-1})(x)x^{-1} \\
			  		 & = \psi^{-1}(x)x^{-1}
\end{align*}
and therefore $\psi(x)^2=1$. The group $F/E$ being a $3$-group, it follows that $\psi$ coincides with the trivial map and therefore $\eta$ and $\beta_E$ are equal.
\end{proof}

\begin{lemma}\label{deltainner}
Let $N=\phi^{-1}(G_4)$ and let $\Delta$ denote the subgroup of $\Aut(F/D_N)$ consisting of all those maps inducing the identity on both $F/E$ and $E/D_N$.
Then $\Delta$ is contained in $\Inn(F/D_N)$.
\end{lemma}

\begin{proof}
The group $N$ belongs to $\cor{N}_3$, because $G/G_4$ is a $\kappa$-group of class $3$. As a consequence of Lemma \ref{nondeg with DN}, the commutator map induces an injective homomorphism
$\varphi:L/N\rightarrow \Hom(F/F_2,E/D_N)$.
Combining Lemmas \ref{LsuN} and \ref{orderE}, we get that the orders of $\Hom(F/F_2,E/D_N)$ and $L/N$ are the same and therefore $\varphi$ is also surjective. 
It follows that, for each element $f$ of $\Hom(F/F_2,E/D_N)$, there exists $l\in L$ such that $f$ equals $xF_2\mapsto [l,x]D_N$. Set $\overline{F}=F/D_N$ and use the bar notation for the subgroups of $\overline{F}$.
We now prove that $\Delta$ is contained in $\Inn(\overline{F})$. Let $\delta\in\Delta$. Then, as a consequence of Lemma \ref{maps1-1}, there exists a homomorphism $f:\overline{F}\rightarrow\overline{E}$ whose kernel contains $\overline{E}$ and such that, for each $x\in\overline{F}$, one has $\delta(x)=f(x)x$. Fix such $f$. The group $\overline{E}$ being elementary abelian, the kernel of $f$ contains $\overline{F_2}$ and therefore $f$ factors as a homomorphism $F/F_2\rightarrow\overline{E}$. As a result, there exists $l\in\overline{L}$ such that, for each $x\in\overline{F}$, one has $f(x)=[l,x]$ and thus $\delta(x)=[l,x]x=lxl^{-1}$. In particular, $\delta$ is an inner automorphism of $\overline{F}$ and, the choice of $\delta$ being arbitrary, $\Delta$ is contained in $\Inn(\overline{F})$.
\end{proof}

\begin{lemma}\label{Dente}
Let $N$ be an element of $\cor{N}_3$. Then $\beta(D_N)=D_N$.
\end{lemma}

\begin{proof}
As a consequence of Lemma \ref{beta layers}, the group $N$ is $\gen{\beta}$-stable and therefore so is $D_N$.
\end{proof}

\begin{lemma}\label{innermodDN}
Let $N=\phi^{-1}(G_4)$.
Let $\eta\in\Aut(F/D_N)$ be of order $2$ and assume that $\eta$ induces the inversion map on $F/F_2$.
Assume moreover that $\beta$ and $\eta$ induce the same automorphism of $F/E$. Then there exists $\psi_N\in\Inn(F/D_N)$ such that, for all $x\in F$, one has $\beta(x)\equiv(\psi_N\eta\psi_N^{-1})(x)\bmod D_N$.
\end{lemma}

\begin{proof}
Set $\overline{F}=F/D_N$ and use the bar notation for the subgroups of $\overline{F}$. Thanks to Lemma \ref{Dente}, the map $\beta$ induces an automorphism of $\overline{F}$, which we denote by $\overline{\beta}$.
Let $\Delta$ denote the subgroup of $\Aut(\overline{F})$ consisting of all those elements $\delta$ such that $\delta$ induces the identity on both $\overline{E}$ and $\overline{F}/\overline{E}$. 
Then, as a consequence of Lemmas \ref{nondeg with DN} and \ref{lemma product of characters}, the automorphism $\eta^{-1}\overline{\beta}$ belongs to $\Delta$ and thus, thanks to Lemma \ref{deltainner}, we have  $\eta^{-1}\overline{\beta}\in\Inn(\overline{F})$. Applying the Schur-Zassenhaus theorem to $\Inn(\overline{F})\rtimes\gen{\eta}$, we get that there exists $\overline{\psi}\in\Inn(\overline{F})$ such that $\overline{\beta}=\overline{\psi}\eta\overline{\psi}^{-1}$. This concludes the proof.
\end{proof}

\begin{proposition}\label{proposition beta a meno di interni}
Let $\alpha$ be an automorphism of order $2$ of $G$ that induces the inversion map on $G/G_2$. 
Then there exists $\gamma\in\Inn(F)$ such that 
$\alpha\phi=\phi(\gamma\beta\gamma^{-1})$.
\end{proposition}

\begin{proof}
Thanks to proposition \ref{proposition extension alpha}, there exists an automorphism $\epsilon$ of $F/F_5$ of order $2$ such that $\alpha\phi_5=\phi_5\epsilon$. 
As a consequence, the map $\epsilon$ induces the inversion map on $F/F_2$. 
Let now $M=\ker\phi$ and let $N=ME$. 
Thanks to Lemma \ref{Dker}, the group $N$ belongs to $\cor{N}_3$ and $D_N\subseteq M$. One easily shows that $F_5$ is contained in $D_N$. It follows that $\epsilon$ induces an automorphism $\eta$ of order $2$ of $F/D_N$. 
Let $\eta_L$ be the automorphism that $\eta$ induces on $F/L$. Then, via the choice of a representative, Lemma \ref{innermodL} ensures that there exists an inner automorphism $\varphi_N$ of $F/D_N$ such that $\beta$ and $\varphi_N\eta\varphi_N^{-1}$ induce the same automorphism of $F/L$. Fix such $\varphi_N$ and define $\eta_1=\varphi_N\eta\varphi_N^{-1}$. Since $\eta$ has order $2$, the order of $\eta_1$ is equal to $2$.
Lemma \ref{innermodE} yields that in fact $\eta_1$ and $\beta$ are the same modulo $E$. At last, let $\psi_N$ be as in Lemma \ref{innermodDN} and define $\eta_2=\psi_N\eta_1\psi_N^{-1}$. As a consequence of Lemma \ref{innermodDN}, the maps $\eta_2$ and $\beta$ induce the same map on $F/D_N$. 
Via the choice of a representative, the inner automorphism $\psi_N\varphi_N$ of $F/D_N$ lifts to an inner automorphism $\gamma$ of $F$ with the property that $\eta$ and $\gamma\beta\gamma^{-1}$ induce the same automorphism on $F/D_N$. To conclude, let $\phi_N:F/D_N\rightarrow G$ be the map induced by $\phi$. 
Since $\alpha\phi_5=\phi_5\epsilon$, one gets 
$\alpha\phi_N=\phi_N\eta$ and therefore $\alpha\phi=\phi(\gamma\beta\gamma^{-1})$.
\end{proof}

\begin{lemma}\label{n4}
There exists $M\in\cor{N}_4$ such that $F/M$ is isomorphic to $\yo$. 
Moreover, $\cor{N}_4$ is non-empty.
\end{lemma}

\begin{proof}
The group $\yo$ is a $\kappa$-group, by Lemma \ref{yo kappa}, and 
$\yo_4$ has cardinality $3$, thanks to Lemma \ref{yo structure}($1$). 
By Lemma \ref{yo class}, the class of $\yo$ is $4$ and moreover, thanks to Lemma \ref{yo inversion}, the group $\yo$ possesses an automorphism that induces the inversion map on the quotient $\yo/\yo_2$. 
By the universal property of free groups, there exists $M\in\cor{N}_4$ such that $F/M$ is isomorphic to $\yo$. In particular, $\cor{N}_4$ is non-empty.
\end{proof}

\begin{lemma}\label{n4 betastable}
For each $M\in\cor{N}_4$, one has $\beta(M)=M$. 
\end{lemma}

\begin{proof}
Let $M\in\cor{N}_4$. Without loss of generality $G=F/M$ and so $M=\ker\phi$. Let moreover $\alpha$ be an automorphism of $G$ of order $2$ that induces the inversion map on $G/G_2$. Then, thanks to Proposition \ref{proposition beta a meno di interni}, there exists $\gamma\in\Inn(F)$ such that $\alpha\phi=\phi(\gamma\beta\gamma^{-1})$.
It follows that 
\[
\graffe{1}=\alpha(\phi(M))=\phi(\gamma\beta\gamma^{-1})(M)=\phi\beta(M)
\]
and therefore $\beta(M)$ is contained in $\ker\phi=M$. 
Since $\beta$ induces an automorphism of each quotient $F/F_k$ and since, for large enough $k$ one has $F_k\subseteq M$, we have in fact that $\beta(M)=M$.
\end{proof}

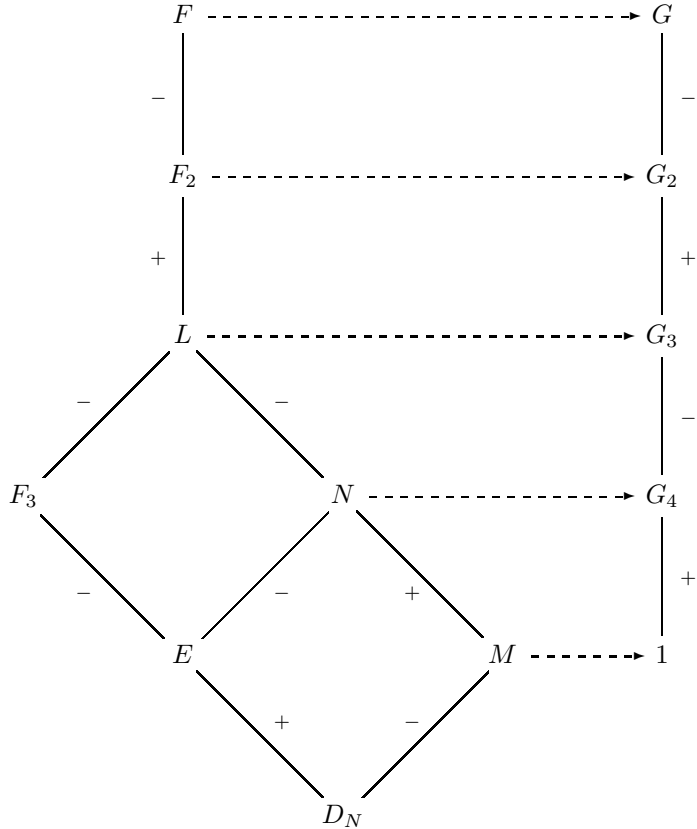
\[
\begin{diagram}[nohug]
       &          &  F  &  & & &  &  \rDashto   &  G  \\
       &          & \dLine^{- \ }  & &  &  & &  &    \dLine_{\ -}  \\
       &          &  F_2    &   & & &   &  \rDashto   & G_2    \\
       &          &  \dLine^{+ \ } & &  &  & &     &  \dLine_{\ +} \\
       &          &  L  &   & &   &   &  \rDashto(6,0)  & G_3    \\
       &  \ldLine^{- \ }  &   & \rdLine^{\ -} & & & &  &    \dLine_{\ -}   \\
F_3  &     &    &     & N & &  &  \rDashto(4,0)  & G_4    \\
       & \rdLine_{- \ } & & \ldLine_{\ -} & & \rdLine_{\ +} & && \dLine_{\ +}\\
       &         &  E  &   &   &  & M  & \rDashto   &   1\\
       &   &  & \rdLine^{\ +}  & & \ldLine^{\ -}  & & & \\
       &   &  &  &  D_N & & & & 
\end{diagram}
\vspace{10pt}
\]

\begin{lemma}\label{complementE}
Let $N$ be an element of $\cor{N}_3$ and write $\overline{F}=F/D_N$. Set moreover $\overline{N}=N/D_N$ and $\overline{E}=E/D_N$.
Define $\overline{\beta}$ to be the map that is induced by $\beta$ on $\overline{F}$ and set 
$$\overline{N}^{\,+}=\graffe{\overline{x}\in\overline{N} : \overline{\beta}(\overline{x})=\overline{x}} \ \ \text{and} \ \ \overline{N}^{\,-}=\graffe{\overline{x}\in\overline{N} : \overline{\beta}(\overline{x})=\overline{x}^{\,-1}}.$$ 
Then $\overline{N}^{\,+}=\overline{E}$ and $\overline{N}^{\,-}$ is the unique 
$\gen{\overline{\beta}}$-stable complement of $\overline{E}$ in $\overline{N}$.
\end{lemma}

\begin{proof}
As a consequence of Lemma \ref{beta layers}, the group $\overline{N}$ is $\gen{\overline{\beta}}$-stable and, being central in $\overline{F}$, it is also abelian. 
Write now $B=\gen{\beta}$ and let $\sigma:B\rightarrow\graffe{\pm 1}$ be the isomorphism mapping $\beta$ to $-1$. By Lemma \ref{free 2}, the group $B$ acts on $F/F_2$ through $\sigma$ and, by Lemma \ref{beta layers}, the induced action of $B$ on $L/E$ is through $\sigma$. As a consequence, the induced action of $B$ on both $L/N$ and $N/E$ is through $\sigma$.  
It follows from Lemmas \ref{nondeg with DN} and \ref{lemma product of characters} that $\beta$ induces the identity map on $\overline{E}$ and so, thanks to Theorem \ref{lambda mu}, the subgroup $\overline{E}$ has a unique $\gen{\overline{\beta}}$-stable complement in $\overline{N}$, which coincides with $\overline{N}^{\,-}$. 
\end{proof}

\begin{lemma}\label{injection34}
The map $\cor{N}_4\rightarrow\cor{N}_3$ that is defined by $M\mapsto ME$ is an injection respecting the natural actions of $\Aut(F)$. 
\end{lemma}

\begin{proof}
The map $\cor{N}_4\rightarrow\cor{N}_3$ is well-defined, thanks to Lemma \ref{Dker}, and it is clear that it respects the action of $\Aut(F)$. We prove injectivity. 
To this end, let $M_1$ and $M_2$ be elements of $\cor{N}_4$ such that $M_1E=M_2E$ and set $N=M_1E=M_2E$. Since $M_1$ and $M_2$ belong to $\cor{N}_4$, Lemma \ref{n4 betastable} yields $\beta(M_1)=M_1$ and $\beta(M_2)=M_2$. 
It follows then from Lemma \ref{complementE} that both $\overline{M_1}$  and
$\overline{M_2}$ are the unique $\gen{\overline{\beta}}$-stable complement of $\overline{E}$ and so $M_1=M_2$.
\end{proof}

\begin{lemma}\label{bijection34}
The map $\cor{N}_4\rightarrow\cor{N}_3$ that is defined by $M\mapsto ME$ is a bijection respecting the natural actions of $\Aut(F)$. 
\end{lemma}

\begin{proof}
The map $\cor{N}_4\rightarrow\cor{N}_3$ is well-defined, injective, and respects the action of $\Aut(F)$ thanks to Lemma \ref{injection34}. We prove surjectivity. 
To this end, let $N$ be an element of $\cor{N}_3$. Write $\overline{F}=F/D_N$ and use the bar notation for the subgroups of $\overline{F}$. Let moreover $\overline{N}^{\,-}$ be as in Lemma \ref{complementE}.
As a consequence of the definition of $D_N$, the subgroup $\overline{N}$ is central in $\overline{F}$ and so $\overline{N}^{\,-}$ is normal in $\overline{F}$. Let $M$ be the unique normal subgroup of $F$ containing $D_N$ such that $\overline{M}=\overline{N}^{\,-}$. Then, as a consequence of Lemma \ref{complementE}, one has $N=ME$.
Write $H=F/M$ and denote by $\pi$ the canonical projection $F\rightarrow H$.
We will prove that 
$M\in\cor{N}_4$.
Thanks to the isomorphism theorems, the groups $\pi(N)$ and $\overline{E}$ are naturally isomorphic and, by Lemma \ref{orderE}, the group $\overline{E}$ has order $3$. It follows that $|N:M|=3$. Moreover, the group $N$ being an element of $\cor{N}_3$, the quotient $F/N$ has class $3$ and so $M\subseteq\pi^{-1}(H_4)\subseteq N$. Only two cases can occur: either $N=\pi^{-1}(H_4)$ or $M=\pi^{-1}(H_4)$. Assume by contradiction that $M=\pi^{-1}(H_4)$ and so that $H$ has class $3$. Since $H/\pi(N)$ is isomorphic to $F/N$, Lemma \ref{description L} yields that $\pi(L)=H_3\pi(N)$ and, since $N$ is central modulo $D_N$, the subgroup $\pi(N)$ is central in $H$.
It follows that 
\[
\pi([F,L])=[\pi(F),\pi(L)]=[H,H_3\pi(N)]=[H,H_3]=H_4=\graffe{1}
\] 
and therefore $[F,L]$ is contained in $M$. 
We get then 
\[
E=[F,L]F_2^3\subseteq[F,L]F_2^3[F_2,F_2]\subseteq [F,L]D_N\subseteq M,
\]
and thus $N=ME=M$, which is a contradiction.
We have proven that $N=\pi^{-1}(H_4)$, from which it follows in particular that $|H_4|=|N:M|=3$ and so $H$ has class $4$. 
Moreover, $H$ is a $\kappa$-group, because $F/N$ is.
To prove that $M$ belongs to $\cor{N}_4$, we are left with proving that $H$ has an automorphism of order $2$ that induces the inversion map on $H/H_2$ and in fact such an automorphism can be gotten, for example, by inducing $\beta$ to $H$.
We have proven that $M\in\cor{N}_4$ and so, the choice of $N$ being arbitrary, the map $\cor{N}_3\rightarrow\cor{N}_4$ is surjective.
\end{proof}

\begin{corollary}\label{cor n4 has3elemn}
The set $\cor{N}_4$ has $3$ elements and the action of $\Aut(F)$ on $\cor{N}_4$ is transitive.
\end{corollary}

\begin{proof}
Combine Proposition \ref{transitive action N} and Lemma \ref{bijection34}.
\end{proof}

\noindent
We are now ready to prove Proposition \ref{proposition unique with automorphism}. 
By Lemma \ref{n4}, there exists an element $M$ of $\cor{N}_4$ with the property that $F/M$ is isomorphic to $\yo$. 
As a consequence of Corollary \ref{cor n4 has3elemn}, 
the natural action of $\Aut(F)$ on $\cor{N}_4$ is transitive and therefore $G$ and $\yo$ are isomorphic. The proof of Proposition \ref{proposition unique with automorphism} is complete.

\section{Intensity}\label{section 3-gps sufficient}

\noindent
In Section \ref{section extension} we have proven Corollary \ref{3-gps class at most 4}, which asserts that finite $3$-groups of intensity larger than $1$ have class at most $4$. We will prove in this section that the bound is best possible by showing that the group $\yo$, introduced at the beginning of this chapter and whose structure we investigated in Section \ref{section yo}, has intensity $2$. Thanks to results coming from the previous sections, we will, at the end of the current section, finally be able to give the proof of Theorem \ref{theorem 3-groups}.

\begin{proposition}\label{ppp}
The group $\yo$ has intensity $2$.
\end{proposition}

\noindent
We will devote a big part of the present section to the proof of Proposition \ref{ppp}. To this end, let the following assumptions hold until the end of Section \ref{section 3-gps sufficient}. Set $G=\yo$ and denote by $(G_i)_{i\geq 1}$ its lower central series. For all $i\in\Z_{\geq 1}$, write $\wt_G(i)=w_i$.
By Lemma \ref{yo class}, the group $G$ has class $4$ and order $729$. Moreover, thanks to Lemmas \ref{yo structure}($1$) and \ref{yo kappa}, the group $G$ is a $\kappa$-group satisfying $(w_1,w_2,w_3,w_4)=(2,1,2,1)$.
Let $\alpha$ be as in Lemma \ref{yo inversion} and set $A=\gen{\alpha}$.
In concordance with the notation from Section \ref{section involutions},
we define $G^+=\graffe{x\in G\ :\ \alpha(x)=x}$ and $G^-=\graffe{x\in G\ :\ \alpha(x)=x^{-1}}$. Moreover, for each subgroup $H$ of $G$, we denote $H^+=H\cap G^+$ and 
$H^-=H\cap G^-$.

\begin{lemma}\label{normalizer G2}
Let $H$ be a subgroup of $G_2$ and let $g$ be an element of $G$. Then the following hold. 
\begin{itemize}
 \item[$1$.] The group $G_2$ normalizes $H$.
 \item[$2$.] If both $H$ and $gHg^{-1}$ are $A$-stable, then $gHg^{-1}=H$.
\end{itemize}  
\end{lemma}

\begin{proof}
The group $G_2$ is abelian, by Corollary \ref{yoG2}, and in particular it normalizes each of its subgroups. As a consequence of Lemma \ref{order +- jumps}($1$), the subgroup $G^+$ is contained in $G_2$ and we conclude combining ($1$) with Lemma \ref{conjugate stable iff in nor-times-G+}.
\end{proof}

\begin{lemma}\label{base induzione 3}
Let $H$ be a subgroup of $G$ that contains $G_4$. Then there exists $g\in G$ such that $gHg^{-1}$ is $A$-stable.
\end{lemma}

\begin{proof}
We denote by $\alpha_4$ the automorphism of $G/G_4$ that is induced by $\alpha$. 
By Proposition \ref{p^4 has cp}, the automorphism $\alpha_4$ is intense so, by Lemma \ref{equivalent intense coprime-pgrps}, there exists $g\in G$ such that 
$gHg^{-1}/G_4$ is $\gen{\alpha_4}$-stable. It follows from the definition of $\alpha_4$ that $gHg^{-1}$ is $A$-stable.
\end{proof}

\noindent
We recall that a positive integer $j$ is a jump of a subgroup $H$ of $G$ if and only if $H\cap G_j\neq H\cap G_{j+1}$. For the theory of jumps we refer to Section \ref{section jumps}.

\begin{lemma}\label{HHHH}
Let $H$ be a subgroup of $G$ such that $H\cap G_4=\graffe{1}$. Assume moreover that $H$ is not contained in $G_2$.
Then there exists $x\in G\setminus G_2$ such that $H=\gen{x}$.
\end{lemma}

\begin{proof}
The subgroup $H$ is different from $G$, since $H\cap G_4=\graffe{1}$, and it is therefore contained in a maximal subgroup $C$ of $G$. Moreover, $H$ not being contained in $G_2$, we have $\wt_H^G(1)=1$. We first show that $H$ is abelian. The subgroup $[H,H]$ is contained in $[C,C]$ and $[C,C]=[C,G_2]$, as a consequence of Lemma \ref{cyclic quotient commutators}. Thanks to Lemma \ref{commutator indices}, the subgroup $[H,H]$ is contained in $G_3$. By Lemma \ref{centro 3}, the centre of $G$ is equal to $G_4$ and so, by Lemma \ref{class 4 comm map non-deg}, the map 
$\gamma:G/G_2\times G_3/G_4\rightarrow G_4$ that is induced from the commutator map is non-degenerate.
Since $C=HG_2$, we get $[C,[H,H]]=[H,[H,H]]\subseteq H\cap G_4$ and so, since $H\cap G_4=\graffe{1}$, the subgroup $[H,H]$ is contained in $\ZG(C)$. It follows that $[H,H]$ is contained in $\ZG(C)\cap[C,C]\cap H$ and so, thanks to Corollary \ref{yoG2} and Lemma \ref{properties C}, the commutator subgroup of $H$ is trivial.
The group $H$ being abelian, it follows, from the non-degeneracy of $\gamma$, that $\wt_H^G(3)\leq 1$ and, from Lemma \ref{commumino}, that $\wt_H^G(2)=0$.
Let now $x$ be an element of $H\setminus G_2$. Then $1$ is a jump of $\gen{x}$ in $G$ and, the group $G$ being a $\kappa$-group, it follows that $x^3\in G_3\setminus G_4$. As a consequence of Lemma \ref{order product orders jumps}, we get
\[
|\gen{x}|\geq 3^{\wt_{\gen{x}}^G(1)}3^{\wt_{\gen{x}}^G(3)}\geq 9\geq
3^{\wt_{H}^G(1)}3^{\wt_{H}^G(3)}=\prod_{i=1}^43^{\wt_{H}^G(i)}=|H|
\]
and therefore $H$ is cyclic generated by $x$.
\end{proof}

\begin{lemma}\label{lH1}
Let $H$ be a subgroup of $G$ such that $H\cap G_4=\graffe{1}$. Assume that $H$ is not contained in $G_2$. Then $H$ and $\alpha(H)$ are conjugate in $G$.
\end{lemma}

\begin{proof}
By Lemma \ref{HHHH}, the group $H$ is cyclic. We define $T=H\oplus G_4$ so, by Lemma \ref{base induzione 3}, there exists $g\in G$ such that $gTg^{-1}$ is $A$-stable. We fix such $g$ and denote $T'=gTg^{-1}$ and $H'=gHg^{-1}$. The subgroup $G_4$ being characteristic, it follows that $H'\oplus G_4=T'=\alpha(H')\oplus G_4$. Let $\cor{C}$ denote the collection of complements of $G_4$ in $T'$.
By Lemma \ref{graph+complements}, the elements of $\cor{C}$ are in bijection with the elements of $\Hom(H',G_4)$, which is naturally isomorphic to $\Hom(H'/\Phi(H'), G_4)$, because $G_4$ has order $3$. 
The group $H'$ is cyclic, so $\Phi(H')=(H')^3$ and, the group $G$ being a $\kappa$-group, one gets 
$\Phi(H')=H'\cap G_2$. 
By Lemma \ref{frattini comm}, the quotient $G/G_2$ is elementary abelian and the restriction map $\Hom(G/G_2,G_4)\rightarrow\Hom(H'G_2/G_2,G_4)$ is thus surjective. 
Moreover, by Lemma \ref{centro 3}, the subgroup $G_4$ coincides with $\ZG(G)$ so, as a consequence of Lemma \ref{class 4 comm map non-deg}, the map $G_3/G_4\rightarrow\Hom(G/G_2,G_4)$, defined by $xG_4\mapsto(tG_2\mapsto [x,t])$, is an isomorphism. 
It follows from Lemma \ref{graph+complements} that, for each $K\in\mathcal{C}$, there exists $x\in G$ such that $K=\graffe{[x,t]t=xtx^{-1}\ |\ t\in H'}$. As a consequence, there is $x\in G$ such that $\alpha(H')=xH'x^{-1}$ and so, since $H'=gHg^{-1}$, also $\alpha(H)$ and $H$ are conjugate in $G$.
\end{proof}

\begin{lemma}\label{lH2}
Let $H$ be a subgroup of $G_3$ such that $H\cap G_4=\graffe{1}$. Then $H$ and $\alpha(H)$ are conjugate in $G$.
\end{lemma}

\begin{proof}
Let $T=HG_4$.
The group $G_3$ is elementary abelian, as a consequence of Corollary \ref{yoG2}, and therefore so is $T=H\oplus G_4$.  
Let $g\in G$ be such that $gTg^{-1}$ is $A$-stable, as in Lemma \ref{base induzione 3}, and set $T'=gTg^{-1}$ and $H'=gHg^{-1}$. 
Let moreover $\cor{C}$ be the set of complements of $G_4$ in $T'$ and note that, $G_4$ being characteristic, both $H'$ and $\alpha(H')$ belong to $\cor{C}$. 
Thanks to Lemma \ref{graph+complements}, the elements of $\cor{C}$ are in bijection with the elements of $\Hom(H',G_4)$. 
By the isomorphism theorems, $H'$ is isomorphic to $(H'G_4)/G_4$ and the restriction map induces a surjection $\Hom(G_3/G_4,G_4)\rightarrow\Hom(H',G_4)$. By Lemma \ref{centro 3}, the subgroup $G_4$ coincides with $\ZG(G)$ so, as a consequence of Lemma \ref{class 4 comm map non-deg}, the map $G/G_2\rightarrow\Hom(G_3/G_4,G_4)$, defined by $xG_4\mapsto(tG_2\mapsto [x,t])$ is an isomorphism. It follows from Lemma \ref{graph+complements} that each element of $\mathcal{C}$ is of the form $\graffe{[x,t]t=xtx^{-1}\ |\ t\in H'}=xH'x^{-1}$, for some $x\in G$. 
In particular, $\alpha(H')$ and $H'$ are conjugate in $G$.
The groups $H'$ and $H$ being conjugate in $G$, it follows that $H$ and $\alpha(H)$ are conjugate, too. 
\end{proof}

\begin{lemma}\label{lH3}
Let $H$ be a subgroup of $G$ such that $H\oplus G_4= G_2$. Then $H$ has an $A$-stable conjugate in $G$.
\end{lemma}

\begin{proof}
We define $X$ to be the collection of all subgroups $K$ of $G$ for which 
$G_2=K\oplus G_4$.
The group $G_2$ is elementary abelian, by Corollary \ref{yoG2}, so the set $X$ is non-empty.
Moreover, as a consequence of Lemma \ref{graph+complements}, the cardinality of $X$ is equal to the cardinality of 
$\Hom(H,G_4)$, which is $27$. 
We define $X^+=\graffe{K\in X\ :\ \alpha(K)=K}$ and we will show, with a counting argument, that $H$ is conjugate to an element of $X^+$. Let $K\in X^+$. Thanks to Corollary \ref{abelian sum +-}, we can write $K=K^+\oplus K^-$ and, as a consequence of Lemma \ref{order +- jumps}, the subgroup $K^-$ is equal to $G_2^-$. 
Again by Lemma \ref{order +- jumps}, the subgroup  $G_2^+$ has order $9$ and it contains $G_4$.
It follows that $|X^+|$ is equal to the number of $1$-dimensional subspaces of $G_2^+$ that are different from $G_4$, i.e. $|X^+|=3$. By Lemma \ref{normalizer G2}($1$), the group $G_2$ normalizes $K$, but in fact $G_2=\nor_G(K)$, as a consequence of Lemma \ref{class 4 comm map non-deg}. 
It follows that the orbit of $K$ in $X$ has size $|G:G_2|=9$ so, thanks to Lemma \ref{normalizer G2}($2$), the element $K$ is the only element of $X^+$ belonging to its orbit under $G/G_2$. The number $|X|/|X^+|$ being equal to $9$, it follows that each orbit of the action of $G/G_2$ on $X$ has a representative in $X^+$. The same holds for the orbit of $H$.
\end{proof}

\begin{lemma}\label{lH4}
Let $H$ be a subgroup of $G$ such that $H\oplus G_3=G_2$. Then $H$ has an $A$-stable conjugate in $G$.
\end{lemma}

\begin{proof}
Let $X$ be the collection of all complements of $G_3$ in $G_2$. The group $G_2$ is elementary abelian, by Corollary \ref{yoG2}, and so, from Lemma \ref{graph+complements} it follows that the cardinality of $X$ is equal to 
$|\Hom(H,G_3)|=27$. We define $X^+=\graffe{K\in X\ :\ \alpha(K)=K}$. As a consequence of Lemma \ref{order +- jumps}, if $K$ is an element of $X^+$, then $K=K^+$. The elements of $X^+$ are thus exactly the one-dimensional subspaces of $G_2^+$ that are different from $G_4$ and so we have that $|X^+|=3$. Fix $K\in X^+$. 
Then, as a consequence of Lemma \ref{commumino}, the commutator map induces an isomorphism $G/G_2\otimes K\rightarrow G_3/G_4$. It follows that $\nor_G(K)$ is contained in $G_2$ so, thanks to Lemma \ref{normalizer G2}($1$), one has $\nor_G(K)=G_2$. Lemma \ref{normalizer G2}($2$) yields that $K$ is the only element of $X^+$ belonging to its orbit under the action of $G/G_2$ on $X$. The number 
$|X|/|X^+|$ being equal to $9$, it follows that each orbit of the action of $G/G_2$ on $X$ has a representative in $X^+$ so, in particular, $H$ has an $A$-stable conjugate in $G$.
\end{proof}

\begin{lemma}\label{giovy}
Let $x\in G\setminus G_2$ and let $a\in G_2^+\setminus G_3$. 
Then $[x,a]$ does not belong to $G^-$.
\end{lemma}

\begin{proof}
Let $C=\Cyc_G([x,a])$ and $D=\gen{x,G_2}$. 
By Lemma \ref{commumino}, the element $[x,a]$ belongs to $G_3\setminus G_4$ so, as a consequence of Lemmas \ref{centro 3} and \ref{class 4 comm map non-deg}, the index of $C$ in $G$ is equal to $3$. In particular, both $C$ and $D$ are maximal subgroups of $G$.
Assume now by contradiction that $[x,a]\in G^-$. Since $x$ belongs to $G\setminus G_2$, there exists $\gamma\in G_2$ such that $\alpha(x)=x^{-1}\gamma$ and so we have
$$[x,a]^{-1}=\alpha([x,a])=[x^{-1}\gamma,a].$$
The group $G_2$ is elementary abelian, by Corollary \ref{yoG2}, and therefore, applying Lemma \ref{multiplication formulas commutators}($2$), one gets
\begin{align*}
[x,a]^{-1} & = [x^{-1}\gamma,a] \\
           & = x^{-1}[\gamma,a]x[x^{-1},a] \\
           & = [x^{-1},a] \\
           & = x^{-1}axa^{-1} \\
           & = x^{-1}[a,x]x \\
           & = [x^{-1},[a,x]][a,x] \\
           & = [x^{-1},[x,a]^{-1}][x,a]^{-1}.
\end{align*}
As a result, the element $[x^{-1},[x,a]^{-1}]$ is trivial, and so $x\in C$. It follows that $C=D$, and thus
$[x,a]$ belongs to $[C,C]\cap\ZG(C)$. Lemma \ref{properties C} yields $[x,a]\in G_4$. Contradiction.
\end{proof}

\begin{lemma}\label{lH5}
Let $x\in G_2\setminus G_3$ and $y\in G_3\setminus G_4$. Define $H=\gen{x,y}$. Then $H$ has an $A$-stable conjugate.
\end{lemma}

\begin{proof}
The group $G_2$ is elementary abelian, by Corollary \ref{yoG2}, and therefore $H=\gen{x}\oplus\gen{y}$.
Let $X$ be the set consisting of all subgroups of $G_2$ of the form $\gen{u}\oplus\gen{v}$, where $u\in G_2\setminus G_3$ and 
$v\in G_3\setminus G_4$. The cardinality of $X$ is then equal to $108$. We define $X^+=\graffe{K\in X\ :\alpha(K)=K}$ and we fix $K\in X^+$. By Corollary \ref{abelian sum +-}, the subgroup $K$ decomposes as 
$K=K^+\oplus K^-$ and, as a consequence of Lemma \ref{order +- jumps}, there exists $a\in G_2^+$ such that $K=\gen{a}\oplus K^-$. Fix such $a$.
Again thanks to Lemma \ref{order +- jumps}, we get that 
$|X^+|=12$. 
We want to count the conjugates of $K$. By Lemma \ref{normalizer G2}($1$), the subgroup $G_2$ is contained in $\nor_G(K)$ and, if $x\in\nor_G(K)$, 
then $[x,a]\in K\cap G_3$. The intersection $K\cap G_3$ being equal to $K^-$, it follows from Lemma \ref{giovy} that $\nor_G(K)=G_2$. 
As a consequence of Lemma \ref{normalizer G2}($2$), the element $K$ is the only element of $X^+$ belonging to its orbit under $G/G_2$ so, from the equality $|X|/|X^+|=9$, we can deduce that each orbit of the action of $G/G_2$ on $X$ has a representative in $X^+$. The same holds for the orbit of $H$.
\end{proof}

\begin{lemma}\label{malditesta}
The automorphism $\alpha$ is intense. Moreover, the intensity of $G$ is equal to $2$.
\end{lemma}

\begin{proof}
Let $H$ be a subgroup of $G$. If $H$ contains $G_4$, then, by Lemma \ref{base induzione 3}, there exists a conjugate of $H$ that is $A$-stable.
Assume that $H\cap G_4=\graffe{1}$. If $H$ is not contained in $G_2$, then $H$ is conjugate to $\alpha(H)$, thanks to Lemma \ref{lH1}.
Assume that $H$ is contained in $G_2$. If $2$ is not a jump of $H$ in $G$, then, by Lemma \ref{lH2}, the subgroups $H$ and $\alpha(H)$ are conjugate in $G$.  
We suppose that $2$ is a jump of $H$ in $G$.
By Corollary \ref{yoG2}, the group $G_2$ is elementary abelian and so $H$ is a subspace of $G_2$, not contained in $G_3$, that trivially intersects $G_4$. 
The combination of Lemmas \ref{lH3}, \ref{lH4}, and \ref{lH5} guarantees that $H$ has an $A$-stable conjugate in $G$.
The choice of $H$ being arbitrary, it follows from Lemma \ref{equivalent intense coprime-pgrps} that $\alpha$ is intense.
The intensity of $G$ is at least $2$, because $\alpha$ has order $2$, but in fact $\inte(G)=2$, as a consequence of Theorem \ref{theorem class at least 3}($1$).
\end{proof}

\noindent
We remark that, thanks to Lemma \ref{malditesta}, the proof of Proposition \ref{ppp} is complete. 
Moreover, we are now also able to prove Theorem \ref{theorem 3-groups}. 
The implication $(2)\Rightarrow(1)$ is clear and the implication $(3)\Rightarrow(2)$ is given by the combination of Proposition \ref{ppp} and Lemma \ref{yo class}. We now prove $(1)\Rightarrow(2)$. To this end, let $Q$ be a finite $3$-group of class at least $4$ with $\inte(Q)>1$.
Because of Corollary \ref{3-gps class at most 4}, the class of $Q$ is equal to $4$ so, as a consequence of Theorem \ref{theorem dimensions}, the order of $Q$ is equal to $729$. The intensity of $Q$ is equal to $2$, thanks to Theorem \ref{theorem class at least 3}($1$). We have concluded the proof of $(1)\Rightarrow(2)$ and, to finish the proof of Theorem \ref{theorem 3-groups}, we will next prove $(2)\Rightarrow(3)$. Let $Q$ be a finite $3$-group of class $4$ and intensity $2$. Then, by Lemma \ref{quotient kappa}, the group $Q$ is a $\kappa$-group and, as a consequence of Proposition \ref{proposition -1^i}, it possesses an automorphism of order $2$ that induces the inversion map on $Q/Q_2$. By Theorem \ref{theorem dimensions}, the order of $Q_4$ is $3$. Proposition \ref{proposition unique with automorphism} yields that $Q$ is isomorphic to $\yo$. The proof of Theorem \ref{theorem 3-groups} is now complete.





\chapter{Obelisks}\label{chapter obelisks}

Let $p>3$ be a prime number.
A \emph{$p$-obelisk} \index{$p$-obelisk} is a finite $p$-group $G$ for which the following hold.
\begin{itemize}
 \item[$1$.] The group $G$ is not abelian.
 \item[$2$.] One has $|G:G_3|=p^3$ and $G_3=G^p$.
\end{itemize}
The following proposition will immediately clarify our interest in $p$-obelisks.

\begin{proposition}\label{class 4 obelisk}
Let $p>3$ be a prime number and let $G$ be a finite $p$-group of class at least $4$. If $\inte(G)>1$, then $G$ is a $p$-obelisk.
\end{proposition}

\begin{proof}
Combine Theorems \ref{theorem dimensions} and \ref{theorem exp not p}.
\end{proof}

\noindent
Chapter \ref{chapter obelisks} will be entirely devoted to understanding the structure of $p$-obelisks and that of their subgroups. Some of the results, especially coming from Section \ref{section subgroups obelisks}, are rather technical and their relevance will become evident in Chapter \ref{chapter complete char}.

\section{Some properties}

\noindent
We remind the reader that, if $p$ is a prime number and $G$ is a finite $p$-group, then $\wt_G(i)=\log_p|G_i:G_{i+1}|$ where $(G_i)_{i\geq 1}$ denotes the lower central series of $G$. 

\begin{lemma}\label{obelisk basic}
Let $p>3$ be a prime number and let $G$ be a $p$-obelisk. Let $(G_i)_{i\geq 1}$ denote the lower central series of $G$. Then the following hold.
\begin{itemize}
 \item[$1$.] The class of $G$ is at least $2$.
 \item[$2$.] One has $\wt_G(1)=2$ and $\wt_G(2)=1$.
 \item[$3$.] The group $G/G_3$ is extraspecial of exponent $p$.
\end{itemize}
\end{lemma}

\begin{proof}
The group $G$ is non-abelian and thus $G_2\neq G_3$. The index $|G:G_3|$ being equal to $p^3$, it follows from Lemma \ref{index G'} that $\wt_G(1)=2$ and $\wt_G(2)=1$. We denote now $\overline{G}=G/G_3$ and use the bar notation for the subgroups of $\overline{G}$. Then $\overline{G_2}$ is contained in $\ZG(\overline{G})$ and $\overline{G_2}=\ZG(\overline{G})$, as a consequence of Lemma \ref{quotient by centre not cyclic}. The exponent of $\overline{G}$ is $p$, because $G^p$ is contained in $G_3$.
\end{proof}

\begin{lemma}\label{obelisk regular}
Let $p>3$ and let $G$ be a $p$-obelisk. Then $G$ is regular.
\end{lemma}

\begin{proof}
This follows directly from Lemma \ref{regular if omega small}.
\end{proof}

\begin{proposition}\label{blackburn 2-1-2}
Let $p>3$ be a prime number and let $G$ be a $p$-obelisk. Let $(G_i)_{i\geq1}$ be the lower central series of $G$ and let $c$ denote the class of $G$.
Then the following hold.
\begin{itemize}
 \item[$1$.] For all $i\in\Z_{\geq 1}$, one has $\wt_G(i)\wt_G(i+1)\leq 2$.
 \item[$2$.] If $\wt_G(i)\wt_G(i+1)=1$, then $i=c-1$.
 \item[$3$.] For all positive integers $k$ and $l$, not both even, one has
 			 $[G_k,G_l]=G_{k+l}$.
\end{itemize}
\end{proposition}

\begin{proof}
Proposition \ref{blackburn 2-1-2} is a simplified version of Theorem $4.3$ from \cite{blackburn}, which can also be found in Chapter $3$ of \cite{huppert} as Satz $17.9$.
\end{proof}

\noindent
We remark that the term \emph{$p$-obelisk} does not appear in \cite{blackburn} or \cite{huppert} and is of our own invention.
Moreover, originally Proposition \ref{blackburn 2-1-2}($1$-$2$) was phrased in the following way: if $G$ is a $p$-obelisk, then 
$$(\wt_G(i))_{i\geq 1}=(2,1,2,1,\ldots,2,1,f,0,0,\ldots) \ \ \text{where} \ \ f\in\graffe{0,1,2}.$$

\begin{lemma}\label{parity obelisks}
Let $p>3$ be a prime number and let $G$ be a $p$-obelisk. Let $c$ denote the class of $G$ and let $i\in\graffe{1,\ldots,c-1}$. Then the following hold.
\begin{itemize}
 \item[$1$.] The index $i$ is odd if and only if $\wt_G(i)=2$.
 \item[$2$.] The index $i$ is even if and only if $\wt_G(i)=1$.
 \item[$3$.] If $\wt_G(c)=2$, then $c$ is odd.
 \item[$4$.] If $c$ is even, then $\wt_G(c)=1$.
\end{itemize}
\end{lemma}

\begin{proof}
For all $j\in\graffe{1,\ldots,c-1}$, denote $w_j=\wt_G(j)$.
Thanks to Lemma \ref{obelisk basic}, we have $w_1=2$ and $w_2=1$.
As a consequence of Proposition \ref{blackburn 2-1-2}, whenever $i<c-1$, the product $w_iw_{i+1}$ is equal to $2$ and, for all indices $i,j\in\graffe{1,\ldots,c-1}$, one has $w_i=w_j$ if and only if $i$ and $j$ have the same parity. It follows from Proposition \ref{blackburn 2-1-2}($1$) that $\wt_G(c)$ can be $2$ only if $c$ is odd.
\end{proof}

\noindent
We recall that, if $G$ is a $p$-group, then $\rho$ denotes the map $x\mapsto x^p$ on $G$.

\begin{lemma}\label{black core}\label{black p-power jumps 2}
Let $p>3$ be a prime number and let $G$ be a $p$-obelisk.
Then, for all $i,k\in\Z_{>0}$, one has $\rho^{k}(G_i)=G_{2k+i}$. 
\end{lemma}

\noindent
In his original proof of Proposition \ref{blackburn 2-1-2}, Blackburn also proves  Lemma \ref{black p-power jumps 2}. Blackburn's proof strongly relies on the fact that $p$-obelisks are regular and it makes use of some technical lemmas that can be found in \cite[Ch.~$\mathrm{III}$]{huppert}.

\begin{proposition}\label{obelisk centre at the bottom}
Let $p>3$ be a prime number and let $G$ be a $p$-obelisk. Let $(G_i)_{i\geq 1}$ be the lower central series of $G$ and let $c$ denote its nilpotency class. Then $\ZG(G)=G_c$. 
\end{proposition}

\begin{proof}
We work by induction on $c$. If $c=2$, then, by Lemma \ref{obelisk basic}($3$), the group $G$ is extraspecial so $G_2=\ZG(G)$. Assume now that $c>2$. 
The subgroup $G_c$ is central, because $G$ has class $c$, and, by the induction hypothesis, $\ZG(G/G_c)=G_{c-1}/G_c$. 
It follows that $G_c\subseteq\ZG(G)\subseteq G_{c-1}$ and $\ZG(G)\neq G_{c-1}$. 
Moreover, by Proposition \ref{blackburn 2-1-2}($1$), the width $\wt_G(c-1)$ is either $1$ or $2$. If $\wt_G(c-1)=1$, then $\ZG(G)=
G_c$; we assume thus that $\wt_G(c-1)=2$. 
By Lemma \ref{parity obelisks}($1$), there exists a positive integer $k$ such that $c-1=2k+1$ so, from Lemma \ref{black core}, we get $G_{c-1}=\rho^k(G)$ and $G_c=\rho^k(G_2)$.
As a consequence of Proposition \ref{blackburn 2-1-2}($1$), the subgroup $G_c$ has order $p$. Let us assume by contradiction that $\ZG(G)\neq G_c$, in other words $|G_{c-1}:\ZG(G)|=|\ZG(G):G_c|=p$. Let $N=\Cyc_G(G_{c-1})$. The commutator map $G/G_2\times G_{c-1}/G_c\rightarrow G_c$ is bilinear by Lemma \ref{bilinear LCS} and it factors as a surjective non-degenerate map 
$G/N\times G_{c-1}/\ZG(G)\rightarrow G_c$. It follows from Lemma \ref{non-degenerate dim 1} that $G/N$ is cyclic of order $p$ so, by Lemma \ref{cyclic quotient commutators}, one has $G_2=[N,G]$. Lemma \ref{regular technical} yields 
$$\rho^k([N,G])=[N,\rho^k(G)]=[N,G_{c-1}]=\graffe{1}$$ and so $G_c=\rho^k(G_2)=\graffe{1}$. Contradiction.
\end{proof}

\begin{lemma}\label{dedekind law}
Let $G$ be a group and let $N$ be a normal subgroup of $G$. Let moreover $H$ and $K$ be subgroups of $G$ such that $K\subseteq H$. 
Then one has $(H\cap N)K=(KN)\cap H$.
\end{lemma}

\begin{proof}
Easy exercise.
\end{proof}

\begin{lemma}\label{non-ab quotient obelisk}
Let $p>3$ be a prime number and let $G$ be a $p$-obelisk. 
Then each non-abelian quotient of $G$ is a $p$-obelisk.
\end{lemma}

\begin{proof}
Let $N$ be a normal subgroup of $G$ such that $G/N$ is not abelian. 
We claim that $N$ is contained in $G_3$. Denote first $H=G/N$.
Then we have 
$|H:H_2|\leq |G:G_2|$. Moreover, $H$ being non-abelian, Lemma \ref{index G'} yields 
$|H:H_2|\geq p^2$, and therefore, from Lemma \ref{obelisk basic}($2$), it follows that $N\subseteq G_2$.
If $N\cap G_3=N$, then $N$ is contained in $G_3$ and we are done. Assume by contradiction that $N\cap G_3\neq N$. 
As a consequence of Lemma \ref{obelisk basic}($2$), the subgroup $N$ does not contain $G_3$. 
Let now $M$ be a normal subgroup of $G$ such that $N \cap G_3\subseteq M\subseteq G_3$ and $|G_3:M|=p$, as given by Lemma \ref{normal index p}. Then $\overline{G}=G/M$ has class $3$ and $\overline{N}\neq \graffe{1}$. But, by Lemma \ref{class 3 G3=Z}, the centre of $\overline{G}$ is equal to $\overline{G_3}$ so, $\overline{G_3}$ having order $p$, Lemma \ref{normal intersection centre trivial} yields $\overline{G_3}\subseteq \overline{N}$. In particular, $G_3$ is contained in $MN$. 
Thanks to Lemma \ref{dedekind law}, we get that 
$G_3=(G_3\cap N)M=M$, 
which gives a contradiction. It follows that $N\subseteq G_3$, as claimed, and thus we have
$|H:H_3|=|G:G_3|$. It is moreover clear that $H^p=H_3$, and so we have proven that $H$ is a $p$-obelisk. 
\end{proof}

\begin{lemma}\label{obelisks normal squeezed}
Let $p>3$ be a prime number and let $G$ be a $p$-obelisk. 
Then the following hold.
\begin{itemize}
 \item[$1$.] If $H$ is a quotient of $G$ of class $i$, then $\ZG(H)=H_i$.
 \item[$2$.] Let $N$ be a subgroup of $G$. Then $N$ is normal in $G$ if and only if there exists $i\in\Z_{>0}$ such that $G_{i+1}\subseteq N\subseteq G_i$.
\end{itemize}
\end{lemma}

\begin{proof}
($1$) Let $H$ be a quotient of $G$ and let $i$ denote the class of $H$. Let moreover $c$ denote the class of $G$ and note that $i\leq c$.  
If $i=0$ or $i=1$, the group $H$ is abelian and $\ZG(H)=H$. Assume now that $i>1$. Then $H$ is a non-abelian quotient of a $p$-obelisk so, by Lemma \ref{non-ab quotient obelisk}, it is a $p$-obelisk itself.
To conclude, apply Proposition \ref{obelisk centre at the bottom}. For the proof of ($2$), we combine ($1$) with Lemma \ref{normali schiacchiati generale}.
\end{proof}

\section{Power maps and commutators}\label{section obelisk diagrams}

\noindent
Throughout Section \ref{section obelisk diagrams} we will faithfully follow the notation from the List of Symbols. In particular, if $p$ is a prime number and $G$ is a finite $p$-group, then $\rho$ denotes the map $G\rightarrow G$ that is defined by $x\mapsto x^p$. We remind the reader that $\rho$ is in general not a homomorphism.

\begin{lemma}\label{maps incrociate}
Let $p>3$ be a prime number and let $G$ be a $p$-obelisk. 
Then the following hold. 
\begin{itemize}
 \item[$1$.] For all $i,k\in\Z_{>0}$ the map $\rho^k:G_i\rightarrow G_i$ induces a 
 			 surjective homomorphism 
 			 $\rho^k_i:G_i/G_{i+1}\rightarrow G_{2k+i}/G_{2k+i+1}$.
 \item[$2$.] For all $h,k\in\Z_{>0}$ not both even, the commutator map induces a
 			 bilinear map
 			 $\gamma_{h,k}:G_h/G_{h+1}\times G_k/G_{k+1}\rightarrow
 			 G_{h+k}/G_{h+k+1}$ whose image generates $G_{h+k}/G_{h+k+1}$. 
\end{itemize}
\end{lemma}

\begin{proof}
($1$) Let $i$ and $k$ be positive integers and, without loss of generality, assume that $G_{2k+i+1}=\graffe{1}$. We work by induction on $k$ and we start by taking $k=1$. As a consequence of Lemma \ref{commutator indices}, the group $[G_i,G_i]$ is contained in $G_{2i}$ so, from Lemma \ref{black core}, it follows that $[G_i,G_i]^p$ is contained in $G_{2i+2}$. The index $i$ being positive, $G_{2i+2}$ is contained in $G_{i+3}=\graffe{1}$. Now, the prime $p$ is larger than $3$ so $G_{ip}$ is also contained in $G_{i+3}=\graffe{1}$. It follows from Lemma \ref{commutator indices}, that $(G_i)_p$ is contained in $G_{ip}$, and so, thanks to Corollary \ref{p map petrescu}, the map ${\rho:G_i\rightarrow G_i}$ is a homomorphism. The function $\rho$ factors as a surjective homomorphism $\rho^1_i:G_i/G_{i+1}\rightarrow G_{i+2}$, thanks to Lemma \ref{black core}. This finishes the proof for $k=1$. Assume now that $k>1$ and define $$\rho_i^k=\rho^1_{2k+i-1}\circ\rho^1_{2k+i-3}\circ \ldots \circ \rho_{i+2}^1 \circ \rho^1_i.$$
As a consequence of the base case, the map $\rho_i^k$ is a surjective homomorphism $\rho^k_i:G_i/G_{i+1}\rightarrow G_{2k+i}/G_{2k+i+1}$ and, by its definition, it is induced by $\rho^k$.
This proves ($1$). To prove ($2$) combine Proposition \ref{blackburn 2-1-2}($3$) with Lemma \ref{bilinear LCS hk}.
\end{proof}

\begin{corollary}\label{power iso bottom}\label{power iso}
Let $p>3$ be a prime number and let $G$ be a $p$-obelisk. Let $c$ be the class of $G$.
Let moreover $i$ and $j$ be integers of the same parity such that 
$1\leq i\leq j\leq c$ and one of the following holds.
\begin{itemize}
 \item[$1$.] The number $j$ is even.
 \item[$2$.] One has $\wt_G(j)=2$.
\end{itemize}
Define $m=\frac{j-i}{2}$.
Then the map $\rho^{m}:G_i\rightarrow G_i$ induces an isomorphism of groups $\rho^{m}_i:G_i/G_{i+1}\rightarrow G_{j}/G_{j+1}$.
\end{corollary}

\begin{proof}
By Lemma \ref{maps incrociate}($1$), the map $\rho^{m}:G_i\rightarrow G_i$ induces a surjective homomorphism $\rho^{m}_i:G_i/G_{i+1}\rightarrow G_{j}/G_{j+1}$.
Now, $i$ and $j$ having the same parity, it follows from Lemma \ref{parity obelisks} that $\wt_G(i)=\wt_G(j)$ and $\rho^m_i$ is a bijection.
\end{proof}

\begin{lemma}\label{obelisk non-deg}
Let $p>3$ be a prime number and let $G$ be a $p$-obelisk. Denote by $c$ the class of $G$. Let moreover $h$ and $k$ be positive integers, not both even, such that 
$h+k\leq c$. Assume additionally that, if $h+k$ is odd, then $\wt_G(h+k)=2$.
Then the map
$\gamma_{h,k}$ from Lemma \ref{maps incrociate} is non-degenerate.
\end{lemma}

\begin{proof}
Without loss of generality, assume that $c=h+k$ and so $G_{h+k+1}=\graffe{1}$.
We prove non-degeneracy of $\gamma_{h,k}$ by looking at the parity of $h+k$. 
Assume first that $h+k$ is odd and, without loss of generality, $h$ is odd and $k$ is even. From Lemma \ref{parity obelisks}, it follows that $\wt_G(h)=2$ and $\wt_G(k)=1$. 
Moreover, by assumption, $\wt_G(h+k)=2$.
Since the image of $\gamma_{h,k}$ generates $G_{h+k}$, the map $\gamma_{h,k}$ is non-degenerate.
Let now $h+k$ be even. The numbers $h$ and $k$ are both odd so $\wt_G(h)=\wt_G(k)=2$, by Lemma \ref{parity obelisks}($2$). 
Assume without loss of generality that $h\leq k$. Then, by Lemma \ref{black core}, the set $\rho^{\frac{k-h}{2}}(G_h)$ coincides with the subgroup $G_k$. 
Let now $C=\Cyc_{G_h}(G_k)$ and $D=\Cyc_{G_k}(G_h)$. Since $\gamma_{h,k}\neq 1$, Lemma \ref{commutator indices} yields that 
$G_{h+1}\subseteq C\subsetneq G_h$ and $G_{k+1}\subseteq D\subsetneq G_k$.
The commutator map induces a non-degenerate map $G_h/C\times G_k/D\rightarrow G_{h+k}$ so, $\wt_G(h+k)$ being equal to $1$, Lemma \ref{non-degenerate dim 1} yields that
$|G_h:C|=|G_k:D|$. Now, by Lemma \ref{obelisk regular}, the group $G$ is regular, and therefore so is $C$. Thanks to Lemma \ref{regular implies power abelian}($1$), the set $\rho^{\frac{k-h}{2}}(C)$ is a subgroup of $C$ and so, thanks to Lemma \ref{regular technical}, one has 
$[\rho^{\frac{k-h}{2}}(C),G_h]=[C,\rho^{\frac{k-h}{2}}(G_h)]=[C,G_k]=\graffe{1}$. In particular,  
$\rho^{\frac{k-h}{2}}(C)\subseteq D$. Since $|G_h:C|=|G_k:D|$ and $\wt_G(h)=\wt_G(k)=2$, we derive from Corollary \ref{power iso bottom} that $\rho^{\frac{k-h}{2}}(C)=D$.
Assume now by contradiction that there exists $x\in G_h$ such that $G_h=\gen{x,C}$. Then 
$G_k=\gen{\rho^{\frac{k-h}{2}}(x),D}$ and therefore, the commutator map being alternating, one has $G_{h+k}=[G_k,G_h]=\gen{[x,\rho^{\frac{k-h}{2}}(x)]}=\graffe{1}$. Contradiction to the class of $G$ being $h+k$. It follows that the quotient $G_h/C$ is not cyclic and so $C=G_{h+1}$ and $D=G_{k+1}$.
In particular, $\gamma_{h,k}$ is non-degenerate.
\end{proof}

\begin{corollary}\label{delta surjective}
Let $p>3$ be a prime number and let $G$ be a $p$-obelisk. Denote by $c$ the class of $G$. Let moreover $l\in\graffe{1, \ldots, c-1}$ be such that $c-l$ is odd.
Then the map $G_{c-l}/G_{c-l+1}\rightarrow\Hom(G_l/G_{l+1},G_c)$ that is defined by 
$$t\,G_{c-l+1}\ \ \mapsto \ \ (x\,G_{l+1}\mapsto [t,x])$$ 
is a surjective homomorphism of groups.
\end{corollary}

\begin{proof}
As a consequence of Lemma \ref{black core}, the groups $G_l/G_{l+1}$, $G_{c-l}/G_{c-l+1}$, and $G_c$ are elementary abelian and the map $\gamma_{c-l,l}$ from Lemma \ref{maps incrociate} is thus a bilinear map of $\F_p$-vector spaces. Respecting the notation from Section \ref{section linear algebra}, we define
$$\delta:G_{c-l}/G_{c-l+1}\rightarrow\Hom(G_l/G_{l+1},G_c)$$ to be the map sending each element 
$v\in G_{c-l}/G_{c-l+1}$ to ${_{v}}(\gamma_{c-l,l})$.
In other words, if $v=t\,G_{c-l+1}$, then $\delta(v):G_l/G_{l+1}\rightarrow G_c$ is
defined by $x\,G_{l+1}\mapsto [t,x]$. 
As a consequence of Lemma \ref{maps incrociate}($2$), the function $\delta$ is a homomorphism of groups and $\delta$ differs from the zero map. Let us now, for all $i\in\graffe{1,\ldots,c}$, denote $w_i=\wt_G(i)$. It follows that the dimension of $\Hom(G_l/G_{l+1},G_c)$ is equal to $w_lw_c$ and, if $w_lw_c=1$, then $\delta$ is surjective. We assume that $w_lw_c\neq 1$. The index $c-l$ being odd, it follows that either $l$ or $c$ is even. Proposition \ref{blackburn 2-1-2} yields $w_{c-l}=w_lw_c$ and, if $l$ is even, then $w_c=2$.    
As a consequence of Lemma \ref{obelisk non-deg}, the map $\delta$ is injective and so $\delta$ is also surjective. 
\end{proof}

\section{Framed obelisks}\label{section framed}

\noindent
Let $p>3$ be a prime number and let $G$ be a $p$-obelisk. Then $G$ is \emph{framed} if, for each maximal subgroup $M$ of $G$, one has $\Phi(M)=G_3$. \index{$p$-obelisk!framed}

\begin{lemma}\label{commutative (id,rho)}
Let $p>3$ be a prime number and let $G$ be a $p$-obelisk.
Let moreover $h,k\in\Z_{>0}$, with $h$ odd and $k$ even, and
$n\in\Z_{\geq 0}$. Then the following diagram is commutative.
\vspace{8pt}
\[
\begin{diagram}
G_h/G_{h+1}\times G_k/G_{k+1}  &\ \ \rTo^{\gamma_{h,k}}\ \ & G_{h+k}/G_{h+k+1}  \\
\dTo^{(\id_{G_h/G_{h+1}}\, ,\ \rho^n_k)} &   & \dTo_{\rho^n_{h+k}}\\
G_{h}/G_{h+1}\times G_{k+2n}/G_{k+2n+1} &\ 
 \rTo^{\gamma_{h,k+2n}}\ \ & G_{h+k+2n}/G_{h+k+2n+1}  \\
\end{diagram}
\]
\vspace{8pt}
\end{lemma}

\begin{proof}
The maps from the above diagram are defined in Lemma \ref{maps incrociate}. Assume without loss of generality that $G_{h+k+2n+1}=\graffe{1}$ so that $G_{h+k+2n}$ is central. The diagram is clearly commutative for $n=0$. We will prove the most delicate case, i.e. when $n=1$, and leave the general case to the reader. Set $n=1$. Let $(x,y)\in G_h\times G_k$. We will show, and that suffices, that 
$[x,y^{p}]=[x,y]^{p}$.  
Thanks to Lemma \ref{multiplication formulas commutators}($4$), one gets 
$$[x,y]^{-p}[x,y^p]=\prod_{r=1}^{p-1}[[y^r,x],y].$$
Applying Lemma \ref{commutator indices} twice, one gets that, for each index $r$, the element $[[y^r,x],y]$ belongs to $G_{h+2k}$, which is itself contained in the central subgroup $G_{h+k+2}$. Moreover, the group $G_{h+k}$ being central modulo $G_{h+k+1}$, Lemma \ref{tgt} yields $[y^r,x]\equiv [y,x]^r\bmod G_{h+k+1}$. Thanks to Lemma \ref{bilinear LCS hk}, the commutator map on $G$ induces a bilinear map $G_{h+k}/G_{h+k+1}\times G_{k}/G_{k+1}\rightarrow G_{h+2k}$ and therefore we get
$[[y^r,x],y]=[[y,x]^r,y]=[[y,x],y]^r$. It follows that
$$[x,y]^{-p}[x,y^p]=\prod_{r=1}^{p-1}[[y^r,x],y]=\prod_{r=1}^{p-1}[[y,x],y]^r=[[y,x],y]^{\binom{p}{2}}$$
and, the prime $p$ being larger than $2$, the number $\binom{p}{2}$ is a multiple of $p$. Since $[[y,x],y]$ belongs to $G_{h+k+2}$, it follows from Lemma \ref{black core} that $[x,y^p]=[x,y]^p$. This concludes the case $n=1$.
\end{proof}

\begin{lemma}\label{commutative (rho,id)}
Let $p>3$ be a prime number and let $G$ be a $p$-obelisk.
Let moreover $h,k\in\Z_{>0}$, with $h$ odd and $k$ even, and $m\in\Z_{\geq 0}$. Then the following diagram is commutative.
\vspace{8pt}
\[
\begin{diagram}
G_h/G_{h+1}\times G_k/G_{k+1}  &\ \ \rTo^{\gamma_{h,k}}\ \ & G_{h+k}/G_{h+k+1}  \\
\dTo^{(\rho^m_h\, ,\ \id_{G_k/G_{k+1}})} &   & \dTo_{\rho^m_{h+k}}\\
G_{h+2m}/G_{h+2m+1}\times G_k/G_{k+1} &\ 
 \rTo^{\gamma_{h+2m,k}}\ \ & G_{h+k+2m}/G_{h+k+2m+1}  \\
\end{diagram}
\]
\vspace{8pt}
\end{lemma}

\begin{proof}
The maps in the diagram are as in Lemma \ref{maps incrociate} and they are 
well-defined. Assume without loss of generality that $G_{h+k+2m+1}=\graffe{1}$ and 
so $G_{h+k+2m}$ is central.
Let $(x,y)\in G_h\times G_k$. The diagram is clearly commutative if $m=0$; we will prove commutativity when $m=1$, the most difficult case, and we will leave the general case to the reader.
Set $m=1$. We will prove, and that suffices, that $[x^p,y]=[x,y]^p$.
Applying Lemma \ref{multiplication formulas commutators}($3$) twice, we get
\begin{align*}
[x^p,y][x,y]^{-p} & = 
\prod_{s=1}^{p-1}[x,[x^{p-s},y]] \\
& \prod_{s=1}^{p-1}\Big{[}x\,,\,\Big{(}\,
\prod_{j=1}^{p-s-1}
[x\,,\,[x^{p-s-j},y]]\,\Big{)}[x,y]^{p-s}\Big{]}.
\end{align*}
Thanks to Lemma \ref{commutator indices}, each element $[x,[x^{p-s-j},y]]$ belongs to $G_{k+2h}$ and, the group $G_{h+k+3}$ being trivial, $G_{k+2h}$ centralizes $G_{h+k}$. Again by Lemma \ref{commutator indices}, for each index $s$, the element $[x,y]^{p-s}$ belongs to $G_{h+k}$ and applying
Lemma \ref{multiplication formulas commutators}($1$) twice yields  
\begin{align*}
[x^p,y][x,y]^{-p} & =
\prod_{s=1}^{p-1}\Big{[}x\,,\,\prod_{j=1}^{p-s-1}[x,[x^{p-s-j},y]]\Big{]}\prod_{s=1}^{p-1}[x,[x,y]^{p-s}] \\
& = \prod_{s=1}^{p-1}\Big{[}x\,,\,\prod_{j=1}^{p-s-1}[x,[x^{p-s-j},y]]\Big{]}\prod_{s=1}^{p-1}[x,[x,y]]^{p-s}.
\end{align*}
Thanks to Lemma \ref{bilinear LCS hk}, given any two positive integers $i$ and $j$, the commutator map induces a bilinear map $G_i/G_{i+1}\times G_j/G_{j+1}\rightarrow G_{i+j}/G_{i+j+1}$. By taking consecutively 
$(i,j)=(h,k)$ and $(i,j)=(h,h+k)$, we get respectively that 
$[x^{p-s-j},y]\equiv[x,y]^{p-s-j}\bmod G_{h+k+1}$ and so 
$$[x,[x^{p-s-j},y]]\equiv [x,[x,y]^{p-s-j}]\equiv [x,[x,y]]^{p-s-j}\bmod G_{2h+k+1}.$$ 
By taking $(i,j)=(h,2h+k)$, we derive that 
\begin{align*}
[x^p,y][x,y]^{-p} & =\prod_{s=1}^{p-1}\Big{[}x\,,\,\prod_{j=1}^{p-s-1}[x,[x,y]]^{p-s-j}\Big{]}\prod_{s=1}^{p-1}[x,[x,y]]^{p-s} \\
& = \prod_{s=1}^{p-1}\Big{[}x,[x,[x,y]]^{\binom{p-s}{2}}\Big{]}\prod_{s=1}^{p-1}[x,[x,y]]^{p-s} \\ 
& = \prod_{s=1}^{p-1}[x,[x,[x,y]]]^{\binom{p-s}{2}}\prod_{s=1}^{p-1}[x,[x,y]]^{p-s} 
= [x,[x,[x,y]]]^{\binom{p}{3}}[x,[x,y]]^{\binom{p}{2}}.
\end{align*}
The prime $p$ being larger than $3$, both $\binom{p}{2}$ and $\binom{p}{3}$  are multiples of $p$. As both $[x,[x,y]]$ and $[x,[x,[x,y]]]$ belong to $G_{h+k+1}$, it follows from Lemma \ref{black core} that $[x^p,y]=[x,y]^p$. This concludes the proof for $m=1$.
\end{proof}

\begin{proposition}\label{proposition commutative diagram}
Let $p>3$ be a prime number and let $G$ be a $p$-obelisk.
Let moreover $h,k\in\Z_{>0}$, with $h$ odd and $k$ even, and let $m,n\in\Z_{\geq 0}$. Then the following diagram is commutative. 
\vspace{8pt} \\
\makebox[\textwidth][l]{
\begin{diagram}
G_h/G_{h+1}\times G_k/G_{k+1}  & \rTo^{\gamma_{h,k}} & G_{h+k}/G_{h+k+1}  \\
\dTo^{(\rho^m_h\, ,\ \rho^n_k)} &   & \dTo_{\rho^{m+n}_{h+k}}\\
G_{h+2m}/G_{h+2m+1}\times G_{k+2n}/G_{k+2n+1} &
 \rTo^{\gamma_{h+2m,k+2n}}& G_{h+k+2(m+n)}/G_{h+k+2(m+n)+1}  \\
\end{diagram}
}
\vspace{8pt}
\end{proposition}

\begin{proof}
Combine Lemmas \ref{commutative (id,rho)} and Lemma \ref{commutative (rho,id)}.
\end{proof}

\begin{lemma}\label{maximal non-framed}
Let $p>3$ be a prime number and let $G$ be a $p$-obelisk of class at least $3$. Let moreover $M$ be a maximal subgroup of $G$. Then $[M,M]=[M,G_2]$ and, whenever $\wt_G(3)=2$, the following are equivalent. 
\begin{itemize}
 \item[$1$.] One has $\Phi(M)\neq G_3$.
 \item[$2$.] One has $[M,M]=M^p=\Phi(M)$.
\end{itemize}
\end{lemma}

\begin{proof}
The subgroups $M^p$ and $[M,M]$ are both characteristic in the normal subgroup $M$; thus both $M^p$ and $[M,M]$ are normal in $G$. 
By Lemma \ref{obelisk basic}($2$), the quotient $G/G_2$ has order $p^2$ and so $|G:M|=|M:G_2|=p$.
It follows from Lemma \ref{cyclic quotient commutators} that $[M,M]=[M,G_2]$ and so, as a consequence of Corollary \ref{power iso bottom} and Lemma \ref{obelisk non-deg}, the least jumps of $[M,M]$ and $M^p$ in $G$ are both equal to $3$ and of width $1$. In particular, $\Phi(M)$ is contained in $G_3$ and Lemma \ref{obelisks normal squeezed}($2$) yields 
$G_4\subseteq M^p\cap [M,M]$. 
If the third width of $G$ is equal to $2$, then it follows that $\Phi(M)\neq G_3$ if and only if $[M,M]=\Phi(M)=M^p$.
\end{proof}

\noindent
We remark that, as a consequence of Lemma \ref{black core}, quotients of consecutive elements of the lower central series of a $p$-obelisk are vector spaces over $\F_p$ and therefore, in ($2$) and ($3$) from Proposition \ref{lines}, it makes sense, for each positive integer $i$, to talk about subspaces of $G_i/G_{i+1}$.

\begin{proposition}\label{lines}
Let $p>3$ be a prime number and let $G$ be a $p$-obelisk. Then the following conditions are equivalent. 
\begin{itemize}
 \item[$1$.] The $p$-obelisk $G$ is framed.
 \item[$2$.] For each $1$-dimensional subspace $\ell$ of $G/G_2$, the quotient $G_3/G_4$ is generated by $\rho^1_1(\ell)$ and $\gamma_{1,2}(\graffe{\ell}\times G_2/G_3)$.
 \item[$3$.] For each $h,k\in\Z_{>0}$, with $h$ odd and $k$ even, and for each $1$-dimensional subspace $\ell$ in $G_h/G_{h+1}$, the spaces $\rho^{k/2}_h(\ell)$ and $\gamma_{h,k}(\graffe{\ell}\times G_k/G_{k+1})$ generate $G_{h+k}/G_{h+k+1}$.
\end{itemize}
\end{proposition}

\begin{proof}
$(1)\Leftrightarrow(2)$ Let $\pi:G\rightarrow G/G_2$ denote the natural projection. Then, through $\pi$, there is a bijection between the maximal subgroups of $G$ and the $1$-dimensional subspaces of $G/G_2$.
For any maximal subgroup $M$ of $G$, we know from Lemma \ref{maximal non-framed} that $[M,G_2]=[M,M]$ and therefore $(2)$ holds if and only if, given any maximal subgroup $M$ of $G$, one has $\Phi(M)G_4=G_3$. Lemma \ref{obelisks normal squeezed}($2$) yields that $(2)$ is satisfied if and only if, for any maximal subgroup $M$ of $G$, one has $\Phi(M)=G_3$.
We now deal with $(2)\Leftrightarrow(3)$. The implication $\Leftarrow$ is proven by taking $h=1$ and $k=2$, so we will prove that $(2)$ implies $(3)$. Let $\ell$ be a $1$-dimensional subspace of $G_h/G_{h+1}$. Define moreover $m=\frac{h-1}{2}$, $n=\frac{k-2}{2}$, and $S=m+n=\frac{h+k-3}{2}$. Thanks to Lemma \ref{maps incrociate}($1$), there exists a $1$-dimensional subspace $\ell'$ of $G/G_2$ such that $\rho^m_1(\ell')=\ell$ and, moreover, $\rho^n_2(G_2/G_3)=G_k/G_{k+1}$.
By assumption $G_3/G_4$ is generated by $\rho^1_1(\ell')$ and $\gamma_{1,2}(\graffe{\ell'}\times G_2/G_3)$, so it follows from Lemma \ref{maps incrociate}($1$) that $\rho^{S}_3(\rho^1_1(\ell'))$ and $\rho^{S}_3(\gamma_{1,2}(\graffe{\ell'}\times G_2/G_3))$ together span $G_{h+k}/G_{h+k+1}$. We now have 
\[
\rho^{S}_3(\rho^1_1(\ell'))=\rho^{S+1}_1(\ell')=\rho^{k/2}_h(\ell)
\] 
and, thanks to Proposition \ref{proposition commutative diagram}, we also have 
\[
\rho^{S}_3(\gamma_{1,2}(\graffe{\ell'}\times G_2/G_3))=
\gamma_{h,k}(\rho^m_1(\ell')\times\rho_2^n(G_2/G_3))=
\gamma_{h,k}(\graffe{\ell}\times G_k/G_{k+1}).
\]
This completes the proof.
\end{proof}

\section{Subgroups of obelisks}\label{section subgroups obelisks}

\noindent
The major goal of this section is to link structural properties of subgroups of a $p$-obelisk to the parities and widths of their jumps. The importance of Section \ref{section subgroups obelisks} will become clear in Chapter \ref{chapter bilbao}.

\begin{proposition}\label{obelisk in obelisk}
Let $p>3$ be a prime number and let $G$ be a $p$-obelisk.
Let $H$ be a subgroup of $G$ that is itself a $p$-obelisk. 
Then $H=G$.
\end{proposition}

\begin{proof}
The subgroup $H$ is non-abelian, by definition of a $p$-obelisk, and
it is in particular non-trivial. Let $l$ denote the least jump of $H$ in $G$.
Then, as a consequence of Lemma \ref{commutator indices}, the subgroup $H_2=[H,H]$ is contained in $G_{2l}$. Moreover, since $H^p$ is equal to $H_3$, the subgroup $H^p$ is contained in $H_2$. It follows from Corollary \ref{power iso bottom} that the minimum jump of $H_2$ is at most $l+2$: we get that $2l\leq l+2$ and therefore $l\leq 2$. 
We will show that $HG_2=G$. Assume by contradiction that $G\neq HG_2$. 
Then, as a consequence of Lemma \ref{obelisk basic}($2$), the width $\wt_H^G(l)$ is equal to $1$ and so Lemma \ref{cyclic quotient commutators} yields that 
$H_2=[H,H\cap G_{l+1}]$. Thanks to Lemma \ref{commutator indices}, the subgroup $H_2$ is contained in $G_{2l+1}$ and therefore $2l+1\leq l+2$. It follows that $l=1$ and that $H_2$ is contained in $G_3$.
Define now $\overline{G}=G/G_4$ and use the bar notation for the subgroups of $\overline{G}$. By the isomorphism theorems, the groups $\overline{H}$ and $H/(H\cap G_4)$ are isomorphic and so, as a consequence of Lemma \ref{non-ab quotient obelisk}, the group $\overline{H}$ is abelian or a $p$-obelisk. The minimum jump of $H^p$ in $G$ being equal to $3$, we have that $3$ is a jump of $\overline{H_2}$ in $\overline{G}$ and so $\overline{H}$ is a $p$-obelisk. Now, the group $\overline{G_3}$ is central in $\overline{G}$ and so, the group $\overline{H_2}$ being non-trivial, the quotient $\overline{H}/(\overline{H}\cap\overline{G_3})$ is not cyclic.
It follows that $2$ is a jump of $\overline{H}$ in $\overline{G}$ and, from the combination of Lemmas \ref{parity obelisks} and \ref{maps incrociate}($2$), that $\overline{H_2}$ has order $p$. Since $\overline{H_2}$ contains $\overline{H}^p$, we get $\overline{H_2}=\overline{H}^p=\overline{H_3}$. Contradiction to $\overline{H}$ being non-abelian.
We have proven that $G=HG_2$, from which we derive $G=H\Phi(G)$. 
Lemma \ref{subgroup times frattini} yields $H=G$.
\end{proof}

\begin{lemma}\label{cyclic then 1-dim jumps}
Let $p>3$ be a prime number and let $G$ be a $p$-obelisk. 
Let $H$ be a cyclic subgroup of $G$.
Then all jumps of $H$ in $G$ have the same parity and width $1$.
\end{lemma}

\begin{proof}
Let $H$ be a cyclic subgroup of $G$. 
Then, for all $i\in\Z_{>0}$, there exists $k\in\Z_{\geq 0}$ such that 
$H\cap G_i=H^{p^k}$.
Moreover, $i\in\Z_{>0}$ is a jump of $H$ in $G$ if and only if there exists $k\in\graffe{0,1,\ldots,\log_p|H|-1}$ such that $H\cap G_i=H^{p^k}$ and $H\cap G_{i+1}=H^{p^{k+1}}$.
We conclude thanks to Lemma \ref{maps incrociate}($1$).
\end{proof}

\begin{lemma}\label{obelisk cyclic subgroups}
Let $p>3$ be a prime number and let $G$ be a $p$-obelisk. 
Let $c$ denote the nilpotency class of $G$ and assume that one of the following holds.
\begin{itemize}
 \item[$1$.] The number $c$ is even.
 \item[$2$.] One has $\wt_G(c)=2$.
\end{itemize}
If $H$ is a subgroup such that all of its jumps in $G$ have the same parity and width $1$, then $H$ is cyclic.
\end{lemma}

\begin{proof}
Without loss of generality we assume that $H$ is non-trivial and we take $l$ to be the least jump of $H$ in $G$. 
Let moreover $\cor{J}(H)$ denote the collection of jumps of $H$ in $G$ and
define $J=\graffe{l+2k\ :\ k\in\Z_{\geq 0},\ k\leq (c-l)/2}$.
Let $x$ be an element of $H$ such that $\dpt_G(x)=l$; the existence of $x$ is guaranteed by Lemma \ref{jumps and depth}.  
Write $K=\gen{x}$ and let 
$\cor{J}(K)$ be the collection of jumps of $K$ in $G$. 
By assumption $J$ contains $\cor{J}(H)$ and, as a consequence of Corollary \ref{power iso bottom}, the set $J$ is contained in $\cor{J}(K)$. Keeping in mind that each jump of $H$ in $G$ has width $1$, one derives  
$$
|K|=\prod_{j\in\cor{J}(K)}p^{\wt_K^G(j)}
\geq \prod_{j\in J}p^{\wt_K^G(j)}
\geq \prod_{j\in J}p^{\wt_H^G(j)}
\geq \prod_{j\in\cor{J}(H)}p^{\wt_H^G(j)}=|H|.
$$
It follows that $K=H$ and $H$ is cyclic.
\end{proof}

\begin{lemma}\label{cyclic trivial intersection}
Let $p>3$ be a prime number and let $G$ be a $p$-obelisk. 
Let $c$ denote the nilpotency class of $G$ and let $H$ be a subgroup of $G$ such that $H\cap G_c=\graffe{1}$. If all jumps of $H$ in $G$ have the same parity and width $1$, then $H$ is cyclic.
\end{lemma}

\begin{proof}
We denote $\overline{G}=G/G_c$ and we will use the bar notation for the subgroups of $\overline{G}$. As a consequence of Lemma \ref{quotient obelisk}, the group $\overline{G}$ is abelian or it is a $p$-obelisk. If $\overline{G}$ is abelian, then $c=2$ and so, by Lemma \ref{obelisk cyclic subgroups}, the subgroup $H$ is cyclic. Assume now that $\overline{G}$ is non-abelian and thus a $p$-obelisk.
The group $\overline{G}$ has class $c-1$ and, as a consequence of Corollary \ref{parity obelisks}, either $c-1$ is even or $\wt_{\overline{G}}(c-1)=2$. It follows from Lemma \ref{obelisk cyclic subgroups} that $\overline{H}$ is cyclic and, the intersection $H\cap G_c$ being trivial, so is $H$. 
\end{proof}

\begin{lemma}\label{same same}
Let $p>3$ be a prime number and let $G$ be a $p$-obelisk. Let $c$ denote the nilpotency class of $G$ and let $H$ be a non-trivial subgroup of $G$ such that $H\cap G_c=\graffe{1}$.
Let $l$ be the least jump of $H$ in $G$ and assume that all jumps of $H$ in $G$ have the same parity and the same width.
Then the following hold.
\begin{itemize}
 \item[$1$.] The group $H$ is abelian.
 \item[$2$.] One has $\Phi(H)=H\cap G_{l+1}$.
\end{itemize}
\end{lemma}

\begin{proof}
Let $\cor{J}(H)$ denote the collection of jumps of $H$ in $G$. 
We first assume $\wt_H^G(l)=1$. By Lemma \ref{cyclic trivial intersection}, the subgroup $H$ is cyclic and $\Phi(H)$ has index $p$ in $H$. It follows that $\Phi(H)=H\cap G_{l+1}$. 
Assume now that $\wt_H^G(l)=2$. Then, thanks to Lemma \ref{parity obelisks}($3$), the jump $l$ is odd. 
The subgroup $[H,H]$ is contained in $G_{2l}$, thanks to Lemma \ref{commutator indices}, and therefore, $2l$ being even, Lemma \ref{maps incrociate}($2$) yields $2l>c$. In particular, one has $[H,H]=\graffe{1}$ so $\Phi(H)=H^p$. Moreover, as a consequence of Lemma \ref{maps incrociate}($1$), the set of jumps of $H^p$ in $G$ is equal to 
$\cor{J}(H)\setminus\graffe{l}$ and each jump of $H^p$ has width $2$. 
It follows that $H^p=H\cap G_{l+1}$.
Thanks to Proposition \ref{blackburn 2-1-2} the width $\wt_H^G(l)$ is either $1$ or $2$ and the proof is thus complete.
\end{proof}

\begin{lemma}\label{supertecnico 1}
Let $p>3$ be a prime number and let $G$ be a $p$-obelisk.
Let $c$ be the class of $G$ and
let $H$ be a non-trivial subgroup of $G$ such that $H\cap G_c=\graffe{1}$. 
Denote by $l$ the least jump of $H$ and assume that $H\cap G_{l+1}=\Phi(H)$. Finally, assume that $c-l$ is odd.
Then, for each complement $K$ of $G_c$ in $HG_c$, there exists $t\in G_{c-l}$ such that $K=tHt^{-1}$.
\end{lemma}

\begin{proof}
The subgroup $G_c$ is central in $G$, because $G$ has class $c$, and so, by Lemma \ref{graph+complements}, all complements of 
$G_c$ in $T=HG_c$ are of the form $\graffe{f(h)h : h\in H}$ as $f$ varies in $\Hom(H,G_c)$. The subgroup $G_c$ is elementary abelian, as a consequence of Lemma \ref{black core}, and therefore $\Hom(H,G_c)$ is naturally isomorphic to $\Hom(H/\Phi(H),G_c)=\Hom(H/(H\cap G_{l+1}),G_c)$. By assumption, $c-l$ is odd so, thanks to Corollary \ref{delta surjective}, the homomorphism 
$G_{c-l}/G_{c-l+1}\rightarrow\Hom(G_l/G_{l+1},G_c)$, defined by $tG_{c-l}\mapsto (xG_{l+1}\mapsto [t,x])$, is surjective.
By Lemma \ref{black core}, the quotient $G_l/G_{l+1}$ is elementary abelian and therefore the restriction map 
$$\Hom(G_l/G_{l+1},G_c)\rightarrow\Hom(HG_{l+1}/G_{l+1},G_c)$$ is surjective. By the isomorphism theorems, $HG_{l+1}/G_{l+1}$ and $H/(H\cap G_{l+1})$ are isomorphic and so every homomorphism $H\rightarrow G_c$ is of the form $x\mapsto [t,x]$, for some 
$t\in G_{c-l}$.
For each complement $K$ of $G_c$ in $T$ there exists thus $t\in G_{c-l}$ such that 
$K=\graffe{[t,x]x\ :\ x\in H}=\graffe{txt^{-1}\ :\ x\in H}=tHt^{-1}$. 
\end{proof}


\chapter{The most intense chapter}\label{chapter complete char}

\noindent
Let $p>3$ be a prime number. We recall that a \emph{$p$-obelisk} is a finite $p$-group $G$ of class at least $2$ that satisfies $G_3=G^p$ and $|G:G_3|=p^3$. A $p$-obelisk $G$ is \emph{framed} if, for each maximal subgroup $M$ of $G$, one has $\Phi(M)=G_3$. Some theory about $p$-obelisks is developed in Chapter \ref{chapter obelisks}.
\vspace{8pt} \\
\noindent
The main results of this chapter are summarized in Theorems \ref{theorem class 4 iff} and \ref{theorem complete char pt.1}, which are proven in Section \ref{section proof main}.

\begin{theorem}\label{theorem class 4 iff}
Let $p>3$ be a prime number and let $G$ be a finite $p$-group of class $4$.
Let $\alpha$ be an automorphism of order $2$ of $G$. Then the following conditions are equivalent.
\begin{itemize}
 \item[$1$.] The group $G$ is a $p$-obelisk and the automorphism 
 			 $G/G_2\rightarrow G/G_2$ that is induced by $\alpha$ is equal
 			 to the inversion map $\overline{x}\mapsto\overline{x}^{\,-1}$.
 \item[$2$.] The automorphism $\alpha$ is intense.
\end{itemize}
\end{theorem}

\noindent
An analogue of Theorem \ref{theorem class 4 iff} for higher nilpotency classes is proven in Chapter \ref{chapter lines}: the next theorem gives an essential contribution to its proof.

\begin{theorem}\label{theorem complete char pt.1}
Let $p>3$ be a prime number and let $G$ be a framed $p$-obelisk.
Let $\alpha$ be an automorphism of order $2$ of $G$ and assume that 
the automorphism $G/G_2\rightarrow G/G_2$ that is induced by $\alpha$ is equal to the inversion map $\overline{x}\mapsto\overline{x}^{\,-1}$.
Then  $\alpha$ is intense.
\end{theorem}

\noindent
We remark that the structure of Chapter \ref{chapter complete char} is quite rigid and is meant to ease the understanding of the strategy behind the proof of Theorem \ref{theorem complete char pt.1}. We will prove Theorem \ref{theorem complete char pt.1} by induction on the nilpotency class $c$ of the group $G$ and we will separate the cases according to the parity of $c$. Propositions \ref{proposition even}, \ref{proposition odd 1}, and \ref{proposition odd 2} will be the building blocks of the whole theory and will be verified respectively in Sections \ref{even class}, \ref{section odd 1}, and \ref{section odd 2}. We will use several results from Section \ref{section subgroups obelisks} to understand the structure of the subgroups of $G$, according to the size of their intersection with $G_c$. Moreover, the arguments that we will apply will heavily depend on the knowledge of the jumps of each subgroup in $G$. For more detailed information about jumps, we refer to Section \ref{section jumps}.

\section{The even case}\label{even class}

\noindent
The next proposition is proven for any $p$-obelisk, where $p$ is a prime number greater than $3$. We want to stress that, on the contrary, in Propositions \ref{proposition odd 1} and \ref{proposition odd 2} we ask for the $p$-obelisk to be framed.

\begin{proposition}\label{proposition even}
Let $p>3$ be a prime number and let $G$ be a $p$-obelisk of class $c$. Assume that $c$ is even.
Let moreover $\alpha$ be an automorphism of $G$ of order $2$ and assume that the map $\alpha_c:G/G_c\rightarrow G/G_c$ that is induced by $\alpha$ is intense.
Then $\alpha$ is intense.
\end{proposition}

\noindent
We give the proof of Proposition \ref{proposition even} in Section \ref{proof even}, after some preparation.

\subsection{Some lemmas}\label{section lemmas even}

\noindent
We will work under the hypotheses of Proposition \ref{proposition even} until the end of Section \ref{section lemmas even}. The class $c$ of $G$ being even, Lemma \ref{parity obelisks}($4$) yields that $G_c$ has order $p$.
We denote moreover $A=\gen{\alpha}$ and we recall that a subgroup $H$ of $G$ is said to be $A$-stable if the action of $A$ on $G$ induces an action of $A$ on $H$.

\begin{lemma}\label{induction even}
Let $H$ be a subgroup of $G$ containing $G_c$. Then there exists $g\in G$ such that $gHg^{-1}$ is $A$-stable. 
\end{lemma}

\begin{proof}
The automorphism $\alpha_c$ is intense so, by Lemma \ref{equivalent intense coprime-pgrps}, there exists $g\in G$ such that $(gHg^{-1})/G_c$ is $\gen{\alpha_c}$-stable. It follows from the definition of $\alpha_c$ that $gHg^{-1}$ is $A$-stable.
\end{proof}

\begin{lemma}\label{tutti jump odd}
Let $H$ be a subgroup of $G$ such that $H\cap G_c=\graffe{1}$. Then all jumps of $H$ in $G$ are odd. 
\end{lemma}

\begin{proof}
The subgroup $H$ has trivial intersection with $G_c$ and $c$ is even. It follows from Corollary \ref{power iso bottom} that $H$ cannot have even jumps in $G$. 
\end{proof}

\begin{lemma}\label{prodotti stabili}\label{gamma}
Let $H$ be a subgroup of $G$ such that $H\cap G_c=\graffe{1}$. 
Define $T=HG_c$ and assume that $\alpha(T)=T$.
Then, for each subgroup $K$ of $T$, one has $\alpha(KG_c)=KG_c$.
Moreover, for each $x\in H$, there exists $\gamma\in G_c$ such that 
$\alpha(x)=x^{-1}\gamma$ and $\alpha(\gamma)=\gamma$.
\end{lemma}

\begin{proof}
We denote $\overline{G}=G/G_c$ and we use the bar notation for its subgroups.
By Lemma \ref{tutti jump odd}, all jumps of $H$ in $G$ are odd and so all jumps of $\overline{T}$ in $\overline{G}$ are odd. The subgroup $\overline{T}$ is $\gen{\alpha_c}$-stable so, as a consequence of Lemma \ref{order +- jumps}, each element of $\overline{T}$ is sent to its inverse by $\alpha_c$. Every subgroup of $\overline{T}$ is thus $\gen{\alpha_c}$-stable and, in particular, so is $\overline{K}$. It follows from the definition of $\alpha_c$ that $KG_c$ is $A$-stable. Moreover, every element of $H$ is inverted, modulo $G_c$, by $\alpha$ and the restriction of $\alpha$ to $G_c$ is the identity map, thanks to Lemma 
\ref{action chi^i general}.
\end{proof}

\begin{lemma}\label{same dimension}
Let $H$ be a non-trivial subgroup of $G$ such that $H\cap G_c=\graffe{1}$. 
Let $l$ denote the least jump of $H$ and assume that all jumps of $H$ in $G$ have the same width. Assume moreover that $\alpha(HG_c)=HG_c$.
Then there exists $g\in G_{c-l}$ such that $gHg^{-1}$ is $A$-stable.
\end{lemma}

\begin{proof}
Define $T=HG_c$. 
All jumps of $H$ in $G$ are odd, by Lemma \ref{tutti jump odd},
and $H$ is abelian, by Lemma \ref{same same}($1$). The subgroup $G_c$ being central, the group $T$ is in fact equal to $H\oplus G_c$. 
Moreover, the subgroup $G_c$ being characteristic, $T=\alpha(T)=\alpha(H)\oplus G_c$ and $\alpha(H)$ is a complement of $G_c$ in $T$. 
By Lemma \ref{same same}($2$), the Frattini subgroup of $H$ is equal to $H\cap G_{l+1}$ so it follows from
Lemma \ref{supertecnico 1} that there exists $t\in G_{c-l}$ such that $\alpha(H)=tHt^{-1}$. Thanks to Lemma \ref{equivalent intense coprime}, there exists $g\in G_{c-l}$ such that $gHg^{-1}$ is $A$-stable.
\end{proof}

\begin{lemma}\label{same dimension general}
Let $H$ be a subgroup of $G$ such that $H\cap G_c=\graffe{1}$. 
Assume that all jumps of $H$ in $G$ have the same width. 
Then there exists $g\in G$ such that $gHg^{-1}$ is $A$-stable.
\end{lemma}

\begin{proof}
Denote $T=HG_c$. Thanks to Lemma \ref{induction even}, there exists $a\in G$ such that $aTa^{-1}$ is $A$-stable. 
Write $T'=aTa^{-1}$ and $H'=aHa^{-1}$. Then $T'=H'G_c$ and, thanks to Lemma \ref{same dimension}, there exists $b\in G$ such that $bH'b^{-1}$ is $A$-stable. To conclude, define $g=ba$.
\end{proof}

\begin{lemma}\label{gran finale even}
Let $H$ be a subgroup of $G$ such that $H\cap G_c=\graffe{1}$. 
Then there exists $g\in G$ such that $gHg^{-1}$ is $A$-stable.
\end{lemma}

\noindent
We devote the remaining part of this section to the proof of Lemma \ref{gran finale even}. We warn the reader that the following assumptions will be valid until the end of Section \ref{section lemmas even}. 
\vspace{8pt} \\
\noindent
Let $H$ be a subgroup of $G$ such that $H\cap G_c=\graffe{1}$. Without loss of generality we assume that $H$ is non-trivial and, in view of Lemma \ref{same dimension general}, that the jumps of $H$ in $G$ do not all have the same width. As a consequence of Proposition \ref{blackburn 2-1-2}($1$), each jump of $H$ in $G$ will have width $1$ or $2$.
Let $l$ and $j$ denote respectively the least jump of width $1$ and the least jump of width $2$ of $H$ in $G$.
Write $T=HG_c$.

\begin{lemma}\label{dipassaggio}
Let $i$ and $h$ be jumps of $H$. Assume that $\wt_H^G(i)=1$ and that $\wt_H^G(h)=2$. Then $i<h$.
\end{lemma}

\begin{proof}
By Lemma \ref{tutti jump odd}, both $i$ and $h$ are odd so, the class $c$ being even, Corollary \ref{power iso} yields $i<h$.
\end{proof}

\begin{lemma}\label{jl}
The following hold.
\begin{itemize}
 \item[$1$.] One has $l<j$.
 \item[$2$.] One has $j+l>c$.
 \item[$3$.] The subgroup $H$ is abelian.
\end{itemize}
\end{lemma}

\begin{proof}
Part $(1)$ follows directly from Lemma \ref{dipassaggio}. We prove $(2)$ and $(3)$ together.
As a consequence of Lemma \ref{obelisk cyclic subgroups}, the group $H/(H\cap G_j)$ is cyclic and, thanks to Lemma \ref{cyclic quotient commutators}, one gets $[H,H]=[H,H\cap G_j]$. 
The number $l+j$ being even, it follows from Lemma \ref{tutti jump odd} that $l+j$ is not a jump of $H$ in $G$. Lemma \ref{obelisk non-deg} yields $l+j>c$ and, as a result, $[H,H]\subseteq G_{c+1}=\graffe{1}$ so $H$ is abelian. 
\end{proof}

\begin{lemma}\label{J-L-H}
There exist cyclic subgroups $J$ and $L$ of $H$ such that $H=J\oplus L$ and $j$ and $l$ are respectively the least jump of $J$ and the least jump of $L$ in $G$. 
\end{lemma}

\begin{proof}
The subgroup $H$ is abelian by Lemma \ref{jl}($3$) and, as a consequence of Lemma \ref{dipassaggio}, the subgroup $H\cap G_j$ has only jumps of width $2$. The smallest jump of $H$ in $G$ is $l$ so, thanks to Lemma \ref{jumps and depth}, there is an element $z$ in $H$ with $\dpt_G(z)=l$. Define $L=\gen{z}$. Then $L$ is a subgroup of $H$ and $l$ is the least jump of $L$ in $G$. Moreover, thanks to Lemma \ref{cyclic then 1-dim jumps}, all jumps of $L$ are odd and of width $1$ in $G$.
Now, $l$ is smaller than $j$, by Lemma \ref{jl}($1$), and $j$ is a jump of $L$ in $G$ as a consequence of Corollary \ref{power iso bottom}. However, $j$ is a jump of width $2$ of $H$, and thus there exists an element $x$ in $H\setminus L$ such that $\dpt_G(x)=j$. Define $J=\gen{x}$. The group $H$ being abelian, Corollary \ref{power iso} yields $L\cap J=\graffe{1}$. Now, every jump $l\leq i<j$ of $L$ in $G$ is also a jump of $H$ and it has width $1$ by definition of $j$. Moreover, each jump $j\leq i<c$ of $J\oplus L$ is a 
jump of width $2$ of $H$. Corollary \ref{power iso bottom} ensures that all odd integers $l\leq i<c$ are jumps of $H$ in $G$, so Lemma \ref{order product orders jumps} yields
$$
|J\oplus L|=\prod_{i=l}^{c-1}p^{\dim_{J\oplus L}^G(i)}=\prod_{i=l}^{c-1}p^{\dim_{H}^G(i)}=|H|.
$$
It follows that $H$ and $J\oplus L$ coincide.
\end{proof}

\begin{lemma}\label{polline}
Let $J$ and $L$ be as in Lemma \ref{J-L-H}.
Assume that $\alpha(T)=T$.
Then there exists $g\in G_{c-l}$ such that the following hold.
\begin{itemize}
 \item[$1$.] The group $gLg^{-1}$ is $A$-stable.
 \item[$2$.] One has $gTg^{-1}=T$ and $gJg^{-1}=J$.
\end{itemize} 
\end{lemma}

\begin{proof}
We define $R=LG_c$. By Lemma \ref{prodotti stabili}, the group $R$ is $A$-stable.
The subgroup $L$ is cyclic so, by Lemma \ref{cyclic then 1-dim jumps}, all its jumps in $G$ are odd and of width $1$. With $L$ in the role of $H$, it follows from Lemma \ref{same dimension} that there exists $g\in G_{c-l}$ such that $gLg^{-1}$ is $A$-stable. We fix such an element $g$ and prove that $g$ normalizes both $J$ and $T$. The least jump of $J$ is $j$ and therefore $[g,J]=\graffe{[g,x]\ :\ x\in J}$ is contained in $[G_{c-l},G_j]$. As a consequence of Lemma \ref{commutator indices}, the set $[g,L]$ is contained in $G_{c-l+j}$, which is itself contained in $G_{c+1}$ thanks to Lemma \ref{jl}($1$). It follows that $g$ centralizes $J$ and $gJg^{-1}=J$. 
To conclude, we show that $gTg^{-1}=T$. The subgroup $T$ is contained in $G_l$, so $[g,T]$ is contained in $[G_{c-l},G_l]$. Again applying Lemma \ref{commutator indices}, we get that $[g,T]$ is contained in $G_c$. The subgroup $G_c$ being contained in $T$, we are done by Lemma \ref{commutators normalizer}.
\end{proof}

\begin{lemma}\label{nelmezzo}
Let $J$ and $L$ be as in Lemma \ref{J-L-H}.
Assume that $\alpha(T)=T$ and $\alpha(L)=L$.
Then there exists $g\in G_{c-j}$ such that $gHg^{-1}$ is $A$-stable.
\end{lemma}

\begin{proof}
We will construct $g$.
Let $x,z\in H$ be such that $J=\gen{x}$ and $L=\gen{z}$. As a consequence of Lemma \ref{J-L-H}, one has $\dpt_G(x)=j$ and 
$\dpt_G(z)=l$. Let moreover $\gamma\in G_c$ be such that $\alpha(x)=x^{-1}\gamma$ and $\alpha(\gamma)=\gamma$; the existence of $\gamma$ is confirmed by Lemma \ref{gamma}.
Define $m=(j+l-c)/2$, which is a positive integer thanks to Lemma \ref{jl}($2$).
By Lemma \ref{black core}, there exists $a$ of depth $c-j$ in $G$ such that $\rho^m(a)=z$. We fix $a$ and remark that $a$ belongs to $\Cyc_G(L)$, because $[a,z]=[a,\rho^m(a)]=1$ and $z$ generates $L$. Now, $x$ does not belong to $L$, but, as a consequence of Corollary \ref{power iso bottom}, the jump $j$ of $J$ is also a jump of $L$ in $G$. Since $\wt_G(j)=2$, it follows from Lemma \ref{obelisk non-deg} that there exists $s\in\Z$ such that $[a^s,x]=\gamma^{\frac{p-1}{2}}$.
We define $g=a^s$ and claim that $gHg^{-1}$ is $A$-stable.
We recall that $\gamma$ belongs to the central subgroup $G_c$ and that, because of Lemma \ref{black core}, the exponent of $G_c$ is $p$. We compute
$$
\alpha(gxg^{-1})=\alpha([g,x]x)=\alpha(\gamma^{\frac{p-1}{2}}x)=
\alpha(\gamma^{\frac{p-1}{2}})\alpha(x)=
\gamma^{\frac{p-1}{2}}x^{-1}\gamma=
$$
$$
\gamma^{\frac{p+1}{2}}x^{-1}=
(\gamma^{\frac{p-1}{2}}x)^{-1}=([g,x]x)^{-1}=(gxg^{-1})^{-1}
$$
and so $gJg^{-1}$ is $A$-stable. Moreover, $g$ centralizes $L$ and therefore 
$gHg^{-1}=gJg^{-1}\oplus L$. As a consequence, $gHg^{-1}$ is itself $A$-stable.
\end{proof}

\begin{lemma}\label{wojo}
Let $J$ and $L$ be as in Lemma \ref{J-L-H}.
Assume that $\alpha(T)=T$.
Then there exists $g\in G_{c-j}$ such that $gHg^{-1}$ is $A$-stable.
\end{lemma}

\begin{proof}
As a consequence of Lemma \ref{polline}, there exists $a\in G_{c-l}$ such that $aLa^{-1}$ is $A$-stable, $aTa^{-1}=T$, and $aHa^{-1}=J\oplus aLa^{-1}$. We fix such $a$ and we take $h\in G_{c-j}$ making $h(aHa^{-1})h^{-1}$ stable under the action of $A$. With $H$ replaced by $aHa^{-1}$, Lemma \ref{nelmezzo} guarantees the existence of $h$. We define $g=ha$ and we claim that $g\in G_{c-j}$. The jump $j$ is larger than the jump $l$, by Lemma \ref{jl}($1$), and therefore $G_{c-l}\subseteq G_{c-j}$. It follows that the product $ah$ belongs to $G_{c-l}$.   
\end{proof}

\noindent
To conclude the proof of Lemma \ref{gran finale even}, we construct $g\in G$ such that $gHg^{-1}$ is $A$-stable. Let $b\in G$ be such that $bTb^{-1}$ is $A$-stable and observe that such element $b$ exists by Lemma \ref{induction even}. Lemma \ref{wojo}, with $H$ replaced by $bHb^{-1}$, provides an element $a\in G$ such that $a(bHb^{-1})a^{-1}$ is $A$-stable. We define $g=ab$ and the proof of Lemma \ref{gran finale even} is complete.

\subsection{The induction step}\label{proof even}

\noindent
Under the hypotheses of Proposition \ref{proposition even}, we want to show that $\alpha$ is an intense automorphism of $G$. In view of this, let $H$ be a subgroup of $G$. If $H$ contains $G_c$, then $H$ has an $A$-stable conjugate in $G$ by Lemma \ref{induction even}. We assume that $H\cap G_c\neq G_c$. The class $c$ being even, it follows from Lemma \ref{parity obelisks}($4$) that $|G_c|=p$ and so $H$ intersects $G_c$ trivially. By Lemma \ref{gran finale even}, there exists an element $g\in G$ such that $gHg^{-1}$ is $A$-stable. The automorphism $\alpha$ is intense as a consequence of Lemma \ref{equivalent intense coprime-pgrps} and the fact that $H$ was chosen arbitrarily. Proposition \ref{proposition even} is proven.

\section{The odd case, part \textrm{I}}\label{section odd 1}

\noindent
In Proposition \ref{proposition odd 1} an additional assumption compared to Proposition \ref{proposition even} is made: that $G$ be a \emph{framed} $p$-obelisk. We recall that, if $p$ is a prime number, then a $p$-obelisk $G$ is framed if, for each maximal subgroup $M$ of $G$, one has $\Phi(M)=G_3$. We refer to Section \ref{section framed} for useful facts related to framed $p$-obelisks.

\begin{proposition}\label{proposition odd 1}
Let $p>3$ be a prime number and let $G$ be a framed $p$-obelisk of class $c$. Assume that $c$ is odd and that $G_c$ has order $p$.
Let $\alpha$ be an automorphism of $G$ of order $2$ and assume that the map $\alpha_c:G/G_c\rightarrow G/G_c$ that is induced by 
$\alpha$ is intense.
Then $\alpha$ is intense.
\end{proposition}

\noindent
The proof of Proposition \ref{proposition odd 1} is given in Section \ref{proof odd 1}.

\subsection{Some lemmas}\label{section lemmas odd 1}

The goal of this section is to give all ingredients for the proof of Proposition \ref{proposition odd 1} so we will keep the following assumptions until the end of Section \ref{section lemmas odd 1}. 
Let $p>3$ be a prime number and let $G$ be a $p$-obelisk of class $c$. Assume that $c$ is odd and that $G_c$ has order $p$. 
Let moreover $\alpha$ be an automorphism of $G$ of order $2$ and assume that the map $\alpha_c:G/G_c\rightarrow G/G_c$ that is induced by 
$\alpha$ is intense.
Set $A=\gen{\alpha}$ and, in concordance with Section \ref{section involutions}, write $G^+=\graffe{x\in G\ :\ \alpha(x)=x}$ and $G^-=\graffe{x\in G\ :\ \alpha(x)=x^{-1}}$. For a subgroup $H$ of $G$, we denote $H^+=H\cap G^+$ and $H^-=H\cap G^-$ and we use the same ``plus-minus" notation for any subgroup of $G/G_c$ with respect to $\alpha_c$. 
\vspace{8pt} \\
\noindent
We have intentionally not yet asked for $G$ to be framed: we will make such assumption right after stating Lemma \ref{gran finale odd 1}.

\begin{lemma}\label{induction odd 1}
Let $H$ be a subgroup of $G$ containing $G_c$. Then there exists $g\in G$ such that $gHg^{-1}$ is $A$-stable. 
\end{lemma}

\begin{proof}
By assumption, the automorphism $\alpha_c$ is intense and, by Lemma \ref{equivalent intense coprime-pgrps}, there exists $g\in G$ such that $(gHg^{-1})/G_c$ is $\gen{\alpha_c}$-stable. It follows from the definition of $\alpha_c$ that $gHg^{-1}$ is $A$-stable.
\end{proof}

\begin{lemma}\label{dub fx}
Let $H$ be a subgroup of $G$ such that $H\cap G_c=\graffe{1}$.
Then all jumps of $H$ in $G$ have width $1$.
\end{lemma}

\begin{proof}
As a consequence of Proposition \ref{blackburn 2-1-2}($1$), every jump of $H$ in $G$ has width at most $2$. Assume by contradiction that $l$ is a jump of $H$ in $G$ of width $2$. The jump $l$ is odd, thanks to Lemma \ref{parity obelisks}($1$), and $G_l/G_{l+1}=(H\cap G_l)G_{l+1}/G_{l+1}$. Looking at 
$\rho^{(c-l)/2}_l:G_l/G_{l+1}\rightarrow G_c$, it follows from Lemma \ref{maps incrociate}($1$) that $H\cap G_c\neq\graffe{1}$. Contradiction.
\end{proof}

\begin{lemma}\label{all even odd 1}
Let $H$ be a subgroup of $G$ such that $H\cap G_c=\graffe{1}$. Assume that all jumps of $H$ in $G$ are even. Then there exists $g\in G$ such that $gHg^{-1}$ is $A$-stable.
\end{lemma}

\begin{proof}
All jumps of $H$ in $G$ are even so, by Lemma \ref{dub fx}, they also all have width $1$. Let now $l$ be the least jump of $H$ in $G$. Then, by Lemma \ref{cyclic trivial intersection}, the subgroup $H$ is cyclic and, by Lemma \ref{same same}, the subgroups $\Phi(H)$ and $H\cap G_{l+1}$ are the same. 
Let $T=H\oplus G_c$. Assume first that $\alpha(T)=T$. Then $T=\alpha(H)\oplus G_c$ and, by Lemma \ref{supertecnico 1}, there exists $t\in G$ such that $\alpha(H)=tHt^{-1}$. Thanks to Lemma \ref{equivalent intense coprime}, there exists thus $t\in G$ such that $tHt^{-1}$ is $A$-stable.
In general, by Lemma \ref{induction odd 1}, there exists $a\in G$ such that $aTa^{-1}$ is $A$-stable. There now exists $t\in G$ such that $t(aHa^{-1})t^{-1}$ is $A$-stable, so we conclude by defining $g=ta$.
\end{proof}

\begin{lemma}\label{all odd odd 1}
Let $H$ be a subgroup of $G$ such that $H\cap G_c=\graffe{1}$. Assume that all jumps of $H$ in $G$ are odd. Then there exists $g\in G$ such that $gHg^{-1}$ is $A$-stable.
\end{lemma}

\begin{proof}
Let $T=HG_c$. The class of $G$ being odd, it follows from the assumptions that all jumps of $T$ in $G$ are odd. By Lemma \ref{induction odd 1}, there exists $g\in G$ such that $gTg^{-1}$ is $A$-stable. By Lemma \ref{same jumps conjugates}, the subgroups $gTg^{-1}$ and $T$ have the same jumps in $G$ so, as a consequence of Lemma \ref{order +- jumps}, we get that $gTg^{-1}=(gTg^{-1})^-$. In particular, $gHg^{-1}=(gHg^{-1})^-$ and $gHg^{-1}$ is $A$-stable.
\end{proof}

\begin{lemma}\label{gran finale odd 1}
Let $H$ be a subgroup of $G$ such that $H\cap G_c=\graffe{1}$. Assume moreover that $G$ is framed.
Then there exists $g\in G$ such that $gHg^{-1}$ is $A$-stable.
\end{lemma}

\noindent
The remaining part of Section \ref{section lemmas odd 1} will be entirely dedicated to the proof of Lemma \ref{gran finale odd 1}. For this purpose, all assumptions that we now make will hold until the end of the very same section. 
\vspace{8pt} \\
\noindent
Assume that $G$ is a framed $p$-obelisk. Let moreover $H$ be a subgroup of $G$ that trivially intersects $G_c$. 
If all jumps of $H$ in $G$ have the same parity, then we are done by Lemmas \ref{all even odd 1} and \ref{all odd odd 1}.
We assume that $H$ has jumps of each parity and we define $i$ and $j$ respectively to be the least odd jump and the least even jump of $H$ in $G$. Write $T=HG_c$. We recall that, the class of $G$ being $c$, the subgroup $G_c$ is central in $G$.

\begin{lemma}\label{ij}
The following hold. 
\begin{itemize}
 \item[$1$.] One has $i+j>c$.
 \item[$2$.] The subgroups $H$ and $T$ are abelian.
\end{itemize}
\end{lemma}

\begin{proof}
The numbers $i$ and $j$ having different parities, their sum $m=i+j$ is odd. 
Let $k=\max\graffe{i,j}$. Then, as a consequence of Lemma \ref{dub fx}, all jumps of $H$ in $G$ that are smaller than $k$ have width $1$ and so, by Lemma \ref{obelisk cyclic subgroups}, the group $H/(H\cap G_k)$ is cyclic. From Lemma \ref{cyclic quotient commutators}, we get $[H,H]=[H,H\cap G_k]$.  By Lemma \ref{commutator indices}, the subgroup $[H,H]$ is contained in 
$G_m$. If $m>c$, then $G_m\subseteq G_{c+1}=\graffe{1}$, and thus $(1)$ and $(2)$ are proven. Assume by contradiction that $m\leq c$. Let $y$ and $x$ be elements of $H$ respectively of depth $i$ and $j$ in $G$. Then the image of $\gen{y}$ under the natural projection $G\rightarrow G/G_{i+1}$ is a $1$-dimensional subspace of $G_i/G_{i+1}$. Thanks to Proposition \ref{lines}($3$), with $h=i$ and $k=j$,
the elements $y^{p^{j/2}}$ and $[y,x]$ of $H$ span $G_m/G_{m+1}$. It follows from Lemma \ref{dub fx} that $m$ is a jump of $H$ of width $1$ in $G$ so, from Lemma \ref{parity obelisks}($1$), we derive $m=c$. Contradiction to $H$ trivially intersecting $G_c$.
\end{proof}

\begin{lemma}\label{cyclic pmH}
Let $\pi:G\rightarrow G/G_c$ denote the natural projection. Assume that $\alpha(T)=T$. Then $\pi(H)$ is $\gen{\alpha_c}$-stable and $\pi(H)=\pi(H)^+\oplus\pi(H)^-$. Moreover, both $\pi(H)^+$ and $\pi(H)^-$ are cyclic.
\end{lemma}

\begin{proof}
To lighten the notation, we will denote $\overline{G}=\pi(G)$ and we will use the bar notation for the subgroups of $\overline{G}$. By assumption, $\alpha(T)=T$ and thus $\alpha_c(\overline{T})=\overline{T}$. Moreover, $\overline{H}$ is equal to 
$\overline{T}$, so $\overline{H}$ is itself $\gen{\alpha_c}$-stable. As a consequence of Lemma \ref{ij}($2$), the group $\overline{H}$ is abelian so, by Corollary \ref{abelian sum +-}, it decomposes as 
$\overline{H}=\overline{H}^{\,+}\oplus\overline{H}^{\,-}$.
It follows from Lemma \ref{order +- jumps} that $\overline{H}^{\,+}$ and $\overline{H}^{\,-}$ have respectively only even jumps and only odd jumps in $\overline{G}$. Moreover, thanks to Lemma \ref{dub fx}, all jumps of $\overline{H}$, and thus of its subgroups, in $\overline{G}$ have width $1$. Lemma \ref{obelisk cyclic subgroups} yields that both $\overline{H}^{\,+}$ and $\overline{H}^{\,-}$ are cyclic.
\end{proof}

\begin{lemma}\label{elenanonrisp}
Assume that $\alpha(T)=T$.
Then there exist cyclic subgroups $I$ and $J$ of $H$, with least jumps in $G$ respectively equal to $i$ and $j$, such that the following hold.
\begin{itemize}
 \item[$1$.] One has $H=I\oplus J$.
 \item[$2$.] The group $I$ is $A$-stable and $I=I^-$.
 \item[$3$.] The group $S=J\oplus G_c$ is $A$-stable and $S=S^+\oplus G_c$.
\end{itemize}
\end{lemma}

\begin{proof}
We denote $\overline{G}=G/G_c$ and we will use the bar notation for the subgroups of $\overline{G}$. By Lemma \ref{cyclic pmH}, the subgroup $\overline{H}$ is $\gen{\alpha_c}$-stable and it decomposes as 
$\overline{H}=\overline{H}^{\,+}\oplus\overline{H}^{\,-}$, where both $\overline{H}^{\,+}$ and $\overline{H}^{\,-}$ are cyclic. Let $R$ and $S$ be subgroups of $G$, containing $G_c$, such that $\overline{S}=\overline{H}^{\,+}$ and $\overline{R}=\overline{H}^{\,-}$. Because of their definitions, both $R$ and $S$ are $A$-stable. The subgroup $G_c$ is contained in $G^-$ as a consequence of Lemma \ref{action chi^i general}, so it follows from Lemma \ref{abelian>minus quotients} that $R=R^-$. Moreover, by Corollary \ref{abelian sum +-}, one has $S=S^+\oplus S^-$. However, as $\overline{S}=\overline{S}^{\,+}$, the subgroups $S^-$ and $G_c$ are equal, and hence $S=S^+\oplus G_c$. 
We define $I=H\cap R$ and $J=H\cap S$.
The subgroup $I$, being contained in $R=R^-$, is itself $A$-stable and $I=I^-$. 
Moreover, with $G$ and $N$ respectively replaced by $T$ and $H$, Lemma \ref{dedekind law} yields $JG_c=(H\cap S)G_c=S\cap(HG_c)=S$. Since $H\cap G_c=\graffe{1}$, we get $S=J\oplus G_c$. In the same way, we have $R=I\oplus G_c$. It follows that $J$ and $I$ are respectively isomorphic to $\overline{H}^{\,+}$ and $\overline{H}^{\,-}$, and therefore they are cyclic.
What is left to show is that indeed $H=I\oplus J$. The subgroup $H$ is abelian by Lemma \ref{ij}($2$) and $I\cap J=\graffe{1}$, since $R\cap S=G_c$. The subgroup $I\oplus J$ is contained in $H$ and
\[
\overline{I\oplus J}=\overline{I}\oplus\overline{J}=\overline{R}\oplus\overline{S}=
\overline{H}^+\oplus\overline{H}^-=\overline{H},
\]
so we derive $H=I\oplus J$.
\end{proof}

\begin{lemma}\label{jurassic5}
Let $\gamma\in G_c$ and let $x,y$ be elements of $G$ be such that $\dpt_G(x)=j$ and $\dpt_G(y)=i$. 
Then there exist $n\in\Z$ and $d\in\Cyc_G(y)\cap G_{c-j}$ such that 
$\gamma=y^n[d,x]$. 
\end{lemma}

\begin{proof}
By Lemma \ref{ij} the sum $i+j$ is larger than $c$ and so $i>c-j$. Define
$$r=\frac{i-(c-j)}{2}\ \ \text{and} \ \ s=\frac{j}{2}-r.$$ 
Let now $a\in G_{c-j}\setminus G_{c-j+1}$ be such that $\rho^r(a)=y$; the existence of $a$ is granted by Lemma \ref{black core}. As a consequence of Proposition \ref{lines}, the subgroup $G_c$ is generated by $\rho^{\frac{j}{2}}(a)$ and $[a,x]$. There exist thus $A,B\in\Z$ such that 
$$\gamma=\rho^{\frac{j}{2}}(a)^A[a,x]^B.$$
We recall that, for any $k\in\Z_{\geq 0}$, the map $\rho^k$ is given by $z\mapsto z^{p^k}$, hence 
$$\rho^{\frac{j}{2}}(a)=\rho^{s+r}(a)=\rho^s(\rho^r(a))=\rho^s(y)=y^{p^s}.$$
Now, the group $G_c$ being central, Lemma \ref{tgt} implies that the commutator map
$G_{c-j}\times G_j\rightarrow G_c$ is bilinear so $[a,x]^B=[a^B,x]$. 
We define 
$$n=Ap^s \ \ \text{and}\ \ d=a^B$$
and get $\gamma=y^n[d,x]$. To conclude, the element $d$ belongs to $\Cyc_G(y)$, because $d$ and $y$ belong to $\gen{a}$.
\end{proof}

\begin{lemma}\label{nervi}
Assume that $\alpha(T)=T$.
Let $I$ and $J$ be as in Lemma \ref{elenanonrisp}. 
Then there exists $g\in\Cyc_G(I)$ such that $\alpha(gJg^{-1})\subseteq gHg^{-1}$.
\end{lemma}

\begin{proof}
Let $y$ be a generator of $I$ and let $x$ be a generator of $J$. Then one has $\dpt_G(y)=i$ and $\dpt_G(x)=j$. 
As a consequence of Lemma \ref{elenanonrisp}($3$), there exists $\gamma\in G_c$ such that $\alpha(x)=x\gamma$. We fix $\gamma$ and we want to construct $g$.
Let $n\in\Z$ and $d\in \Cyc_G(y)\cap G_{c-j}$ be such that $\gamma=y^n[d,x]$; the existence of $n$ and $d$ is given by Lemma \ref{jurassic5}. We define $g=d^{\frac{p+1}{2}}$ and we claim that $\alpha(gxg^{-1})$ belongs to $gHg^{-1}$. We will use
some properties of $G_c$ that we list here. The group $G_c$ is central and annihilated by $p$, by hypothesis. Moreover, as a consequence of Lemma \ref{action chi^i general}, the restriction of $\alpha$ to $G_c$ coincides with the map $z\mapsto z^{-1}$. To conclude, the commutator map $G_{c-j}\times G_j\rightarrow G_c$ is bilinear by Lemma \ref{tgt}. We compute
$$
\alpha(gxg^{-1})=\alpha([g,x]x)=
\alpha([g,x])\alpha(x)=[g,x]^{-1}x\gamma=[g^{-1},x]x\gamma=
$$
$$
[g^{-1},x]xy^n[d,x]=[g^{-1},x][d,x]xy^n=[g^{-1}d,x]xy^n=
[d^{\frac{p-1}{2}}d,x]xy^n=
$$
$$
[d^{\frac{p+1}{2}},x]xy^{n}=[g,x]xy^{n}=(gxg^{-1})y^n.
$$
The element $g$ centralizes $y$, because $d$ does, so $\alpha(gxg^{-1})=g(xy^n)g^{-1}$ belongs to $gHg^{-1}$.
In particular, $\alpha(gJg^{-1})\subseteq gHg^{-1}$.
\end{proof}

\noindent
We conclude the proof of Lemma \ref{gran finale odd 1}.
By Lemma \ref{induction odd 1} there exists $a\in G$ such that $aTa^{-1}$ is $A$-stable. We fix $a$ and write $aHa^{-1}=I\oplus J$, with $I$ and $J$ as in Lemma \ref{elenanonrisp} and $H$ replaced by $aHa^{-1}$. By Lemma  \ref{nervi}, there exists an element $b\in G$ such that $bIb^{-1}=I$ and $\alpha(bJb^{-1})$ is contained in $baHa^{-1}b^{-1}$. We select such an element $b$ and define $g=ba$. Then $I$ is contained in $gHg^{-1}$ and
$$\alpha(gHg^{-1})=\alpha(baHa^{-1}b^{-1})=
\alpha(bIb^{-1}\oplus bJb^{-1})=\alpha(I\oplus bJb^{-1})=
$$
$$
\alpha(I)\oplus\alpha(bJb^{-1})=I\oplus\alpha(bJb^{-1}) \subseteq gHg^{-1}.$$
It follows that $\alpha(gHg^{-1})=gHg^{-1}$ and $gHg^{-1}$ is itself $A$-stable.
The proof of Lemma \ref{gran finale odd 1} is now complete.

\subsection{The induction step}\label{proof odd 1}

\noindent
In this paragraph we give the proof of Proposition \ref{proposition odd 1} and we work thus under the assumptions of the very same proposition. Let $H$ be a subgroup of $G$; we will show that $H$ has an $A$-stable conjugate. 
If $H$ contains $G_c$, then, by Lemma \ref{induction odd 1}, there exists $g\in G$ such that $gHg^{-1}$ is $A$-stable. Assume now that $G_c$ is not contained in $H$, i.e.
$H\cap G_c=\graffe{1}$. Thanks to Lemma \ref{gran finale odd 1}, the subgroup $H$ has an $A$-stable conjugate. In other words, by Lemma \ref{equivalent intense coprime-pgrps}, the subgroups $H$ and $\alpha(H)$ are conjugate in $G$. As the choice of $H$ was arbitrary, $\alpha$ is intense and the proof of Proposition \ref{proposition odd 1} is complete.

\section{The odd case, part \textrm{II}}\label{section odd 2}

\begin{proposition}\label{proposition odd 2}
Let $p>3$ be a prime number and let $G$ be a framed $p$-obelisk of class $c$. Assume that $c$ is odd and that $G_c$ has order $p^2$.
Let $\alpha$ be an automorphism of $G$ of order $2$ and assume that the map $\alpha_c:G/G_c\rightarrow G/G_c$ that is induced by 
$\alpha$ is intense.
Then $\alpha$ is intense.
\end{proposition}

\noindent
The proof of Proposition \ref{proposition odd 2} is given in Section \ref{proof odd 2}.

\subsection{Some lemmas}\label{section lemmas odd 2}

The purpose of this section is laying the ground for the proof of Proposition \ref{proposition odd 2}. We will therefore, until the end of Section \ref{section lemmas odd 2}, work under the assumptions of Proposition \ref{proposition odd 2}. Denote $A=\gen{\alpha}$.

\begin{lemma}\label{induction odd 2}
Let $H$ be a subgroup of $G$ containing $G_c$. Then there exists $g\in G$ such that $gHg^{-1}$ is $A$-stable.
\end{lemma}

\begin{proof}
We write $\overline{G}=G/G_c$ and we use the bar notation for the subgroups of $\overline{G}$. By hypothesis, the automorphism $\alpha_c$ is intense so, by Lemma \ref{equivalent intense coprime-pgrps}, there exists $g\in G$ such that 
$\alpha_c(\overline{gHg^{-1}})=\overline{gHg^{-1}}$. The map $\alpha_c$ being induced from $\alpha$, it follows that $\alpha(gHg^{-1})=gHg^{-1}$ and $gHg^{-1}$ is $A$-stable.
\end{proof}

\begin{lemma}\label{odd 2 nontrivial intersection}
Let $H$ be a subgroup of $G$ such that $H\cap G_c\neq\graffe{1}$. Then there exists $g\in G$ such that $gHg^{-1}$ is $A$-stable.
\end{lemma}

\begin{proof}
Let $N=H\cap G_c$. 
If $N=G_c$, then, by Lemma \ref{induction odd 2}, there exists $g\in G$ such that $gHg^{-1}$ is $A$-stable. We assume that $N\neq G_c$. The group $N$ being non-trivial, it follows from Proposition \ref{blackburn 2-1-2}($1$) that $G_c$ and $N$ have orders respectively $p^2$ and $p$.
Moreover, the group $G_c$ being central, $N$ is normal in $G$.
It follows from Lemma \ref{intense properties}($2$) that the action of $A$ on $G$ induces an action of $A$ on $\overline{G}=G/N$. Moreover, $\overline{G}$ has class $c$ and the subgroup $\overline{H}=H/N$ has trivial intersection with $\overline{G_c}=G_c/N$.
By Lemma \ref{gran finale odd 1}, there exists $\bar{g}\in\overline{G}$ such that $\bar{g}\overline{H}\bar{g}^{\,-1}$ is $A$-stable, and so there exists $g\in G$ such that $gHg^{-1}$ is $A$-stable.
\end{proof}

\begin{lemma}\label{feltrefaschifo}
Let $H$ be a subgroup of $G$ such that $H\cap G_c=\graffe{1}$. Then $H$ has only even jumps.
\end{lemma}

\begin{proof}
The index $c$ is odd and $H\cap G_c=\graffe{1}$. It follows from Corollary \ref{power iso bottom} that $H$ cannot have odd jumps.
\end{proof}

\begin{lemma}\label{quasi gran finale odd 2}
Let $H$ be a subgroup of $G$ such that $H\cap G_c=\graffe{1}$.
Let $T=HG_c$ and assume that $\alpha(T)=T$.
Then there exists $g\in G$ such that $gHg^{-1}$ is $A$-stable.
\end{lemma}

\begin{proof}
Let $l$ denote the least jump of $H$ in $G$.
By Lemma \ref{feltrefaschifo}, all jumps of $H$ in $G$ are even so, as a consequence of Lemma \ref{parity obelisks}($2$), all jumps of $H$ have width $1$.
It follows from Lemma \ref{same same} that $H$ is abelian and $\Phi(H)=H\cap G_{l+1}$. The subgroup $G_c$ being central, we get $T=H\oplus G_c$.
Now, by Lemma \ref{order +- jumps}, the subgroup $T^+=\graffe{t\in T : \alpha(t)=t}$ has the same jumps as $H$, and it is therefore a complement of $G_c$ in $T$. Thanks to Lemma \ref{supertecnico 1}, the subgroups $H$ and $T^+$ are conjugate in $G$. In particular, $H$ has an $A$-stable conjugate.
\end{proof}

\begin{lemma}\label{gran finale odd 2}
Let $H$ be a subgroup of $G$ such that $H\cap G_c=\graffe{1}$.
Then there exists $g\in G$ such that $gHg^{-1}$ is $A$-stable.
\end{lemma}

\begin{proof}
Define $S=HG_c$. By Lemma \ref{induction odd 2}, there exists $a\in G$ such that $aSa^{-1}$ is $A$-stable. Let now $T=aSa^{-1}$. Then $\alpha(T)=T$ and $T=a(HG_c)a^{-1}=aHa^{-1}G_c$. Moreover, the intersection $aHa^{-1}\cap G_c$ is trivial. Thanks to Lemma \ref{quasi gran finale odd 2} (with $aHa^{-1}$ in the place of $H$), there exists $b\in G$ such that $b(aHa^{-1})b^{-1}$ is $A$-stable. To conclude, we define $g=ba$.
\end{proof}

\subsection{The last step}\label{proof odd 2}

\noindent
We give here the proof of Proposition \ref{proposition odd 2} and we make thus all assumptions from Proposition \ref{proposition odd 2} hold, until the end of Section \ref{proof odd 2}. To show that $\alpha$ is intense, we will show that each subgroup of $G$ has an $A$-stable conjugate. 
Let $H$ be a subgroup of $G$. If $H$ trivially intersects $G_c$, then, by Lemma \ref{gran finale odd 2}, there exists $g\in G$ such that $gHg^{-1}$ is $A$-stable. If, on the contrary, $H\cap G_c\neq\graffe{1}$, then, by Lemma \ref{odd 2 nontrivial intersection}, there exists a conjugate of $H$ in $G$ that is $A$-stable.
We have proven that, in any case, $H$ has an $A$-stable conjugate and, by Lemma \ref{equivalent intense coprime}, the subgroups 
$H$ and $\alpha(H)$ are conjugate in $G$. The choice of $H$ being arbitrary, the automorphism $\alpha$ is intense and we have proven Proposition \ref{proposition odd 2}.

\section{Proving the main theorems}\label{section proof main}

\noindent
In Sections \ref{pf1} and \ref{pf2} we finally prove the two main results of this Chapter, which were stated at the beginning of it.

\subsection{The proof of Theorem \ref{theorem class 4 iff}}\label{pf1}

\noindent
We work under the assumptions of Theorem \ref{theorem class 4 iff}. 
The implication $(2)\Rightarrow(1)$ follows from the combination of Propositions \ref{class 4 obelisk} and \ref{proposition -1^i}. We now prove $(1)\Rightarrow(2)$.
To this end, denote by $\overline{G}$ the quotient $G/G_4$ and by $\alpha_4$ the automorphism of $\overline{G}$ that is induced by $\alpha$. The map $\alpha$ induces the inversion map on $G/G_2$ and thus so does $\alpha_4$ on $\overline{G}/\overline{G}_2$. It follows from Proposition \ref{p^4 has cp} that $\alpha_4$ is intense and consequently, from Proposition \ref{proposition even}, that $\alpha$ is intense too. The proof of Theorem \ref{theorem class 4 iff} is complete.

\subsection{The proof of Theorem \ref{theorem complete char pt.1}}\label{pf2}


\noindent
Under the hypotheses of Theorem \ref{theorem complete char pt.1}, we will work by induction on the class $c$ of $G$. 
As a consequence of Lemma \ref{obelisk basic}($1$-$3$), the group $G$ has class at least $2$ and $G/G_3$ is extraspecial of exponent $p$. 
If $c=2$, then Lemma \ref{extraspecial action on first quotient} yields that $\alpha$ is intense. We assume that $c>2$ and 
denote by $\overline{G}$ the quotient $G/G_c$. We denote moreover by $\alpha_c$ the automorphism of $\overline{G}$ that is induced by $\alpha$ and assume that $\alpha_c$ is intense.
The group $\overline{G}$ is a framed obelisk, because $c>2$, and $\alpha_c$ induces the inversion map on $\overline{G}/\overline{G}_2$, because $\alpha$ does. 
If $c$ is even, then, by Proposition \ref{proposition even}, the map $\alpha$ is intense. Suppose that $c$ is odd. From Proposition \ref{blackburn 2-1-2}($1$) it follows that the cardinality of $G_c$ is $p$ or $p^2$. In the first case we apply Proposition \ref{proposition odd 1}, in the second Proposition \ref{proposition odd 2}.
Theorem \ref{theorem complete char pt.1} is now proven.


\chapter{A characterization for high classes}\label{chapter lines}

\noindent
Let $p>3$ be a prime number and let $G$ be a finite $p$-group. We recall that, for each positive integer $i$, the \emph{$i$-th width} of $G$ is $\wt_G(i)=\log_p|G_i:G_{i+1}|$. The group $G$ is a \emph{$p$-obelisk} if it is non-abelian, satisfying $G_3=G^p$ and 
$|G:G_3|=p^3$. A $p$-obelisk $G$ is \emph{framed} if, for each maximal subgroup $M$ of $G$, one has $\Phi(M)=G_3$. For more information about $p$-obelisks, we refer to Chapter \ref{chapter obelisks}.
\vspace{8pt} \\
\noindent
In this chapter we prove the following result.

\begin{theorem}\label{theorem we need lines}
Let $p$ be a prime number and let $G$ be a finite $p$-group with $\wt_G(5)=2$.
Then the following are equivalent.
\begin{itemize}
 \item[$1$.] One has $\inte(G)>1$.
 \item[$2$.] One has $p>3$, the group $G$ is a framed $p$-obelisk, and there exists an automorphism $\alpha$ of $G$ of order $2$ that induces the inversion map on $G/G_2$. 
\end{itemize}
\end{theorem}

\noindent
We would like to stress that, from the combination of Lemma \ref{parity obelisks} with Theorem \ref{theorem we need lines}, it follows that each finite $p$-group $G$ of class at least $6$ with $\inte(G)>1$ is a framed $p$-obelisk.

\section{A special case}\label{section class 5 special}

The main result of this section is the following.

\begin{proposition}\label{centralizer G4 2-gen}
Let $p>3$ be a prime number and let $G$ be a $p$-obelisk. Write $C=\Cyc_G(G_4)$.
Assume that $\wt_G(5)=1$ and $\inte(G)>1$.
Then one has $\Phi(C)=G_3$.
\end{proposition}

\noindent
The goal of Section \ref{section class 5 special} is proving Proposition \ref{centralizer G4 2-gen}, so all assumption that we will make throughout the text (right now and right after Lemma \ref{Dp central}) will hold until the end of Section \ref{section class 5 special}.
\vspace{8pt} \\
\noindent
Let $p>3$ be a prime number and let $G$ be a $p$-obelisk. Let $(G_i)_{i\geq 1}$ denote the lower central series of $G$. Assume that $\wt_G(5)=1$ so, thanks to Proposition \ref{blackburn 2-1-2}($2$), the class of $G$ is equal to $5$. 
Write $C=\Cyc_G(G_4)$.

\begin{lemma}\label{D maximal}
The subgroup $C$ is maximal in $G$. 
\end{lemma}

\begin{proof}
The commutator map induces a bilinear map $G/G_2\times G_4/G_5\rightarrow G_5$ whose image generates $G_5$, thanks to Lemma \ref{bilinear LCS}, and whose left kernel is $C/G_2$. As a consequence of Lemma \ref{black core}, all quotients $G_i/G_{i+1}$ are $\F_p$-vector spaces and, by assumption, $\wt_G(5)=1$. By Lemma \ref{parity obelisks}($2$), the dimension of $G_4/G_5$ is equal to $1$ and so Lemma \ref{non-degenerate dim 1} yields $|G:C|=p$. In other words, $C$ is a maximal subgroup of $G$.
\end{proof}

\begin{lemma}\label{comprised}
One has $G_4\subseteq C^p\subseteq G_3$ and $|G_3:C^p|=|C^p:G_4|=p$.
\end{lemma}

\begin{proof}
The subgroup $C$ is maximal, by Lemma \ref{D maximal}, and it is thus normal of index $p$ in $G$.
It follows that $C^p$ is normal in $G$ and, as a consequence of Corollary \ref{power iso}, the number $3$ is a $1$-dimensional jump of $C^p$ in $G$. Lemma \ref{obelisks normal squeezed}($2$) yields $G_4\subseteq C^p\subseteq G_3$ and thus, thanks to Lemma \ref{parity obelisks}($2$), we get $|G_3:C^p|=|C^p:G_4|=p$. 
\end{proof}

\begin{lemma}\label{pioggiasempre}
The subgroup $C^p$ centralizes $G_2$.
\end{lemma}

\begin{proof}
Each $p$-obelisk is regular, by Lemma \ref{obelisk regular}, so, as a consequence of Lemma \ref{regular bilinear}, the subgroups $[C,G_2^p]$ and $[C^p,G_2]$ are the same.
Now, $G_2^p$ is equal to $G_4$, by Lemma \ref{black core}, and $[C,G_4]=\graffe{1}$, by definition of $C$. It follows that $C^p$ centralizes $G_2$.
\end{proof}

\begin{lemma}\label{Dp central}
The group $C^p$ is contained in $\ZG(C)$.
\end{lemma}

\begin{proof}
The subgroup $C^p$ is contained in $G_3$, by Lemma \ref{comprised}, and the commutator map $C\times C^p\rightarrow G_4$ is bilinear by Lemma \ref{tgt}.
Such commutator map factors as 
$\gamma: C/G_2\times C^p/G_4\rightarrow G_4$, as a consequence of Lemma \ref{pioggiasempre} and of the definition of $C$. Moreover, thanks to Corollary \ref{power iso bottom}, if $C=\gen{\graffe{x}\cup G_2}$ then $C^p=\gen{\graffe{x^p}\cup G_4}$. The map  $\gamma$ being alternating, it follows that $\gamma$ is the trivial map and so $C^p$ centralizes $C$.
\end{proof}

\noindent
Let now $\alpha$ be an intense automorphism of $G$ of order $2$ and write $A=\gen{\alpha}$.
Set $G^+=\graffe{x\in G\ :\ \alpha(x)=x}$ and $G^-=\graffe{x\in G\ :\ \alpha(x)=x^{-1}}$ and, for each subgroup $H$ of $G$, denote $H^+=H\cap G^+$ and $H^-=H\cap G^-$.
We will prove Proposition \ref{centralizer G4 2-gen} \emph{by contradiction} and, to this end, we assume that $\Phi(C)\neq G_3$.
Let $X$ be the collection of subgroups $H$ of $C$ of the form $H=\gen{x,y}$, where $x\in C\setminus G_2$ and $y\in G_4\setminus G_5$. Then $A$ acts on $X$ in a natural way. Let $X^+$ be the collection of fixed points of $X$ under $A$.

\begin{lemma}\label{exp D}
The exponent of $C$ divides $p^2$.
\end{lemma}

\begin{proof}
By definition, the subgroup $C^{p^2}$ is contained in $(C^p)^p$. By Lemma \ref{maximal non-framed}, we have $[C,C]=C^p$ so it follows from Lemma \ref{regular bilinear} that $(C^p)^p=[C,C]^p=[C,C^p]$. As a consequence of Lemma \ref{Dp central}, the subgroup $[C,C^p]$ is trivial, and thus 
$(C^p)^p=\graffe{1}$. In particular, the exponent of $C$ divides $p^2$. 
\end{proof}

\begin{lemma}\label{order elementino}
Let $x$ be an element of $C\setminus G_2$. Then $x$ has order $p^2$.
\end{lemma}

\begin{proof}
As a consequence of Corollary \ref{power iso bottom}, the element $x^p$ is non-trivial so the order of $x$ is divisible by $p^2$. We conclude by Lemma \ref{exp D}.
\end{proof}

\begin{lemma}\label{abellli}
Let $H\in X$. Then $H$ is abelian and $H\cap G_5=\graffe{1}$.
Moreover, if $x,y\in H$ satisfy $\dpt_G(x)=1$ and $\dpt_G(y)=4$, then $H=\gen{x}\oplus\gen{y}$.
\end{lemma}

\begin{proof}
Let $(x,y)\in (C\setminus G_2)\times (G_4\setminus G_5)$ be such that $H=\gen{x,y}$. Then $y\in\ZG(C)$ and the group $H$ is commutative. Moreover, as a consequence of Lemma \ref{jumps cyclic sbgs all same parity}, the subgroups $\gen{x}$ and $\gen{y}$ have respectively only odd and only even jumps. In particular, $\gen{x}\cap\gen{y}=\graffe{1}$ and $H=\gen{x}\oplus\gen{y}$. In addition, it follows from Lemma \ref{maps incrociate}($1$) that $5$ is a jump of $H$ in $G$ if and only if $x^{p^2}\neq 1$. Lemma \ref{order elementino} yields $H\cap G_5=\graffe{1}$.
\end{proof}

\begin{lemma}\label{jumps di sto qua}
Let $H\in X$ and, for each $i\in\Z_{\geq 1}$, denote $u_i=\wt_H^G(i)$. 
Then $(u_1,u_2,u_3,u_4,u_5)=(1,0,1,1,0)$ and $H$ has order $p^3$.
\end{lemma}

\begin{proof}
For each $i\in\Z_{\geq 1}$, write $w_i=\wt_G(i)$. Thanks to Lemma \ref{parity obelisks}, we have $(w_1,w_2,w_3,w_4,w_5)=(2,1,2,1,1)$.
Let $x,y$ be as in Lemma \ref{abellli}: then $u_1,u_4\geq 1$ and $u_5=0$. 
Since, for each $i\geq 1$, one has $u_i\leq w_i$, we get $u_4=1$.
Moreover, Lemma \ref{maps incrociate}($1$) ensures that $u_3\geq 1$. Let now $N=\gen{y}G_5$, which is a normal subgroup of $G$ thanks to Lemma \ref{obelisks normal squeezed}. Then $N\cap H=\gen{y}$ and, the quotient $H/\gen{y}$ being cyclic, so is $HN/N$. Thanks to Lemma \ref{cyclic then 1-dim jumps}, all jumps of $HN/N$ have the same dimension and width $1$ in $G/N$. As a result, $2$ is not a jump of $HN/N$ in $G/N$ and, since $\gen{y}$ is contained in $G_4$, we have $u_2=0$ and $u_1=u_3=1$.
The group $H$ has order $p^3$, by Lemma \ref{order product orders jumps}.
\end{proof}

\begin{lemma}\label{chill}
The cardinality of $X$ is $p^4$.
\end{lemma} 

\begin{proof}
Thanks to Lemma \ref{abellli}, the set $X$ consists of subgroups of the form $\gen{x}\oplus\gen{y}$, with $x\in C\setminus G_2$ and $y\in G_4\setminus G_5$. The cardinality of $X$ will be thus equal to the quotient $\frac{n}{m}$, where $n$ is the cardinality of 
$(C\setminus G_2)\times(G_4\setminus G_5)$ and $m$ denotes the number of elements of $(C\setminus G_2)\times(G_4\setminus G_5)$ that generate the same subgroup. Let $H$ be in $X$ and let $x$ and $y$ be generators of $H$, as described before. Then, as a consequence of Lemma \ref{jumps di sto qua}, the orders of $x$ and $y$ are respectively $p^2$ and $p$. It follows that $m=(p^3-p^2)(p-1)$ so, in view of Lemmas \ref{jumps di sto qua} and \ref{obelisk basic}, we get 
\[
|X|=\frac{n}{m}=\frac{(p^6-p^5)(p^2-p)}{(p^3-p^2)(p-1)}=p^4.
\]
\end{proof}

\begin{lemma}\label{shape H}
Let $H\in X$. Then the following are equivalent.
\begin{itemize}
 \item[$1$.] The subgroup $H$ is $A$-stable.
 \item[$2$.] There exists $x\in C^-\setminus G_2$ such that $H=\gen{x}\oplus G_4^+$.
\end{itemize} 
\end{lemma}

\begin{proof}
To prove that $(2)$ implies $(1)$ is an easy exercise; we prove the other implication. Assume ($1$).
The group $H$ is abelian, by Lemma \ref{abellli}, and it is $A$-stable. By Corollary \ref{abelian sum +-}, it decomposes as $H=H^+\oplus H^-$. In view of Lemmas \ref{jumps di sto qua} and \ref{order +- jumps}($1$), we have that $H^+=G_4^+$ and that $H\cap G_4=G_4^+$. It follows from Lemma \ref{abellli} that there exists a cyclic subgroup $Q$ of $H$ such that $H=Q(H\cap G_4)$, and thus $H^-$ is cyclic. The proof is now complete.
\end{proof}

\begin{lemma}\label{hop}
The cardinality of $X^+$ is $p^2$.
\end{lemma}

\begin{proof}
Let $\cor{C}$ denote the collection of subgroups $\gen{x}$ of $C$, where $x$ is an element of $C^-\setminus G_2$.
Thanks to Lemma \ref{shape H}, one can define the map $\cor{C}\rightarrow X^+$, by $Q\mapsto Q\oplus G_4^+$, which is easily shown to be a bijection.
In particular, the cardinality of $X^+$ is equal to that of $\cor{C}$. 
Now, the group $C$ is normal in $G$, as a consequence of Lemma \ref{D maximal}, and therefore it is $A$-stable.
By Lemma \ref{order elementino}, each element of $C^-\setminus G_2$ has order $p^2$ and, as a consequence of Proposition \ref{proposition -1^i}, the set $C^-\setminus G_2$ is equal to $C^-\setminus G_3^-$. 
It follows from Lemma \ref{order +- jumps} that  
$$|X^+|=\frac{|C^-|-|G_3^-|}{p^2-p}=\frac{p^4-p^3}{p^2-p}=p^2.$$
\end{proof}

\begin{lemma}\label{map into G4+}
For each subgroup $L$ of $C^-$, the commutator map induces a bilinear map 
$L\times G_3\rightarrow G_4^+$.
\end{lemma}

\begin{proof}
The subgroup $G_4$ is central in $C$ so, by Lemma \ref{tgt}, the commutator map $L\times G_3\rightarrow G_4$ is bilinear. 
Since $L$ is contained in $C=\Cyc_G(G_4)$, the commutator map induces a bilinear map 
$L\times G_3/G_4\rightarrow G_4$. Now, thanks to Proposition \ref{proposition -1^i}, the map $\alpha$ induces the inversion map on $G_3/G_4$ and so, thanks to Lemma \ref{lemma product of characters}, we get $[L,G_3]=[L,G_3]^+$. In particular, $[L,G_3]$ is contained in $G_4^+$ and the proof is complete.
\end{proof}

\begin{lemma}\label{normalizer contains G3}
Let $H\in X^+$. Then $G_3\subseteq\nor_G(H)$.
\end{lemma}

\begin{proof}
By Lemma \ref{shape H}, the subgroup $H$ is of the form $\gen{x}\oplus G_4^+$, for some element $x\in C^-\setminus G_2$. As a consequence of Lemma \ref{commutator indices}, the subgroup $[G_3,G_4^+]$ is trivial so, from Lemma \ref{multiplication formulas commutators}($1$), it follows that $[G_3,H]=[G_3,\gen{x}]$.
Lemma \ref{map into G4+} yields that $[G_3,\gen{x}]$ is contained in $G_4^+$, a subgroup of $H$, and so, by Lemma \ref{commutators normalizer}, one has $G_3\subseteq\nor_G(H)$. 
\end{proof}

\noindent
We will now prove Proposition \ref{centralizer G4 2-gen} by building a contradiction. We remind the reader that we have assumed that $\Phi(C)\neq G_3$.
\vspace{8pt} \\
\noindent
Let $H$ be an element of $X^+$ with the property that $|G:\nor_G(H)|$ is maximal. Let moreover $\cor{J}$ denote the collection of jumps of $\nor_G(H)$ in $G$. 
As a consequence of Lemma \ref{normalizer contains G3}, the normalizer of $H$ contains $HG_3$. It follows from Lemma \ref{jumps di sto qua} that $\graffe{1,3,4,5}$ is contained in $\cor{J}$ and, thanks also to Lemma \ref{obelisk basic}($2$), that ${|G:\nor_G(H)|\leq |G:HG_3|=p^2}$.
Now, by Lemmas \ref{chill} and \ref{hop}, the cardinalities of $X$ and $X^+$ are respectively $p^4$ and $p^2$.
It follows from Lemma \ref{orbits} that
\[
p^4=|X|\leq\sum_{K\in X^+}|G:\nor_G(K)|\leq |X^+||G:\nor_G(H)|\leq p^2p^2=p^4,
\]
and therefore ${|G:\nor_G(H)|=p^2}$. In particular, we get $\nor_G(H)=HG_3$ and $\cor{J}=\graffe{1,3,4,5}$. 
Moreover, again by Lemma \ref{orbits}, no two elements of $X^+$ are conjugate in $G$. 
As a consequence of Lemma \ref{conjugate stable iff in nor-times-G+}, the subgroup $G^+$ is contained in $\nor_G(H)=HG_3$ and so, thanks to Lemma \ref{order +- jumps}($1$), the number $2$ is a jump of $\nor_G(H)$ in $G$. Contradiction.

\section{The last exotic case}\label{section last odd}

The aim of Section \ref{section last odd} is that of exploring the last exotic case for what concerns the structure of finite $p$-groups of intensity greater than $1$. As a consequence of Theorem \ref{theorem we need lines}, the finite $p$-groups of ``high class'' and intensity greater than $1$ all need to be framed obelisks. Theorem \ref{proposition last odd case} is the last result we present that still allows some ``structural freedom'' to $p$-obelisks.

\begin{theorem}\label{proposition last odd case}
Let $p$ be a prime number and let $G$ be a finite $p$-group with $\wt_G(5)=1$.
Write $C=\Cyc_G(G_4)$.
Then the following are equivalent.
\begin{itemize}
 \item[$1$.] One has $\inte(G)>1$.
 \item[$2$.] One has $p>3$, the group $G$ is a $p$-obelisk, and 
$\Phi(C)=G_3$. Moreover, there exists an automorphism $\alpha$ of $G$ of order $2$ that induces the inversion map on $G/G_2$. 
\end{itemize}
\end{theorem}

\noindent
The remaining part of Section \ref{section last odd} will be devoted to the proof of Theorem \ref{proposition last odd case} and we will thus work under the hypotheses of such theorem. 
\vspace{8pt} \\
\noindent
Assume first $(1)$. As a consequence of Proposition \ref{proposition 2gps} and Corollary \ref{3-gps class at most 4}, the prime $p$ is larger than $3$ and so, thanks to Proposition \ref{class 4 obelisk}, the group $G$ is a $p$-obelisk.
Thanks to Theorem \ref{theorem class at least 3}($1$), there exists an intense automorphism $\alpha$ of order $2$ of $G$, which induces the inversion map on $G/G_2$ by Proposition \ref{proposition -1^i}. Proposition \ref{centralizer G4 2-gen} yields $\Phi(C)=G_3$.
\vspace{8pt}\\
\noindent
Assume now that $p>3$, that $G$ is a $p$-obelisk, and that $\Phi(C)=G_3$. Let moreover $\alpha$ be an automorphism of order $2$ of $G$ that induces the inversion map on $G/G_2$. We will prove $(1)$.
Set $A=\gen{\alpha}$ and, for each $i\in\Z_{\geq 1}$, denote $w_i=\wt_G(i)$.
Thanks to Lemma \ref{parity obelisks}, we have  $(w_1,w_2,w_3,w_4,w_5)=(2,1,2,1,1)$ and so, thanks to Proposition \ref{blackburn 2-1-2}($2$), the class of $G$ is equal to $5$. We remind the reader that, for each $k\in\Z_{\geq 0}$, the map $G\rightarrow G$ sending $x$ to $x^{p^k}$ is denoted by $\rho^k$. Furthermore, by Lemma \ref{obelisk regular}, the group $G$ is regular and so, given any subgroup $K$ of $G$, Lemma \ref{regular implies power abelian} yields that $\rho^k(K)=K^{p^k}$. 

\begin{lemma}\label{fromthenews}
One has $\rho^2(C)=G_5$.
\end{lemma}

\begin{proof}
The group $\Phi(C)$ is equal to $C^p[C,C]$ and $\Phi(C)=G_3$, by assumption.
It follows from Lemma \ref{maps incrociate}($1$) that $\Phi(C)^p=G_5$. Now, the group $C$ is maximal in $G$, by Lemma \ref{D maximal}, and so, as a consequence of Lemma \ref{obelisk basic}($2$), the quotient $C/G_2$ is cyclic. Lemma \ref{cyclic quotient commutators} yields $[C,C]=[C,G_2]$. Now, by Lemma \ref{regular technical}, one has $[C,G_2]^p=[C,G_2^p]$ and so, thanks to Lemma \ref{black core}, one gets $[C,G_2]^p=[C,G_4]=\graffe{1}$. It follows that $\Phi(C)^p$ is equal to $C^{p^2}$ and therefore $\rho^2(C)=G_5$.
\end{proof}

\begin{lemma}\label{come in odd 1}
Let $\alpha_5:G/G_5\rightarrow G/G_5$ denote the automorphism that is induced by $\alpha$. Then $\alpha_5$ is intense.
\end{lemma}

\begin{proof}
Let $\overline{G}$ denote $G/G_5$ and use the bar notation for the subgroups of $\overline{G}$. The automorphism
$\alpha_5$ induces the inversion map on $\overline{G}/\overline{G_2}$, because $\alpha$ does so on $G/G_2$. Moreover, thanks to Lemma \ref{non-ab quotient obelisk}, the group $\overline{G}$ is a $p$-obelisk of class $4$. We conclude by applying Theorem \ref{theorem class 4 iff}.
\end{proof}

\noindent
Let $H$ be a subgroup of $G$ and, for each $i\in\Z_{\geq 1}$, write $u_i=\wt_H^G(i)$. We will show that $H$ has an $A$-stable conjugate in $G$. We assume, without loss of generality, that $H$ is non-trivial. As a consequence of Lemma \ref{come in odd 1}, the automorphism that $\alpha$ induces on $G/G_5$ is intense. If $G_5$ is contained in $H$, then, thanks to Lemma \ref{induction odd 1}, there exists $g\in G$ such that $gHg^{-1}$ is $A$-stable. Since $G_5$ has order $p$, we now assume that $H\cap G_5=\graffe{1}$. By Lemma \ref{dub fx}, all jumps of $H$ in $G$ have dimension $1$ and, if they all have the same parity, Lemmas \ref{all even odd 1} and \ref{all odd odd 1} yield that $H$ has an $A$-stable conjugate. We assume now that $H$ has jumps of both parities and we denote by $i$ and $j$ respectively the least odd and the least even jump of $H$ in $G$.

\begin{lemma}\label{stacce}
One has $u_4=1$.
\end{lemma}

\begin{proof}
The group $G$ having class $5$, we have $j\in\graffe{2,4}$. It follows from Corollary \ref{power iso bottom} that $0\neq u_4\leq w_4$ and therefore $u_4=1$.
\end{proof}

\begin{lemma}\label{i=3}
One has $i=3$.
\end{lemma}

\begin{proof}
Since $u_5=0$, the index $i$ is different from $5$ and so $i\in\graffe{1,3}$. Assume by contradiction that $i=1$. As a consequence of Lemma \ref{stacce}, the subgroups $G_4$ and $(H\cap G_4)G_5$ are equal. The group $G_5$ being central, it follows from Lemma \ref{multiplication formulas commutators}($1$) that 
$$G_5=[G,G_4]=[G,(H\cap G_4)G_5]=[G,H\cap G_4].$$
Thanks to Lemma \ref{commutator indices}, the group $[G_2,G_4]$ is trivial ans so Lemma \ref{multiplication formulas commutators}($2$) yields 
$$[HG_2,G_4]=[H,G_4]=[H,H\cap G_4]\subseteq H\cap G_5=\graffe{1}.$$ 
In particular, $H$ is contained in $C$ and so, as a consequence of Lemma \ref{maps incrociate}($1$), we get $\rho^2(H)=\rho^2(C)$.
It follows from Lemma \ref{fromthenews} that $H$ contains $G_5$. Contradiction to $H\cap G_5=\graffe{1}$.
\end{proof}

\noindent
Let $D$ be a maximal subgroup of $G$ with the property that $(H\cap G_3)G_4=D^pG_4$ and note that, thanks to Corollary \ref{power iso bottom}, the subgroup $D$ is uniquely determined by $H$. Since $D^p$ is characteristic in the normal subgroup $D$, Lemma \ref{obelisks normal squeezed} yields $D^p=D^pG_4$ and therefore, from Corollary \ref{power iso bottom}, one gets $|D^p:G_4|=p$.

\begin{lemma}\label{Dp2 trivial}
One has $\rho^2(D)=\graffe{1}$.
\end{lemma}

\begin{proof}
From the definition of $D$ together with Lemma \ref{maps incrociate}($1$), it follows that $\rho^2(D)=\rho(D^p)=\rho(H\cap G_3)$. As a consequence of Lemma \ref{black core}, the subgroup $\rho(H\cap G_3)$ is contained in $H\cap G_5=\graffe{1}$ and thus $\rho^2(D)=\graffe{1}$. 
\end{proof}

\begin{lemma}\label{cnod}
One has $D\neq C$ and $[D,G_4]=G_5$.
\end{lemma}

\begin{proof}
The subgroups $D$ and $C$ are both maximal in $G$ and so, as a consequence of Lemmas \ref{fromthenews} and \ref{Dp2 trivial}, they are distinct. Moreover, the class of $G$ being $5$, the subgroup $[D,G_4]$ is non-trivial. Lemma \ref{commutator indices} gives that $[D,G_4]$ is contained in $G_5$ and, since $w_5=1$, we get $[D,G_4]=G_5$.
\end{proof}

\begin{lemma}\label{wakeupp}
One has $[G_2,D^p]=G_5$.
\end{lemma}

\begin{proof}
The group $G$ is regular, by Lemma \ref{obelisk regular}, and therefore, by Lemma \ref{regular technical}, the subgroups $[G_2,D^p]$ and $[G_2^p,D]$ are equal. By Lemma \ref{black core}, we have $G_2^p=G_4$ and so, from Lemma \ref{cnod}, we derive $[G_2,D^p]=G_5$.
\end{proof}

\begin{lemma}\label{Habeliano}
The subgroup $H$ is abelian.
\end{lemma}

\begin{proof}
As a consequence of Lemma \ref{i=3}, the subgroup $H$ is contained in $G_2$ and, since $w_2=1$, the quotient $H/(H\cap G_3)$ is cyclic. It follows from Lemma \ref{cyclic quotient commutators} that $[H,H]=[H,H\cap G_3]$ and so, thanks to Lemma \ref{commutator indices}, one gets $[H,H]\subseteq H\cap G_5=\graffe{1}$. In particular, $H$ is abelian.
\end{proof}

\begin{lemma}\label{ij=34}
One has $(i,j)=(3,4)$.
\end{lemma}

\begin{proof}
By Lemma \ref{i=3}, the jump $i$ is equal to $3$ and, by definition of $D$, we have $D^p=(H\cap G_3)G_4$.
Moreover, since $G$ has class $5$, the jump $j$ belongs to $\graffe{2,4}$. Assume by contradiction that $j=2$. 
Then one has $u_2=w_2=1$ and so $G_2=HG_3$. 
By Lemma \ref{wakeupp}, the subgroups $[G_2,D^p]$ and $G_5$ coincide and so, the group $G_5$ being central, Lemma \ref{tgt} ensures that the commutator map 
$G_2\times D^p\rightarrow G_5$ is bilinear and differs from the trivial map.
Now, the induced map $G_2/G_3\times D^p/G_4\rightarrow G_5$, derived from Lemma \ref{commutator indices}, is non-trivial and so 
$[H,H\cap G_3]\neq 1$. Contradiction to Lemma \ref{Habeliano}.
\end{proof}

\begin{lemma}\label{rockshow}
Let $x$ and $y$ be elements of $H$, respectively belonging to ${D^p\setminus G_4}$ and $G_4\setminus G_5$. Then $H=\gen{x}\oplus\gen{y}$ and $(u_1,u_2,u_3,u_4,u_5)=(0,0,1,1,0)$.
\end{lemma}

\begin{proof}
Thanks to Lemma \ref{ij=34}, we have $(u_1,u_2,u_3,u_4,u_5)=(0,0,1,1,0)$.
The subgroup $H$ is thus contained in $G_3$ and so, by Lemma \ref{black core}, one has $H^p\subseteq G_5\cap H=\graffe{1}$. It follows from Lemma \ref{Habeliano} that $H$ is elementary abelian. Given any two elements $x$ and $y$ of $H$, satisfying $x\in D^p\setminus G_4$ and $y\in G_4\setminus G_5$, Lemma \ref{jumps and depth} now yields $H=\gen{x}\oplus\gen{y}$.
\end{proof}

\noindent
We define $X$ to be the collection of all subgroups of $G$ of the form $\gen{x}\oplus\gen{y}$, where $(x,y)$ belongs to $(D^p\setminus G_4)\times (G_4\setminus G_5)$. Thanks to Lemmas \ref{Dp2 trivial} and \ref{Habeliano}, each such subgroup is elementary abelian and thus $X$ is well defined.
We remark that, the group $D^p$ being normal in $G$, the group $G$ acts naturally on $X$ by conjugation. 
Write $X^+=\graffe{K\in X : \alpha(K)=K}$.

\begin{lemma}\label{seminarin10}
The cardinality of $X$ is $p^2$.
\end{lemma}

\begin{proof}
Let $K$ be an element of $X$. Then there exist elements $x$ and $y$ of order $p$, respectively of depth $3$ and $4$ in $G$, such that $x\in D^p$ and $K=\gen{x}\oplus\gen{y}$.
Since $|D^p:G_4|=p$ and $(w_4,w_5)=(1,1)$, we get 
\[
|X|= \frac{(p^3-p^2)(p^2-p)}{(p-1)p(p-1)}=p^2.
\]
\end{proof}

\begin{lemma}\label{invoice}
One has $\nor_G(H)\cap D=\nor_G(H)\cap G_2$.
\end{lemma}

\begin{proof}
Assume by contradiction that $\nor_G(H)\cap D\neq \nor_G(H)\cap G_2$. 
As a consequence of Lemma \ref{stacce}, we have that $(H\cap G_4)G_5=G_4$ and, the group $G_5$ being central, it follows from Lemma \ref{multiplication formulas commutators}($1$) that $[D,G_4]=[D,H\cap G_4]$. Lemma \ref{commutators normalizer} yields 
$[D,G_4]\subseteq H$ and thus, by Lemma \ref{cnod}, the subgroup $G_5$ is contained in $H$. Contradiction.
\end{proof}

\begin{lemma}\label{tosse2}
One has $\nor_G(H)\cap G_2=G_3$.
\end{lemma}

\begin{proof}
As a consequence of Lemma \ref{ij=34}, the subgroup $H$ is contained in $G_3$ and, thanks to Lemma \ref{commutator indices}, one has $[G_3,G_3]\subseteq G_6=\graffe{1}$. In particular, $G_3$ normalizes $H$.
Assume by contradiction that $2$ is a jump of $\nor_G(H)$ in $G$. 
Thanks to Lemma \ref{commutator indices}, the group $G_2$ centralizes $G_4$ and, by definition of $D$, we have $(H\cap G_3)G_4=D^p$. It follows from Lemma \ref{multiplication formulas commutators}($1$) that $[G_2,D^p]=[G_2,H\cap G_3]$ and so, thanks to Lemma \ref{commutators normalizer}, the subgroup $[G_2,D^p]$ is contained in $H$. Lemma \ref{wakeupp} yields $G_5\subseteq H$. Contradiction.
\end{proof}

\noindent
We claim that the action of $G$ on $X$ is transitive.
As a consequence of Lemma \ref{seminarin10}, we have that 
$p^2=|X|\geq |G:\nor_G(H)|$
and therefore, applying Lemmas \ref{invoice} and \ref{tosse2}, we get
$$p^2\geq |G:\nor_G(H)|\geq |D:G_2||G_2:G_3|=|D:G_3|.$$
As a consequence of Lemma \ref{obelisk basic}($2$), the index $|D:G_3|$ is equal to $p^2$ and therefore the number of conjugates of $H$ in $G$ is equal to $p^2$. This proves the claim. To conclude, we remark that $\alpha(H)$ is an element of $X$ and therefore $\alpha(H)$ and $H$ are conjugate. 
The choice of $H$ being arbitrary, Lemma \ref{equivalent intense coprime-pgrps} yields that $\alpha$ is intense and so $\inte(G)>1$. 
The proof of Theorem \ref{proposition last odd case} is now complete.

\section{Proving the main theorem}\label{section brevissima}

\noindent
In this section we prove Proposition \ref{proposition eq class 5} and Theorem \ref{theorem we need lines}. We remind the reader that a $p$-obelisk $G$ is framed if, for each maximal subgroup $M$ of $G$, one has $\Phi(M)=G_3$.

\begin{proposition}\label{proposition eq class 5}
Let $p>3$ be a prime number and let $G$ be a finite $p$-group of class at least $5$. Assume that $\inte(G)>1$. Then $G$ is a $p$-obelisk and one of the following holds.
\begin{itemize}
 \item[$1$.] One has $\wt_G(5)=1$ and $G$ has class $5$.
 \item[$2$.] One has $\wt_G(5)=2$ and $G$ is framed.
\end{itemize}
\end{proposition}

\begin{proof}
By Proposition \ref{class 4 obelisk}, the group $G$ is a $p$-obelisk so, thanks to Proposition \ref{blackburn 2-1-2}($1$), the width $\wt_G(5)$ is either $1$ or $2$. 
The $4$-th width of $G$ is $1$, thanks to Lemma \ref{parity obelisks}($2$), so, if $\wt_G(5)=1$, then Proposition \ref{blackburn 2-1-2}($2$) yields that $G$ has class $5$. Assume now that $\wt_G(5)=2$. We will show that, for each maximal subgroup $M$ of $G$, one has $\Phi(M)=G_3$. To this end, let $M$ be a maximal subgroup of $G$.
By Lemma \ref{parity obelisks}, the widths $\wt_G(1)$ and $\wt_G(4)$ are respectively $2$ and $1$ so, the index $|G:M|$ being $p$, it follows from Lemma \ref{obelisk non-deg}, that $5$ is a jump of $[M,G_4]$ of width $1$ in $G$. Moreover, $5$ is the smallest jump of $[M,G_4]$ in $G$, and so Lemma \ref{obelisks normal squeezed} yields $G_6\subseteq [M,G_4]$.  
We denote $\overline{G}=G/[M,G_4]$ and use the bar notation for the subgroups and the elements of $\overline{G}$. We remark that, by construction, we have $\overline{M}\subseteq \Cyc_{\overline{G}}(\overline{G_4})$ and 
$\wt_{\overline{G}}(5)=1$. The class of $\overline{G}$ being $5$, we have in fact that $\overline{M}=\Cyc_{\overline{G}}(\overline{G_4})$ and so
Proposition \ref{centralizer G4 2-gen} yields $\Phi(\overline{M})=\overline{G_3}$. The subgroup $\Phi(M)$ being normal in $G$, it follows from Lemma \ref{obelisks normal squeezed} that 
$\Phi(M)=\graffe{x\in G : \overline{x}\in\Phi(\overline{M})}$ and therefore 
$\Phi(M)=G_3$. The choice of $M$ being arbitrary, the proof is complete. 
\end{proof}

\noindent
We are finally ready to prove Theorem \ref{theorem we need lines}. Let $p$ be a prime number and let $G$ be a finite $p$-group with $\wt_G(5)=2$. The implication $(2)\Rightarrow(1)$ is given by Theorem \ref{theorem complete char pt.1}. Assume now $(1)$. Since $\wt_G(5)\neq 1$, the class of $G$ is at least $5$. Moreover, thanks to Proposition \ref{proposition 2gps} and Corollary \ref{3-gps class at most 4}, the prime $p$ is larger than $3$. Proposition \ref{proposition eq class 5} yields that $G$ is a framed $p$-obelisk. As a consequence of Theorem \ref{theorem class at least 3}, the intensity of $G$ is equal to $2$ and so, thanks to the Schur-Zassenhaus theorem, $G$ has an intense automorphism of order $2$ that, by Proposition \ref{proposition -1^i}, induces the inversion map on $G/G_2$. The proof of Theorem \ref{theorem we need lines} is complete.


\chapter{A generalization to profinite groups}\label{chapter bilbao}

Let $G$ be a profinite group and let $\alpha$ be an automorphism of $G$. Then $\alpha$ is \emph{topologically intense} \index{topologically intense automorphism}
if, for every closed subgroup $H$ of $G$, there exists $x\in G$ such that $\alpha(H)=xHx^{-1}$. Topologically intense automorphisms are automatically continuous, because they stabilize each open normal subgroup of the group on which they are defined.
We denote by $\Int_{\cc}(G)$ the group of topologically intense automorphisms of a profinite group $G$.
\vspace{8pt} \\
\noindent
Topologically intense automorphisms are a generalization of intense automorphisms to profinite groups. 
In Section \ref{section topintchar}, we will show that, the group of topologically intense automorphisms of a profinite group is itself profinite and moreover, if $p$ is a prime number and $G$ is a pro-$p$-group, then 
$\Int_{\cc}(G)$ is isomorphic to $S\rtimes C$, where $S$ is a pro-$p$-subgroup of $\Int_{\cc}(G)$ and $C$ is a subgroup of $\F_p^*$.  
The \emph{intensity} of a pro-$p$-group $G$ is then defined to be the cardinality of $C$ and it is denoted by $\inte(G)$. The question we ask is: \emph{What are the infinite pro-$p$-groups that have intensity greater than $1$?}
We answer this question with Theorem \ref{theorem unique infinite}, which we state after fixing some notation. Let $p$ be an odd prime number and take $t\in\Z_p$ to be a quadratic non-residue modulo $p$. 
We define $\Delta_p$ to be the quaternion algebra 
$\Z_p\oplus\Z_p\mathrm{i}\oplus\Z_p\mathrm{j}\oplus\Z_p\mathrm{k}$
with defining relations 
$\mathrm{i}^2=t$, $\mathrm{j}^2=p$, and $\mathrm{k}=\mathrm{ij}=-\mathrm{ji}$.
We denote by $\Sn(\Delta_p)$ the pro-$p$-subgroup of the multiplicative group $(1+\mathrm{j}\Delta_p)$ that consists of all elements $x=a+b\mathrm{i}+c\mathrm{j}+d\mathrm{k}$ satisfying 
$a^2-tb^2-pc^2+tpd^2=1$.

\begin{theorem}\label{theorem unique infinite}
Let $p$ be a prime number and let $G$ be an infinite pro-$p$-group. Then $\inte(G)>1$ if and only if exactly one of the following holds. 
\begin{itemize}
 \item[$1$.] One has $p>2$ and $G$ is abelian.
 \item[$2$.] One has $p>3$ and $G$ is topologically isomorphic to $\Sn(\Delta_p)$.
\end{itemize}
Moreover, one has $\inte(\Sn(\Delta_p))=2$ and, if $G$ is abelian, then $\inte(G)=p-1$.
\end{theorem}

\noindent
Let $p$ be a prime number and let $G$ be a pro-$p$-group.
We will show, in Section \ref{section topintchar}, that 
$\inte(G)=\gcd\graffe{\inte(G/N) : N \ \text{normal open in} \ G, \ N\neq G}$
and, thanks to this last characterization, we will derive the following theorem as a corollary of Theorem \ref{theorem unique infinite}.

\begin{theorem}\label{theorem any class}
Let $p>3$ be a prime number. Then, for any positive integer $c$, there exists a finite $p$-group $G$ of class 
$c$ and intensity greater than $1$.
\end{theorem}

\noindent
The pace of Chapter \ref{chapter bilbao} will be slightly faster, compared to the previous ones, in the sense that we will assume the reader is familiar with some basic facts about profinite groups (which can however all be found in Chapters $0$ and $1$ from \cite{analytic}). We will give some extra background in Section \ref{section background}.  
In Section \ref{section topintchar}, we will prove several properties of topologically intense automorphisms and give an analogue of Theorem \ref{theorem abelian} for pro-$p$-groups. In the subsequent sections we will pave the way to proving Theorem \ref{theorem unique infinite}.
In Section \ref{section if top intense}, we will give some limitations, for $p>3$, to the structure of infinite non-abelian pro-$p$-groups of intensity greater than $1$. In Section \ref{section na2}, we will discover that, if such groups exist, they can be continuously embedded in one of two infinite pro-$p$-groups (one of them being $\Sn(\Delta_p)$). We will study the structure of those two groups in Section \ref{section two gps} and, in Section \ref{section na2}, we will prove that, if $p>3$ is a prime number and $G$ is an infinite non-abelian pro-$p$-group with $\inte(G)>1$, then $G$ is topologically isomorphic to $\Sn(\Delta_p)$. The results from Section \ref{section sl1} will ensure that $\inte(\Sn(\Delta_p))>1$. 
We will conclude the proof of Theorem \ref{theorem unique infinite} in Section \ref{i} and give that of Theorem \ref{theorem any class} in Section \ref{ii}.
We will close Chapter \ref{chapter bilbao} with Section \ref{iii}, where we will draw a bridge between Theorem \ref{theorem unique infinite} and Theorem \ref{theorem any class}.

\section{Some background}\label{section background}

\noindent
This section is a collection of definitions and results from \cite{analytic}. If $G$ is a profinite group and $S$ is a subset of it, we denote by 
$\cl(S)$ the closure of $S$ in $G$. A full list of the symbols we use can be found at the beginning of this thesis (see List of Symbols).

\begin{definition}
Let $G$ be a profinite group. A \indexx{discrete quotient} of $G$ is a quotient of $G$ by an open normal subgroup. A \indexx{proper quotient} of $G$ is a quotient of $G$ by a closed normal subgroup that is different from $\graffe{1}$.  
\end{definition}

\begin{definition}\label{def topological generators}
Let $G$ be a profinite group. Then a set $X$ is a \emph{set of} \indexx{topological generators} of $G$ if $G=\cl(\gen{X})$. The group $G$ is \emph{topologically finitely generated} if it admits a finite set of topological generators.
\end{definition}

\begin{definition}\label{lcs profinite def}
Let $G$ be a profinite group. The \emph{lower central series} \index{lower central series} $(G_i)_{i\geq 1}$ of $G$ is defined by
\[G_1=G \ \ \text{and} \ \ G_{i+1}=\cl([G,G_i]).\] 
\end{definition}

\noindent
In Section \ref{section commutators}, we have defined the lower central series for any abstract group, which should not be confused with that of a profinite group. In the case of finite groups, they however coincide.
We recall that, as defined in Section \ref{rank}, the rank of a finite group $H$ is the smallest integer $r$ such that every subgroup of $H$ can be generated by $r$ elements.

\begin{definition}\label{rank profinite}
Let $G$ be a profinite group. The \emph{rank} \index{rank} of $G$ is 
$$\rk(G)=\sup\graffe{\rk(G/N) : N\ \text{is normal open in}\ G}.$$
\end{definition}

\noindent
Let $G$ be a profinite group. It follows from the definition that $\rk(G)$ belongs to $\Z\cup\graffe{\infty}$ and, if $G$ has finite rank, that $G$ is also finitely generated.
Moreover, when $G$ is finite, the definition of rank given in Section \ref{rank} is equivalent to the one from Definition \ref{rank profinite}. 
In \cite[Proposition $3.11$]{analytic}, a series of equivalent definitions of rank is given. 

\begin{definition}
A \emph{$p$-adic analytic} \index{$p$-adic analytic group} group is a profinite group that 
contains an open pro-$p$-subgroup of finite rank.
\end{definition}

\noindent
Our definition of a $p$-adic analytic group is not among the standard ones, but it serves our purposes the best. In general, $p$-adic analytic groups are defined to be topological groups that present the structure of a $p$-adic manifold. 
The equivalence of the two definitions, for profinite groups, is given by Corollary $9.35$ from \cite{analytic}.
For more information about the topic, see \cite[Ch.~$9$]{analytic}.

\begin{definition}
A profinite group is \emph{just-infinite} \index{just-infinite group} if it is infinite and each of its proper quotients is finite.
\end{definition}

\begin{definition}
Let $p$ be a prime number and let $G$ be a pro-$p$-group. The \indexx{Frattini subgroup}
of $G$ is $\Phi(G)=\cl(G^p[G,G])$.
\end{definition}

\noindent
As for the case of finite $p$-groups, the Frattini subgroup $\Phi(G)$ of a pro-$p$-group $G$ is the unique closed normal subgroup of $G$ minimal with the property that $G/\Phi(G)$ is a vector space over $\F_p$.

\begin{lemma}\label{frattini open}
Let $p$ be a prime number and let $G$ be a pro-$p$-group. Then $G$ is topologically finitely generated if and only if $\Phi(G)$ is open in $G$.
\end{lemma}

\begin{proof}
This is Proposition $1.14$ from \cite{analytic}.
\end{proof}

\noindent
In Chapter $3$ of \cite{analytic} it is proven that, if $G$ is a finitely generated group, then the cardinality of a minimal set of topological generators of $G$ is equal to $\dim_{\F_p}(G/\Phi(G))$.



{
\begin{definition}
Let $p$ be an odd prime number and let $U$ be a pro-$p$-group.
Then $U$ is \emph{uniform} if the following hold.\index{uniform pro-$p$-group}
\begin{itemize}
\item[$1$.] The group $U$ is topologically finitely generated.
\item[$2$.] The quotient $U/\cl(U^p)$ is abelian.
\item[$3$.] The group $U$ is torsion-free.
\end{itemize}
\end{definition}
}

\noindent
The definition of uniform group we give is slightly different from the one that is given in \cite{analytic}. 
However, the equivalence of the two is proven in \cite[Theorem $4.8$]{analytic}.

\begin{definition}
Let $p$ be an odd prime number and let $U$ be a uniform pro-$p$-group. The \emph{dimension} \index{dimension} of 
$U$ is the cardinality of a minimal set of topological generators of $U$. The dimension of $U$ is denoted $\dim(U)$.
\end{definition}

\begin{lemma}\label{uniform characteristic}
Let $p$ be an odd prime number and let $G$ be a pro-$p$-group of finite rank.
Then $G$ has a characteristic open uniform subgroup. 
\end{lemma}

\begin{proof}
See Corollary $4.3$ from \cite{analytic}.
\end{proof}

\begin{lemma}\label{uniform same dimension}
Let $p$ be an odd prime number and let $G$ be a pro-$p$-group. Then all open uniform subgroups of $G$ have the same dimension.
\end{lemma}

\begin{proof}
See \cite[Corollary $4.6$]{analytic}.
\end{proof}

\begin{definition}\label{def dim}\index{dimension}
Let $p$ be an odd prime number and let $G$ be a pro-$p$-group of finite rank. The \emph{dimension} of $G$ is the dimension of any of its open uniform subgroups.
\end{definition}

\noindent
Lemmas \ref{uniform characteristic} and \ref{uniform same dimension} guarantee the consistency of Definition \ref{def dim}.

\section{Properties and intensity}\label{section topintchar}

In Section \ref{section topintchar} we give several properties of topologically intense automorphisms and, for a given prime number $p$, we define the intensity of a pro-$p$-group.

\begin{lemma}\label{bristol}
Let $G$ be a profinite group and let $\alpha$ be a topologically intense automorphism of $G$.
Then $\alpha$ induces an intense automorphism on each discrete quotient of $G$. 
\end{lemma}

\begin{proof}
Let $N$ be an open normal subgroup of $G$. Then $\alpha(N)=N$ and $\alpha$ induces an automorphism $\overline{\alpha}$ of $G/N$. Now each subgroup $\overline{H}$ of $G/N$ corresponds to an open subgroup $H$ of $G$, which is sent to a conjugate by $\alpha$. As a consequence, also $\overline{\alpha}(\overline{H})$ and $\overline{H}$ are conjugate in $G/N$ and, the choice of $\overline{H}$ being arbitrary, it follows that $\overline{\alpha}$ is intense.
\end{proof}

\noindent
If $G$ is a profinite group and $\Upsilon$ denotes the set of open normal subgroups of $G$, then $\Aut(G)$ has a natural topology, the ``congruence topology'', for which a basis of open neighbourhoods of the identity is given by 
$$\graffe{\Gamma(N)=\graffe{\alpha\in\Aut(G) : \alpha \equiv \id \bmod N}}_{N\in\Upsilon}.$$ For more information on the subject see for example 
\cite[Ch.~$5.2$]{analytic}.

\begin{lemma}\label{intense group is profinite}
Let $G$ be a profinite group and let $\Upsilon$ denote the collection of open normal subgroups of $G$. Then one has
$$\Int_{\cc}(G)=\underset{N\in\Upsilon}{\varprojlim}\Int(G/N).$$
\end{lemma}

\begin{proof}
Let $\Upsilon$ denote the collection of open normal subgroups of $G$. Then, thanks to Lemma \ref{bristol}, we have a natural homomorphism 
$$\pi:\Int_{\cc}(G)\rightarrow \prod_{N\in\Upsilon}\Int(G/N).$$
The map $\pi$ is injective, because $\ker\pi$ is contained in $\cap_{N\in\Upsilon}\Gamma(N)=\graffe{1}$, and the image of $\pi$ is equal to 
$\underset{N\in\Upsilon}{\varprojlim}\Int(G/N)$, thanks to Lemma \ref{intense properties}($2$).
\end{proof}

\begin{lemma}\label{inverse limit}
Let $\graffe{X_\lambda}_{\lambda\in \Lambda}$ be an inverse system of finite non-empty sets over a directed set $\Lambda$. 
Then $\varprojlim X_\lambda$ is non-empty.
\end{lemma}

\begin{proof}
This is Proposition $1.4$ from \cite{analytic}.
\end{proof}

\begin{proposition}\label{openclosedtuttouguale}
Let $G$ be a profinite group and let $\alpha$ be an automorphism of $G$. Then the following are equivalent.
\begin{itemize}
 \item[$1$.] The automorphism $\alpha$ is topologically intense.
 \item[$2$.] For every open subgroup $H$ of $G$, there exists an element $x\in G$ such that $\alpha(H)=xHx^{-1}$.
\end{itemize}
\end{proposition}

\begin{proof}
As every open subgroup is also closed, $(1)$ clearly implies $(2)$. 
Assume now ($2$) and let $H$ be a closed subgroup of $G$. We will construct $x\in G$ such that $\alpha(H)=xHx^{-1}$.
Let $\Lambda$ denote the collection of all discrete quotients of $G$ and let moreover $\Upsilon$ be the collection of all open normal subgroups of $G$. Then there is a natural bijection $\Upsilon\rightarrow\Lambda$, given by $N\mapsto G/N$. 
Now, thanks to Lemma \ref{bristol}, the automorphism $\alpha$ induces an intense automorphism on each element of $\Lambda$. Hence, if $\overline{G}$ is an element of $\Lambda$ and $\overline{H}$ denotes the image of $H$ in $\overline{G}$, then there exists $x\in\overline{G}$ such that $\overline{\alpha(H)}=x\overline{H}x^{-1}$. For each $\overline{G}\in \Lambda$ define 
$X_{\overline{G}}=\graffe{x\in \overline{G} : \overline{\alpha(H)}=x\overline{H}x^{-1}}$ and observe that $X_{\overline{G}}$ is finite and non-empty. Let now $\overline{G}$ and $\overline{G}\,'$ be elements of $\Lambda$ such that $\overline{G}\,'$ is a quotient of $\overline{G}$. Then the natural projection $\overline{G}\rightarrow\overline{G}\,'$ induces a well-defined map 
$X_{\overline{G}}\rightarrow X_{\overline{G}\,'}$. It follows that $\graffe{X_{\overline{G}}}_{\overline{G}\in\Lambda}$ is an inverse system of finite non-empty sets so, by Lemma \ref{inverse limit}, the set $X=\varprojlim X_{\overline{G}}$ is non-empty. 
Let $x\in X$. As a consequence of the definition of $X$, for each element $N$ of $\Upsilon$, the element $xN$ belongs to 
$X_{G/N}$ and thus, for each $N\in\Upsilon$, we have $\alpha(HN)=xHx^{-1}N$.
The map $\alpha$ is continuous, because it stabilizes each open normal subgroup, and so it follows that 
\[
\alpha(H)=\cl(\alpha(H))=\bigcap_{N\in\Upsilon}\alpha(H)N=\bigcap_{N\in\Upsilon}\alpha(HN)=
\]
\[
=\bigcap_{N\in\Upsilon}xHx^{-1}N=\cl(xHx^{-1})=xHx^{-1}.
\]
This proves ($1$), and therefore the proof is complete.
\end{proof}

\noindent
In the proof of the following result we will use the generalization to profinite groups of Schur-Zassenhaus's theorem (see for example Theorem $2.3.15$ from \cite{profinite}).

\begin{proposition}\label{inte profinite}
Let $p$ be a prime number and let $G$ be a pro-$p$-group. Then 
$$\Int_{\cc}(G)=S\rtimes C,$$
where $S$ is a Sylow pro-$p$-subgroup of $\Int_{\cc}(G)$ and $C$ is isomorphic to a subgroup of $\F_p^*$. Moreover, one has
$$|C|=\gcd\graffe{\inte(G/N) : N \ \text{normal open in} \ G, \ N\neq G}.$$
\end{proposition}

\begin{proof}
Let $\Upsilon$ denote the collection of open normal subgroups of $G$. For each $N\in\Upsilon$, denote by $\pi_N:\Int(G/N)\rightarrow\Int((G/N)/\Phi(G/N))$ the map from Lemma \ref{intense properties}($2$) and set $K_N=\ker\pi_N$ and 
$I_N=\pi_N(\Int(G/N))$. 
For each $N\in\Upsilon$, we then get a short exact sequence 
$$1\rightarrow K_N\rightarrow\Int(G/N)\rightarrow I_N \rightarrow 1$$
which induces, thanks to Lemma \ref{intense group is profinite} and the exactness of $\varprojlim$, the short exact sequence
$$1\rightarrow\underset{N\in\Upsilon}{\varprojlim} K_N\rightarrow\Int_{\cc}(G)\rightarrow\underset{N\in\Upsilon}{\varprojlim} I_N \rightarrow 1.$$
Define $S=\underset{N\in\Upsilon}{\varprojlim} K_N$ and 
$C=\underset{N\in\Upsilon}{\varprojlim} I_N$. 
As a consequence of Lemma \ref{intensity of quotients}, whenever $M,N\in\Upsilon\setminus\graffe{G}$ and $N\subseteq M$, the natural map 
$I_N\rightarrow I_M$ is injective and therefore,
$\underset{N\in\Upsilon}{\varprojlim} \Int((G/N)/\Phi(G/N))$
being equal to $\Int(G/\Phi(G))$, Lemma \ref{intense vector space} yields that $C$ is isomorphic to a subgroup of $\F_p^*$. Moreover, thanks to Lemma \ref{formulation}, the group $S$ is a pro-$p$-subgroup of $\Int_{\cc}(G)$. The order of $C$ being coprime to $p$, it follows that in fact $S$ is a Sylow pro-$p$-subgroup of $\Int_{\cc}(G)$ and, from the generalization of Schur-Zassenhaus's theorem to profinite groups, that $\Int_{\cc}(G)=S\rtimes C$. Moreover, the fact that $|C|$ is equal to the greatest common divisor of the $\inte(G/N)$, as $N$ varies in $\Upsilon\setminus\graffe{G}$, is a direct consequence of Lemma \ref{intensity of quotients}.
\end{proof}

\noindent
Let $p$ be a prime number and let $G$ be a pro-$p$-group. Let moreover $C$ be as in Proposition \ref{inte profinite}. The \indexx{intensity} $\inte(G)$ of $G$ is defined to be the cardinality of $C$. Thanks to Proposition \ref{inte profinite}, the intensity of $G$ is also equal to the greatest common divisor of the set 
$\graffe{\inte(G/N) : N \ \text{normal open in} \ G, \ N\neq G}$.

\begin{corollary}\label{pro-p-abelian}
Let $p$ be a prime number and let $G$ be an abelian pro-$p$-group.
If $G$ is non-trivial, then $\inte(G)=p-1$.
\end{corollary}

\begin{proof}
If $G$ is non-trivial, then Theorem \ref{theorem abelian} yields $\inte(G)=p-1$.
\end{proof}

\section{Non-abelian groups, part I}\label{section if top intense}

The main purpose of Section \ref{section if top intense} is to give a proof of Proposition \ref{proposition infinite analytic}. We refer to Section \ref{section background} for the definitions of just-infinite profinite groups and of $p$-adic analytic groups and their dimensions.

\begin{proposition}\label{proposition infinite analytic}
Let $p>3$ be a prime number and let $G$ be a non-abelian infinite 
pro-$p$-group. Assume that $\inte(G)>1$. 
Then $G$ is a just-infinite $p$-adic analytic group of dimension $3$.
\end{proposition}

\noindent
The following assumptions will be valid until the end of Section \ref{section if top intense}.
Let $p$ be an odd prime number and let $G$ be an infinite non-abelian pro-$p$-group of intensity greater than $1$. 
Let $(G_i)_{i\geq 1}$ denote the lower central series of $G$, as defined in Section \ref{section background}, and let $\alpha$ be a topologically intense automorphism of $G$ of order $2$. The existence of $\alpha$ is guaranteed by the combination of Propositions \ref{inte profinite} and \ref{proposition divisible by 2}.

\begin{lemma}\label{inte quozienti}
The automorphism $\alpha$ induces an intense automorphism of order $2$ on each non-trivial discrete quotient of $G$. 
\end{lemma}

\begin{proof}
Let $C$ be as in proposition \ref{inte profinite} and, without loss of generality, assume that $\alpha\in C$. Then, as a consequence of Lemma \ref{bristol}, given any open normal subgroups $N$ and $M$ of $G$ such that $N\subseteq M\neq G$, we get a commutative diagram
$$
\begin{diagram}
C & \rTo & \Int(G/M) \\
\dTo  & \ruTo &       \\
\Int(G/N)  
\end{diagram}.
\vspace{3pt}\\
$$
Moreover, the map $\alpha$ being non-trivial, there exists 
a discrete quotient of $G$ on which $\alpha$ induces an automorphism of order $2$. The choice of $M$ and $N$ being arbitrary, it follows that $\alpha$ induces an intense automorphism of order $2$ on each non-trivial discrete quotient of $G$.
\end{proof}

\begin{lemma}\label{quotient obelisk}
Assume that $p>3$.
Then each discrete quotient of $G$ of class at least $4$ is a $p$-obelisk.
\end{lemma}

\begin{proof}
Let $\overline{G}$ be a discrete quotient of class at least $4$ of $G$. By Lemma \ref{inte quozienti}, the map $\alpha$ induces an intense automorphism of order $2$ of $\overline{G}$. It follows from Proposition \ref{class 4 obelisk} that $\overline{G}$ is a $p$-obelisk.
\end{proof}

\begin{lemma}\label{quotients of any class}
Let $c$ be a non-negative integer. Then $G$ has a discrete quotient of class $c$.
\end{lemma}

\begin{proof}
Assume by contradiction that there exists an upper bound on the class of the discrete quotients of $G$ and let $C\in\Z_{\geq 0}$ be minimal with this property. Since $G$ is non-abelian, one has $C\geq 2$.
Let us now denote by $\Upsilon$ the collection of open normal subgroups of $G$. Then $G=\underset{N\in\Upsilon}{\varprojlim}G/N$ and so $G$ has class $C$. The group $G$ being infinite, it follows from Theorem \ref{theorem consecutive layers} that $C<3$ and so $C=2$. Let now $M,N\in\Upsilon$ be such that $G/N$ has class $2$ and $M\subsetneq N$. Let $K=G/M$ and
let $\pi:G\rightarrow K$ denote the canonical projection. Then $K$ has class $2$ and, as a consequence of Lemma \ref{inte quozienti}, the intensity of $K$ is greater than $1$. By Theorem \ref{theorem class2 complete}, the group $K$ is extraspecial and, $\pi(N)$ being non-trivial and normal in $K$, it follows from Lemma \ref{normal intersection centre trivial} that $\pi(N)$ contains $\ZG(K)=[K,K]$. In particular, $K/\pi(N)$ is abelian and therefore so is $G/N$. Contradiction. 
\end{proof}

\begin{lemma}\label{lcs open}
The set $\{G_i\}_{i\geq 1}$ is a base of open neighbourhoods of $1$ in $G$.
\end{lemma}

\begin{proof}
Let moreover $\Upsilon$ denote the collection of all open normal subgroups $N$ of $P$ with the property that $G/N$ has class at least $3$. As a consequence of Lemma \ref{quotients of any class}, the group $G$ is equal to 
$\underset{N\in\Upsilon}{\varprojlim} G/N$. Moreover, each subgroup $G_i$ being closed, we also have 
$G_i=\underset{N\in\Upsilon}{\varprojlim} (G/N)_i$. 
Thanks to Lemma \ref{inte quozienti}, whenever $N\in\Upsilon$, the quotient $G/N$ has intensity greater than $1$ and so Theorem \ref{theorem dimension layers} yields that $\graffe{G_i}_{i\geq 1}$ is a base of open neighbourhoods of $1$ in $G$.
\end{proof}

\begin{lemma}\label{finite rank}
Assume that $p>3$. 
Then $\rk(G)=3$ and $G$ is $p$-adic analytic.
\end{lemma}

\begin{proof}
By Lemma \ref{inte quozienti}, the automorphism $\alpha$ induces an intense automorphism of order $2$ on each non-trivial discrete quotient of $G$ and, as a consequence of Lemma \ref{quotients of any class}, the group $G$ has finite quotients of any possible class. 
It follows that 
$\rk(G)=\sup\graffe{\rk(G/N) : G/N\ \text{has class at least} \ 4}$
and hence Proposition \ref{rank 3} yields $\rk(G)=3$. The group $G$ being a pro-$p$-group of finite rank, it is $p$-adic analytic.
\end{proof}

\begin{lemma}\label{just-infinite}
Assume that $p>3$. 
Let $N$ be a non-trivial closed subgroup of $G$.
Then the following are equivalent.
\begin{itemize}
 \item[$1$.] The subgroup $N$ is normal.
 \item[$2$.] There exists $l\in\Z_{\geq 1}$ such that 
 $G_{l+1}\subseteq N\subseteq G_l$. 
\end{itemize}
Moreover, $P$ is just-infinite.
\end{lemma}

\begin{proof}
The implication $(2)\Rightarrow(1)$ is clear; we prove $(1)\Rightarrow(2)$. Thanks to Lemma \ref{lcs open}, every element of the lower central series of $G$ is open and $\graffe{G_i}_{i\geq 1}$ is a base of open neighbourhoods of $1$. 
For all $k\in\Z_{\geq 1}$, denote by $\pi_k:G\rightarrow G/G_k$ the canonical projection and set ${l=\max\graffe{k:\pi_k(N)=1}}$. The index $l$ is well-defined, because $N\neq 1$, and 
$N$ is contained in $G_l$, but not in $G_{l+1}$, by the maximality of $l$. In particular, for each $k>l$, the minimum jump of $\pi_k(N)$ in $G/G_k$ is $l$. 
Now, by Lemma \ref{quotient obelisk}, whenever $k\geq 5$, the quotient $G/G_k$ is a $p$-obelisk. It follows from Lemma \ref{obelisks normal squeezed}($2$) that, whenever $k>\max\graffe{l,5}$, the subgroup $G_{l+1}$ is contained in $NG_k$, and therefore
\[G_{l+1}\,\subseteq\, \bigcap_{k>\max\graffe{l,5}}NG_k\,=\,\bigcap_{k\geq 1}NG_k\,=\,\cl(N)\,=\,N.\]
We have proven that $G_{l+1}\subseteq N\subseteq G_l$ and thus also that $(1)$ implies $(2)$. 
To conclude, since each $G_k$ is open in $G$, the subgroup $N$ is open and the quotient $G/N$ is finite. 
Because of the arbitrary choice of $N$, the group $G$ is just-infinite.
\end{proof}

\begin{lemma}\label{torsion-free}
Assume that $p>3$. 
Then $G$ is torsion-free.
\end{lemma}

\begin{proof}
By Lemma \ref{quotient obelisk}, whenever $k$ is at least $5$, the quotient $G/G_k$ is a $p$-obelisk. It follows from Corollary \ref{power iso bottom} that, for each non-negative integer $i$, raising to the power $p$ induces a well-defined isomorphism 
$G_i/G_{i+1}\rightarrow G_{i+2}/G_{i+3}$. By Lemma \ref{quotients of any class}, there is no bound on the class of the finite quotients of $G$, and therefore $G$ is torsion-free.
\end{proof}

\begin{lemma}\label{dimension 3}
Assume that $p>3$. 
Then $G_2$ is open, uniform, and has dimension $3$.
\end{lemma}

\begin{proof}
Let $\overline{G}$ be a discrete quotient of class at least $5$ of $G$, which exists by Lemma \ref{quotients of any class}. 
As a consequence of Lemma \ref{quotient obelisk}, the group $\overline{G}$ is a $p$-obelisk and so Lemma \ref{parity obelisks} yields $|\overline{G}_2:\overline{G}_4|=p^3$. The subgroup $\overline{G}_2^p$ is equal to $\overline{G}_4$, thanks to Lemma \ref{black core}, and so, as a consequence of Lemma \ref{commutator indices}, the quotient $\overline{G}_2/\overline{G}_2^p=\overline{G}_2/\overline{G}_4$ is elementary abelian. It follows that each generating set of $\overline{G}_2$ has at least $3$ elements. However, the rank of $G$ is equal to $3$, thanks to Lemma \ref{finite rank}, and therefore $\overline{G}_2$ is generated by exactly $3$ elements.
Since $\overline{G}$ was chosen arbitrarily, the quotient $G_2/\cl(G_2^p)$ is abelian and any minimal set of topological generators of $G_2$ has $3$ elements. 
Now, as a consequence of Lemma \ref{torsion-free}, the torsion of $G_2$ is trivial and hence $G_2$ is uniform of dimension $3$. Moreover, the subgroup $G_2$ is open thanks to Lemma \ref{lcs open}.
\end{proof}

\noindent
We conclude Section \ref{section if top intense} by giving the proof of Proposition \ref{proposition infinite analytic}. Assume that $p>3$. Then $G$ is $p$-adic analytic, by Lemma \ref{finite rank}, and it has dimension $3$ thanks to Lemma \ref{dimension 3}. Moreover, $G$ is just-infinite by Lemma \ref{just-infinite}. The proof of Proposition \ref{proposition infinite analytic} is now complete.

\section{Two infinite groups}\label{section two gps}

In Section \ref{section two gps} we present two infinite pro-$p$-groups, which are $p$-adic analytic. We will see, in Section \ref{section na2}, the role they play in the proof of Theorem \ref{theorem unique infinite}.

\subsection{The first group}\label{section sl2}

\noindent
Let $p>3$ be a prime number and let $\pi:\SL_2(\Z_p)\rightarrow \SL_2(\F_p)$ be the canonical projection.
Let $\SL_2^\triangle(\F_p)$ denote the subgroup of $\SL_2(\F_p)$ consisting of those elements of the form 
$$
\begin{pmatrix}
1 & x \\
0 & 1 \\
\end{pmatrix}
\ \  \text{where} \ \ x\in\F_p.
$$
We define $\SL_2^{\triangle}(\Z_p)=\pi^{-1}(\SL_2^\triangle(\F_p))$ and remark that $\SL_2^\triangle(\Z_p)$ is a pro-$p$-group. Our notation is consistent with that of \cite{small}; however, we will make use of several facts coming from \cite[Ch.~$\mathrm{III}.17$]{huppert}, where the group $\SL_2^\triangle(\Z_p)$ is denoted by $\mathfrak{M}_{0,1,1}$.

\begin{lemma}\label{sl2 quotients obelisks}
Let $p>3$ be a prime number and let $G=\SL_2^\triangle(\Z_p)$. 
Denote by $(G_i)_{i\geq 1}$ the lower central series of $G$.
Then, for each $k\in\Z_{\geq 3}$, the quotient $G/G_k$ is a $p$-obelisk.
\end{lemma}

\begin{proof}
This is a reformulation of Satz $17.8$ from \cite[Ch.~$\mathrm{III}$]{huppert}.
\end{proof}

\begin{lemma}\label{aiutino}
Let $p>3$ be a prime number and let $G=\SL_2^\triangle(\Z_p)$. 
Denote by $(G_i)_{i\geq 1}$ the lower central series of $G$.
Then there exist $x\in G\setminus G_2$ and $a\in G_2\setminus G_3$ such that $[x,a]\in \cl(\gen{x})$.
\end{lemma}

\begin{proof}
This proof relies on several lemmas from \cite[Ch.~$\mathrm{III}.17$]{huppert}; we will respect Huppert's notation.
Let 
$$
x=B(1)=
\begin{pmatrix}
1 & 1 \\
0 & 1 
\end{pmatrix}
\ \ \text{and} \ \ 
a=D(1+p)=
\begin{pmatrix}
(1+p)^{-1} & 0 \\
0 & (1+p)
\end{pmatrix}.
$$
Satz $17.4$ from \cite[Ch.~$\mathrm{III}.17$]{huppert} gives a concrete characterization of the lower central series of $G$, from which it directly follows that $x\in G\setminus G_2$ and $a\in G_2\setminus G_3$.
As a consequence of Hilfssatz $17.2$(a), the element $x$ generates topologically the subgroup 
$\mathfrak{B}_0$ consisting of all matrices of the form 
$$
\begin{pmatrix}
1 & x \\
0 & 1 \\
\end{pmatrix}
\ \  \text{where} \ \ x\in\Z_p.
$$
To conclude, Hilfssatz $17.3$ guarantees that there exists an element $b\in\Z_p$ such that 
$$
[x,a]=
\begin{pmatrix}
1 & pb \\
0 & 1 \\
\end{pmatrix}
$$
so, in particular, $[x,a]$ belongs to the subgroup $\mathfrak{B}_0=\cl(\gen{x})$.
\end{proof}

\noindent
We recall that a $p$-obelisk is a non-abelian finite $p$-group $G$ satisfying $G_3=G^p$ and $|G:G_3|=p^3$. A $p$-obelisk $G$ is framed if, given any maximal subgroup $M$ of $G$, one has $\Phi(M)=G_3$. For more information about $p$-obelisks, we refer to Chapter \ref{chapter obelisks}.

\begin{lemma}\label{fuori bello!}
The group $\SL_2^\triangle(\Z_p)$ has a discrete quotient of class $6$ that is not a framed $p$-obelisk.
\end{lemma}

\begin{proof}
Let $G=\SL_2^\triangle(\Z_p)$ and denote by $(G_i)_{i\geq 1}$ the lower central series of $G$, as defined in Section \ref{section background}.
We define $\overline{G}=G/G_7$ and we use the bar notation for subgroups and elements of $\overline{G}$. 
The group $\overline{G}$ has class $6$ and it is a $p$-obelisk, by Lemma \ref{sl2 quotients obelisks}. 
Let now $x\in G$ be as in Lemma \ref{aiutino} and set $\ell=\gen{\overline{xG_2}}$, which is a $1$-dimensional subspace of
$\overline{G}/\overline{G}_2$. Let moreover
$$\rho_1^1:\overline{G}/\overline{G}_2\rightarrow\overline{G_3}/\overline{G}_4$$
and
$$\gamma_{1,2}:\overline{G}/\overline{G}_2\times \overline{G_2}/\overline{G_3}\rightarrow \overline{G}_3/\overline{G}_4$$
denote the maps from Lemma \ref{maps incrociate}.
As a consequence of Lemma \ref{aiutino}, the elments $\rho_1^1(\ell)$ and $\gamma_{1,2}(\graffe{\ell}\times\overline{G}_2/\overline{G}_3)$ generate a $1$-dimensional subspace $\ell'$ of $\overline{G}_3/\overline{G}_4$. 
By Lemma \ref{parity obelisks}($1$), the width $\wt_{\overline{G}}(3)$ is equal to $2$ so 
$\ell'$ is different from $\overline{G}_3/\overline{G}_4$.
Proposition \ref{lines} yields that $\overline{G}$ is not framed.
\end{proof}

\subsection{The second group}\label{section sl1}

\noindent
Let $p>3$ be a prime number and let $t\in\Z_p$ be a quadratic non-residue modulo $p$. 
Define $\Delta_p$ to be $\big(\frac{t\,,\, p}{\Z_p}\big)$, i.e., the quaternion algebra 
$$\Delta_p=\Z_p\oplus\Z_p\mathrm{i}\oplus\Z_p\mathrm{j}\oplus\Z_p\mathrm{k}$$
with defining relations 
$$\mathrm{i}^2=t, \, \mathrm{j}^2=p,\ \text{and} \ \mathrm{k}=\mathrm{ij}=-\mathrm{ji}.$$

\noindent
The quaternion algebra $\Delta_p$ is equipped with a bar map, defined by 
$$x=a+b\mathrm{i}+c\mathrm{j}+d\mathrm{k} \ \mapsto \
\overline{x}=a-b\mathrm{i}-c\mathrm{j}-d\mathrm{k},$$ 
which is an anti-homomorphism of order $2$.
The algebra $\Delta_p$ has, in addition, a unique maximal (left/right/two-sided) ideal $\mathfrak{m}$, which is principal generated by $\mathrm{j}$, i.e.
$\mathfrak{m}=\Delta_p\,\mathrm{j}$. It follows that an element $x=a+b\mathrm{i}+c\mathrm{j}+d\mathrm{k}$ belongs to $\mathfrak{m}$ if and only if both $a$ and $b$ belong to $p\Z_p$. Moreover, for each $k\in\Z_{\geq 1}$, the ideal $\mathfrak{m}^k$ is principal generated by $\mathrm{j}^k$ and therefore, for each $s\in\Z_{\geq 0}$, one has 
\[\mathfrak{m}^{2s}=p^s\Delta_p \ \ \text{and} \ \ \mathfrak{m}^{2s+1}=p^s\mathfrak{m}. \]
As a result, for each $k\in\Z_{\geq 1}$, the quotient $\mathfrak{m}^k/\mathfrak{m}^{k+1}$ is a vector space over $\F_p$ of dimension $2$. Now, for each $k\in\Z_{\geq 1}$, the set $1+\mathfrak{m}^k$ is easily seen to be a subgroup of $\Delta_p^*$ and the natural map 
$(1+\mathfrak{m}^k)/(1+\mathfrak{m}^{k+1})\rightarrow \mathfrak{m}^k/\mathfrak{m}^{k+1}$  is an isomorphism of groups. It follows that $1+\mathfrak{m}$ is a pro-$p$-subgroup of $\Delta_p^*$.
Define $$\Sn(\Delta_p)=(1+\mathfrak{m})\cap\graffe{x\in\Delta_p : \overline{x}=x^{-1}}.$$ 
Then $\Sn(\Delta_p)$ is a closed subgroup of $1+\mathfrak{m}$ and thus a pro-$p$-group itself. We have here lightened the notation from \cite{small}, where the group $\Sn(\Delta_p)$ is denoted by $\SL_1^1(\Delta_p)$.

\begin{lemma}\label{lower sl1}
Let $p>3$ be a prime number and let $G=\Sn(\Delta_p)$.
Denote by $(G_i)_{i\geq 1}$ the lower central series of $G$. 
Then, for each $k\in\Z_{\geq 1}$, one has $G_k=(1+\mathfrak{m}^k)\cap G$.
\end{lemma}

\begin{proof}
We sketch here the proof, but leave out the computations. 
For all $i\in\Z_{\geq 1}$, denote $M_i=(1+\mathfrak{m}^i)\cap G$. 
We remark that all $M_i$ are normal in $G$ and they form a base of open neighbourhoods of $1$ in $G$.
It is easy to check that $(M_i)_{i\geq 1}$ is a central series, in other words for all $i\in\Z_{\geq 1}$ the subgroup 
$[M_1,M_{i}]$ is contained in $M_{i+1}$. 
As a consequence of Lemma \ref{tgt}, for each index $i$, the commutator map induces a bilinear map 
$\gamma_i:M_1/M_2\times M_i/M_{i+1} \rightarrow M_{i+1}/M_{i+2}$. Next, by direct computation, one gets that, for every $i\in\Z_{\geq 1}$, the image of $\gamma_i$ generates $M_i/M_{i+1}$, and therefore $M_{i+1}=[M_1,M_{i}]M_{i+2}$.   
Fix $i$. By induction one shows that, for each positive integer $n$, one has $M_{i+1}=[M_1,M_i]M_{i+n}$, and hence
\[M_{i+1}=\bigcap_{n\geq 1}[M_1,M_i]M_{i+n}=\cl([M_1,M_{i}]).\]
Since $M_1=G$, we get that $M_{i+1}=\cl([G,G_i])=G_{i+1}$. The choice of $i$ being arbitrary, the proof is complete.
\end{proof}

\begin{lemma}\label{metà framed potenze}
Let $p>3$ be a prime number and let $G=\Sn(\Delta_p)$.
Denote by $(G_i)_{i\geq 1}$ the lower central series of $G$. 
Then, for each $i\in\Z_{\geq 1}$, the map ${x\mapsto x^p}$ on $G$ induces an isomorphism 
$\rho_i:G_i/G_{i+1}\rightarrow G_{i+2}/G_{i+3}$. 
\end{lemma}

\begin{proof}
By Lemma \ref{lower sl1}, given any positive integer $i$, one has $G_i=(1+\mathfrak{m}^i)\cap G$.
Fix $i\in\Z_{\geq 1}$ and let $1+x$ be an element of $G_i$. One shows that $(1+x)^p\equiv 1+px\bmod G_{i+3}$. It is now easy to conclude.
\end{proof}

\begin{lemma}\label{metà framed}
Let $p>3$ be a prime number and let $G=\Sn(\Delta_p)$.
Denote by $(G_i)_{i\geq 1}$ the lower central series of $G$. Let $x\in G\setminus G_2$ and let $y\in G_2\setminus G_3$. 
Then $G_3$ is generated by $x^p$ and $[x,y]$ modulo $G_4$.
\end{lemma}

\begin{proof}
Straightforward computation.
\end{proof}

\noindent
We remind the reader that, as defined in Chapter \ref{chapter obelisks}, a $p$-obelisk is a finite non-abelian $p$-group $G$ such that $|G:G_3|=p^3$ and $G^p=G_3$. A $p$-obelisk is said to be framed if, for each maximal subgroup $M$ of $G$, one has $\Phi(M)=G_3$.

\begin{lemma}\label{sl1 quotients framed}
Let $p>3$ be a prime number and let $G=\Sn(\Delta_p)$.
Denote by $(G_i)_{i\geq 1}$ the lower central series of $G$. 
Then, for each $k\in\Z_{\geq 3}$, the quotient $G/G_k$ is a framed $p$-obelisk.
\end{lemma}

\begin{proof}
Let $k\in\Z_{\geq 3}$ and denote $\overline{G}=G/G_k$. The group $\overline{G}$ is non-abelian and it is finite. Moreover, as a consequence of Lemma \ref{lower sl1}, one can easily compute that $|\overline{G}:\overline{G}_3|=|G:G_3|=p^3$ and, thanks to Lemma \ref{metà framed potenze}, one has $\overline{G}^p=\overline{G}_3$. It follows that $\overline{G}$ is a $p$-obelisk.
To show that $\overline{G}$ is framed, combine Lemma \ref{metà framed} and Proposition \ref{lines}.
\end{proof}

\begin{lemma}\label{fischiettando}
Let $p>3$ be a prime number and let $G=\Sn(\Delta_p)$.
Let moreover $\alpha:G\rightarrow G$ be defined by
$$a+b\mathrm{i}+c\mathrm{j}+d\mathrm{k} \ \mapsto \
a+b\mathrm{i}-c\mathrm{j}-d\mathrm{k}.$$  
Then $\alpha$ is a continuous automorphism of $G$ and the map $G/G_2\rightarrow G/G_2$ that is induced by $\alpha$ is equal to the inversion map $a\mapsto a^{-1}$.
\end{lemma}

\begin{proof}
The map $\alpha$ coincides with conjugation by $\mathrm{i}$ and it is therefore a continuous automorphism. 
Moreover, thanks to Lemma \ref{lower sl1}, the subgroup $G_2$ coincides with $(1+\mathfrak{m}^2)\cap G$.
Since each element $x$ of $G$ 
can be written in the form $x=1+c\mathrm{j}+d\mathrm{k}+m$, with $c,d\in\Z_p$ and $m\in \mathfrak{m}^2$, we get that 
$\alpha(x)\equiv \overline{x}\bmod G_2$. The elements $\overline{x}$ and $x^{-1}$ being equal, it follows that 
$\alpha(x)\equiv x^{-1}\bmod G_2$.
\end{proof}

\begin{lemma}\label{sl1 has inte 2}
Let $p>3$ be a prime number and let $G=\Sn(\Delta_p)$.
Define moreover $\alpha:G\rightarrow G$ by
$$a+b\mathrm{i}+c\mathrm{j}+d\mathrm{k} \ \mapsto \
a+b\mathrm{i}-c\mathrm{j}-d\mathrm{k}.$$ 
Then $\alpha$ is a topologically intense automorphism of $G$ of order $2$ and $\inte(G)=2$.
\end{lemma}

\begin{proof}
By Lemma \ref{fischiettando}, the map $\alpha$ is a continuous automorphism of $G$ and, by its definition, it clearly has order $2$.
We prove that $\alpha$ is topologically intense.
To this end, let $H$ be an open subgroup of $G$. As a consequence of Lemma \ref{lcs open}, there exists a positive integer $k$ such that $G_k$ is contained in $H$. Fix such integer $k$ and define $K=\max\graffe{k,4}$. Denote $\overline{G}=G/G_K$ and use the bar notation for the subgroups of $\overline{G}$. Denote moreover by $\alpha_K$ the automorphism that is induced on $\overline{G}$ by $\alpha$. Then $\alpha_K$ induces the inversion map on $\overline{G}/\overline{G_2}$, as a consequence of Lemma \ref{fischiettando} and the definition of $\alpha_K$. 
Moreover, the class of $\overline{G}$ is at least $3$ so, thanks to Lemma \ref{sl1 quotients framed}, the group $\overline{G}$ is a framed obelisk. It follows from Theorem \ref{theorem complete char pt.1} that $\alpha_K$ is intense, so there exists $g\in G$ such that $\alpha_K(\overline{H})=\overline{gHg^{-1}}$. 
Furthermore, we have that $\alpha(H)=gHg^{-1}$ and, the choice of $H$ being arbitrary, it follows from Proposition \ref{openclosedtuttouguale} that $\alpha$ is topologically intense. In particular, $\inte(G)$ is even.
The intensity of $G$ is equal to $2$, as a consequence of Proposition \ref{inte profinite} and Theorem \ref{theorem class at least 3}.
\end{proof}

\section{Non-abelian groups, part II}\label{section na2}

The aim of this section is to give a proof of the following proposition.

\begin{proposition}\label{proposition unique infinite}
Let $p>3$ be a prime number and let $P$ be a non-abelian infinite 
pro-$p$-group. Assume that $\inte(P)>1$. 
Then $P$ is topologically isomorphic to $\Sn(\Delta_p)$.  
\end{proposition}

\noindent
Until the end of Section \ref{section na2}, let the following assumptions be valid.
Let $p>3$ be a prime number and let $P$ be an infinite non-abelian pro-$p$-group of intensity greater than $1$. 
Let $(P_i)_{i\geq 1}$ denote the lower central series of $P$, as defined in Section \ref{section background}, and let $\alpha$ be a topologically intense automorphism of $P$ of order $2$, which exists thanks to Proposition \ref{inte profinite}.
In the proof of Proposition \ref{proposition unique infinite}, we will make heavy use of results coming from Chapters \ref{chapter obelisks} and \ref{chapter complete char}.

\begin{definition}
Let $G$ be a group. The \index{derived series} \emph{derived series} $(G^{(i)})_{i\geq 0}$ of $G$ is defined recursively by 
$$G^{(0)}=G \ \ \text{and} \ \ G^{(i+1)}=[G^{(i)},G^{(i)}].$$
The group $G$ is \index{solvable group} \emph{solvable} if there exists $k\in\Z_{\geq 0}$ such that $G^{(k)}=\graffe{1}$.
\end{definition}

\begin{lemma}\label{jon solvable torsion}
Every solvable just-infinite pro-$p$-group other than $\Z_p$ has torsion.
\end{lemma}

\begin{proof}
This is Proposition $6.1$ in \cite{small}.
\end{proof}

\noindent
We remind the reader that, for each prime number $p>3$, the groups $\SL_2^\triangle(\Z_p)$ and $\Sn(\Delta_p)$ have been defined in Section \ref{section two gps}.

\begin{lemma}\label{dim3jon}
Let $p>3$ be a prime number and let $G$ be a $p$-adic analytic group.
Assume that $\dim(G)=3$ and that $G$ is both torsion-free and non-solvable. 
Then $G$ is topologically isomorphic to an open subgroup of $\Sn(\Delta_p)$ or $\SL_2^\triangle(\Z_p)$.  
\end{lemma}

\begin{proof}
See \cite[Section $7.3$]{small}.
\end{proof}

\begin{lemma}\label{open in one of the two}
The group $P$ is topologically isomorphic to an open subgroup of $\Sn(\Delta_p)$ or $\SL_2^\triangle(\Z_p)$.  
\end{lemma}

\begin{proof}
The group $P$ is a just-infinite $p$-adic analytic group of dimension $3$, by Proposition \ref{proposition infinite analytic}. By Lemma \ref{torsion-free}, the torsion of $P$ is trivial and so, by Lemma \ref{jon solvable torsion}, the group $P$ is not solvable. 
It follows from Lemma  \ref{dim3jon} that $P$ is isomorphic to an open subgroup of 
$\Sn(\Delta_p)$ or $\SL_2^\triangle(\Z_p)$.
\end{proof}

\begin{lemma}\label{only two left}
The group $P$ is topologically isomorphic to $\Sn(\Delta_p)$ or $\SL_2^\triangle(\Z_p)$.  
\end{lemma}

\begin{proof}
Let $G\in\graffe{\Sn(\Delta_p),\SL_2^\triangle(\Z_p)}$ and let $(G_i)_{i\geq 1}$ denote the lower central series of $G$.
From the combination of Lemmas \ref{sl2 quotients obelisks} and \ref{sl1 quotients framed}, we know that, for each $k\geq 3$, the quotient $G/G_k$ is a $p$-obelisk.
Let now $H$ be an open subgroup of $G$, such that $P$ is topologically isomorphic to $H$. The existence of $H$ is ensured by Lemma \ref{open in one of the two}.
By Lemma \ref{quotients of any class}, the group $H$ has discrete quotients of any class and, thanks to Lemma \ref{quotient obelisk}, each such quotient, of class at least $4$, is a $p$-obelisk. 
The subgroup $H$ being open, it follows from Lemma \ref{lcs open} that there exists $k\in\Z_{>4}$ such that $G_k$ is contained in $H$ so $H/G_k$ is a 
$p$-obelisk. 
Proposition \ref{obelisk in obelisk} yields $G=H$.
\end{proof}

\begin{lemma}\label{quotients framed}
Each discrete quotient of $P$ of class at least $6$ is a framed $p$-obelisk.
\end{lemma}

\begin{proof}
Let $\overline{G}$ be a discrete quotient of $P$ of class at least $6$. By Lemma \ref{inte quozienti}, the map $\alpha$ induces an intense automorphism of order $2$ on $\overline{G}$ and, by Lemma \ref{quotient obelisk}, the group $\overline{G}$ is a $p$-obelisk. By Lemma \ref{parity obelisks}($1$), the number $\wt_{\overline{G}}(5)$ is equal to $2$ so, by Theorem \ref{theorem we need lines}, the $p$-obelisk $\overline{G}$ is framed. 
\end{proof}

\noindent
We are finally ready to give the proof of Proposition \ref{proposition unique infinite}. Thanks to Lemma \ref{only two left}, there are only two possibilities for the isomorphism type of $P$: that of $\Sn(\Delta_p)$ or that of $\SL_2^\triangle(\Z_p)$. By Lemma \ref{quotients framed}, every discrete quotient of $P$ of class $6$ is a framed $p$-obelisk so, in view of Lemma \ref{fuori bello!}, the group $\SL_2^\triangle(\Z_p)$ is not isomorphic to $P$. It follows that $P$ is topologically isomorphic to $\Sn(\Delta_p)$ and so the proof of Proposition \ref{proposition unique infinite} is complete.

\section{Proving the main theorems and more}\label{section unique infinite}

In Sections \ref{i} and \ref{ii} we prove respectively Theorem \ref{theorem unique infinite} and Theorem \ref{theorem any class}. The last two theorems are the most important results of Chapter \ref{chapter bilbao}: we are able to draw a bridge between the two thanks to Proposition \ref{proposition function}, which is proven in Section \ref{iii}.

\subsection{The proof of Theorem \ref{theorem unique infinite}}\label{i}

Let $p$ be a prime number. As a consequence of Proposition \ref{inte profinite}, the intensity of a pro-$p$-group divides $p-1$ and so there are no pro-$2$-groups of intensity greater than $1$. Assume now that $p$ is odd. Then, thanks to Corollary \ref{pro-p-abelian}, each infinite abelian pro-$p$-group has intensity $p-1$, which, $p$ being odd, is greater than $1$. Let now $G$ be a non-abelian infinite pro-$p$-group with $\inte(G)>1$. Then $G$ has a discrete quotient of any class, thanks to Lemma \ref{quotients of any class}, so Theorem \ref{theorem 3-groups} yields that $p$ is larger than $3$.
By Proposition \ref{proposition unique infinite}, the group $G$ is topologically isomorphic to $\Sn(\Delta_p)$, which, by Lemma \ref{sl1 has inte 2}, has indeed intensity $2$. 
The proof of Theorem \ref{theorem unique infinite} is now complete.

\subsection{The proof of Theorem \ref{theorem any class}}\label{ii}

Let $p>3$ be a prime number and let $c$ be a positive integer. Write $P=\Sn(\Delta_p)$ and let $(P_i)_{i\geq 1}$ denote the lower central series of $P$, as defined in Section \ref{section background}. Then the group $P/P_{c+1}$ has class $c$ and it is finite, as a consequence of Lemma \ref{lcs open}. The group $P$ being a pro-$p$-group, $P/P_{c+1}$ is a finite $p$-group. Moreover, by Theorem \ref{theorem unique infinite}, the intensity of $P$ is greater than $1$ and so, thanks to Proposition \ref{inte profinite}, we get $\inte(P/P_{c+1})>1$. The number $c$ was chosen arbitrarily and therefore Theorem \ref{theorem any class} is proven.

\subsection{A bridge between finite and infinite}\label{iii}

The purpose of Section \ref{iii} is to compare, for a fixed prime $p>3$, the finite $p$-groups of intensity greater than $1$ with the discrete quotients of $\Sn(\Delta_p)$.

\begin{proposition}\label{proposition function}
Let $p>3$ be a prime number and write $P=\Sn(\Delta_p)$. Denote by $(P_i)_{i\geq 1}$ the lower central series of $P$.
Then there exists a function $f:\Z_{>0}\rightarrow\Z_{\geq 0}$ 
with the following properties. 
\begin{itemize}
 \item[$1$.] One has $\lim_{c\rightarrow\infty}f(c)=\infty$.
 \item[$2$.] For each finite $p$-group $G$ of class $c$ with $\inte(G)>1$, the quotients $G/G_{f(c)}$ and $P/P_{f(c)}$ are isomorphic.
\end{itemize}
\end{proposition}

\begin{proof}
For each positive integer $c$, let $\Int(p,c)$ denote the collection of all finite $p$-groups of class $c$ and intensity greater than $1$.
We define $f:\Z_{>0}\rightarrow\Z_{\geq 0}$ by mapping each element
$c\in\Z_{>0}$ to the maximum index $m\in\Z_{>0}$ for which, whenever $G\in\Int(p,c)$, 
the quotients $G/G_{m}$ and $P/P_{m}$ are isomorphic. 
The map $f$ is well-defined, thanks to Theorem \ref{theorem any class}, and it follows directly from the definition of $f$ that $(2)$ is satisfied. 
Moreover, thanks to Lemma \ref{intensity of quotients}, the function $f$ is non-decreasing. 
We prove $(1)$ by contradiction. Let $C\in\Z_{\geq 0}$ be such that, for all $c\geq C$, one has $f(c)=f(C)$. In other words, for each $c\in\Z_{\geq C}$, there exists $G\in\Int(p,c)$ such that $G/G_{f(C)}$ and $P/P_{f(C)}$ are isomorphic, but $G/G_{f(C)+1}$ and $P/P_{f(C)+1}$ are not. 
For all $c\geq C$, call $X_c$ the collection of such $G$ and note that, for each $c\geq C$, the set $X_c$ is non-empty. 
Thanks to Lemma \ref{intensity of quotients}, for each $c\in\Z_{>C}$, we have a natural map $X_{c+1}\rightarrow X_c$, which is defined by $G\mapsto G/G_{c+1}$. The collection $\graffe{X_c}_{c>C}$ is thus an inverse system of non-empty sets. As a consequence of Theorem \ref{theorem intensity class 3}, the constant $C$ is at least $3$ and so, for each $c>C$,
Theorem \ref{theorem dimension layers} yields that $X_c$ is finite. By Lemma \ref{inverse limit}, the set $X=\underset{c>C}{\varprojlim}X_c$ is non-empty and therefore there exists an infinite non-abelian pro-$p$-group of intensity larger than $1$ and which is, by construction, not isomorphic to $P$. Contradiction to Theorem \ref{theorem unique infinite}. It follows that $(2)$ is satisfied and the proof is complete.
\end{proof}

\noindent
Proposition \ref{proposition function} is the last result of this thesis and, in summary, it states that, for $p>3$, each finite $p$-group $G$ with $\inte(G)> 1$ shares a ``relatively big'' quotient (growing in size with the class of $G$) with the infinite group $\Sn(\Delta_p)$. One can then ask: if $p>3$ and $G$ is a finite $p$-group of intensity greater than $1$, then ``how far is $G$ from being a quotient of $\Sn(\Delta_p)$''? More precisely, if $G$ is a finite $p$-group of class $c$ with $\inte(G)>1$ and $f$ is as in Proposition \ref{proposition function}, then what is the average size of $G_{f(c)}$? Is there an absolute constant $B$ such that, for each $c\in\Z_{>0}$ and for each finite $p$-group $G$ of class $c$ and intensity greater than $1$, one has $|G_{f(c)}|\leq p^B$? 
In view of Theorem \ref{theorem we need lines}, we can surely answer this question if we manage to classify, for each given prime $p>3$, all framed $p$-obelisks that have an automorphism of order $2$ that induces the inversion map on the Frattini quotient of the group.


\bibliographystyle{alpha}


\vspace{-10pt}

\printindex

\newpage


\pagestyle{empty}

\chapter*{Summary}
\addcontentsline{toc}{chapter}{Summary}

\pagestyle{empty}
\vspace{-50pt}
{\LARGE \textbf{Intense automorphisms of finite groups}}
\vspace{30pt} \\
\noindent
Let $G$ be a group. An automorphism $\alpha$ of $G$ is \emph{intense} if, for every subgroup $H$ of $G$, there exists an element $x$ of $G$ such that $\alpha(H)=xHx^{-1}$. The collection of intense automorphisms of $G$ is a subgroup of $\Aut(G)$, which is denoted by $\Int(G)$.
\newline
\newline
In this thesis we classify the pairs $(p,G)$, where $p$ is a prime number and $G$ is a finite $p$-group but $\Int(G)$ is not. To this end, for each finite $p$-group $G$, we define the \emph{intensity} $\inte(G)$ of $G$ to be the index of any Sylow $p$-subgroup of $\Int(G)$ in $\Int(G)$. We prove that finite $2$-groups have intensity $1$. Next, we prove that, for every prime number $p$, each non-trivial finite abelian $p$-group has intensity $p-1$. We proceed with the classification by progressively increasing the nilpotency class of the groups we are looking at. 
Let $p$ be an odd prime number. We show that the finite $p$-groups of class $2$ and intensity greater than $1$ are exactly the \emph{extraspecial} $p$-groups of exponent $p$. We prove moreover that, if the class is $3$, then a finite $p$-group has intensity greater than $1$ if and only if its abelianization has order $p^2$. The classification process becomes more difficult as the class increases. We prove that there exists a unique $3$-group, up to isomorphism, of class at least $4$ and intensity greater than $1$; that group has order $729$. In contrast with the case of $3$-groups, we demonstrate that, if $p>3$, there exists, for each positive integer $c$, a $p$-group $G$ of class $c$ for which $\Int(G)$ is not itself a $p$-group. To do so, we extend the notion of intensity to pro-$p$-groups and, if $p>3$, we construct an infinite non-abelian pro-$p$-group of intensity greater than $1$. We later prove that the infinite group we constructed is the unique infinite non-abelian pro-$p$-group of intensity greater than $1$, up to isomorphism. 
In conclusion, for each prime number $p>3$, we define a new family of $2$-generated finite $p$-groups, which we call \emph{$p$-obelisks}, and we show that they have exceptionally pleasant properties. The classification of finite $p$-groups of intensity greater than $1$ is completed, modulo the existence of a special kind of automorphisms of $p$-obelisks. 

\chapter*{Samenvatting}
\pagestyle{empty}
\addcontentsline{toc}{chapter}{Samenvatting}
\pagestyle{empty}
\vspace{-50pt}
{\LARGE \textbf{Intense automorfismen van eindige groepen}}
\vspace{30pt} \\
\noindent
Zij $G$ een groep. Een automorfisme $\alpha$ van $G$ heet \emph{intens} als er voor elke ondergroep $H$ van $G$ een element $x$ in $G$ bestaat waarvoor geldt $\alpha(H)=xHx^{-1}$. De verzameling $\Int(G)$ van alle intense automorfismen van $G$ is een ondergroep van $\Aut(G)$.
\newline
\newline
In dit proefschrift classificeren we de paren $(p,G)$, met $p$ een priemgetal en $G$ een eindige $p$-groep waarvoor $\Int(G)$ geen $p$-groep is. Daartoe defini\"eren we voor elke $p$-groep $G$ de \emph{intensiteit} $\inte(G)$ van $G$ als de index van een willekeurige Sylow $p$-ondergroep van $\Int(G)$ in $\Int(G)$. We bewijzen dat eindige $2$-groepen intensiteit $1$ hebben. Vervolgens bewijzen we dat, als $p$ een priemgetal is, elke eindige non-triviale abelse $p$-groep intensiteit $p-1$ heeft. We vervolgen de classificatie door de nilpotentie klasse van de groepen die we bekijken te laten oplopen. 
Zij $p$ een oneven priemgetal. We laten zien dat de $p$-groepen van klasse $2$ en intensiteit groter dan $1$ precies de \emph{extraspeciale} $p$-groepen met exponent $p$ zijn. We bewijzen bovendien dat, als de klasse $3$ is, een eindige $p$-groep intensiteit groter dan $1$ heeft dan en slechts dan als zijn verabelisering orde $p^2$ heeft. De classificatie wordt moeilijker naarmate de klasse groeit. 
We bewijzen dat er op isomorfie na een unieke $3$-groep van klasse groter dan $4$ en intensiteit groter dan $1$ bestaat; deze groep heeft orde $729$. In tegenstelling tot het geval van $3$-groepen tonen we aan dat er, voor $p>3$ en $c$ een positief getal, een $p$-groep $G$ van klasse $c$ bestaat waarvoor $\Int(G)$ geen $p$-groep is. Hiertoe breiden we het begrip intensiteit uit naar pro-$p$-groepen en voor $p>3$ construeren we een oneindige niet-abelse pro-$p$-groep van intensiteit groter dan $1$. Later bewijzen we dat de oneindige groep die we hebben geconstrueerd op isomorfie na de unieke oneindige non-abelse pro-$p$-groep van intensiteit groter dan $1$ is.
Ten slotte defini\"eren we voor elk priemgetal $p>3$ een nieuwe familie van $2$-voortgebrachte eindige $p$-groepen, die we \emph{$p$-obelisken} noemen, en we laten zien dat ze bijzonder aangename eigenschappen hebben. De classificatie van eindige $p$-groepen van intensiteit groter dan $1$ is voltooid, op het bestaan van een speciaal soort automorfismen van $p$-obelisken na.

\chapter*{Acknowledgements}
\addcontentsline{toc}{chapter}{Acknowledgements}
I would like to thank: 
\begin{itemize}
\item[] my supervisor, Hendrik Lenstra, for his guidance and constant help. I am grateful for the beautiful mathematics we have done together and for everything he has taught me in the last years.
\item[] the members of the Doctorate Committee, for the time they spent reading my thesis. 
\item[] Andrea Lucchini, for his help through the years, both mathematical and otherwise. 
\item[] Ellen Henke, for her accurate comments.
\item[] Bart de Smit, for helping me with numerous issues; so many that I lost count. 
\item[] Jon Gonz\`alez S\`anchez, for sharing his knowledge and precious ideas. 
\item[] Martin Bright, for helping me with elliptic curves and teaching issues.
\item[] Carlo Pagano, for all the maths we discussed together.
\item[] the algebra, geometry, and number theory group in Leiden; and all its evolutions over the last five years. I couldn't have asked for a better environment to grown in, as a mathematician and as a person. A big thanks to Ronald.
\item[] my friends in Leiden, back home, and all around the globe, for their support and for all the fun times that coloured my PhD adventure. 
\item[] my family. 
\end{itemize}

\chapter*{Curriculum Vitae}
\addcontentsline{toc}{chapter}{Curriculum Vitae}
Mima Stanojkovski was born on August $21$, $1989$, in Sarajevo (YU). In $1992$, her family moved to Fiera di Primiero (TN), Italy, where she entered elementary school. After the primary studies, she moved, together with her family, to Feltre (BL), Italy. There, she completed her pre-university education. From $2000$ to $2003$, she studied at Scuola Media G.\ Rocca and, from $2003$ to $2008$, she was a student at Liceo Scientifico G.\ Dal Piaz. 
\newline
\newline
After finishing high school, Mima started her bachelor studies in mathematics at Universit\`{a} degli Studi di Trento, Italy. She graduated in $2011$, and wrote a thesis titled \emph{``Controesempi sui compatti''}, supervised by Prof.\ Giuseppe Vigna Suria.
In $2011$, she joined the ALGANT Master program. She spent, between $2011$ and $2013$, one year at Universit\`{a} degli Studi di Padova and one year at Universiteit Leiden. She obtained her master's degree, from both Padova and Leiden, in $2013$. Her master thesis, entitled \emph{``Evolving groups''}, was supervised by Prof.\ Hendrik Lenstra from Universiteit Leiden. 
\newline
\newline
In $2013$ she started her PhD at Universiteit Leiden, under the supervision of Prof.\ Hendrik Lenstra. In relation to the PhD project, in $2014$, she visited both Dr.\ Jon Gonz\'{a}lez 
S\'{a}nchez and Prof.\ Andrea Lucchini, respectively in Bilbao and Padova.
Mima expects to spend the next two academic years at Universit\"{a}t Bielefeld, under the mentorship of Prof.\ Christopher Voll.
\newpage
~
\clearpage

\end{document}